\documentclass[11pt]{amsart}

\usepackage{amssymb,amsfonts}
\usepackage{hyperref}
\usepackage{mdwlist}

\usepackage{amssymb,textcomp}
\usepackage{enumerate}

\usepackage[utf8]{inputenc}
\usepackage[T1]{fontenc}

\usepackage[mathscr]{eucal}

\usepackage{hyperref}
 \hypersetup{
     colorlinks=true,
     linktocpage=true,
     linkcolor=red,
     filecolor=blue,
     citecolor = blue,
     urlcolor=cyan,
     }

\usepackage[a4paper, twoside=false, vmargin={2cm,3cm}, includehead]{geometry}

\usepackage{comment}
\newtheorem{lemma}{Lemma}[section]
\newtheorem{theorem}[lemma]{Theorem}
\newtheorem*{theorem*}{Theorem}

\newtheorem{corollary}[lemma]{Corollary}

\newtheorem{proposition}[lemma]{Proposition}
\newtheorem*{proposition*}{Proposition}

\newtheorem{problem}{Problem}
\newtheorem{heuristic}[lemma]{Heuristic}

\newtheorem*{problem*}{Problem}

\theoremstyle{definition}
\newtheorem*{claim*}{Claim}
\newtheorem{convention}{Convention}

\newtheorem{definition}[lemma]{Definition}

\newtheorem{example}[lemma]{Example}

\let\ringaccent\r
\def\aa{\ringaccent a}
\def\aa{\ringaccent a}

\DeclareMathOperator*{\E}{\mathbb{E}}

\newcommand{\C}{{\mathbb C}}

\newcommand{\F}{{\mathbb F}}
\newcommand{\N}{{\mathbb N}}
\renewcommand{\P}{{\mathbb P}}
\newcommand{\Q}{{\mathbb Q}}
\newcommand{\R}{{\mathbb R}}
\renewcommand{\S}{\mathbb{S}}
\newcommand{\T}{{\mathbb T}}
\newcommand{\Z}{{\mathbb Z}}

\newcommand{\CA}{{\mathcal A}}

\newcommand{\CC}{{\mathcal C}}
\newcommand{\CD}{{\mathcal D}}
\newcommand{\CE}{{\mathcal E}}

\newcommand{\CH}{{\mathcal H}}
\newcommand{\CI}{{\mathcal I}}
\newcommand{\CK}{{\mathcal K}}
\newcommand{\CL}{{\mathcal L}}

\newcommand{\CP}{{\mathcal P}}

\newcommand{\CR}{{\mathcal R}}

\newcommand{\CT}{{\mathcal T}}
\newcommand{\CX}{{\mathcal X}}
\newcommand{\CY}{{\mathcal Y}}

\newcommand{\CZ}{{\mathcal Z}}

\newcommand{\ba}{\mathbf{a}}
\renewcommand{\b}{{\mathbf{b}}}
\newcommand{\bc}{{\mathbf{c}}}

\newcommand{\bn}{{\mathbf{n}}}
\newcommand{\be}{{\mathbf{e}}}

\newcommand{\bh}{{\mathbf{h}}}

\renewcommand{\r}{{\mathbf{r}}}

\newcommand{\bv}{{\mathbf{v}}}

\newcommand{\veps}{\varepsilon}

\newcommand{\eps}{\epsilon}
\newcommand{\ueps}{{\underline{\epsilon}}}

\newcommand{\Id}{\textrm{Id}}

\newcommand{\norm}[1]{\left\Vert #1\right\Vert}
\newcommand{\nnorm}[1]{\lvert\!|\!| #1|\!|\!\rvert}

\newcommand{\inv}{^{-1}}

\DeclareMathOperator{\spec}{Spec}
\DeclareMathOperator{\Spec}{Spec}

\DeclareMathOperator{\reel}{Re}
\renewcommand{\Re}{\reel}

\newcommand{\Krat}{{\CK_{\text{\rm rat}}}}

\newcommand{\abs}[1]{\mathopen{}\left| #1\mathclose{}\right|}

\newcommand{\brac}[1]{\mathopen{}\left( #1 \mathclose{}\right)}

\newcommand{\floor}[1]{{\left \lfloor #1 \right \rfloor}}
\newcommand{\sfloor}[1]{{\lfloor #1 \rfloor}}

\newcommand{\ceil}[1]{{\left \lceil #1 \right \rceil}}
\newcommand{\rem}[1]{\left \{ #1 \right \}}

\newcommand{\srem}[1]{ \{ #1 \}}

\usepackage[dvipsnames]{xcolor}

\title[]{Joint ergodicity - 40 years on}
\author{Borys Kuca}

\address[Borys Kuca]{Jagiellonian University, Faculty of Mathematics and Computer Science, 30-348 Krak\'ow, Poland}
\email{borys.kuca@uj.edu.pl}
\date{}

\thanks{The author was supported by the National Science Center (NCN Poland) Sonata grant No. 2024/55/D/ST1/00468 and would also like to acknowledge the Foundation of Polish Science (FNP) Start stipend.}

\begin{document}

\begin{abstract}
Recent years have seen dramatic progress in the study of joint ergodicity, i.e. a scenario in which a multiple ergodic average converges in norm to the product of integrals of individual functions. This survey, accompanying the talk given by the author in the Perspectives on Ergodic Theory and its Interactions conference to celebrate Vitaly Bergelson's 75th birthday, aims to summarize these recent advances, outline crucial new tools, present various open problems, and highlight the main challenges currently faced in the study of multiple ergodic averages.
\end{abstract}

\maketitle

\tableofcontents

\section{Introduction}\label{S: introduction}
\subsection{Multiple ergodic averages: what are they, and why should you care?}

Furstenberg's celebrated proof of the Szemer\'edi theorem \cite{Fu77} initiated the systematic investigation of \emph{multiple ergodic averages}, a class of multilinear operators of dynamical origin that arise in the search for arithmetic patterns in sets of positive density and other ``large'' sets. Given some integer sequences $a_1, \ldots, a_\ell:\N\to\Z$, % (polynomials, primes, Hardy sequences, etc.), 
we are interested in the (norm, weak, or pointwise) limit of
\begin{align*}
A_N(f_1, \ldots, f_\ell):=\E_{n\in[N]} T_1^{a_1(n)}f_1\cdots T_\ell^{a_\ell(n)}f_\ell
\end{align*}
as $N\to\infty$, where:
\begin{itemize}
\item $\E_{n\in[N]} = \frac{1}{N}\sum_{n=1}^N$ is a finite C\`esaro average;
\item $(X, \CX, \mu,$\! $T_1, \ldots, T_\ell)$ is a \emph{system}, i.e. a collection of measure-preserving transformations  $T_1, \ldots, T_\ell$ acting on a standard probability space $(X, \CX, \mu)$ (which throughout this survey we assume to be \emph{invertible} and \emph{commuting} unless stated otherwise);
%\item $T_1, \ldots, T_\ell$ are some invertible (and usually commuting) measure-preserving transformations on a standard probability space $(X, \CX, \mu)$;
\item $f_1, \ldots, f_\ell\in L^\infty(\mu)$ are bounded functions.
\end{itemize}
%Here and throughout, we define $(Tf)(x) = f(Tx)$ as well as $T^n f = T$.

Multiple ergodic averages have been investigated for many classes of integer sequences, including polynomials, generalized polynomials, Hardy sequences, tempered functions, sequences involving primes, and random sequences. %This survey focuses on polynomials and Hardy sequences: two classes of iterates for which the greatest progress has been made in the previous few years. 
%In general, we first fix the integer sequences or their family (e,g. integer polynomials), and 

% $[N]:=\{1, \ldots, N\}$; $\E_{i\in I} = \frac{1}{|I|}\sum_{i\in I}$ for any finite set $I$;  $T_1, \ldots, T_\ell$ are some invertible (and usually commuting) measure-preserving transformations on a standard probability space $(X, \CX, \mu)$; and $f_1, \ldots, f_\ell\in L^\infty(\mu)$ are bounded functions.
%, and $a_1, \ldots, a_\ell:\N\to\Z$ are integer sequences (polynomials, primes, Hardy sequences, etc.).
Broadly speaking, the 
%questions surrounding 
study of multiple ergodic averages can be split into the following classes of problems:
\begin{enumerate}
\item (Multiple recurrence) Show that for any $E\in\CX$ with $\mu(E)>0$, there exists $n\in\N$ for which
\begin{align*}
\mu(E\cap T_1^{-a_1(n)}E\cap\cdots\cap T_\ell^{-a_\ell(n)}E) >0.
\end{align*}
\item (Convergence) Show that $A_N(f_1, \ldots, f_\ell)$ converges as $N\to\infty$, whether in $L^2(\mu)$, weakly, or pointwise almost everywhere.
\item (Structure) Describe the limit of $A_N(f_1, \ldots, f_\ell)$ as $N\to\infty$.
\end{enumerate}
While the multiple recurrence problem is strictly speaking not about multiple ergodic averages, the ergodic way to prove it is by showing that the average
\begin{align*}
\E_{n\in[N]} \mu(E\cap T_1^{-a_1(n)}E\cap\cdots\cap T_\ell^{-a_\ell(n)}E) = \int 1_E \cdot \E_{n\in[N]} T_1^{a_1(n)}1_E \cdots T_\ell^{a_\ell(n)}1_E\; d\mu
\end{align*}
has positive liminf as $N\to\infty$. Such results have yielded new extensions of Szemer\'edi's theorem when combined with the Furstenberg correspondence principle, a variant of which is presented below. This transference principle serves as the main gateway between ergodic theory and additive combinatorics.
\begin{theorem}[Furstenberg's correspondence principle (cf. {\cite[Theorem 1.1]{Ber87b}})]
   Let $k,\ell\in\N$. For every $A\subseteq\Z^\ell$ there exists a system $(X, \CX, \mu,$\! $T_1, \ldots, T_\ell)$ and a set $E\in\CX$ with the following properties:
   \begin{enumerate}
       \item $\overline{d}(A) = \mu(E)$;
       \item for all $\bn_1, \ldots, \bn_k\in\Z^\ell$ given by $\bn_i = (n_{i1}, \ldots, n_{i\ell})\in\Z^\ell$, we have
       \begin{align*}
           \overline{d}\brac{A\cap (A-\bn_1)\cap \cdots \cap (A-\bn_k)}\geq \mu(E\cap \prod_{j=1}^\ell T_j^{-n_{1j}}E\cap \cdots \cap \prod_{j=1}^\ell T_j^{-n_{kj}}E).
       \end{align*}
   \end{enumerate}
\end{theorem}

Our understanding of the aforementioned problems is sharpest for integer polynomials, for which the state of the art can be summarized as follows. Modulo some notable exceptions (see Problem \ref{Pr: joint intersective}), the questions regarding multiple recurrence are largely resolved thanks to the works of Furstenberg \cite{Fu77}, Furstenberg-Katznelson \cite{FuKa78}, Bergelson-Leibman \cite{BL96}, Leibman \cite{L98}, Bergelson-Leibman-Lesigne \cite{BLL08}, and others.
%, at least when $a_1, \ldots, a_\ell$ are integer polynomials and/or sequences involving primes. 
The same holds for norm convergence, due to the remarkable result of Walsh \cite{W12} vastly generalizing earlier works of Conze-Lesigne \cite{CL84}, Furstenberg-Weiss \cite{FW96}, Host-Kra \cite{HK05b, HK05a}, Ziegler \cite{Z07}, Leibman \cite{L05a}, Tao \cite{T08}, Towsner \cite{To09}, Austin \cite{A10b}, Host \cite{H09}, among others.
%Towsner \cite{To09}, and others. 
By contrast, the pointwise almost everywhere convergence of multiple ergodic averages, a topic first studied by Bourgain \cite{Bo90} over 35 years ago, remains a challenging open problem despite impressive recent breakthroughs due to Krause-Mirek-Tao \cite{KMT20}, Kosz-Mirek-Peluse-Wan-Wright \cite{KMPWW24}, and others. 

%With a considerable, perhaps outrageous, degree of oversimplification, the state of the art on the questions above can be summarized as follows. Modulo some notable exceptions to be described later, the questions regarding multiple recurrence are largely resolved thanks to the works of Furstenberg, Furstenberg-Katznelson, Bergelson-Leibman, Leibman, and others, at least when $a_1, \ldots, a_\ell$ are integer polynomials and/or sequences involving primes. A similar claim can be made regarding norm convergence, due to the remarkable result of Walsh extending earlier works of Host-Kra, Ziegler, Leibman, Tao, Austin. By contrast, the pointwise almost everywhere convergence of multiple ergodic averages remains a challenging open problem despite highly impressive recent breakthroughs due to Krause-Mirek-Tao and others \BK{(cite)}. 

Although the proofs of multiple recurrence, norm convergence, and pointwise almost everywhere convergence differ in many aspects, 
% and use different tricks, but 
they share a common core: they all require \text{some} structural description of the limit of the associated multiple ergodic average. Hence the last of the presented problems, that of the structure of the limit,
%description of limits of multiple ergodic averages, 
lies at the heart of the others. 
%Proving a multiple recurrence of convergence result typically requires \emph{some} understanding of the structure of the limit. 
The most general known multiple recurrence and convergence results rely on a rather basic description of the limit.
%(take e.g. Bergelson-Leibman's proof of the polynomial Szemer\'edi theorem, or Walsh's norm convergence of nilpotent ergodic averages); 
Yet there remain strong incentives to describe the limits as precisely as possible. Knowing the limit is without a doubt more intellectually satisfying than simply knowing its existence; but if this reason is not convincing enough, there are a few ``practical'' ones:
\begin{enumerate}
    \item knowing the limit can yield new recurrence results and combinatorial consequences (see e.g. Theorem \ref{T: Fr Hardy along APs}) as well as strengthen the existing ones, implying e.g. the existence of so-called \emph{popular common differences} (see e.g. Proposition \ref{P: recurrence for jointly ergodic});
    \item for many averages, finding the limit is both the most straightforward way to establish its norm convergence (see e.g. Theorem \ref{T: Fr distinct-growth Hardy}) and can also serve as a step to prove pointwise convergence (see e.g. \cite{DS21, HSY19b});
    \item the predominantly analytic methods developed for this purpose can often be adapted to the discrete or Euclidean setting, reinforcing the fruitful interactions between the infinitary world of ergodic theory, the finitary realm of additive combinatorics, and the continuous Euclidean universe.
\end{enumerate}
% First, knowing the limit is without a doubt more intellectually satisfying than simply knowing its existence. Second, a precise understanding of the limit 
% %often yields stronger 
% gives new recurrence results and combinatorial consequences as well as strengthens the existing ones, implying e.g. the existence of so-called \emph{popular common differences} (see ...). Third, the predominantly analytic methods developed for this purpose can often be adapted to the discrete or Euclidean setting, reinforcing the fruitful interactions between the infinitary world of ergodic theory, the finitary realm of additive combinatorics, and the continuous Euclidean universe.
% %(where daunting technicalities tend to obfuscate the underlying idea) and
% %they have played a major part both in developing both qualitative results on the existence of patterns (e.g. the Green-Tao theorem) and bounds for the sets avoiding certain configurations. 
%merits the most nuanced answer. 
Given this motivation, it is clear that the study of the limiting behavior of multiple ergodic averages occupies a central role in developing new extensions of the Szemer\'edi theorem.  

Up to a few years ago, the structure of multiple ergodic averages along polynomials and other classes of iterates
%, primes, and Hardy sequences 
was fairly well-understood in the case of a \emph{single transformation} (i.e. when $T_1 = \cdots = T_\ell$). However, little had been known for commuting transformations other than for a few isolated cases. This changed substantially in the last couple of years, with the advent of potent new techniques capable of answering many previously intractable problems.
%as powerful new techniques have been developed to study this question. 
This survey aims to summarize these advances and highlight remaining open problems, some old and some new. We have chosen to do so from the perspective of joint ergodicity, which has served as the focal point of many of these recent breakthroughs. 
%of averages along commuting transformations. Even so, there 

\subsection{How to describe the limits of multiple ergodic averages?}

Before delving into the definition of joint ergodicity, we briefly discuss the surprisingly intricate question of what it even means to \emph{describe} the limit of a multiple ergodic average. 
\subsubsection{Explicit description of a limit}
Given a multiple ergodic average, we would ideally like to identify an explicit limiting function. This best-case scenario is exemplified by the mean ergodic theorem below.
\begin{theorem}[Von Neumann's mean ergodic theorem]
Let $(X, \CX, \mu, T)$ be a system and $f\in L^2(\mu)$. Then
\begin{align*}
\lim_{N\to\infty}\norm{\frac{1}{N}\sum_{n=1}^N T^n f - \E(f|\CI(T))}_{L^2(\mu)} = 0,
\end{align*}
where $\CI(T) := \{E\in\CX\colon\; T\inv E = E\}$ is the $\sigma$-algebra of $T$-invariant sets. 
%$$\CI(T) := \{E\in\CX\colon T\inv E = E\}.$$ 
In particular, if $T$ is ergodic, then
\begin{align*}
\lim_{N\to\infty}\norm{\frac{1}{N}\sum_{n=1}^N T^n f - \int f\; d\mu}_{L^2(\mu)} = 0.
\end{align*}
\end{theorem}
% which asserts that the $L^2(\mu)$ limit of the single ergodic average 
%\begin{align*}
%\frac{1}{N}\sum_{n=1}^N T^n f
%\end{align*}
%as $N\to\infty$ is $\E(f|\CI(T))$, the conditional expectation onto the invariant factor $$\CI(T) = \{E\in\CX\colon T\inv E = E\}.$$ 
Unfortunately, the mean ergodic theorem is rather exceptional in the level of precision with which it describes the limiting function. In most scenarios, life is not that simple, and finding an explicit limiting function is either infeasible or unhelpful, as illustrated by the example below.
\begin{example}[Local obstructions for polynomials]\label{Ex: n^2}
Let $q\in\N$, and let $(X, \CX, \mu, T)$ be the rotation $Tx = x + 1$ mod $q$ on  $X = \Z/q\Z$. Then 
\begin{align*}
	\lim_{N\to\infty} \frac{1}{N}\sum_{n=1}^N T^{n^2} f = \sum_{n\in\Z/q\Z}\frac{|\{i\in\Z/q\Z\colon i \equiv n^2\!\!\! \mod q\}|}{q} T^n f,
\end{align*}
where the limit can be taken in any suitable sense (in norm, weakly, or pointwise). 
For instance, when $q = 3$, we get
\begin{align*}
	\lim_{N\to\infty} \frac{1}{N}\sum_{n=1}^N T^{n^2} f = \frac{f + 2T f}{3},
\end{align*}
so in particular the limit is different from $\int f \; d\mu = \frac{f+Tf+T^2f}{3}$. The issue here is the \emph{local obstructions}, i.e. the fact that integer polynomials are generally not equidistributed in residue classes. On the other hand, if $(X, \CX, \mu, T)$ is \emph{totally ergodic} (i.e. $T, T^2, T^3, \ldots$ are all ergodic), then 
\begin{align*}
\lim_{N\to\infty} \norm{\frac{1}{N}\sum_{n=1}^N T^{n^2} f - \int f\; d\mu}_{L^2(\mu)} = 0.
\end{align*}
Describing the limit in a general system therefore requires us to merge the two extreme cases of periodic and totally ergodic systems. In particular, when the system admits complicated periodic factors, then any explicit formula for the limit would likely be too unsavoury to serve any practical purpose. Proposition \ref{P: Krat} below will describe a more compact and ready-to-apply way of describing the limit, even if a less explicit one.
% While for a particular value of $q$ the limit can still be described explicitly, providing a limiting formula for a general system would require us to explicitly describe the limit on the periodic factors of the system. This is certainly possible to do, but if the system admits complicated periodic factors, then the resulting formula would be too unsavoury to serve any practical purpose. 
\end{example}

%This approach is exemplified by the mean ergodic theorem, which asserts that the $L^2(\mu)$ limit of the single ergodic average 
%\begin{align*}
%\frac{1}{N}\sum_{n=1}^N T^n f
%\end{align*}
%as $N\to\infty$ is $\E(f|\CI(T))$, the conditional expectation onto the invariant factor $$\CI(T) = \{E\in\CX\colon T\inv E = E\}.$$ Unfortunately, life is not that simple, and in most scenarios, finding an explicit limiting function is either infeasible or unhelpful. 

\subsubsection{Describing the limit via a characteristic factor or a seminorm}
Faced with the difficulties inherent in explicit descriptions of limits, 
%Therefore, in many cases we are content with 
we usually resort to describing the limit implicitly, by identifying a factor or seminorm that controls a given average. 
\begin{definition}[Factor and seminorm control]\label{D: control}
    Let $(X, \CX, \mu,$\! $T_1, \ldots, T_\ell)$ be a system, $a_1, \ldots, a_\ell:\N\to\Z$ be sequences, and $1\leq j\leq \ell$ be an index. We say that a factor $\CY_j\subseteq \CX$ or a seminorm $\norm{\cdot}_j$ on $L^\infty(\mu)$ \emph{controls the sequences on the system at index $j$} if for arbitrary $f_1, \ldots, f_\ell\in L^\infty(\mu)$,
    %the limiting behavior of $A_N(f_1, \ldots, f_\ell)$ in the sense that
\begin{align}\label{E: control}
\lim_{N\to\infty}\norm{A_N(f_1, \ldots, f_\ell)}_{L^2(\mu)} = 0 \quad \textrm{whenever}\quad \E(f_j|\CY_j) = 0 \quad \textrm{or}\quad \norm{f_j}_j = 0.
\end{align}
If the same factor/seminorm can be chosen for any $j$, we simply say that it \emph{controls the sequences on the system}, \emph{controls the average}, or \emph{controls the tuple $(T_1^{a_1(n)}, \ldots, T_\ell^{a_\ell(n)})_n$}.
\end{definition}
% This is typically done by identifying a factor $\CY_j\subseteq \CX$ (often called \emph{characteristic factor}) or a seminorm $\norm{\cdot}_j$ on $L^\infty(\mu)$ that \emph{controls} the limiting behavior of $A_N(f_1, \ldots, f_\ell)$ in the sense that
% \begin{align}\label{E: control}
% \lim_{N\to\infty}\norm{A_N(f_1, \ldots, f_\ell)}_{L^2(\mu)} = 0 \quad \textrm{whenever}\quad \E(f_j|\CY_j) = 0 \quad \textrm{or}\quad \norm{f_j}_j = 0.
% \end{align}
The factor $\CY_j$ is also known as \emph{(partial) characteristic factor (at index $j$)}. The name comes from the work of Furstenberg and Weiss \cite{FW96}, but the underlying idea already arises in the foundational work of Furstenberg \cite{Fu77}. 

The factor and seminorm will in general depend on the index $j$, particularly while dealing with several different transformations. In the context of Example \ref{Ex: n^2}, such an implicit description is provided by the \emph{rational Kronecker factor}
\begin{align}\label{E: Krat}
	\Krat(T) := \bigvee_{r\in\N}\CI(T^r) = \langle E\in \CX\colon T^{-r}E = E\;\textrm{for some}\; r\in\N\rangle,
\end{align}
i.e. the smallest factor of $\CX$ containing all its periodic factors.
%as illustrated by the result below
\begin{proposition}[Rational Kronecker factor controls single polynomial averages {\cite[Lemma 3.14]{Fu81}}]\label{P: Krat}
Let $a\in\Z[t]$ be nonconstant, $(X, \CX, \mu, T)$ be a system, and $f\in L^\infty(\mu)$. Then
\begin{align*}
\lim_{N\to\infty} \norm{\frac{1}{N}\sum_{n=1}^N T^{a(n)} f}_{L^2(\mu)} = 0\quad \textrm{whenever}\quad \E(f|\Krat(T)) = 0.
\end{align*}
\end{proposition}
Proposition \ref{P: Krat} formalizes the idea that the only obstructions in understanding single ergodic averages along integer polynomials come from periodic systems. 

The control of the limit by a seminorm as in \eqref{E: control} is a purely analytic statement, usually achieved by chiefly analytic means. However, we are typically only interested in those seminorms $\norm{\cdot}_j$ that admit a dynamical interpretation, i.e. to which we can associate a factor $\CY_j$ satisfying
\begin{align*}
\norm{f}_j = 0 \quad \Longleftrightarrow \quad \E(f|\CY_j) = 0.
\end{align*} 
In many cases, the factors of interest admit a useful structural description, i.e. they are either (inverse limits of) algebraically structured systems (as in Example \ref{Ex: HK example} below) or joins of simple and explicit factors of the system (see Example \ref{Ex: Host example}). The control of an average by a factor plus the structure theory of the factor then gives us a dynamical description of the limit. The smaller the characteristic factor, the more precise our knowledge of the limit. Two quintessential instances of this procedure are presented below; for the definition of the seminorms and factors arising in both examples, see Appendix \ref{A: Host-Kra theory}.
\begin{example}[Host-Kra theory for Furstenberg averages]\label{Ex: HK example}
    Let $(X, \CX, \mu, T)$ be a system. In their seminal work \cite{HK05a}, Host and Kra showed that for every $\ell\in\N$ there exists a constant $C_\ell>0$ and a seminorm $\nnorm{\cdot}_{\ell, T}$ on $L^\infty(\mu)$ (called \emph{degree-$\ell$ Host-Kra seminorm}) such that for all 1-bounded functions $f_1, \ldots, f_\ell\in L^\infty(\mu)$, we have
\begin{align}\label{E: HK seminorm control}
\limsup_{N\to\infty}\norm{\E_{n\in[N]} T^n f_1 \cdots T^{\ell n} f_\ell}_{L^2(\mu)} \leq C_\ell \nnorm{f_\ell}_{\ell, T}.
\end{align}
They then constructed a factor $\CZ_{\ell-1}(T)\subseteq \CX$ such that
\begin{align}\label{E: HK factor property}
\nnorm{f}_{\ell, T} = 0 \quad \Longleftrightarrow \quad \E(f|\CZ_{\ell-1}(T)) = 0.
\end{align} 
Splitting each $f_j$ into $\E(f_j|\CZ_{\ell-1}(T))$ and its orthogonal complement, and using the seminorm control \eqref{E: HK seminorm control} together with the factor property \eqref{E: HK factor property} to annihilate the contribution of the orthogonal complement, they showed that it suffices to study the $L^2(\mu)$-limit of 
\begin{align}\label{E: Furstenberg averages}
    \E_{n\in[N]} T^n f_1 \cdots T^{\ell n} f_\ell
\end{align}
for $\CZ_{\ell-1}(T)$-measurable functions $f_1, \ldots, f_\ell$. For ergodic systems, they then proved the following structure theorem.
\begin{theorem}[Host-Kra structure theorem]\label{T: HK structure theorem}
Let $\ell\in \N$ and $(X, \CX, \mu, T)$ be an ergodic system. Then $(X, \CZ_\ell, \mu, T)$ is isomorphic to an inverse limit of $\ell$-step nilsystems.
\end{theorem}
Hence, passing to the ergodic components of $\mu$ if needed and using everything above, Host and Kra reduced the study of the $L^2(\mu)$ convergence of \eqref{E: Furstenberg averages} to the analogous problem for $(\ell-1)$-step nilsystems. The (everywhere) convergence on nilsystems then follows from the work of Leibman \cite{L05b}, with the limit described by Ziegler \cite[Theorem 1.2]{Z05}.
\end{example}

%is provided by combining the following results of Host and Kra.
% \begin{theorem}
% Let $\ell\in \N$ and $(X, \CX, \mu, T)$ be a system. Then there exists a constant $C_\ell>0$ and a seminorm $\nnorm{f}_{\ell, T}$ on $L^\infty(\mu)$ such that for any 1-bounded $f_1, \ldots, f_\ell\in L^\infty(\mu)$, we have
% \begin{align*}
% \limsup_{N\to\infty}\norm{\E_{n\in[N]} T^n f_1 \cdots T^{\ell n} f_\ell}_{L^2(\mu)} \leq C_\ell \nnorm{f_\ell}_{\ell, T}.
% \end{align*}
% \end{theorem}
% \begin{theorem}
% Let $\ell\in \N$ and $(X, \CX, \mu, T)$ be a system. Then there exists a factor $\CZ_{\ell-1}(T)\subseteq \CX$ such that
% \begin{align*}
% \nnorm{f}_{\ell, T} = 0 \quad \Longleftrightarrow \quad \E(f|\CZ_{\ell-1}(T)) = 0.
% \end{align*} 
% \end{theorem}
% \begin{theorem}
% Let $\ell\in \N$ and $(X, \CX, \mu, T)$ be an ergodic system. Then $(X, \CZ_\ell, \mu, T)$ is isomorphic to an inverse limit of $\ell$-step nilsystems.
% \end{theorem}
% Combining these bits together, we deduce that to understand the limit of Furstenberg's averages
% \begin{align*}
% \E_{n\in[N]} T^n f_1 \cdots T^{\ell n} f_\ell,
% \end{align*}
% it suffices to understand them in the case when the underlying system is an $(\ell-1)$-step nilsystem.

\begin{example}[Host's magic extensions]\label{Ex: Host example}
Building on earlier work of Austin \cite{A10b}, Host \cite{H09} extended bits of the Host-Kra theory to (the $L^2(\mu)$ limits of) the averages
\begin{align}\label{E: commuting averages}
    \E_{n\in[N]} T_1^n f_1 \cdots T_\ell^n f_\ell,
\end{align}
coming from an arbitrary system $(X, \CX, \mu,$\! $T_1, \ldots, T_\ell)$ (where, recalling our standing assumption, $T_1, \ldots, T_\ell$ are invertible and commute with each other). Specifically, he constructed a \emph{box seminorm} $\nnorm{\cdot}_{T_\ell, T_\ell T_1\inv, \ldots, T_\ell T_{\ell-1}\inv}$ on $L^\infty(\mu)$ and a \emph{box factor} $$\CZ(T_\ell, T_\ell T_1\inv, \ldots, T_\ell T_{\ell-1}\inv)\subseteq\CX$$ such that
\begin{align}\label{E: box seminorm control}
\limsup_{N\to\infty}\norm{\E_{n\in[N]} T_1^n f_1 \cdots T_\ell^n f_\ell}_{L^2(\mu)} \leq \nnorm{f_\ell}_{T_\ell, T_\ell T_1\inv, \ldots, T_\ell T_{\ell-1}\inv}
\end{align}
for all $1$-bounded functions $f_1, \ldots, f_\ell\in L^\infty(\mu)$, and moreover 
\begin{align}\label{E: box factor property}
\nnorm{f}_{T_\ell, T_\ell T_1\inv, \ldots, T_\ell T_{\ell-1}\inv} = 0 \quad \Longleftrightarrow \quad \E(f|\CZ(T_\ell, T_\ell T_1\inv, \ldots, T_\ell T_{\ell-1}\inv)) = 0.
\end{align} 
Combining \eqref{E: box seminorm control} and \eqref{E: box factor property}, he reduced the study of \eqref{E: commuting averages} to the case when $f_\ell$ is measurable with respect to $\CZ(T_\ell, T_\ell T_1\inv, \ldots, T_\ell T_{\ell-1}\inv)$.\footnote{For Furstenberg averages \eqref{E: Furstenberg averages}, i.e. when $T_j = T^j$ for some fixed transformation, we have $\nnorm{f}_{T_\ell, T_\ell T_1\inv, \ldots, T_\ell T_{\ell-1}\inv}\asymp \nnorm{f}_{\ell, T}$ and $\CZ(T_\ell, T_\ell T_1\inv, \ldots, T_\ell T_{\ell-1}\inv) = \CZ_{\ell-1}(T)$ by combining \eqref{E: scaling 1}, \eqref{E: scaling 2}, and \eqref{E: factor property}. Hence box seminorms and factors extend the Host-Kra seminorms and factors to the setting of commuting transformations.} 

It turns out that in contrast to Host-Kra factors, box factors do not seem to admit a nice structural description inside the system $(X, \CX,\mu, T_1, \ldots, T_\ell)$ itself. However, the system admits an extension $\pi:(\widetilde X, \widetilde\CX,\widetilde\mu, \widetilde T_1, \ldots, \widetilde T_\ell)\to (X, \CX,\mu, T_1, \ldots, T_\ell)$ (which Host called \emph{magic extension}) on which
\begin{align}\label{E: magic extension property}
    \CZ(\widetilde T_\ell, \widetilde T_\ell \widetilde T_1\inv, \ldots, \widetilde T_\ell \widetilde T_{\ell-1}\inv) = \CI(\widetilde T_\ell)\vee \CI(\widetilde T_\ell \widetilde T_{\ell-1}) \vee \cdots \vee \CI(\widetilde T_\ell \widetilde T_1\inv).
\end{align}
This means that every function $\widetilde f$ measurable with respect to \eqref{E: magic extension property} can be approximated by linear combinations of functions $g_0 g_1 \cdots g_{\ell-1}$, where each $g_j$ is measurable with respect to $\CI(\widetilde T_\ell \widetilde T_j\inv)$ (where $\widetilde T_0 := \Id_{\widetilde X}$). 
In particular,
\begin{align}\label{E: magic functional property}
    \widetilde T_\ell (g_0 g_1 \cdots g_{\ell-1}) = g_0 \cdot \widetilde T_1 g_1\cdots \widetilde T_{\ell-1}g_{\ell-1},
\end{align}
and so $\widetilde T_\ell$ disappears when acting on functions measurable with regards to \eqref{E: magic extension property}.
Hence, lifting the functions $f_j\in L^\infty(\mu)$ to functions $\widetilde f_j = f_j\circ \pi\in L^\infty(\widetilde \mu)$, using the properties \eqref{E: box seminorm control}, \eqref{E: box factor property}, and \eqref{E: magic extension property}, approximating $\widetilde f_\ell$ by a linear combination of functions $g_0 g_1 \cdots g_{\ell-1}$ as above, and using the identity \eqref{E: magic functional property}, the study of $L^2(\mu)$ convergence of \eqref{E: commuting averages} can be reduced to the one of 
\begin{align*}
        \E_{n\in[N]} \widetilde T_1^n (\widetilde f_1 \cdot g_1) \cdots \widetilde T_{\ell-1}^n (\widetilde f_{\ell-1} \cdot g_{\ell-1}).
\end{align*}
Inducting on $\ell$ in this fashion, we can infer the norm convergence of \eqref{E: commuting averages} from the mean ergodic theorem.

The idea of finding a structured extension on which a multiple ergodic average admits a simpler structural description dates back to Furstenberg and Weiss \cite{FW96}. For the average \eqref{E: commuting averages}, Austin \cite{A10b} first constructed a structured extension (which he called \emph{pleasant extension}) satisfying the property \eqref{E: magic extension property} using Furstenberg self-joinings rather than box seminorms. In subsequent works (e.g. \cite{A15a, A15b}), he constructed further extensions of a similar flavor to study other multiple ergodic averages. Host's construction came later, but it was more explicit, and his machinery of box seminorms more robust, playing a crucial role in the developments summarized in this survey.
\end{example}

\subsubsection{Describing a limit by comparison}\label{SSS: limit by comparison}
In certain cases, we may want to describe a limit that we a priori know very little about by comparing it with a limit of which we have fairly good understanding. A good example of this strategy is provided by the following result due to Frantzikinakis \cite{Fr10}, a special case of Theorem \ref{T: Fr Hardy along APs}.
\begin{theorem}[Identity for arithmetic progressions along fractional powers]
Let $\ell\in \N$ and $(X, \CX, \mu, T)$ be a system. Then for any $f_1, \ldots, f_\ell\in L^\infty(\mu)$, we have
\begin{align}\label{E: AP comparison Hardy}
\lim_{N\to\infty}\norm{\E_{n\in[N]} T^{\sfloor{n^{3/2}}} f_1 \cdots T^{\ell \sfloor{n^{3/2}}} f_\ell-\E_{n\in[N]} T^n f_1 \cdots T^{\ell n} f_\ell}_{L^2(\mu)}  = 0.
\end{align}
\end{theorem}
Since we know via Example \ref{Ex: HK example} that the average on the right of \eqref{E: AP comparison Hardy} is controlled by the seminorm $\nnorm{\cdot}_{\ell, T}$ and the factor $\CZ_{\ell-1}(T)$, the same holds for the average on the left. 

\subsection{Joint ergodicity: what is it, and why is it useful?}
%The overview of various strategies for describing the limit may leave the reader disheartened about the level of our understanding of the multiple ergodic averages and their limits. However, it 
The preceding sections may leave the reader with the disheartening impression that we can explicitly identify the limit of multiple ergodic averages only in exceedingly rare cases. Fortunately, there exists a pretty vast class of averages that admit a very simple, explicit description of the limit. This brings us finally to the concept of joint ergodicity. %: these are precisely the averages satisfying joint ergodicity.
\begin{definition}[Joint ergodicity]\label{D: joint ergodicity}
    Let $a_1, \ldots, a_\ell:\N\to\Z$ be sequences and $(X, \CX, \mu,$\! $T_1, \ldots, T_\ell)$ be a system. We say that 
 \begin{enumerate}
\item $a_1, \ldots, a_\ell$ are \emph{jointly ergodic} for $(X, \CX, \mu,$\! $T_1, \ldots, T_\ell)$ if
\begin{align}\label{E: joint ergodicity}
        \lim_{N\to\infty}\norm{\E_{n\in[N]}\prod_{j=1}^\ell T_j^{{a_{j}(n)}}f_j - \prod_{j=1}^\ell \int f_j\, d\mu}_{L^2(\mu)} = 0
    \end{align}
    holds for all $f_1, \ldots, f_\ell\in L^\infty(\mu)$. %{For $\ell=1,$ we simply say that $a_1$ is \emph{ergodic} for $(X, \CX, \mu, T_1)$ (and respectively we call $(T_1^{{a_1(n)}})_n$ \emph{ergodic}).}
    \item $a_1, \ldots, a_\ell$ are \emph{jointly ergodic} if they are jointly ergodic for every system $(X, \CX, \mu,$\! $T_1, \ldots, T_\ell)$ with $T_1, \ldots, T_\ell$ ergodic.\footnote{This is a stronger assumption than the ergodicity of the joint action of $T_1, \ldots, T_\ell$. For instance, the joint action of $T_1(x,y) = (x+\sqrt{2},y)$ and $T_2(x,y)=(x,y+\sqrt{3})$ on $X=\T^2$ ergodic (in that no set of Lebesgue measure in $(0,1)$ is invariant under both $T_1, T_2$), but neither $T_1$ nor $T_2$ is individually ergodic.} For $\ell=1$, we simply say that $a_1$ is \emph{ergodic}. 
    \end{enumerate}
%     {For $\ell=1,$ we say that $a_1$ is \emph{ergodic/weakly ergodic} for $(X, \CX, \mu, T_1)$ (and respectively we call $(T_1^{\floor{a_1(n)}})_n$ \emph{ergodic/weakly ergodic}).}
\end{definition}
When $a_1, \ldots, a_\ell$ are jointly ergodic for $(X, \CX, \mu,$\! $T_1, \ldots, T_\ell)$, we also say that the tuple $(T_1^{{a_1(n)}}, \ldots, T_\ell^{{a_\ell(n)}})_n$ is \emph{jointly ergodic} for $(X, \CX, \mu)$. If $(T_1^{{a(n)}}, \ldots, T_\ell^{{a(n)}})_n$ is \emph{jointly ergodic} for $(X, \CX, \mu)$, then we also say that $a$ is jointly ergodic for $(X, \CX, \mu,$\! $T_1, \ldots, T_\ell)$. In the case $\ell=1$, we simply say that $a$ is \emph{ergodic} for $(X, \CX, \mu, T)$.

A lot of sequences under consideration in this survey will be real-valued (e.g. $n^{3/2}$ or $n\log n$), therefore we need some rounding function to turn them into integer sequences. To simplify the discussion, we adopt the following convention.
\begin{convention}\label{Conv: real sequences}
    Let $\CP$ be a property defined for sequences $a_1, \ldots, a_\ell:\N\to\Z$. We say that sequences $a_1, \ldots, a_\ell:\N\to\R$ satisfy the property $\CP$ if their integer parts do.
\end{convention}
%In this case, we say that $a_1, \ldots, a_\ell$ are jointly ergodic if their integer parts are. 
For most of the results, the choice of the floor function versus other rounding functions (e.g. ceiling function, nearest integer) does not matter, and we will remark explicitly when it does. 

Joint ergodicity was first studied by Berend and Bergelson in the 1980s \cite{BB84, BB86}, who fully characterized the joint ergodicity of the averages \eqref{E: commuting averages}.
\begin{theorem}[Characterization of joint ergodicity for $(T_1^n, \ldots, T_\ell^n)_n$ {\cite[Theorem 3.1]{BB84}}]\label{T: BB characterization}
    Let $(X, \CX, \mu,$\! $T_1, \ldots, T_\ell)$ be a system. Then $(T_1^n, \ldots, T_\ell^n)_n$ is jointly ergodic for $(X, \CX, \mu)$ if and only if the following two conditions are satisfied:
    \begin{enumerate}
        \item (Product ergodicity condition) $T_1\times \cdots \times T_\ell$ is ergodic for $(X^\ell, \CX^{\otimes \ell}, \mu^\ell)$;
        \item (Difference ergodicity condition) $T_iT_j\inv$ is ergodic for $(X, \CX, \mu)$ for all distinct $1\leq i, j \leq \ell$. 
    \end{enumerate}
\end{theorem}

One should see joint ergodicity as the best-case scenario, as it provides the nicest possible limiting formula for a multiple ergodic average. A clear necessary condition for joint ergodicity is that $T_1, \ldots, T_\ell$ are all ergodic (\emph{not} to be confused with the ergodicity of their joint action). When the transformations are not ergodic, then the natural generalization of joint ergodicity is given by the concept below.
% \begin{definition}[Invariant factor control\footnote{In \cite{FrKu22b}, being controlled by the invariant factor was called \emph{weak joint ergodicity}. However, given that this notion holds under more general assumptions on the system (no ergodicity required), it is contentious whether one should think of it as a weak or strong variant of joint ergodicity.}]
%         Let $a_1, \ldots, a_\ell:\N\to\R$ be sequences and $(X, \CX, \mu,$\! $T_1, \ldots, T_\ell)$ be a system. We say that 
% \begin{enumerate}
%     \item  $a_1, \ldots, a_\ell$ are \emph{controlled by the invariant factor} for $(X, \CX, \mu,$\! $T_1, \ldots, T_\ell)$ if
%     %(alternatively, the tuple $(T_1^{\floor{a_1(n)}}, \ldots, T_\ell^{\floor{a_\ell(n)}})_n$ is \emph{quasi-jointly ergodic} for $(X, \CX, \mu)$) if
%     \begin{align}\label{E: weak joint ergodicity}
%         \lim_{N\to\infty}\norm{\E_{n\in[N]}\prod_{j=1}^\ell T_j^{\floor{a_{j}(n)}}f_j - \prod_{j=1}^\ell \E(f_j|\CI(T_j))}_{L^2(\mu)} = 0
%     \end{align}
%     holds for all $f_1, \ldots, f_\ell\in L^\infty(\mu)$.
%      \item  $a_1, \ldots, a_\ell$ are \emph{controlled by the invariant factor} if they are controlled by the invariant factor for every system.
% \end{enumerate}
% \end{definition}
\begin{definition}[Invariant factor control\footnote{In \cite{FrKu22b}, being controlled by the invariant factor was called \emph{weak joint ergodicity}. However, given that this notion holds under more general assumptions on the system (no ergodicity required), it is contentious whether one should think of it as a weak or strong variant of joint ergodicity.}]\label{D: invariant control}
        Let $a_1, \ldots, a_\ell:\N\to\Z$ be sequences and $(X, \CX, \mu,$\! $T_1, \ldots, T_\ell)$ be a system. We say that 
\begin{enumerate}
    \item  $a_1, \ldots, a_\ell$ are \emph{controlled by the invariant factor} for $(X, \CX, \mu,$\! $T_1, \ldots, T_\ell)$ if
    %(alternatively, the tuple $(T_1^{\floor{a_1(n)}}, \ldots, T_\ell^{\floor{a_\ell(n)}})_n$ is \emph{quasi-jointly ergodic} for $(X, \CX, \mu)$) if
    \begin{align}\label{E: weak joint ergodicity}
        \lim_{N\to\infty}\norm{\E_{n\in[N]}\prod_{j=1}^\ell T_j^{{a_{j}(n)}}f_j - \prod_{j=1}^\ell \E(f_j|\CI(T_j))}_{L^2(\mu)} = 0
    \end{align}
    holds for all $f_1, \ldots, f_\ell\in L^\infty(\mu)$.
     \item  $a_1, \ldots, a_\ell$ are \emph{controlled by the invariant factor} if they are controlled by the invariant factor for every system.
\end{enumerate}
\end{definition}

%Being controlled by the invariant factor is then a generalization of joint ergodicity where the transformations are not necessarily ergodic.\footnote{In \cite{FrKu22b}, being controlled by the invariant factor was called \emph{weak joint ergodicity}. However, given that this notion holds under more general assumptions on the system (no ergodicity required), it is contentious whether one should think of it as a weak or strong variant of joint ergodicity.}
The utility of this latter concept to combinatorics is best evidenced by the following result; it asserts that sequences controlled by the invariant factor satisfy a strong form of multiple recurrence. In what follows, $\overline{d}$ denotes the positive upper density.
\begin{proposition}[Lower bounds for multiple recurrence]\label{P: recurrence for jointly ergodic}
Suppose that the sequences $a_1, \ldots, a_\ell:\N\to\Z$ are controlled by the invariant factor. Then the following hold:
\begin{enumerate}
	\item For every system $(X, \CX, \mu,$\! $T_1, \ldots, T_\ell)$ and every $E\in\CX$, we have
	\begin{align*}
	\lim_{N\to\infty}\E_{n\in[N]}\mu(E\cap T_1^{-{a_1(n)}}E\cap \cdots \cap T_\ell^{-{a_\ell(n)}}E) \geq \mu(E)^{\ell+1}.
	\end{align*}
 In particular, for every $\veps>0$, the set
 \begin{align}\label{E: popular differences erg}
 	\{n\in\N\colon \mu(E\cap T_1^{-{a_1(n)}}E\cap \cdots \cap T_\ell^{-{a_\ell(n)}}E) \geq \mu(E)^{\ell+1} - \veps\}
 \end{align}
 has positive upper density.
\item For any $A\subseteq\Z^k$ 
    % rather than the arguably more natural quantity $\limsup\limits_{N\to\infty} \frac{|E\cap\{-N, \ldots, N\}^k|}{(2N+1)^k}$ since the Hardy functions under consideration are in general only defined on some halfline $(t_0, \infty)$, and so 
    %{\color{red}{since we are in $\Z,$ how come we don't define it symmetrically around $0$?}} \BK{(Because we study averages along $[N]$, so when we apply our joint ergodicity results, we do so to the asymmetric notion of density, not to the symmetric one.)}} 
    and vectors $\bv_1, \ldots, \bv_\ell\in\Z^k$, we have
	\begin{align*}
		\lim_{N\to\infty}\E_{n\in[N]}\overline{d}(A\cap (A - \bv_1{a_1(n)})\cap \cdots \cap (A - \bv_\ell{a_\ell(n)})) \geq \overline{d}(A)^{\ell+1}.
	\end{align*}
        In particular, for every $\veps>0$, the set
 \begin{align}\label{E: popular differences comb}
 	\{n\in\N\colon \overline{d}(A\cap (A - \bv_1{a_1(n)})\cap \cdots \cap (A - \bv_\ell{a_\ell(n)})) \geq \overline{d}(A)^{\ell+1} - \veps\}
 \end{align}
 has positive upper density.
\end{enumerate}

If furthermore the formulas \eqref{E: joint ergodicity} and \eqref{E: weak joint ergodicity} hold for any F{\o}lner sequence on $\N$ in place of $([N])_{N\in\N}$, then we can conclude that the sets \eqref{E: popular differences erg} and \eqref{E: popular differences comb} are syndetic. 
\end{proposition}
The elements of the sets \eqref{E: popular differences erg} or \eqref{E: popular differences comb} are often called \emph{popular common differences} (although the term does not have a precise definition, given that it depends on $\veps$).
% \begin{proof}
% We start with the first claim. Fix a system $(X, \CX, \mu,$\! $T_1, \ldots, T_\ell)$. The joint ergodicity of the sequences implies that
% \begin{multline*}
% \lim_{N\to\infty}\E_{n\in[N]}\mu(E\cap T_1^{-\floor{a_1(n)}}E\cap \cdots \cap T_\ell^{-\floor{a_\ell(n)}}E)\\
% = \int 1_E \cdot \E(1_E|\CI(T_1)) \cdots \E(1_E|\CI(T_\ell))\; d\mu.
% \end{multline*}
% By \cite[??]{Chu11} (see also Peluse's survey for an interesting discussion), the right hand side is bounded from below by $\mu(E)^{\ell+1}$, implying the first result. 

% Now let 
% \begin{align*}
%     E_n &:= E\cap T_1^{-\floor{a_1(n)}}E\cap \cdots \cap T_\ell^{-\floor{a_\ell(n)}}E\\
%     B_\veps &:= 	\{n\in\N\colon \mu(E\cap T_1^{-\floor{a_1(n)}}E\cap \cdots \cap T_\ell^{-\floor{a_\ell(n)}}E) \geq \mu(E)^{\ell+1} - \veps\}.
% \end{align*}
% Then 
% \begin{align*}
%     \mu(E)^{\ell+1}&\leq \lim_{N\to\infty}\E_{n\in[N]}\mu(E_n)\\
%     &\leq \limsup_{N\to\infty}\E_{n\in[N]}\mu(E_n)1_{B_\veps}(n) + \limsup_{N\to\infty}\E_{n\in[N]}\mu(E_n) 1_{\N\backslash B_\veps}(n)\\
%     &\leq \overline{d}(B_\veps) + \mu(E)^{\ell+1} - \veps;
% \end{align*}
% hence $\overline{d}(B_\veps) \geq \veps$, implying the second result. The last two results follow from Furstenberg's correspondence principle. The deduction of the stronger consequence for the sets \eqref{E: popular differences erg} and \eqref{E: popular differences comb} under the stronger assumption on \eqref{E: joint ergodicity} and \eqref{E: weak joint ergodicity} follows by modifying the arguments above in a straightforward fashion.
% \end{proof}
\begin{proof}
We start with the first claim. Fix a system $(X, \CX, \mu,$\! $T_1, \ldots, T_\ell)$. The joint ergodicity of the sequences implies that
\begin{multline*}
\lim_{N\to\infty}\E_{n\in[N]}\mu(E\cap T_1^{-{a_1(n)}}E\cap \cdots \cap T_\ell^{-{a_\ell(n)}}E)\\
= \int 1_E \cdot \E(1_E|\CI(T_1)) \cdots \E(1_E|\CI(T_\ell))\; d\mu.
\end{multline*}
By \cite[Lemma 1.6]{Chu11}, a consequence of the H\"older inequality, the right hand side is bounded from below by $\mu(E)^{\ell+1}$, implying the first result. 

Now let 
\begin{align*}
    E_n &:= E\cap T_1^{-{a_1(n)}}E\cap \cdots \cap T_\ell^{-{a_\ell(n)}}E\\
    B_\veps &:= 	\{n\in\N\colon \mu(E\cap T_1^{-{a_1(n)}}E\cap \cdots \cap T_\ell^{-{a_\ell(n)}}E) \geq \mu(E)^{\ell+1} - \veps\}.
\end{align*}
Then 
\begin{align*}
    \mu(E)^{\ell+1}&\leq \lim_{N\to\infty}\E_{n\in[N]}\mu(E_n)\\
    &\leq \limsup_{N\to\infty}\E_{n\in[N]}\mu(E_n)1_{B_\veps}(n) + \limsup_{N\to\infty}\E_{n\in[N]}\mu(E_n) 1_{\N\backslash B_\veps}(n)\\
    &\leq \overline{d}(B_\veps) + \mu(E)^{\ell+1} - \veps;
\end{align*}
hence $\overline{d}(B_\veps) \geq \veps$, implying the second result. The consequence for the subsets of $\Z^k$ follows from Furstenberg's correspondence principle. The deduction of the syndeticity of the sets \eqref{E: popular differences erg} and \eqref{E: popular differences comb} under the stronger assumption on \eqref{E: joint ergodicity} and \eqref{E: weak joint ergodicity} follows by modifying the arguments above in a straightforward fashion.
\end{proof}

For certain sequences such as integer polynomials, the control by invariant factor is too much to hope for. Indeed, Example \ref{Ex: n^2} shows that the single ergodic average along $n^2$ is not controlled by the invariant factor due to local obstructions. In this and similar cases, we need to replace the invariant factor by the rational Kronecker factor, defined in \eqref{E: Krat}.
% \begin{definition}[Rational Kronecker factor control]\label{D: Krat control}
%     Let $a_1, \ldots, a_\ell:\N\to\R$ be sequences and $(X, \CX, \mu,$\! $T_1, \ldots, T_\ell)$ be a system. We say that 
%  \begin{enumerate}
%  \item  $a_1, \ldots, a_\ell$ are \emph{controlled by the rational Kronecker factor} for $(X, \CX, \mu,$\! $T_1, \ldots, T_\ell)$ if for all $f_1, \ldots, f_\ell\in L^\infty(\mu)$,
%  %(alternatively, the tuple $(T_1^{\floor{a_1(n)}}, \ldots, T_\ell^{\floor{a_\ell(n)}})_n$ is \emph{quasi-jointly ergodic} for $(X, \CX, \mu)$) if
% \begin{align}\label{E: weak joint ergodicity}
%         \lim_{N\to\infty}\norm{\E_{n\in[N]}\prod_{j=1}^\ell T_j^{\floor{a_{j}(n)}}f_j}_{L^2(\mu)} = 0
%     \end{align}
%     holds whenever $\E(f_j|\Krat(T_j)) = 0$ for some $1\leq j\leq \ell$.
%      \item  $a_1, \ldots, a_\ell$ are \emph{controlled by the rational Kronecker factor} if they are controlled by the rational Kronecker factor for every system.
%     \end{enumerate}
% \end{definition}
\begin{definition}[Rational Kronecker factor control]\label{D: Krat control}
    Let $a_1, \ldots, a_\ell:\N\to\Z$ be sequences and $(X, \CX, \mu,$\! $T_1, \ldots, T_\ell)$ be a system. We say that 
 \begin{enumerate}
 \item  $a_1, \ldots, a_\ell$ are \emph{controlled by the rational Kronecker factor} for $(X, \CX, \mu,$\! $T_1, \ldots, T_\ell)$ if for all $f_1, \ldots, f_\ell\in L^\infty(\mu)$,
 %(alternatively, the tuple $(T_1^{\floor{a_1(n)}}, \ldots, T_\ell^{\floor{a_\ell(n)}})_n$ is \emph{quasi-jointly ergodic} for $(X, \CX, \mu)$) if
\begin{align*}%\label{E: vanishing for Krat control}
        \lim_{N\to\infty}\norm{\E_{n\in[N]}\prod_{j=1}^\ell T_j^{{a_{j}(n)}}f_j}_{L^2(\mu)} = 0
    \end{align*}
    holds whenever $\E(f_j|\Krat(T_j)) = 0$ for some $1\leq j\leq \ell$.
     \item  $a_1, \ldots, a_\ell$ are \emph{controlled by the rational Kronecker factor} if they are controlled by the rational Kronecker factor for every system.
    \end{enumerate}
\end{definition}

Rational Kronecker factor control typically suffices for popular common differences as long as our sequences satisfy good divisibility properties:
\begin{definition}[Good divisibility properties]
    Sequences $a_1, \ldots, a_\ell:\N\to\Z$ satisfy \emph{good divisibility properties} if for all $r\in\N$, the set
    \begin{align*}
        \{n\in\N\colon r|a_1(n), \ldots, a_\ell(n)\}
    \end{align*}
    has positive upper density.
\end{definition}
See e.g. \cite[Theorem 1.3, Corollary 1.7]{FrKr06}, \cite[Corollaries 2.11 and 2.12]{FrKu22a}, and \cite[Theorem 3.4]{FrKu22c} for popular common difference corollaries for integer polynomials and other sequences with good divisibility properties. 
% \begin{definition}[Good divisibility properties]
%     We say that sequences $a_1, \ldots, a_\ell:\N\to\Z$ satisfy \emph{good divisibility properties} if for all $r\in\N$, the set
%     \begin{align*}
%         \{n\in\N\colon r|a_1(n), \ldots, a_\ell(n)\}
%     \end{align*}
%     has positive density.
% \end{definition}
% \begin{proposition}\label{P: recurrence for Krat}
% Let $a_1, \ldots, a_\ell:\N\to\Z$ be sequences, and suppose that they are controlled by the rational Kronecker factor and have good divisibility properties. Then the following hold:
% \begin{enumerate}
% 	\item For every system $(X, \CX, \mu,$\! $T_1, \ldots, T_\ell)$, $E\in\CX$, and $\veps>0$, the set \eqref{E: popular differences erg} has positive upper density.
% \item For any $A\subseteq\Z^k$, vectors $\bv_1, \ldots, \bv_\ell\in\Z^k$, and $\veps>0$, the set \eqref{E: popular differences comb} has positive upper density.
% \end{enumerate}

% If furthermore rational Kronecker factor control hold along any F{\o}lner sequence on $\N$ in place of $([N])_{N\in\N}$, then we can conclude that the sets \eqref{E: popular differences erg} and \eqref{E: popular differences comb} are syndetic. 
% \end{proposition}
%In the definitions and results in this subsection, just like in Definition \ref{D: joint ergodicity}, we abuse the terminology and say that sequences $a_1, \ldots, a_\ell:\N\to\Z$ are controlled by the invariant or rational Kronecker factor if their integer parts are.

%\BK{Add about $\Krat$ and polys}

\subsection{What do we study when we study joint ergodicity?}
We hope to have convinced the reader by now that joint ergodicity and its variants are a worthwhile object of examination: not only do they provide clear limiting formulas, but they also lead to strong forms of recurrence and offer valuable insight into analogous problems in additive combinatorics. But what are the actual questions that we examine when we study joint ergodicity? At least three classes of problems fall under the umbrella of joint ergodicity, and we shall discuss them in turn in the subsequent chapters:
\begin{enumerate}
    \item (Joint ergodicity criteria) Identify simple sufficient and necessary criteria for joint ergodicity, invariant factor control, and rational Kronecker factor control. %and its variants, such as control by the invariant or rational Kronecker factors.
    \item (Joint ergodicity for large classes of sequences and systems) Identify large classes of sequences for which we have joint ergodicity under mild assumptions on the system. 
    \item (Joint ergodicity classification problem) Characterize sequences $a_1, \ldots, a_\ell:\N\to\R$ with the following property: the sequences are jointly ergodic for a system $(X,\CX,\mu,T_1,...,T_\ell)$ if and only if they simultaneously satisfy (an appropriate generalization of) the difference and product ergodicity conditions for the system.
\end{enumerate}

We start this survey with the presentation in Section \ref{S: joint ergodicity criteria} of two approaches, which we call Old and New Joint Ergodicity Strategies, to prove joint ergodicity of multiple ergodic averages. In doing so, we lay out the joint ergodicity criteria (Theorem \ref{T: joint ergodicity criteria}), proved by Frantzikinakis \cite{Fr21} as well as Frantzikinakis and the author \cite{FrKu22a}, that have been one of two main engines behind recent advances on the topic. We then present in Section \ref{S: joint ergodicity for large classes} the main joint ergodicity results for integer polynomials, Hardy sequences, and other families of iterates under the assumptions of ergodicity, total ergodicity, and weak mixing. Section \ref{S: classification problem} discusses the so-called \emph{joint ergodicity classification problem}, i.e. the question of identifying those families of sequences for which joint ergodicity on an arbitrary system is equivalent to the ergodicity of certain actions on this system. Subsequently, we move on to outline the proofs of two major technical breakthroughs that enabled recent joint ergodicity developments: the degree lowering argument and new methods for obtaining Host-Kra seminorm control. This is done in Section \ref{S: technical breakthroughs}. 

Since joint ergodicity is only one possible scenario for what the limit of a multiple ergodic average might be, the survey would be incomplete without the broader picture of what happens in the absence of joint ergodicity. This is the content of Section \ref{S: beyond joint ergodicity}. Likewise, the importance of assuming the commutativity of transformations becomes more pronounced when juxtaposed with its absence; this is what we discuss in Section \ref{S: beyond commutativity}. The proper part of the survey concludes with Section \ref{S: other topics}, a brief enumeration of other relevant topics that we do not discuss at length due to space constraints. Lastly, the appendices at the end summarize the notation used throughout, the basics of Hardy fields, and the rudiments of the Host-Kra theory.

The area is still full of open problems, and we strive to present them consistently throughout the survey. For more open questions on multiple ergodic averages, we advise the reader to consult the excellent survey of Frantzikinakis \cite{Fr16} (as well as his website, which lists up-to-date progress on those problems) and an older survey of Bergelson \cite{Ber96}. A gentle introduction to the study of multiple ergodic averages can be found in surveys of Kra \cite{Kra07} and Host \cite{H06}.
For a thorough summary of nilsystems, Host-Kra theory, and their applications, the textbook of Host and Kra remains the ultimate reference \cite{HK18}. Lastly, recent breakthroughs on joint ergodicity would not be possible without fruitful interchange of ideas with additive combinatorics. The state of the art in the latter field is brilliantly summarized in the recent survey of Peluse \cite{Pel24}.

\subsection{Acknowledgments}
This survey benefited greatly from the collective wisdom and generosity of the ergodic theory community. I am particularly grateful to Andreas Koutsogiannis, Konstantinos Tsinas, and Nikos Frantzikinakis for thorough and thoughtful feedback on the earlier drafts of this survey. I would also like to thank Bryna Kra, Wenbo Sun, Sebasti\'an Donoso, Joel Moreira, Florian Richter, and Vitaly Bergelson for suggestions, corrections, and clarifications.

\section{Joint ergodicity criteria}\label{S: joint ergodicity criteria}
Joint ergodicity is a pleasant and useful thing to have - but how to obtain it? A turning point in the study of joint ergodicity (and related properties from Definitions \ref{D: invariant control} and \ref{D: Krat control}) came in 2021-2022, with the advent of new joint ergodicity criteria summarized in Theorem \ref{T: joint ergodicity criteria} below as well as novel ways of obtaining seminorm control for ergodic averages. This section aims to introduce these new joint ergodicity criteria and juxtapose them with earlier methods.

\subsection{Key definitions}
We begin with several definitions that will play a prominent role in the forthcoming discussion. Before perusing the material below, we invite the reader to familiarize themselves with the background material in Appendix \ref{A: Host-Kra theory} on the Host-Kra theory, nonergodic eigenfunctions, and nilsystems.
%In order to introduce them and show how they differ from earlier methods, we need the following few definitions.
\begin{definition}[Seminorm control]\label{D: seminorm control}
We say that sequences $a_1,\ldots,a_\ell\colon \N\to \Z$:
    \begin{enumerate}
        \item \label{i:local-seminorms} {\em admit degree-$s$ Host-Kra seminorm control for the system  $(X, \CX, \mu,T_1,\ldots, T_\ell)$} if for all  $f_1,\ldots, f_\ell\in L^\infty(\mu)$, we have
\begin{align}\label{E: vanishing}
    \lim_{N\to\infty}\norm{\E_{n\in[N]} T_1^{a_1(n)}f_1\cdots T_\ell^{a_\ell(n)}f_\ell}_{L^2(\mu)} = 0
\end{align}        
whenever $\nnorm{f_j}_{s, T_j} = 0$ for some $1\leq j\leq \ell$;
% $\nnorm{f_m}_{s,T_m}=0$ for some $m\in [\ell]$, and $f_j\in \CE(T_j)$ for $j=m+1,\ldots, \ell$, then \eqref{E:zero} holds
%             		in $L^2(\mu)$.
\item \emph{admit  Host-Kra seminorm control for the system}  $(X, \CX, \mu,T_1,\ldots, T_\ell)$ if they admit degree-$s$ control for some $s\in\N$;
\item \emph{admit (degree-$s$) Host-Kra seminorm control} if they admit it for every system.
    \end{enumerate}
\end{definition}
In fact, our applications require something weaker than (degree-$s$) Host-Kra seminorm control in the sense given by Definition \ref{D: seminorm control}: it suffices to assume that \eqref{E: vanishing} holds as long as $\nnorm{f_m}_{s, T_m} = 0$ whenever $f_{m+1}, \ldots, f_\ell$ are \emph{nonergodic eigenfunctions} of $T_{m+1}, \ldots, T_\ell$ respectively (for all choices of $1\leq m \leq \ell$). This weaker property is stated in Definition \ref{D: P_m} when we outline the proof of Theorem \ref{T: invariant control criteria}. 

%For instance, for the average $\norm{\E_{n\in[N]}T_1^n f_1\cdot T_2^{n^2}f_2}_{L^2(\mu)}$, establishing the control by $\nnorm{f_1}_{s,T_1}$ (for some $s\in\N$) is significantly simpler whenever $f_2$ is a nonergodic eigenfunction of $T_2$, as it

In existing works (e.g. \cite{Fr21, FrKu22a}), sequences admitting Host-Kra seminorm control were also called \emph{good for seminorm estimates/control}.

\begin{definition}[Equidistribution on nilsystems]
We say that sequences $a_1,\ldots,a_\ell\colon \N\to \Z$ \emph{equidistribute on the nilsystem $(Y, \CY, m, S_1, \ldots, S_\ell)$} if for $m$-a.e. $y\in Y$, we have 
    \begin{align*}
        \overline{\rem{(S_1^{a_1(n)}y, \ldots, S_\ell^{a_\ell(n)}y)\colon n\in\N}} = \overline{\rem{(S_1^{n_1} y, \ldots, S_\ell^{n_\ell}y)\colon n_1, \ldots, n_\ell\in\N}};
    \end{align*}
    in particular, the right-hand side equals $Y^\ell$ whenever $S_1, \ldots, S_\ell$ are all ergodic.
\end{definition}

\begin{definition}[Equidistribution and irrational equidistribution]\label{D: equidistribution}
We say that sequences $a_1,\ldots,a_\ell\colon \N\to \Z$ are
    \begin{enumerate}
        \item\label{i:local-rational}    {\em good for equidistribution   for the  system $(X, \CX, \mu,T_1,\ldots, T_\ell)$} if for all nonergodic eigenfunctions $\chi_1\in \CE(T_1), \ldots, \chi_\ell\in\CE(T_\ell)$, we have
    \begin{align*}
        \lim_{N\to\infty}\norm{\E_{n\in[N]} \prod_{j=1}^\ell T_j^{a_j(n)}\chi_j-\prod_{j=1}^\ell\E(\chi_j|\CI(T_j))}_{L^2(\mu)} = 0;
    \end{align*}
    \item \emph{good for equidistribution} if they are good for equidistribution for every system;
    \item\label{i:local-irrational}  {\em good for  irrational equidistribution   for the  system $(X, \CX, \mu,T_1,\ldots, T_\ell)$} if for all nonergodic eigenfunctions $\chi_1\in \CE(T_1), \ldots, \chi_\ell\in\CE(T_\ell)$, we have
    \begin{align}\label{E: Krat identity}
        \lim_{N\to\infty}\norm{\E_{n\in[N]} \prod_{j=1}^\ell T_j^{a_j(n)}\chi_j-\E_{n\in[N]} \prod_{j=1}^\ell T_j^{a_j(n)}\E(\chi_j|\Krat(T_j))}_{L^2(\mu)} = 0;
    \end{align}
    \item \emph{good for irrational equidistribution} if they are good for irrational equidistribution for every system.
    \end{enumerate}
\end{definition}
The necessity to deal with nonergodic eigenfunctions rather than the classical ones comes from the fact that nonergodic eigenfunctions span $\CZ_1$-systems for nonergodic transformations (see Appendix \ref{A: low-degree HK factors}).

A recurrent theme in this survey is an interplay between \emph{local} vs. \emph{global} assumptions and results: those that hold for a particular system vs. those holding for every system from some large class (ergodic, single-transformation, etc.). In many instances, the global assumptions/results take simpler form than the local ones. As an example, we restate below the two global criteria from Definition \ref{D: equidistribution} in terms of exponential sums; for the rather straightforward proofs, see \cite[Lemma 2.4]{FrKu22a}.
\begin{lemma}[Equivalent formulation of global equidistribution properties]
    The sequences $a_1,\ldots,a_\ell\colon \N\to \Z$ are:
    \begin{enumerate}
        \item good for equidistribution if and only if 
    $$
		\lim_{N\to\infty} \E_{n\in [N]}\, e(a_1(n)\beta_1+\cdots+a_\ell(n)\beta_\ell)=0
	$$
		for   $\beta_1,\ldots, \beta_\ell\in [0,1)$ not all zero;
        \item good for irrational equidistribution if and only if 
    $$
		\lim_{N\to\infty} \E_{n\in [N]}\, e(a_1(n)\beta_1+\cdots+a_\ell(n)\beta_\ell)=0
	$$
		for  $\beta_1,\ldots, \beta_\ell\in [0,1)$ not all rational.
    \end{enumerate}
\end{lemma}
\begin{example}
    For noninteger $0<b_1<\cdots < b_\ell$, sequences $n^{b_1}, \ldots, n^{b_\ell}$ are good for equidistribution, whereas any affinely independent polynomials in $\Z[t]$ (e.g. $n, n^2, \ldots, n^\ell$) are good for irrational equidistribution.
\end{example}

Likewise, for systems in which all transformations are ergodic, the two local criteria in Definition \ref{D: equidistribution} can be restated in terms of eigenvalues.
\begin{lemma}[Equidistribution properties for ergodic systems]\label{L: equidistribution condition for ergodic}
    Let $(X, \CX, \mu,$\! $T_1, \ldots, T_\ell)$ be a system with $T_1, \ldots, T_\ell$ ergodic. The sequences $a_1,\ldots,a_\ell\colon \N\to \Z$ are:
    \begin{enumerate}
        \item good for equidistribution for the system if and only if 
    $$
		\lim_{N\to\infty} \E_{n\in [N]}\, e(a_1(n)\beta_1+\cdots+a_\ell(n)\beta_\ell)=0
	$$
		for all  $\beta_1\in \Spec(T_1),\ldots, \beta_\ell\in \Spec(T_\ell)$, not all of which are zero;
        \item good for irrational equidistribution if and only if 
    $$
		\lim_{N\to\infty} \E_{n\in [N]}\, e(a_1(n)\beta_1+\cdots+a_\ell(n)\beta_\ell)=0
	$$
		for all  $\beta_1\in \Spec(T_1),\ldots, \beta_\ell\in \Spec(T_\ell)$, not all of which are rational.
    \end{enumerate}
\end{lemma}
When $T_1, \ldots, T_\ell$ are nonergodic, the property of being good for equidistribution cannot be reduced to the exponential sums in Lemma \ref{L: equidistribution condition for ergodic}, as demonstrated by the following example.
\begin{example}
    Let $X =\T^2$ be endowed with the Borel $\sigma$-algebra and Lebesgue measure, and let
    \begin{align*}
        T_1(x,y) := (x, y + x) \quad \textrm{and}\quad T_2(x,y) = (x, y + x + \beta)
    \end{align*}
    for some $\beta\in\T$.
    It is easy to verify that $T_1, T_2$ have trivial spectrum, and so the exponential sums conditions from Lemma \ref{L: equidistribution condition for ergodic} are satisfied for trivial reasons by any sequences $a_1,a_2:\N\to\Z$.
    Yet $T_1, T_2$ both generate (nonergodic) $\CZ_1$-systems, and every function $\chi(x,y) = e(kx + ly)$ (for $k,l\in\Z$) is a nonergodic eigenfunction of both transformations. Hence to verify the good for equidistribution property for $a_1, a_2$ and $(X, \CX, \mu, T_1,T_2)$, we need to check whether
    \begin{align*}
        \lim_{N\to\infty}\norm{\E_{n\in[N]}e((ka_1(n) + la_2(n))x + la_2(n)\beta)}_{L^2(\mu)} = 0
    \end{align*}
    for all $k,l\in\Z$ not both 0. This fails even when $a_1(n) = a_2(n) = n$, $\beta = \frac{1}{2}$, and $k = -l\in 2\Z$.
\end{example}

In the single-transformation case, we make the following convention.
\begin{convention}
    We say that any property defined for $\ell$ transformations $T_1, \ldots, T_\ell$ holds for a system $(X, \CX, \mu, T)$ if it holds with $T_1 = \cdots = T_\ell = T$.
\end{convention}

\subsection{Old Joint Ergodicity Strategy}
With the definitions out of the way, we are ready to present the main strategies to prove joint ergodicity.
Suppose that we want to establish the joint ergodicity of sequences $a_1, \ldots, a_\ell:\N\to\Z$ for a system $(X, \CX, \mu,$\! $T_1, \ldots, T_\ell)$. Up to the 2021-2022 breakthroughs, we would have proceeded as follows.
%Fix some sequences $a_1, \ldots, a_\ell:\N\to\Z$ and a system $(X, \CX, \mu,$\! $T_1, \ldots, T_\ell)$. Up to the 2021-2022 breakthrough, the proofs of joint ergodicity of $(T_1^{a_1(n)}, \ldots, T_\ell^{a_\ell(n)})_n$ for $(X, \CX, \mu)$ consisted of the following steps.

\textbf{Old Joint Ergodicity Strategy:}
\begin{enumerate}
     \item (Host-Kra seminorm control) Show that the sequences admit Host-Kra seminorm control for the system.
    %Show that the average $A_N(f_1, \ldots, f_\ell)$ is \emph{controlled by a Host-Kra seminorm}, in that there exists $s\in\N$ such that
    % \begin{align*}
    %     \lim_{N\to\infty}\norm{\E_{n\in[N]} T_1^{a_1(n)}f_1\cdots T_\ell^{a_\ell(n)}f_\ell}_{L^2(\mu)} = 0
    % \end{align*}
    % whenever $\nnorm{f_j}_{s,T_j}=0$ for some $1\leq j\leq \ell$.
    \item (Reduction to nilsystems) Use the Host-Kra structure theorem (Theorem \ref{T: HK structure theorem}) to reduce the problem to nilsystems, i.e. derive the identity \eqref{E: joint ergodicity} for all nilsystems $(Y, \CY, m, S_1, \ldots, S_\ell)$ that arise as factors of $(X, \CX, \mu,$\! $T_1, \ldots, T_\ell)$.
    \item (Equidistribution on nilsystems) Show that the sequences equidistribute on all nilsystems as above. 
    % Show that for all nilsystems $(Y, \CY, m, S_1, \ldots, S_\ell)$ as above and $m$-a.e. $y\in Y$, we have 
    % \begin{align*}
    %     \overline{\rem{(S_1^{a_1(n)}y, \ldots, S_\ell^{a_\ell(n)}y)\colon n\in\N}} = \overline{\rem{(S_1^{n_1} y, \ldots, S_\ell^{n_\ell}y)\colon n_1, \ldots, n_\ell\in\N}};
    % \end{align*}
    % in particular, the right-hand side equals $Y^\ell$ whenever $S_1, \ldots, S_\ell$ are all ergodic.
\end{enumerate}
In particular, if the transformations $T_1, \ldots, T_\ell$ are totally ergodic, as is the case in some applications (see e.g. Theorem \ref{T: FrKr affine independence totally ergodic}), the same holds for the nilrotations $S_1, \ldots, S_\ell$.

The first step in the Old Joint Ergodicity Strategy is primarily analytic and usually involves some variant of the PET induction scheme of Bergelson \cite{Ber87}. The details are presented in Section \ref{SSS: PET}. Here, we just remark that it consists of applying the van der Corput inequality, change of variables, and the Cauchy-Schwarz inequality to the average under consideration many times until reaching an average in which all terms are linear in $n$. Each application reduces the average to one of lower ``complexity'' (in some suitable sense), and the careful measure of complexity ensures that one needs only finitely many steps to reach a linear average. At the end, we invoke the seminorm estimates for linear averages of Host-Kra \eqref{E: HK seminorm control} or Host \eqref{E: box seminorm control}, themselves products of several applications of the van der Corput and Cauchy-Schwarz inequalities.

The second step in the Old Joint Ergodicity Strategy is a straightforward application of Theorem \ref{T: HK structure theorem} jointly with an $L^2(\mu)$ approximation argument and the martingale convergence theorem.\footnote{If $\CY = \bigvee_{j=1}^\infty \CY_j$ for some increasing sequence of $\sigma$-algebras $\CY_1\subseteq \CY_2\subseteq \cdots$, then $$\lim\limits_{j\to\infty}\norm{\E(f|\CY_j)-\E(f|\CY)}_{L^2(\mu)} = 0.$$} The third step is considerably more challenging: obtaining a required equidistribution result on nilsystems involves a lot of nasty computations that restrict the classes of sequences tractable using this method (see e.g. the discussion below Theorem \ref{T: BMR weighted}).

\subsection{New Joint Ergodicity Strategy}
The hurdles faced by the seekers of joint ergodicity simplified dramatically in 2021-2022, with the advent of  new joint ergodicity criteria.
\begin{theorem}[Joint ergodicity criteria {\cite[Theorem 1.1]{Fr21}, \cite[Theorem 2.5]{FrKu22a}}]\label{T: joint ergodicity criteria}
        Let $(X, \CX, \mu,$\! $T_1, \ldots, T_\ell)$ be a system with $T_1, \ldots, T_\ell$ ergodic. The sequences $a_1,\ldots,a_\ell\colon \N\to \Z$ are jointly ergodic for the system if and only if the following conditions hold:
        \begin{enumerate}
            \item the sequences admit Host-Kra seminorm control for the system;
            \item the sequences are good for equidistribution for the system.
        \end{enumerate}
\end{theorem}
In other words, Theorem \ref{T: joint ergodicity criteria} combined with Lemma \ref{L: equidistribution condition for ergodic} reduces the arduous step (iii) of the Old Joint Ergodicity Strategy to a much simpler problem involving exponential sums. 

Without assuming ergodicity, Theorem \ref{T: joint ergodicity criteria} takes the following stronger form.
\begin{theorem}[Criteria for invariant factor control {\cite[Theorem 2.5]{FrKu22a}}]\label{T: invariant control criteria}
        Let $(X, \CX, \mu,$\! $T_1, \ldots, T_\ell)$ be a system. The sequences $a_1,\ldots,a_\ell\colon \N\to \Z$ are controlled by the invariant factor for the system if and only if the following conditions hold:
        \begin{enumerate}
            \item the sequences admit Host-Kra seminorm control for the system;
            \item the sequences are good for equidistribution for the system.
        \end{enumerate}
\end{theorem}
It is clear that being controlled by the invariant factor implies both Host-Kra seminorm control and being good for equidistribution. Indeed, since $\CI(T_j) = \CZ_0(T_j)$, the lowest-degree factor in the Host-Kra hierarchy, we have via \eqref{E: factor property} that being controlled by the invariant factor is equivalent to degree-1 Host-Kra seminorm control. Furthermore, by taking the functions $f_1, \ldots, f_\ell$ in \eqref{E: joint ergodicity} to be nonergodic eigenfunctions, we immediately infer being good for equidistribution from joint ergodicity. The other direction is the deep and nontrivial one; we will sketch its proof in Section \ref{S: degree lowering proof}.

If the end goal is to obtain control by the rational Kronecker factor, we can use the following variant of the aforementioned results.
\begin{theorem}[Criteria for rational Kronecker factor control {\cite[Theorem 1.6]{Fr21}, \cite[Theorem 2.6]{FrKu22a}}]\label{T: Krat control criteria}
        Let $(X, \CX, \mu,$\! $T_1, \ldots, T_\ell)$ be a system. The sequences $a_1,\ldots,a_\ell\colon \N\to \Z$ are controlled by the rational Kronecker factor for the system if and only if the following conditions hold:
        \begin{enumerate}
            \item the sequences admit Host-Kra seminorm control for the system;
            \item the sequences are good for irrational equidistribution for the system;
            \item the equation \eqref{E: Krat identity} holds when $f_j$ is an eigenfunction of $T_j$ with rational eigenvalue for all but one value $1\leq j\leq \ell$.\footnote{This last condition, absent in Theorems \ref{T: joint ergodicity criteria} and \ref{T: invariant control criteria}, is needed in the base case of the proof of the difficult direction of this result. In reality, it is easily satisfied for sequences of interest, e.g. for all sequences good for irrational equidistribution \cite[Lemma 5.3]{FrKu22a}.}
        \end{enumerate}
\end{theorem}
Once again, rational Kronecker factor control easily implies the other two conditions while the converse direction is difficult.

The previous three theorems yield the following global corollaries.
\begin{corollary}[Global joint ergodicity criteria {\cite[Theorem 2.1]{FrKu22a}}]\label{C: joint ergodicity criteria global}
    The sequences $a_1,\ldots,a_\ell\colon \N\to \Z$ are jointly ergodic if and only if the following conditions hold:
        \begin{enumerate}
            \item the sequences admit Host-Kra seminorm control for all systems $(X, \CX, \mu,$\! $T_1, \ldots, T_\ell)$ with $T_1, \ldots, T_\ell$ ergodic;
            \item the sequences are good for equidistribution.
        \end{enumerate}
\end{corollary}

\begin{corollary}[Criteria for global invariant/rational Kronecker factor control {\cite[Theorem 2.2]{FrKu22a}}]\label{C: invariant/Krat control criteria global}
    The sequences $a_1,\ldots,a_\ell\colon \N\to \Z$ are controlled by the invariant/rational Kronecker factor respectively if and only if the following conditions hold:
        \begin{enumerate}
            \item the sequences admit Host-Kra seminorm control;
            \item the sequences are good for equidistribution/irrational equidistribution respectively.
        \end{enumerate}
\end{corollary}

The proofs of all of the results above rely on the \emph{degree lowering} method, an extremely versatile analytic argument developed by Peluse \cite{Pel19} and later refined by Peluse-Prendiville \cite{PP19} to obtain quantitative versions of the polynomial Szemer\'edi theorem in the finitary universe. Frantzikinakis then realized its potential value in ergodic theory. Translating the arguments from \cite{PP19} into the language of dynamics, he proved the variants of Theorem \ref{T: joint ergodicity criteria}-Corollary \ref{C: invariant/Krat control criteria global} in the case of a single transformation $T_1 = \cdots = T_\ell$ \cite{Fr21}. The aforementioned results have then been extended to the case of $\Z^\ell$-systems by Frantzikinakis and the author \cite{FrKu22a}. The argument for the single-transformation case uses the ergodic decomposition to reduce matters to the ergodic system. This is no longer possible for several commuting transformations, as there is no variant of ergodic decomposition that would make them simultaneously ergodic.\footnote{Consider $X = \T^2$, $T_1(x,y) = (x+\sqrt{2},y)$, $T_2(x,y) = (x, y+\sqrt{3})$. Then the ergodic decompositions of the Lebesgue measure on $\T^2$ with respect to $T_1, T_2$ are different.}
%Then $\CI(T_1) = \{\T\times A\colon A\in\CB\}$ while $\CI(T_2) = \{A\times \T\colon A\in\CB\}$ (up to null sets), and so the disintegrations}
The lack of ergodicity was the main challenge in extending the results from \cite{Fr21} to \cite{FrKu22a}, and it necessitated the development of several technical results of independent interest. 

With these criteria in mind, the surest path towards proving joint ergodicity/invariant factor control/rational Kronecker factor control of sequences $a_1, \ldots, a_\ell:\N\to\Z$ for a system $(X, \CX, \mu,$\! $T_1, \ldots, T_\ell)$ takes the following form nowadays.

\textbf{New Joint Ergodicity Strategy:}
\begin{enumerate}
     \item (Host-Kra seminorm control) Show that the sequences admit Host-Kra seminorm control for the system.
    %Show that the average $A_N(f_1, \ldots, f_\ell)$ is \emph{controlled by a Host-Kra seminorm}, in that there exists $s\in\N$ such that
    % \begin{align*}
    %     \lim_{N\to\infty}\norm{\E_{n\in[N]} T_1^{a_1(n)}f_1\cdots T_\ell^{a_\ell(n)}f_\ell}_{L^2(\mu)} = 0
    % \end{align*}
    % whenever $\nnorm{f_j}_{s,T_j}=0$ for some $1\leq j\leq \ell$.
    \item (Reduction to an equidistribution problem) Use the relevant one of Theorems \ref{T: joint ergodicity criteria}, \ref{T: invariant control criteria}, or \ref{T: Krat control criteria} to reduce to an equidistribution problem.
    \item (Equidistribution on nilsystems) Show that the sequences are good for (irrational) equidistribution for the system. 
\end{enumerate}
The crux of the matter is then to obtain Host-Kra seminorm control over the averages of interest, as the relevant equidistribution results are often either known or easy to derive from existing ones. This problem is already rather challenging, and up to 2022, Host-Kra seminorm control was only known in two families of examples:
\begin{enumerate}
    \item in the single-transformation case, assuming some form of ``essential distinctness'' of sequences;
    \item in the commuting case, assuming that sequences have sufficiently ``distinct growth''.
\end{enumerate}
The second big breakthrough, in addition to the joint ergodicity criteria, was the development of new tools for Host-Kra seminorm estimates. These new techniques are presented at length in Section \ref{S: seminorm smoothing}, and the specific results on Host-Kra seminorm control proved this way are discussed all over the survey.
%, such as the coefficient tracking method \cite{DFKS22, KKL24a}, the concatenation technique \cite{DFKS22, KKL24a} 
%the \emph{seminorm smoothing} arguments that enabled the proof of Host-Kra seminorm control in the commuting case under the assumptions of ``pairwise independence'' of sequences (see Heuristic \ref{H: pairwise independent seminorm}). Such tools were first developed by Frantzikinakis and the author for integer polynomials \cite{FrKu22b, FrKu22a}, and then extended by Donoso, Koutsogiannis, Sun, Tsinas, and the author to Hardy sequences \cite{DKKST25, DKKST24}. 
%The specific Host-Kra seminorm control results are discussed in Section \ref{S: joint ergodicity for large classes}, and the seminorm smoothing argument in sketched in Section \ref{S: seminorm smoothing}.

Given the difficulties inherent in establishing Host-Kra seminorm control, it is tempting to ask whether the seminorm control conditions can be removed altogether.
\begin{problem}[Simple global joint ergodicity criterion {\cite[Problem 1]{Fr21}}]\label{Pr: joint ergodicity equivalent to equidistribution}
    Prove or disprove: sequences $a_1, \ldots, a_\ell:\N\to\Z$ are jointly ergodic if and only if they are good for equidistribution.
\end{problem}
The version of Problem \ref{Pr: joint ergodicity equivalent to equidistribution} for a fixed system admits a negative answer even for $\ell=1$, as demonstrated by the example below.
\begin{example}[Failure of simple local joint ergodicity criterion]
    Consider the sequence $a(n) = \sfloor{\log_2 n}$. As shown in \cite{DKKST25}, this sequence is ergodic on $(X, \CX, \mu, T)$ if and only if the system is conjugate to a 1-point system, or equivalently $L^\infty(\mu)$ consists of constants. Indeed, if %$\norm{\E_{n\in[N]}T^n f - \int f\; d\mu}_{L^2(\mu)}\to 0$, then also
    \begin{align*}
        &\lim_{N\to\infty}\norm{\E_{n\in[N]}T^{\sfloor{\log_2 n}} f - \int f\; d\mu}_{L^2(\mu)}=0\\ \textrm{then also}\quad &\lim_{N\to\infty}\norm{\E_{n\in[2^N, 2^{N+1})}T^{\sfloor{\log_2 n}} f - \int f\; d\mu}_{L^2(\mu)}=0.
    \end{align*}
    Since $\sfloor{\log_2 n} = N$ for all $n\in [2^N, 2^{N+1})$, the latter convergence is equivalent to 
%We note, however, that for $n\in [2^N, 2^{N+1})$, we have $\sfloor{\log_2 n} = N$, and hence the latter convergence is equivalent to 
    \begin{align*}
        \lim_{N\to\infty}\norm{T^N f - \int f\; d\mu}_{L^2(\mu)}=0.
    \end{align*}
    Measure invariance then implies that $\norm{f - \int f\; d\mu}_{L^2(\mu)} = 0$, which can only happen if $f$ is constant. This shows that ergodicity of $a$ imposes the claimed restriction on the system, and the converse direction is trivial. At the same time, $a$ is good for equidistribution for any weakly mixing systems as those do not have any nontrivial eigenvalues. It follows that $a$ is good for equidistribution but not ergodic on weakly mixing systems that are not conjugate to a 1-point system.
    \end{example}
%In Section \ref{??}, we will outline new techniques for seminorm estimates developed in the last few years. Before that, however, we list the most significant joint ergodicity results proved using the Old and New Joint Ergodicity Strategies. 

Lastly, we record an important corollary of Theorem \ref{T: joint ergodicity criteria} for nilsystems.
\begin{corollary}[Joint ergodicity criteria for nilsystems {\cite[Corollary 1.4]{Fr21}}]\label{C: joint ergodicity criteria for nilsystems}
The sequences $a_1,\ldots,a_\ell\colon \N\to \Z$ are jointly ergodic for all nilsystems $(Y, \CY, \nu, S_1, \ldots, S_\ell)$ with $S_1, \ldots, S_\ell$ ergodic if and only if they are good for equidistribution.

%$(X, \CX, \mu,$\! $T_1, \ldots, T_\ell)$ with $T_1, \ldots, T_\ell$ ergodic if and only if they are good for equidistribution.
\end{corollary}
    The key point about nilsystems is that an $s$-step nilsystem is a $\CZ_s$-system, meaning that $\CZ_s$ equals the entire Borel $\sigma$-algebra. Therefore the Host-Kra seminorm of degree $s+1$ is a norm, and  any family of sequences admits Host-Kra seminorm control on such a system for trivial reasons.

Importantly, joint ergodicity is about \emph{norm} convergence, while being good for equidistribution is about \emph{pointwise} convergence on $\T$. The next problem, mentioned in passing below \cite[Corollary 1.4]{Fr21}, inquires whether Corollary \ref{C: joint ergodicity criteria for nilsystems} can be strengthened to a purely pointwise joint ergodicity condition for nilsystems.
\begin{problem}[Pointwise joint ergodicity criteria for nilsystems]
    Suppose that the sequences $a_1,\ldots,a_\ell\colon \N\to \Z$ are good for equidistribution. 
    Is it true that for every nilsystem $(Y, \CY, \nu, S_1, \ldots, S_\ell)$ with $S_1, \ldots, S_\ell$
    %$(X, \CX, \mu,$\! $T_1, \ldots, T_\ell)$ with $T_1, \ldots, T_\ell$ 
    ergodic, the sequence 
    %and all $f_1, \ldots, f_\ell\in L^\infty(\mu)$, we have
    \begin{align*}
        (S_1^{a_1(n)}y, \ldots, S_\ell^{a_\ell(n)}y)_n
        %(T_1^{a_1(n)}x, \ldots, T_\ell^{a_\ell(n)}x)_n
    \end{align*}
    is equidistributed on $Y^\ell$ for $\mu$-a.e. $y\in Y$?
    %is equidistributed on $X^\ell$ for $\mu$-a.e. $x\in X$?
\end{problem}

\section{Joint ergodicity for large classes of sequences and systems}\label{S: joint ergodicity for large classes}
This section is dedicated to joint ergodicity results and open problems for specific families of iterates: polynomials, Hardy sequences, sequences involving primes, generalized polynomials, tempered functions, and functions of superpolynomial growth or with oscillations. Most results and questions presented in this section can be categorized into one of the two following heuristics.
\begin{heuristic}[Heuristic for joint ergodicity for pairwise independent sequences]\label{H: pairwise independent}
    Any reasonable ``pairwise independent'' sequences $a_1, \ldots, a_\ell:\N\to\Z$ are jointly ergodic for any system $(X, \CX, \mu,$\! $T_1, \ldots, T_\ell)$ with $T_1, \ldots, T_\ell$ weakly mixing.
\end{heuristic}

\begin{heuristic}[Heuristic for joint ergodicity for independent sequences]\label{H: independent}
    Any reasonable ``independent'' sequences $a_1, \ldots, a_\ell:\N\to\Z$ are jointly ergodic for any system $(X, \CX, \mu$, $T_1, \ldots, T_\ell)$ with $T_1, \ldots, T_\ell$ (totally) ergodic.
\end{heuristic}
%We shall not indulge in a discussion on which sequences we find reasonable other than that: 
We shall not offer much justification on which sequences we deem reasonable other than the utilitarian one: these are the well-behaved sequences for which we can prove nice results. At the very least, they should be amenable to usual analytic tools employed to study distribution mod 1, ergodicity, and single recurrence. They should also have a stable limiting behavior: ideally converge or go to infinity, but neither too slowly nor too fast. 

By contrast, we put considerable effort in this section to pin down the notions of ``independence'' and ``pairwise independence''. 
Their precise meaning depends on the context; for instance, it takes a different shape for integer polynomials vs. Hardy sequences. 
%See ... for results formalizing Heuristic \ref{H: pairwise independent} and ... for examples of Heuristic \ref{H: independent}. 
Before we fill these two notions of independence with substance, let us ponder on the origin of the conditions in Heuristics \ref{H: pairwise independent} and \ref{H: independent}. In the light of Theorem \ref{T: joint ergodicity criteria}, joint ergodicity is equivalent to Host-Kra seminorm control and good equidistribution properties. Weakly mixing transformations admit no nonergodic eigenfunctions, so for such systems all sequences are trivially good for equidistribution. Even more pointedly, the Host-Kra seminorms trivialize for a weakly mixing system $T$ in that $\nnorm{f}_{s,T} = \abs{\int f\; d\mu}$ for all $s\in\N$, in which case Host-Kra seminorm control is tantamount to joint ergodicity. Therefore, an appropriate notion of ``pairwise independence'' should capture the requirements for Host-Kra seminorm control.
%without making any assumptions on the system. 
This point is important enough to formalize it in a separate heuristic.
\begin{heuristic}[Heuristic for seminorm control for pairwise independent sequences]\label{H: pairwise independent seminorm}
    Any reasonable ``pairwise independent'' sequences $a_1, \ldots, a_\ell:\N\to\Z$ admit Host-Kra seminorm control.    
\end{heuristic}
%If all the transformations are weakly mixing, then they admit no nonergodic eigenfunctions, so the sequences are trivially good for equidistribution for this system. Therefore the combination of ``pairwise independence'' of sequences and weak mixing of the transformations is expected to yield joint ergodicity.

If our sequences satisfy the stronger conditions for ``independence'' rather than ``pairwise independence'', then we still expect them to admit Host-Kra seminorm control according to Heuristic \ref{H: pairwise independent seminorm}. Additionally, the needed notion of ``independence'' should be strong enough to imply being good for (rational) equidistribution, making the sequences good for equidistribution for any (totally) ergodic system.

\subsection{Integer polynomials}
\subsubsection{Single transformation}
The first result that can be viewed as a joint ergodicity statement is the following result of Furstenberg, which appeared in his foundational work on multiple ergodic averages.
\begin{theorem}%[Arithmetic progressions in weakly mixing systems {\cite{Fu77}}]\label{T: Furstenberg averages weakly mixing}
[{\cite{Fu77}}]\label{T: Furstenberg averages weakly mixing}
    For every $\ell\in\N$, the sequences $n, 2n, \ldots, \ell n$ are jointly ergodic for any weakly mixing system $(X, \CX, \mu, T)$.
    %Let $(X, \CX, \mu, T)$ be weakly mixing. Then $(T^n, T^{2n}, \ldots, T^{\ell n})_n$ is jointly ergodic for every $\ell$.
\end{theorem}
A decade later, Bergelson extended Theorem \ref{T: Furstenberg averages weakly mixing} to all \emph{essentially distinct} polynomials $a_1, \ldots, a_\ell\in\Z[t]$, i.e. those for which $a_j, a_i -a_j$ are nonconstant for every distinct $1\leq i,j\leq \ell$.
\begin{theorem}%[Essentially distinct polynomials in weakly mixing systems, 
[{\cite{Ber87}}]\label{T: Bergelson weakly mixing}
    Any essentially distinct polynomials $a_1, \ldots, a_\ell\in\Z[t]$ are jointly ergodic for any weakly mixing system $(X, \CX, \mu, T)$.
    %Let $a_1, \ldots, a_\ell\in\Z[t]$ be essentially distinct and $(X, \CX, \mu, T)$ be weakly mixing. Then $(T^{a_1(n)}, \ldots, T^{a_\ell(n)})_n$ is jointly ergodic.
\end{theorem}
Theorem \ref{T: Bergelson weakly mixing} exemplifies Heuristic \ref{H: pairwise independent}, showing that in the single-transformation case, the right notion of ``pairwise independence'' of integer polynomials is essential distinctness. The paper \cite{Ber87} was written almost two decades before the arrival of Host-Kra seminorms, and so these notions could not have possibly been used in the proof. However, the proof introduces the PET induction scheme that we use to establish Host-Kra seminorm control nowadays. Indeed, we have the following equivalence.
\begin{theorem}[Essential distinctness vs. Host-Kra seminorm control]\label{T: HK seminorm control for polys and single trafo}
    The polynomials $a_1, \ldots, a_\ell\in\Z[t]$ are essentially distinct if and only if they admit Host-Kra seminorm control for any system $(X, \CX, \mu, T)$.
\end{theorem}
\begin{example}[Necessity of essential distinctness]
    To see the necessity of essential distinctness for Host-Kra seminorm control, consider the average
    \begin{align*}
        \E_{n\in[N]}T^n f_1\cdot T^{n+1}f_2;
    \end{align*}
    suppose also without loss of generality that $T$ is ergodic.
    By the mean ergodic theorem, it converges in norm to $\int f_1\cdot Tf_2\; d\mu$. Taking $f_1 = Tf$, $f_2 = \overline{f}$ for some $f\in L^\infty(\mu)$ we see that the limit is $\int |f|^2\; d\mu$. Hence the average converges to 0 if and only if $f = 0$, which is a much stronger condition than $\nnorm{f}_{s, T} = 0$ for some $s\in\N$.
\end{example}

Just like Theorem \ref{T: Bergelson weakly mixing} instantiates Heuristic \ref{H: pairwise independent}, the following result of Frantzikinakis and Kra epitomizes Heuristic \ref{H: independent}. We call $a_1, \ldots, a_\ell\in\Z[t]$ \emph{affinely independent} if for any $c_1, \ldots, c_\ell\in\Z$, the polynomial $c_1 a_1 + \cdots + c_\ell a_\ell$ is constant unless $c_1 = \cdots = c_\ell = 0$.
\begin{theorem}%[Affinely independent polynomials in totally ergodic systems, 
[{\cite[Theorem 1.1]{FrKr05}}]\label{T: FrKr affine independence totally ergodic}
    Any affinely independent polynomials $a_1, \ldots, a_\ell\in\Z[t]$ are jointly ergodic for any totally ergodic system $(X, \CX, \mu, T)$.
    %Let $a_1, \ldots, a_\ell\in\Z[t]$ be affinely independent and $(X, \CX, \mu, T)$ be totally ergodic. Then $(T^{a_1(n)}, \ldots, T^{a_\ell(n)})_n$ is jointly ergodic.
\end{theorem}
For $n, n^2$, Theorem \ref{T: FrKr affine independence totally ergodic} has previously been proved by Furstenberg and Weiss \cite{FW96}. 
Theorem \ref{T: FrKr affine independence totally ergodic} demonstrates, for instance, that for a totally ergodic $T$, we have
\begin{align*}
    \lim_{N\to\infty}\norm{\E_{n\in[N]}T^n f_1 \cdots T^{n^\ell}f_\ell - \int f_1\; d\mu\cdots \int f_\ell\; d\mu}_{L^2(\mu)} = 0.
\end{align*}
It thus shows that affine independence is the right notion of independence for integer polynomials. 
\begin{example}[Necessity of affine independence]
    To see why affine independence is needed in Theorem \ref{T: FrKr affine independence totally ergodic}, take any $a\in\Z[t]$ (or, indeed, any sequence $a:\N\to\Z$). Let $(X, \CX, \mu, T)$ be a group rotation by some element $\beta\in X$. Then $a(n), 2a(n)$ are not jointly ergodic for $T$. Indeed, by taking $f_1(x) = \chi(2x)$, $f_2(x) = \chi(-x)$ for any nontrivial character (i.e. continuous group homomorphism) $\chi: X\to\S^1$, we see that
    \begin{align*}
        f_1(T^{a(n)}x)\cdot f_2(T^{a(n)}x) = \chi(2(x + \beta a(n)))\cdot \chi(-(x+\beta \cdot 2 a(n))) = \chi(x)
    \end{align*}
    for every $n\in\Z$ and $x\in X$.
    In particular, $\chi(x)$ is the pointwise limit of the corresponding multiple ergodic average, and it is very much not equal to $\int f_1\; d\mu \cdot \int f_2\; d\mu = 0$. This example shows how the absence of affine independence prevents joint ergodicity on any system with a nontrivial Kronecker factor. 
\end{example}

Total ergodicity is needed to prevent local obstructions; see Example \ref{Ex: n^2}. Without this assumption, Frantzikinakis and Kra obtain the following more general result.
\begin{theorem}%[Rational Kronecker factor control for affinely independent polynomials, 
[{\cite[Theorem 1.1]{FrKr06}}]\label{T: FrKr affine independence}
    Let polynomials $a_1, \ldots, a_\ell\in\Z[t]$ be affinely independent. Then they are controlled by the rational Kronecker factor for every system $(X, \CX, \mu, T)$.
\end{theorem}
The proof of Theorems \ref{T: FrKr affine independence totally ergodic} and \ref{T: FrKr affine independence} is a quintessential example of the Old Joint Ergodicity Strategy. Reducing to nilsystems via Theorem \ref{T: HK seminorm control for polys and single trafo} and the Host-Kra structure theorem, the authors spend the bulk of their work establishing the following equidistribution result on nilsystems from which Theorem \ref{T: FrKr affine independence totally ergodic} follows.
\begin{theorem}[Equidistribution of affinely independent polynomials on totally ergodic nilsystems {\cite[Theorem 1.2]{FrKr05}}]
    Let $a_1, \ldots, a_\ell\in\Z[t]$ be affinely independent and $(Y, \CY, m, S)$ be a totally ergodic nilsystem. Then for $m$-a.e. $y\in Y$, we have
    \begin{align*}
        \overline{\{(S^{a_1(n)}y, \ldots S^{a_\ell(n)}y)\colon n\in\N\}} = Y^\ell.
    \end{align*}
\end{theorem}

\subsubsection{Commuting transformations}

Up to a few years ago, very little was known on the structure of multiple ergodic averages of commuting transformations along polynomials except for their norm convergence \cite{W12}. One of the few results in this direction was the following joint ergodicity result 
%that can be extracted from the work 
of Chu, Frantzikinakis, and Host. %\cite{CFH11}. 
\begin{theorem}%[Joint ergodicity of monomials for totally ergodic transformations]
[{\cite[Theorem 1.2]{CFH11}}]\label{T: CFH distinct degrees}
    %For any integers $0<c_1<\cdots < c_\ell$, the polynomials $n^{c_1}, \ldots, n^{c_\ell}$ are jointly ergodic for any system $(X, \CX, \mu,$\! $T_1, \ldots, T_\ell)$ with $T_1, \ldots, T_\ell$ totally ergodic.
    Any nonconstant polynomials $a_1, \ldots, a_\ell\in\Z[t]$ of distinct degrees admit Host-Kra seminorm control. In particular, they are jointly ergodic for any system $(X, \CX, \mu,$\! $T_1, \ldots, T_\ell)$ with $T_1, \ldots, T_\ell$ weakly mixing.
    %Let $a_1, \ldots, a_\ell\in\Z[t]$ be nonconstant polynomials of distinct degrees, and let $(X, \CX, \mu,$\! $T_1, \ldots, T_\ell)$ be a system with $T_1, \ldots, T_\ell$ totally ergodic. Then $(T^{a_1(n)}, \ldots, T^{a_\ell(n)})_n$ is jointly ergodic.
\end{theorem}
%The restriction to monomials has two reasons. 
%One is the difficulty involved in proving the required equidistribution results on nilsystems in \cite[Section 7]{CFH11}. The other one, arguably more important, 
The restriction to distinct-degree polynomials comes from the fact that this is the only class of polynomials for which multiple ergodic averages of commuting transformations could be controlled by Host-Kra seminorms using purely the PET induction scheme. Indeed, the proof of Theorem \ref{T: CFH distinct degrees} breaks even for the average along $n^2, n^2+n$. Furthermore, difficulties inherent in studying equidistribution on nonergodic nilsystems prevented Chu, Frantzikinakis, and Host from deducing joint ergodicity for totally ergodic transformations.
%showing that distinct-degree polynomials are controlled by the rational Kronecker factor for arbitrary system $(X, \CX, \mu,$ $T_1, \ldots, T_\ell)$. 
Further progress on these questions was stalled until the work of Frantzikinakis and the author, in which the ensuing results have been proved.
%The first and arguably more important one is that distinct-degree polynomials form the only class of polynomials for which multiple ergodic averages of commuting transformations could be controlled by Host-Kra seminorms using purely the PET induction scheme. Indeed, the proof of Theorem \ref{T: CFH distinct degrees} breaks even for the average along $n^2, n^2+n$. The second reason is the difficulties inherent in studying equidistribution on nonergodic nilsystems: they prevented Chu, Frantzikinakis, and Host from extending Theorem \ref{T: CFH distinct degrees} to all distinct-degree polynomials and from obtaining rational Kronecker factor control for an arbitrary system $(X, \CX, \mu,$\! $T_1, \ldots, T_\ell)$. Further progress on these questions was stalled until the work of Frantzikinakis and the author, in which the ensuing results have been proved.
\begin{theorem}%[Rational Kronecker factor for affinely independent polynomials, 
[{\cite[Theorem 2.9]{FrKu22a}}]\label{T: FrKu affine independent}
    Any affinely independent polynomials $a_1, \ldots, a_\ell\in\Z[t]$ are controlled by the rational Kronecker factor. In particular, they are jointly ergodic for any system $(X, \CX, \mu,$\! $T_1, \ldots, T_\ell)$ with $T_1, \ldots, T_\ell$ totally ergodic.
    %Let $a_1, \ldots, a_\ell\in\Z[t]$ be affinely independent polynomials, and let $(X, \CX, \mu,$\! $T_1, \ldots, T_\ell)$ be a system with $T_1, \ldots, T_\ell$ totally ergodic. Then $(T^{a_1(n)}, \ldots, T^{a_\ell(n)})_n$ is jointly ergodic.
\end{theorem}
\begin{example}
    The simplest case covered by Theorem \ref{T: FrKu affine independent} that remained out of reach of previous methods was that for any system $(X, \CX, \mu,$\! $T_1, T_2)$ with $T_1, T_2$ totally ergodic, the sequences $n^2, n^2+n$ are jointly ergodic, i.e.
    \begin{align*}
        \lim_{N\to\infty}\norm{\E_{n\in[N]}T_1^{n^2}f_1\cdot T_2^{n^2+n}f_2 - \int f_1\; d\mu\cdot \int f_2\; d\mu}_{L^2(\mu)} = 0
    \end{align*}
    for any $f_1, f_2\in L^\infty(\mu)$. This was not even known under the much stronger assumption that $T_1, T_2$ are weakly mixing. 
\end{example}
Theorem \ref{T: FrKu affine independent} concretizes Heuristic \ref{H: independent} for integer polynomials and commuting transformations. It shows that the correct notion of ``independence'' for integer polynomials is ``affine independence'', just like it was the case for a single transformation.  This stands in contrast to ``pairwise independence'': while for a single transformation this notion took the form of essential distinctness, for commuting transformations we need our polynomials $a_1, \ldots, a_\ell\in\Z[t]$ to be \emph{pairwise independent}, in the sense that for all distinct $1\leq i, j\leq \ell$ and $c_i,c_j\in\Q$, the polynomial $c_i a_i + c_j a_j$ is nonconstant unless $c_i = c_j = 0$.
\begin{theorem}%[Pairwise independent polynomials, 
[{\cite[Theorem 2.8 and Corollary 2.9]{FrKu22a}}]\label{T: FrKu pairwise affine independent}
    Any pairwise independent polynomials $a_1, \ldots, a_\ell\in\Z[t]$:
    \begin{enumerate}
        \item\label{i: poly HK seminorm control} admit Host-Kra seminorm control;
        \item\label{i: poly weakly mixing} are jointly ergodic for any system $(X, \CX, \mu,$\! $T_1, \ldots, T_\ell)$ with $T_1, \ldots, T_\ell$ weakly mixing.
    \end{enumerate}
    %good for Host-Kra seminorm control. Hence they are jointly ergodic for any system $(X, \CX, \mu,$\! $T_1, \ldots, T_\ell)$ with $T_1, \ldots, T_\ell$ weakly mixing.
    %Let $a_1, \ldots, a_\ell\in\Z[t]$ be affinely independent polynomials, and let $(X, \CX, \mu,$\! $T_1, \ldots, T_\ell)$ be a system with $T_1, \ldots, T_\ell$ totally ergodic. Then $(T^{a_1(n)}, \ldots, T^{a_\ell(n)})_n$ is jointly ergodic.
\end{theorem}
The new joint ergodicity criteria (and specifically Theorem \ref{T: Krat control criteria}) play a crucial role in deriving Theorem \ref{T: FrKu affine independent} from Theorem \ref{T: FrKu pairwise affine independent}\eqref{i: poly HK seminorm control}, as they circumvent the study of distribution of polynomial orbits on nonergodic nilsystems. The second ingredient of at least equal importance are the new techniques developed to get the seminorm control in \ref{T: FrKu pairwise affine independent}\eqref{i: poly HK seminorm control}, and which are discussed at length in Section \ref{S: seminorm smoothing}.

\subsection{Hardy sequences}
Integer polynomials form arguably the most natural class of sparse sequences that arise in the study of multiple ergodic averages. Yet they are but a small class of a much broader family of \emph{Hardy sequences of polynomial growth} that include sequences like
\begin{align*}
    \sqrt{2}n^2 + \sqrt{3}n, n^{3/2}, n \log n, \exp(\sqrt{\log n}), 1/n.
\end{align*}
For precise definitions of Hardy sequences, we refer the reader to Appendix \ref{A: Hardy}. Here, we just recall that by $\CH$ we denote the collection of all Hardy functions of polynomial growth lying in some fixed Hardy field that contains $\CL\CE$, all logarithmico-exponential functions, and enjoys a bunch of other nice properties (it is closed under compositions, compositional inverses, and variable shifts).\footnote{Particular results stated in this survey may use only some of these properties, and hence they may hold for more general Hardy fields. However, since we prioritize brevity and ease of exposition over generality, we leave it to the interested reader to consult the relevant literature for the weakest possible condition on $\CH$ required in the particular context.}

The study of Hardy sequences in ergodic theory has a long history, dating back to work of Boshernitzan \cite{Boshernitzan-equidistribution}. %and Boshernitzan-Kolesnik-Quas-Wierdl \cite{BKQW05} 
What makes Hardy sequences reasonably easy to work with are their analytic properties. As long as a function $a\in\CH$ grows faster than $\log$, its derivative $a'(t)$ grows like $a(t)/t$ (up to a multiplicative factor).
%, with a notable exception of slowly-growing functions $a(t)\ll \log t$. 
As a consequence, we can Taylor expand $a$ on short intervals %\footnote{I.e. on intervals $(N, N + L(N)]$ for some $1\prec L(N)\prec N$; often, we can just take $L(N) = N^c$ for some $0<c<1$, and the larger $c$, the higher the degree of the approximating polynomial.} 
as polynomials with well-controlled error terms and explicitly given coefficients.
\begin{example}[Taylor expansion of Hardy sequences on short intervals]
    Let $a(t) = t^{3/2}$. If we choose a function $L\in\CH$ satisfying $N^{1/4}\prec L(N)\prec N^{1/2}$, then we can expand $a(t)$ on the interval $(N, N + L(N)]$ as
    \begin{align*}
        (N+n)^{3/2} = N^{3/2} + \frac{3}{2}N^{1/2}n + \frac{3}{8 N^{1/2}}n^2 - \frac{1}{16\xi_{N,n}^{3/2}}n^3
    \end{align*}
    for every $n\in[L(N)]$ and some $\xi_{N,n}\in(N, N+L(N)]$. The admissible range for $L$ was chosen so that the error term asymptotically vanish while the leading quadratic term does not, in the sense that
    \begin{align*}
        \lim_{N\to\infty}\max_{n\in[L(N)]}\frac{1}{16\xi_{N,n}^{3/2}}n^3 = 0\quad \textrm{and}\quad \lim_{N\to\infty}\max_{n\in[L(N)]} \frac{3}{8 N^{1/2}}n^2 = \infty.
    \end{align*}
    Hence on every interval $(N, N + L(N)]$, $a$ behaves like a quadratic polynomial. If instead we choose $L$ in the range $N^{1-3/(2d)}\prec L(N)\prec N^{1-3/(2(d+1))}$, then the Taylor approximant for $a$ will be a degree-$d$ polynomial.
\end{example}
The fact that on different short intervals we get different approximating polynomials (what some works \cite{Fr16, FrW09, Kouts22} call \emph{variable polynomials}) means that we often need to work with finite averages $\E_{n\in[N]}$, only taking the limit $N\to\infty$ towards the very end of the argument. This necessitates more quantitative arguments than one would use for polynomials. 

The reduction of Hardy sequences to polynomials on short intervals means that any progress on multiple ergodic averages along polynomials is likely to open new paths in the study of Hardy sequences. Not surprisingly, the results presented in this section historically followed analogous developments for integer polynomials. 

\subsubsection{Classification of Hardy sequences}
Before stating joint ergodicity results for Hardy sequences, we want to provide some intuition for various conditions that we will see later. A good starting point that motivates forthcoming definitions is the equidistribution result of Boshernitzan.
\begin{theorem}[Boshernitzan's equidistribution result {\cite[Theorem 1.3]{Boshernitzan-equidistribution}}]\label{T: Boshernitzan}
    Let $a\in\mathcal{H}$. The sequence $(a(n))_{n}$ is equidistributed $\mod~\!1$ if and only if for every $p\in \Q[t]$, we have
    \begin{align}\label{E: staying away}
        \lim_{t\to \infty}\frac{|a(t)-p(t)|}{\log t}=\infty.
    \end{align}
\end{theorem}
If $a$ satisfies \eqref{E: staying away}, we say that it \emph{stays away from rational polynomials}.

In what follows, we will work with the following families of Hardy sequences. We caution the reader that existing research works may name them differently.
\begin{definition}[Independence properties of Hardy sequences]\label{D: independence}
We call $a_1, \ldots, a_\ell\in\CH$:
\begin{enumerate}
\item \emph{essentially distinct} if for all distinct $1\leq i, j\leq \ell$, we have
\begin{align*}
    \lim_{t\to\infty}\frac{|a_j(t)|}{\log t} = \lim_{t\to\infty}\frac{|a_i(t)-a_j(t)|}{\log t}=\infty;
\end{align*}

\item \emph{pairwise independent} if for all distinct $1\leq i, j\leq\ell$ and $c_i, c_j\in\R$ not both zero, we have
\begin{align*}
    \lim_{t\to\infty}\frac{|c_i a_i(t) + c_j a_j(t)|}{\log t}=\infty;
\end{align*}
\item \emph{strongly independent} if every nontrivial linear combination of $a_1, \ldots, a_\ell$ stays logarithmically away from rational polynomials, i.e. for all $c_1, \ldots, c_\ell\in\R$ not all zero and for all $p\in\Q[t]$, we have
\begin{equation}\label{E: lafrp}
    \lim_{t\to\infty}\frac{|c_1 a_1(t) + \cdots + c_\ell a_\ell(t) - p(t)|}{\log t} = \infty;
\end{equation}
    \item \emph{strongly irrationally independent} if \eqref{E: lafrp} holds for all $c_1, \ldots, c_\ell\in\R$ not all rational and for all $p\in\Q[t]$.
    %\item  \emph{irrationally independent} if \eqref{E: lafrp} holds for all $c_1, \ldots, c_\ell\in \mathbb{R}\backslash\mathbb{Q}_{\ast}:=(\mathbb{R}\backslash\mathbb{Q})\cup\{0\}$ not all zero and for all $p\in\Q[t]$.
\end{enumerate}
\end{definition}

%As we see shortly, 
The notions of essential distinctness and pairwise independence generalize the notions of essential distinctness and pairwise affine independence for integer polynomials. By contrast, strong independence and strong irrational independence will play a similar role as affine independence for integer polynomials, even though neither of the former two extends the latter.
%, and strong independence of Hardy sequences will play the same role as the notions of essential distinctness, pairwise affine independence, and affine independence do for integer polynomials. (We do emphasize, however, that strong independence does \emph{not} generalize the concept of affine independence of polynomials.) The notion of strong irrational independence has no analog for integer polynomials since a family of integer polynomials satisfies a nontrivial affine relation if and only if satisfies one with rational coefficients.
%The relation between these notions is reflected below.}}

Like it was the case for integer polynomials, essential distinctness and pairwise independence will be preconditions for Host-Kra seminorm control (and hence for joint ergodicity in weakly mixing systems); strong independence will be necessary for invariant factor control; and strong irrational independence will enable rational Kronecker factor control.

Clearly, strong independence implies strong irrational independence,  %which implies irrational independence, 
which in turn implies pairwise independence. The following examples demonstrate the failure of the converse directions:
\begin{example}
    The sequences $t^{\sqrt{2}},\; t^{\sqrt{2}}+t^{1/2}$ are strongly independent; same with $t^{3/2},\; t^{3/2}+t^{1/2}$ or any other linearly independent \emph{fractional polynomials} $a_j(t) = \sum\limits_{i=1}^d \alpha_{ji}t^{b_i}$ for \emph{noninteger} positive real powers $b_1, \ldots, b_d$.
\end{example}
\begin{example}
    The sequences $t^{\sqrt{2}},\; t^{\sqrt{2}}+t$ are strongly irrationally independent but not strongly independent. More generally, linearly independent fractional polynomials $a_j(t) = \sum\limits_{i=1}^d \alpha_{ji}t^{b_i}$ for (possibly integer) positive reals $b_1, \ldots, b_d$ and rational coefficients $\alpha_{ji}$ will be strongly irrationally independent, but not necessarily strongly independent.
\end{example}
\begin{example}
    Linearly independent real polynomials are pairwise independent, but they need not even be strongly irrationally independent (consider e.g. $\sqrt{2}t^2,\; \sqrt{2}t^2 + \sqrt{3} t$).    
\end{example}

\subsubsection{Single transformation}\label{SSS: Hardy single}
The systematic study of joint ergodicity of Hardy sequences was initiated by Frantzikinakis with the following result.
\begin{theorem}%[Nonpolynomial Hardy sequences of distinct growth, 
[{\cite[Theorem 2.6]{Fr10}}]\label{T: Fr nonpolynomial Hardy}
    Let $a_1, \ldots, a_\ell\in\CL\CE$ satisfy the following two properties:
    \begin{enumerate}
        \item $a_1 \prec \cdots \prec a_\ell$;
        \item for every $1\leq j\leq \ell$ there exists $k_j\in\N_0$ and $0<\veps_j<1$ such that $t^{k_j+\veps_j} \prec a_j(t)\prec t^{k_j+1}$;
    \end{enumerate}
    Then they are controlled by the invariant factor for every system $(X, \CX, \mu, T)$. In particular, they are jointly ergodic for every ergodic system $(X, \CX, \mu, T)$.
    %they are jointly ergodic for every ergodic system $(X, \CX, \mu, T)$ (and hence also are controlled by the invariant factor for every system $(X, \CX, \mu, T)$).
\end{theorem}
We recall Convention \ref{Conv: real sequences} that when we speak of joint ergodicity, invariant factor control, etc. for real-valued sequences, then all these properties refer to the integer parts of these sequences.

In particular, Theorem \ref{T: Fr nonpolynomial Hardy} implies that the corresponding multiple ergodic average converges in norm. In general, for ``nonpolynomial'' Hardy sequences, identifying the limit is the surest way of showing norm convergence, as we do not know how to prove it by other means. 

Theorem \ref{T: Fr nonpolynomial Hardy} thus covers those Hardy sequences that are sufficiently nonpolynomial. The (partially) complimentary case of real polynomials was covered by Karageorgos and Koutsogiannis.
\begin{theorem}%[Strongly independent real polynomials, 
[{\cite[Theorem 2.1]{KK19}}]\label{T: KK strongly independent real polys}
    Any strongly independent real polynomials $a_1, \ldots, a_\ell$ $\in\R[t]$ are controlled by the invariant factor for every system $(X, \CX, \mu, T)$. In particular, they are jointly ergodic for every ergodic system $(X, \CX, \mu, T)$.
    %are jointly ergodic for every ergodic system $(X, \CX, \mu, T)$ (and hence also are controlled by the invariant factor for every system $(X, \CX, \mu, T)$).
\end{theorem}

Further joint ergodicity results for Hardy sequences and a single transformation have been proved by Bergelson, Moreira, and Richter \cite{BMR20} as well as Frantzikinakis \cite{Fr21}; the second of these works was the first in the study of Hardy sequences that relied on the New Joint Ergodicity Strategy. Finally, the problem for a single transformation was effectively resolved by Tsinas \cite{Ts22} with the following two theorems.
\begin{theorem}%[Strongly independent Hardy sequences, 
[{\cite[Theorem 1.2]{Ts22}}]\label{T: Ts strongly independent}
    Any strongly independent $a_1, \ldots, a_\ell\in\CH$ are controlled by the invariant factor for every system $(X, \CX, \mu, T)$. In particular, they are jointly ergodic for every ergodic system $(X, \CX, \mu, T)$.
\end{theorem}
Theorem \ref{T: Ts strongly independent} realizes Heuristic \ref{H: independent} for Hardy sequences of polynomial growth and single transformation, just like the result below realizes Heuristic \ref{H: pairwise independent}.
\begin{theorem}%[Essentially distinct Hardy sequences, 
[{\cite[Theorem 1.3 and Corollary 1.4]{Ts22}}]\label{T: Ts essentially distinct}
    Any essentially distinct $a_1, \ldots, a_\ell\in\CH$:
    \begin{enumerate}
        \item admit Host-Kra seminorm control;
        \item are jointly ergodic for any system $(X, \CX, \mu,$\! $T_1, \ldots, T_\ell)$ with $T_1, \ldots, T_\ell$ weakly mixing.
    \end{enumerate}
\end{theorem}

\subsubsection{Commuting transformations}\label{SSS: Hardy commuting}

Before the breakthroughs in the last few years, joint ergodicity of Hardy sequences for commuting transformations was only known for sequences of sufficiently distinct growth, in analogy with the result of Chu-Frantzikinakis-Host for polynomials (Theorem \ref{T: CFH distinct degrees}).
\begin{theorem}%[Nonpolynomial Hardy sequences of distinct growth, 
[{\cite[Theorem 2.2]{Fr10}}]\label{T: Fr distinct-growth Hardy}
    Let $a_1, \ldots, a_\ell\in\CH$ satisfy the following two properties:
    \begin{enumerate}
        \item $a_1 \prec \cdots \prec a_\ell$;
        \item for every $1\leq j\leq \ell$ there exists $k_j\in\N_0$ such that $t^{k_j}\log t \prec a_j(t)\prec t^{k_j+1}$;
    \end{enumerate}
    Then they are controlled by the invariant factor. In particular, they are jointly ergodic for all systems $(X, \CX, \mu,$\! $T_1, \ldots, T_\ell)$ with $T_1, \ldots, T_\ell$ ergodic. 
\end{theorem}
An important special case of Theorem \ref{T: Fr distinct-growth Hardy} is that for any noninteger $0<b_1<\cdots < b_\ell$, we have
\begin{align*}
    \lim_{N\to\infty}\norm{\E_{n\in[N]}T_1^{\sfloor{n^{b_1}}}f_1\cdots T_\ell^{\sfloor{n^{b_\ell}}}f_\ell - \E(f_1|\CI(T_1))\cdots \E(f_\ell|\CI(T_\ell))}_{L^2(\mu)} = 0.
\end{align*}
\begin{theorem}%[Distinct-degree real polynomials, 
[{\cite[Corollary 1.7]{Ko18}}]\label{T: Kouts distinct-degree polys}
    Any nonconstant polynomials $a_1, \ldots, a_\ell\in\R[t]$ of distinct degrees:
    \begin{enumerate}
        \item admit Host-Kra seminorm control;
        \item are jointly ergodic for any system $(X, \CX, \mu,$\! $T_1, \ldots, T_\ell)$ with $T_1, \ldots, T_\ell$ weakly mixing.
    \end{enumerate}
\end{theorem}

Like with Theorem \ref{T: CFH distinct degrees}, the problem with extending Theorems \ref{T: Fr distinct-growth Hardy} and \ref{T: Kouts distinct-degree polys} to sequences of the same growth was the absence of methods for establishing Host-Kra seminorm control. Following the arrival of such techniques for integer polynomials (discussed in Section \ref{S: seminorm smoothing}), similar tools were developed for Hardy sequences. Their upshot are the following three results of Donoso, Koutsogiannis, Sun, Tsinas, and the author, verifying Heuristics \ref{H: independent}, \ref{H: pairwise independent seminorm}, and \ref{H: pairwise independent} for Hardy sequences.
    \begin{theorem}%[Joint ergodicity for strongly independent Hardy sequences, 
    [{\cite[Theorem 1.4]{DKKST24}}]\label{T: DKKST strongly independent Hardy}
     Any strongly independent Hardy sequences $a_1, \ldots, a_\ell\in \CH$ are controlled by the invariant factor. In particular, they are jointly ergodic for all systems $(X, \CX, \mu,$\! $T_1, \ldots, T_\ell)$ with $T_1, \ldots, T_\ell$ ergodic.  
%     In particular, $a_1, \ldots, a_\ell$ are jointly ergodic for all systems $(X, \CX, \mu,$\! $T_1, \ldots, T_\ell)$ with $T_1, \ldots, T_\ell$ being ergodic.
 	\end{theorem}
As a special case, Theorem~\ref{T: DKKST strongly independent Hardy} gives
   \begin{align*}
       \lim_{N\to\infty}\norm{\E_{n\in[N]} T_1^{\sfloor{n^{\sqrt{2}}}}f_1\cdot T_2^{\sfloor{n^{\sqrt{2}}+n^{1/2}}}f_2 - \E(f_1|\CI(T_1))\cdot \E(f_2|\CI(T_2))}_{L^2(\mu)} = 0,
   \end{align*}
   a result not known before even for weakly mixing $T_1, T_2$. More generally, it yields joint ergodicity for arbitrary linearly independent fractional polynomials with noninteger powers.
   
    \begin{theorem}%[Joint ergodicity for pairwise independent Hardy sequences, 
    [{\cite[Theorems 1.3 and 1.16]{DKKST24}}]\label{T: DKKST pairwise independent Hardy}
    Any pairwise independent Hardy sequences $a_1, \ldots, a_\ell\in \CH$:
    \begin{enumerate}
        \item admit Host-Kra seminorm control;
        \item are jointly ergodic for any system $(X, \CX, \mu,$\! $T_1, \ldots, T_\ell)$ with $T_1, \ldots, T_\ell$ weakly mixing.
    \end{enumerate}    
%    be pairwise independent. Then $a_1, \ldots, a_\ell$ are jointly ergodic for all systems $(X, \CX, \mu,$\! $T_1, \ldots, T_\ell)$ in which $T_1, \ldots, T_\ell$ are all weakly mixing. 
    \end{theorem}
    It is worth emphasizing that the class of pairwise independent Hardy sequences contains the class of pairwise independent integer polynomials. Therefore Theorem \ref{T: DKKST pairwise independent Hardy} generalizes Theorem \ref{T: FrKu pairwise affine independent}.

For strongly irrationally independent Hardy sequences, Theorem \ref{T: DKKST strongly independent Hardy} admits the following variant.
\begin{theorem}%[Joint ergodicity for strongly irrationally independent Hardy sequences, 
[{\cite[Theorem 1.6]{DKKST24}}]\label{T: DKKST strongly irrationally independent Hardy}
     Any strongly irrationally independent Hardy sequences $a_1, \ldots, a_\ell\in \CH$ are controlled by the rational Kronecker factor. In particular, they are jointly ergodic for all systems $(X, \CX, \mu,$\! $T_1, \ldots, T_\ell)$ with $T_1, \ldots, T_\ell$ totally ergodic. 
%     In particular, $a_1, \ldots, a_\ell$ are jointly ergodic for all systems $(X, \CX, \mu,$\! $T_1, \ldots, T_\ell)$ with $T_1, \ldots, T_\ell$ being ergodic.
 	\end{theorem}

    \subsubsection{Weighted averages}\label{SSS: weighted averages}
    One assumption underlying all the results on Hardy sequences so far is that the sequences must grow faster than $\log$. That is because logarithm is bad both for equidistribution and for mean ergodic theorem. Indeed, 
    \begin{align*}
        \E_{n\in[N]}e(\alpha \log n) = \E_{n\in[N]}n^{2\pi i \alpha}\sim \frac{N^{2\pi i \alpha}}{1 + 2\pi i \alpha} = \frac{e(\alpha\log N)}{1 + 2\pi i \alpha},
    \end{align*}
    from which it is clear that this sequence diverges. Likewise, the average
    \begin{align*}
        \E_{n\in[N]}T^{\floor{\log n}}f
    \end{align*}
    does not converge in general since for infinitely many $N$, the function $\floor{\log n}$ equals $\floor{\log N}$ on intervals $[N, (1+c)N]$ for some small $c>0$. These issues can be overcome by studying \emph{weighted averages}. Indeed, if we replace Ces\`aro average by the logarithmic average, then
    \begin{align*}
        \abs{\frac{1}{\log N}\sum_{n=1}^N \frac{e(\alpha \log n)}{n}} = O\brac{\frac{1}{|\alpha|\log N}},
    \end{align*}
    which goes to 0 unless $\alpha = 0$. Similarly, we have
    \begin{align*}
        \lim_{N\to\infty}\frac{1}{\log N}\sum_{n=1}^N \frac{1}{n} T^{\sfloor{\log n}}f = \lim_{N\to\infty}\E_{n\in[N]}T^n f = \E(f|\CI(T))
        %\lim_{N\to\infty}\norm{\frac{1}{\log N}\sum_{n=1}^N T^{\sfloor{\log n}}f - \E(f|\CI(T))}_{L^2(\mu)} =  \lim_{N\to\infty}\norm{\frac{1}{\log N}\sum_{n=1}^N T^{\sfloor{\log n}}f - \E_{n\in[N]}T^n f}_{L^2(\mu)} = 0,
    \end{align*}
    in $L^2(\mu)$,     i.e. the logarithmic ergodic average with the iterate $\sfloor{\log n}$ satisfies the conclusion of the mean ergodic theorem. One can see this as a ``discrete substitution''; indeed, if we replace a discrete transformation $T$ by a flow $R$ and averaging by integrating, then the usual substitution rule gives us
    \begin{align*}
        \frac{1}{\log N}\int_{0}^N \frac{1}{t}f(R^{\log t}x)\; dt = \frac{1}{\log N}\int_{0}^{\log N} f(R^{u}x)\; du,
    \end{align*}
    which is a continuous variant of $\E_{n\in[\sfloor{\log N}]}T^n f$.
    
    In \cite{BMR20}, Bergelson, Moreira, and Richter studied a weighted version of multiple ergodic averages along Hardy sequences. Their setup is as follows.
    Given a Hardy function $W\in\CH$ satisfying $1\prec W(t)\ll t$, set $w(n) := W(n+1)-W(n)$, and denote
    \begin{align*}
        \E_{n\in[N]}^W A(n) := \frac{1}{W(n)}\sum_{n=1}^N w(n) A(n). 
    \end{align*}
    If $W(n) = n$, then $w(n) = 1$ and we recover the Ces\`aro average. By contrast, if $W(N) = \log N$, then $w(n)=(1+o(1))\frac{1}{n}$, and so logarithmic averages morally fall into this framework. We can naturally define a variant of joint ergodicity for $W$-weighted averages.
    \begin{definition}
        Let $a_1, \ldots, a_\ell:\N\to\Z$ be sequences and $(X, \CX, \mu,$\! $T_1, \ldots, T_\ell)$ be a system. We say that
        \begin{enumerate}
            \item $a_1, \ldots, a_\ell$ are \emph{$W$-jointly ergodic} for the system  if 
        \begin{align}\label{E: W-joint ergodicity}
        \lim_{N\to\infty}\norm{\E^W_{n\in[N]}w(n)\cdot \prod_{j=1}^\ell T_j^{{a_{j}(n)}}f_j - \prod_{j=1}^\ell \int f_j\, d\mu}_{L^2(\mu)} = 0
    \end{align}
    holds for all $f_1, \ldots, f_\ell\in L^\infty(\mu)$;
    \item $a_1, \ldots, a_\ell$ are \emph{controlled by the invariant factor for $W$-averages} for the system  if 
        \begin{align*}%\label{E: joint ergodicity}
        \lim_{N\to\infty}\norm{\E^W_{n\in[N]}w(n)\cdot \prod_{j=1}^\ell T_j^{{a_{j}(n)}}f_j - \prod_{j=1}^\ell \E(f_j|\CI(T_j))}_{L^2(\mu)} = 0
    \end{align*}
    holds for all $f_1, \ldots, f_\ell\in L^\infty(\mu)$.
    \item $a_1, \ldots, a_\ell$ \emph{admit Host-Kra seminorm control for $W$-averages} for the system  if there exists $s\in\N$ such that for all $f_1, \ldots, f_\ell\in L^\infty(\mu)$,
        \begin{align*}%\label{E: joint ergodicity}
        \lim_{N\to\infty}\norm{\E^W_{n\in[N]}w(n)\cdot \prod_{j=1}^\ell T_j^{{a_{j}(n)}}f_j}_{L^2(\mu)} = 0
    \end{align*}
    holds whenever $\nnorm{f_j}_{s,T_j} = 0$ for some $s\in\N$.
        \end{enumerate}
    %     We call sequences $a_1, \ldots, a_\ell:\N\to\Z$ \emph{$W$-jointly ergodic} for the system $(X, \CX, \mu,$\! $T_1, \ldots, T_\ell)$ if 
    %     \begin{align}\label{E: joint ergodicity}
    %     \lim_{N\to\infty}\norm{\E^W_{n\in[N]}w(n)\cdot \prod_{j=1}^\ell T_j^{{a_{j}(n)}}f_j - \prod_{j=1}^\ell \int f_j\, d\mu}_{L^2(\mu)} = 0
    % \end{align}
    % holds for all $f_1, \ldots, f_\ell\in L^\infty(\mu)$.
    % %{For $\ell=1,$ we simply say that $a_1$ is \emph{ergodic} for $(X, \CX, \mu, T_1)$ (and respectively we call $(T_1^{{a_1(n)}})_n$ \emph{ergodic}).}
    \end{definition}
    \begin{definition}
        A sequence $a\in\CH$ is \emph{$W$-good} if one of the following holds:
        \begin{enumerate}
            \item either $|a(t)-p(t)|\ll 1$ for some $p\in\R[t]$;
            \item or $|a(t)-p(t)|\gg \log W(t)$ for all $p\in\R[t]$.
        \end{enumerate}
    \end{definition}
    Bergelson, Moreira, and Richter proved the following joint ergodicity result for weighted averages.
    \begin{theorem}%[$W$-joint ergodicity, 
    [{\cite[Theorem B(ii)]{BMR20}}]\label{T: BMR weighted}
        Let $a_1, \ldots, a_\ell\in \CH$, and suppose that any nontrivial linear combination $b$ of elements in
        \begin{align}\label{E: bigger combinations}
            \{a_j^{(k)}:\; 1\leq j\leq \ell,\; k\geq 0\}
            %b(t) = c_1 a_1^{(k_1)} + \cdots + c_\ell a_\ell^{(k_\ell)}
        \end{align}
        (where $a^{(k)}$ is the $k$-th derivative of $a$) is $W$-good and satisfies
        \begin{align}\label{E: distinct from poly}
            \lim\limits_{t\to\infty}|b(t)-p(t)|=\infty\quad \textrm{for\; every}\quad p\in\Q[t].
        \end{align}
                 Then $a_1, \ldots, a_\ell$ are $W$-jointly ergodic for every ergodic system $(X, \CX, \mu, T)$. 
    \end{theorem}
        The proof of Theorem \ref{T: BMR weighted} in \cite{BMR20} follows the Old Joint Ergodicity Strategy and rests on two major ingredients: seminorm estimates and an equidistribution-on-nilsystems result of Richter \cite{R22}. It is the latter result that requires all $b$'s of the form \eqref{E: bigger combinations} to be $W$-good and satisfy \eqref{E: distinct from poly}. The work \cite{BMR20} is one of the last papers that deduces joint ergodicity using Old Joint Ergodicity Strategy and employs rather formidable manipulations on nilsystems.
        
        Notably, the joint ergodicity results for Hardy sequences from Sections \ref{SSS: Hardy single} and \ref{SSS: Hardy commuting} that use New Joint Ergodicity Strategy make less restrictive assumptions on the sequences. It should therefore be possible to weaken the assumptions in Theorem \ref{T: BMR weighted}. Additionally, in the light of the results from Section \ref{SSS: Hardy commuting} for commuting transformations, it would be welcome to merge the results for weighted averages of a single transformation with those of unweighted averages of commuting transformations. We propose several problems in this direction.
    \begin{problem}%[Host-Kra seminorm control for $W$-averages]
        Let $a_1, \ldots, a_\ell\in \CH$ be \emph{$W$-pairwise independent} in that for all distinct $1\leq i, j\leq\ell$ and $c_i, c_j\in\R$ not both zero, we have
        \begin{align*}
            \lim_{t\to\infty}\frac{|c_i a_i(t) + c_j a_j(t)|}{\log W(t)}=\infty;
        \end{align*}
        Show that $a_1, \ldots, a_\ell$ {admit Host-Kra seminorm control for $W$-averages} for any system. 
    \end{problem}
    \begin{problem}%[Invariant factor control for $W$-averages]
    \label{Pr: W-strongly independent}
        Let $a_1, \ldots, a_\ell\in \CH$ be \emph{$W$-strongly independent} in that for all $c_1, \ldots, c_\ell\in\R$ not all zero and for all $p\in\Q[t]$, we have
\begin{equation*}%\label{E: lafrp}
    \lim_{t\to\infty}\frac{|c_1 a_1(t) + \cdots + c_\ell a_\ell(t) - p(t)|}{\log W(t)} = \infty;
\end{equation*}
        Show that $a_1, \ldots, a_\ell$ are {controlled by the invariant factor for $W$-averages} for any system.
        %{are $W$-jointly ergodic} for any system $(X, \CX, \mu,$\! $T_1, \ldots, T_\ell)$ with $T_1, \ldots, T_\ell$ ergodic. 
    \end{problem}    
For the sake of Problem \ref{Pr: W-strongly independent}, we observe that the joint ergodicity criteria can be rather trivially extend to weighted averages once the definitions are appropriately modified. We refer the reader to \cite[Section 2.7]{FrKu22a} for a sample of such results.

    \subsection{Primes}
    A lot of multiple recurrence, convergence, and joint ergodicity results can be refined by allowing the averaging parameter $n$ to run over primes $\P=\{p_1<p_2<\cdots\}$ or shifted primes $\P-1$.\footnote{In the ergodic literature, the name \emph{shifted primes} almost always refers to the sets $\P-1$ or $\P+1$ since these are the only shifts of primes that have \emph{good divisibility property}, i.e. their intersection with $q\Z$ for every $q\in\N$ has positive relative density. In particular, these are the only prime shifts that are sets of recurrence.}
    %equidistribute in residue classes $a+q\Z$ with $\gcd(a,q) = 1$.} 
     The underlying idea is to replace the average along primes by the Ces\`aro average weighed by the von Mangoldt function; indeed, from the prime number theorem, we have
    \begin{align*}
        \lim_{N\to\infty}\norm{\E_{n\in[N]\cap \P}A(n)-\E_{n\in[N]}\Lambda(n) A(n)}_{L^2(\mu)} = 0
        %\lim_{N\to\infty}\E_{n\in[N]\cap \P}A(n) =\lim_{N\to\infty}\E_{n\in[N]}\Lambda(n) A(n),
    \end{align*}
    for any uniformly bounded $A:\N\to L^\infty(\mu)$. We then want to use the Gowers uniformity of $\Lambda-1$ to compare the latter with $\lim\limits_{N\to\infty}\E_{n\in[N]}A(n)$. In reality, $\Lambda-1$ is not Gowers uniform; we instead need to work with appropriately modified von Mangoldt functions $$\Lambda_{w,b}(n) := \frac{\phi(W)}{W}\Lambda(Wn+b),$$ where $w,b\in\N$, $W :=\prod\limits_{\substack{p\in\P\colon p\leq w}}p$ is the $w$-th primorial, and $\phi$ is the Euler totient function. This modification takes care of the irregularities of distribution of primes modulo small primes. It is then a deep result of Green-Tao \cite{GT08b, GT10b, GT12b} that in the limit $w\to\infty$, the functions $\Lambda_{w,b}-1$ have vanishing Gowers norms of any degree. These ideas originating in additive combinatorics have been transplanted to ergodic theory by Frantzikinakis-Host-Kra \cite{FHK07, FHK10}, Wooley-Ziegler \cite{WZ12}, Bergelson-Leibman-Ziegler \cite{BLZ08}, Koutsogiannis \cite{Ko18b}, and Koutsogiannis-Karageorgos \cite{KK19} to prove the polynomial Szemer\'edi theorem along shifted primes.
    
    %These ideas date back to works of Green \cite{???} and Green-Tao \cite{GT08b, GT10b} in additive combinatorics, and they have been adapted to ergodic theory by Frantzikinakis-Host-Kra \cite{FHK07, FHK10} as well as Wooley-Ziegler \cite{WZ12}.
    
    The study of averages along primes then breaks down into two steps: %goal is prove an estimate of the form
    \begin{enumerate}
        \item the uniformity estimate 
    \begin{align}\label{E: uniformity estimate for modified Lambda}
        \lim_{w\to\infty}\lim_{N\to\infty}\max_{\substack{b\in[W]\colon\\ \gcd(b,W) = 1}}\norm{\E_{n\in[N]}(\Lambda_{w,b}(n)-1)A(Wn+b)}_{L^2(\mu)} = 0;
    \end{align}
    \item the examination of 
    \begin{align*}
        \lim_{N\to\infty}\E_{n\in[N]}A(Wn+b)
    \end{align*}
    for $w,b\in\N$.
    \end{enumerate}
    
    There are two renditions of this general strategy, depending on the purpose and properties of the sequences $a_1, \ldots, a_\ell$:
    \begin{enumerate}
        \item if the sequences $a_1, \ldots, a_\ell$ are good for irrational equidistribution and have good divisibility property (or satisfy some variant of these two properties), and we look for new multiple recurrence results, then we study the averages with $n$ restricted to $\P-1$ (or $\P+1$) since the restrictions of $a_1, \ldots, a_\ell$ to this set typically inherit the claimed properties;
        \item in the other cases, e.g.:
            \begin{enumerate}
                \item if the sequences $a_1, \ldots, a_\ell$ are good for equidistribution;
                \item or if the sequences $a_1, \ldots, a_\ell$ are good for irrational equidistribution and we are interested in norm convergence or limiting identities,
            \end{enumerate}
            then we study the averages with $n$ restricted to $\P$; however, the same argument normally works for any shift $\P+c$.
    \end{enumerate}
    %\BK{(See if this makes sense later on)}
    
    We now present joint ergodicity results for sequences evaluated at primes.
    %in each of these two directions, starting with the latter.
    \subsubsection{Sequences along  primes that are good for equidistribution}

    Using the Old Joint Ergodicity Strategy, Karageorgos and Koutsogiannis proved the extension of their Theorem \ref{T: KK strongly independent real polys} along primes.
\begin{theorem}[{\cite[Theorem 2.14]{KK19}}]
    Let $a_1, \ldots, a_\ell\in\R[t]$ be  strongly independent real polynomials. 
    Then the sequences $a_1(p_n), \ldots, a_\ell(p_n)$ are controlled by the invariant factor for any system $(X, \CX, \mu, T)$. In particular, they are jointly ergodic for any ergodic system $(X, \CX, \mu, T)$.
    %Then the sequences $a_1(p_n), \ldots, a_\ell(p_n)$ are jointly ergodic for any ergodic system $(X, \CX, \mu, T)$ (and hence are controlled by the invariant factor for any system $(X, \CX, \mu, T)$).
\end{theorem}
    
    On the other hand, one of the first applications of the New Joint Ergodicity Strategy was the following result of Frantzikinakis for fractional polynomials along primes.
    \begin{theorem}[{\cite[Theorem 1.1]{Fr22}}]\label{T: Fr fractional polynomials along primes}
        Let $a_1, \ldots, a_\ell$ be linearly independent fractional polynomials $a_j(t) = \sum_{i=1}^d \alpha_{ji}t^{b_i}$ with $\alpha_{ji}\in\R$ and $b_i$'s positive nonintegers. %Then the sequences $a_1(p_n), \ldots, a_\ell(p_n)$ are jointly ergodic for any ergodic system $(X, \CX, \mu, T)$ (and hence are controlled by the invariant factor for any system $(X, \CX, \mu, T)$).
            Then the sequences $a_1(p_n), \ldots, a_\ell(p_n)$ are controlled by the invariant factor for any system $(X, \CX, \mu, T)$. In particular, they are jointly ergodic for any ergodic system $(X, \CX, \mu, T)$.
    \end{theorem}
    As an example, Theorem \ref{T: Fr fractional polynomials along primes} yields the identity
    \begin{align*}
        \lim_{N\to\infty}\norm{\E_{n\in[N]\cap\P}T^{\sfloor{n^{b_1}}}f_1 \cdot T^{\sfloor{n^{b_2}}}f_2 - \int f_1\; d\mu \cdot \int f_2\; d\mu}_{L^2(\mu)} = 0
    \end{align*}
    for any noninteger $0<b_1<b_2$ whenever $(X, \CX, \mu, T)$ is ergodic and $f_1, f_2\in L^\infty(\mu)$.

    The same paper inquired about extending Theorem \ref{T: Fr fractional polynomials along primes} to general strongly independent Hardy sequences. This problem was resolved by Koutsogiannis and Tsinas \cite{KoTs23} for the single transformation and by Donoso, Koutsogiannis, Sun, Tsinas, and the author \cite{DKKST24} for commuting transformations.
    \begin{theorem}%[Strongly independent Hardy sequences along primes for single transformation, 
    [{\cite[Theorem 1.4]{KoTs23}}]\label{T: KoTs strongly independent along primes}
        Let $a_1, \ldots, a_\ell\in \CH$ be  strongly independent Hardy sequences.     Then the sequences $a_1(p_n), \ldots, a_\ell(p_n)$ are controlled by the invariant factor for any system $(X, \CX, \mu, T)$. In particular, they are jointly ergodic for any ergodic system $(X, \CX, \mu, T)$.
        %Then the sequences $a_1(p_n), \ldots, a_\ell(p_n)$ are jointly ergodic for any ergodic system $(X, \CX, \mu, T)$ (and hence are controlled by the invariant factor for any system $(X, \CX, \mu, T)$).
    \end{theorem}
    \begin{theorem}%[Strongly independent Hardy sequences along primes for commuting transformations, 
    [{\cite[Theorem 1.12(i)]{DKKST24}}]\label{T: DKKST strongly independent along primes}
        Let $a_1, \ldots, a_\ell\in \CH$ be  strongly independent Hardy sequences. Then the sequences $a_1(p_n), \ldots, a_\ell(p_n)$ are controlled by the invariant factor. In particular, they are jointly ergodic for all systems $(X, \CX, \mu,$\! $T_1, \ldots, T_\ell)$ with $T_1, \ldots, T_\ell$ ergodic.  
    \end{theorem}
    For completeness, we note that the work of Koutsogiannis and Tsinas also proves the extension of Theorem \ref{T: Fr distinct-growth Hardy} along primes \cite[Theorem 1.5]{KoTs23}, a special case of Theorem \ref{T: DKKST strongly independent along primes}.

    While Theorems \ref{T: KoTs strongly independent along primes} and \ref{T: DKKST strongly independent along primes} extend Theorems \ref{T: Ts strongly independent} and \ref{T: DKKST strongly independent Hardy} to averages along primes, the following result does the same with Theorem \ref{T: DKKST pairwise independent Hardy}.
    \begin{theorem}%[Joint ergodicity for pairwise independent Hardy sequences and weakly mixing transformations along primes, 
    [{\cite[Theorem 1.12(ii)]{DKKST24}}]\label{T: DKKST pairwise independent Hardy along primes}
    Let $a_1, \ldots, a_\ell\in \CH$ be pairwise independent Hardy sequences. Then the sequences $a_1(p_n), \ldots, a_\ell(p_n)$:
    \begin{enumerate}
        \item admit Host-Kra seminorm control;
        \item are jointly ergodic for any system $(X, \CX, \mu,$\! $T_1, \ldots, T_\ell)$ with $T_1, \ldots, T_\ell$ weakly mixing.
    \end{enumerate}    
    \end{theorem}

    All the aforementioned results follow from a beautiful transference principle of Koutsogiannis and Tsinas. In what follows, we use the following definition, examples of which are $t + \sqrt{2},\; t^2 + 1/t$.
    \begin{definition}[Almost rational polynomials]\label{D: almost rational poly}
        We call $a\in\CH$ an \emph{almost rational polynomial} if $\lim\limits_{t\to\infty}|a(t) - p(t)| = 0$ for some $p\in\Q[t]+\R$.
    \end{definition}
    \begin{theorem}[Transference principle for Hardy sequences along primes {\cite[Theorem 1.2]{KoTs23}}]\label{T: KoTs transference}
        Let $a_1, \ldots, a_\ell\in \CH$ be Hardy sequences, each one of which either stays logarithmically away from rational polynomials or is an almost rational polynomial.
        % satisfies
        % \begin{align*}
        %     \lim_{t\to\infty}|a(t) - p(t)| = 0
        % \end{align*}
        % for some $p\in\Q[t]+\R$.
        Let $(X, \CX, \mu,$\! $T_1, \ldots, T_\ell)$ be a system and $f_1, \ldots, f_\ell\in L^\infty(\mu)$. Suppose that for any $w,b\in\N$, the averages
        \begin{align*}
            %A_{N,W,b}(f_1, \ldots, f_\ell):=
            \E_{n\in[N]} T_1^{a_1(Wn+b)}f_1\cdots T_\ell^{a_\ell(Wn+b)}f_\ell
        \end{align*}
        converge to the same limit $F$ in $L^2(\mu)$ (where $W =\prod\limits_{\substack{p\in\P\colon p\leq w}}p$ as before). Then 
        \begin{align*}
            \E_{n\in[N]\cap\P} T_1^{a_1(n)}f_1\cdots T_\ell^{a_\ell(n)}f_\ell
        \end{align*}
        also converges to $F$ in $L^2(\mu)$.
    \end{theorem}
    We note that in the statement of \cite[Theorem 1.2]{KoTs23}, the convergence of the averages along $Wn+b$ to $F$ is required for all $W\in\N$, not just primorials (products of consecutive primes). However, the inspection of the proof shows that it suffices for if to hold for primorials.

    Furthermore, \cite[Theorem 1.2]{KoTs23} holds for more general averages in which each term $T_j^{a_j(n)}f_j$ is replaced by a product $T_1^{a_{j1}(n)}\cdots T_\ell^{a_{j\ell}(n)}f_j$.  
    
    \subsubsection{Integer polynomials along primes}
    A transference principle in the spirit of Theorem \ref{T: KoTs transference} also gives extensions of Theorems \ref{T: FrKu affine independent} and \ref{T: FrKu pairwise affine independent} to averages of integer polynomials along primes.
    \begin{theorem}[{\cite[Theorem 2.13]{FrKu22a}}]\label{T: FrKu affine independent along primes}
    Let $a_1, \ldots, a_\ell\in\Z[t]$ be affinely independent polynomials. Then the sequences $a_1(p_n), \ldots, a_\ell(p_n)$
    are controlled by the rational Kronecker factor. In particular, they are jointly ergodic for any system $(X, \CX, \mu,$\! $T_1, \ldots, T_\ell)$ with $T_1, \ldots, T_\ell$ totally ergodic.
    %Let $a_1, \ldots, a_\ell\in\Z[t]$ be affinely independent polynomials, and let $(X, \CX, \mu,$\! $T_1, \ldots, T_\ell)$ be a system with $T_1, \ldots, T_\ell$ totally ergodic. Then $(T^{a_1(n)}, \ldots, T^{a_\ell(n)})_n$ is jointly ergodic.
\end{theorem}
\begin{theorem}\label{T: FrKu pairwise affine independent along primes}
    Let $a_1, \ldots, a_\ell\in\Z[t]$ be pairwise independent polynomials. Then the sequences $a_1(p_n), \ldots, a_\ell(p_n)$:
    \begin{enumerate}
        \item\label{i: poly primes HK seminorm control} admit Host-Kra seminorm control;
        \item\label{i: poly primes weakly mixing} are jointly ergodic for any system $(X, \CX, \mu,$\! $T_1, \ldots, T_\ell)$ with $T_1, \ldots, T_\ell$ weakly mixing.
    \end{enumerate}
    %good for Host-Kra seminorm control. Hence they are jointly ergodic for any system $(X, \CX, \mu,$\! $T_1, \ldots, T_\ell)$ with $T_1, \ldots, T_\ell$ weakly mixing.
    %Let $a_1, \ldots, a_\ell\in\Z[t]$ be affinely independent polynomials, and let $(X, \CX, \mu,$\! $T_1, \ldots, T_\ell)$ be a system with $T_1, \ldots, T_\ell$ totally ergodic. Then $(T^{a_1(n)}, \ldots, T^{a_\ell(n)})_n$ is jointly ergodic.
\end{theorem}
The second of these results is not explicitly stated in \cite{FrKu22a}, but following the remark below \cite[Theorem 2.13]{FrKu22a}, it can be deduced in a straightforward way from Theorem \ref{T: FrKu pairwise affine independent} and \ref{T: KoTs transference}.

    \subsection{Generalized polynomials}\label{SS: generalized polynomials}

Another natural class of sequences of polynomial growth that has received significant attention in dynamics are generalized polynomials.
\begin{definition}
    A function $a:\N\to\R$ is a \emph{generalized polynomial} if it can be formed by finitely many applications of the symbols $+,-,\cdot,\floor{\cdot}$ (and hence also the fractional part $\{\cdot\}$) to constants and the identity function $b(t) =t$.
\end{definition}
Examples of generalized polynomials include
\begin{align*}
    t,\; \sqrt{2}t^2+\sqrt{3}t,\; \pi t\sfloor{\sqrt{5}t},\; e\sfloor{\sqrt{5}t^6 + \sfloor{\sqrt{7}t}}\sfloor{\sqrt{6}t},\; \sqrt{6}t\{\sqrt{2}t^2\}.
\end{align*}
    On the surface, generalized polynomials look similar to polynomials, but they may differ considerably. For instance, there exist nonconstant generalized polynomials that are bounded (e.g. $\rem{t/2}$).

    One reason behind the interest in generalized polynomials is that they arise naturally in higher order Fourier analysis. For instance, the function $n\mapsto e(\sqrt{2}n\floor{\sqrt{3}n})$ can be realized as a 2-step nilsequence on the Heisenberg nilmanifold. See e.g. the paper of Bergelson and Leibman \cite{BL07} for more results concerning the connection between these two classes.

    In contrast to polynomials and Hardy sequences, the structure of multiple ergodic averages along generalized polynomials is relatively unexplored. We therefore pose a number of problems to breathe life into this line of research. For simplicity, we state these problems for a single transformation; however, every single one of them can be extended to commuting transformations.
    The first problem asks about the norm convergence of such averages.
    \begin{problem}[Norm convergence for generalized polynomials {\cite[Problem 14]{Fr16}}]\label{Pr: norm convergence gen polys}
        Let $a_1, \ldots, a_\ell$ be generalized polynomials, $(X, \CX, \mu, T)$ be a system, and $f_1, \ldots, f_\ell\in L^\infty(\mu)$. Does the average
        \begin{align*}
            \E_{n\in[N]}T^{\sfloor{a_1(n)}}f_1\cdots T^{\sfloor{a_\ell(n)}}f_\ell
        \end{align*}
        converge in $L^2(\mu)$?
    \end{problem}
    While the case $\ell=1$ of Problem \ref{Pr: norm convergence gen polys} follows from the spectral theorem and the characterization of bounded generalized polynomials from \cite{BL07}, the problem remains even open for $\ell=2$ and $T$ weakly mixing systems.
    
    The next problem aims at finding the correct notion of pairwise independence for generalized polynomials.
    \begin{problem}[Seminorm control for generalized polynomials] %in the single-transformation case]
        %Let $a_1, \ldots, a_\ell$ be generalized polynomials. 
        Characterize those generalized polynomials that admit Host-Kra seminorm control for every system $(X, \CX, \mu, T)$.
    \end{problem}
    % \begin{problem}[Seminorm control for generalized polynomials in the commuting case]
    %     Let $a_1, \ldots, a_\ell$ be generalized polynomials. Characterize those generalized polynomials that admit Host-Kra seminorm control for every system $(X, \CX, \mu,$\! $T_1, \ldots, T_\ell)$.
    % \end{problem}
    The next two problems ask for the right notions of independence and strong independence for generalized polynomials. Ideally, the requested characterization would take the form of a condition on nontrivial linear combinations of the sequences.
    \begin{problem}[Joint ergodicity for strongly independent generalized polynomials] %in the single-transformation case]
        %Let $a_1, \ldots, a_\ell$ be generalized polynomials. 
        Characterize those generalized polynomials that are jointly ergodic for every ergodic system $(X, \CX, \mu, T)$.
    \end{problem}
    % \begin{problem}[Invariant factor control for strongly independent generalized polynomials] in the commuting case]
    %     Let $a_1, \ldots, a_\ell$ be generalized polynomials. Characterize those generalized polynomials that are controlled by the invariant factor for every system $(X, \CX, \mu,$\! $T_1, \ldots, T_\ell)$.
    % \end{problem}
     \begin{problem}[Joint ergodicity for independent generalized polynomials] %in the single-transformation case]
        %Let $a_1, \ldots, a_\ell$ be generalized polynomials. 
        Characterize those generalized polynomials that are controlled by the rational Kronecker factor for every system $(X, \CX, \mu, T)$.
    \end{problem}
    % \begin{problem}[Joint ergodicity for independent generalized polynomials in the commuting case]
    %     Let $a_1, \ldots, a_\ell$ be generalized polynomials. Characterize those generalized polynomials that are controlled by the rational Kronecker factor for every system $(X, \CX, \mu,$\! $T_1, \ldots, T_\ell)$.
    % \end{problem}

    \subsection{Generalized Hardy sequences}

    Just like generalized polynomials are formed by adding $\sfloor{\cdot}$ to three operations defining polynomials, we can consider a generalization of Hardy sequences that involves sequences like
    \begin{align*}
        \sfloor{n^b}^2,\; \sfloor{n \log n}^{3/2},\; \exp(\sfloor{\log n}).
    \end{align*}
    We refrain from formally defining such a class; we will however refer to it informally as \emph{generalized Hardy sequences}. Sequences like $\sfloor{n^b}^2$ naturally appear in the intermediate steps of PET arguments for Hardy sequences \cite{DKKST24, Ts22}; their equidistribution as well as (single) recurrence and convergence properties were studied by Frantzikinakis \cite{Fr09}. Several problems on such sequences have been posed by Tsinas \cite{Ts22} but have not received much attention so far. We restate them here, adding several of our own. We invite the reader to consult \cite[Section 1.5]{Ts22} for more discussion.
    \begin{problem}[Norm convergence for generalized Hardy sequences {\cite[Conjecture 1]{Ts22}}]
        Let $0<b_1 < \cdots < b_\ell$, $(X, \CX, \mu,T)$ be a system, and $f_1, \ldots, f_\ell\in L^\infty(\mu)$. Does the average
            \begin{align*}
                \E_{n\in[N]}T^{\sfloor{n^{b_1}}^2}f_1\cdots T^{\sfloor{n^{b_\ell}}^2}f_\ell 
            \end{align*}
            converge in $L^2(\mu)$?
    \end{problem}
    \begin{problem}[Seminorm control for generalized Hardy sequences]
        Let $0<b_1 < \cdots < b_\ell$. Do the sequences $\sfloor{n^{b_1}}^2, \cdots, \sfloor{n^{b_\ell}}^2$ admit Host-Kra seminorm control?
    \end{problem}
    \begin{problem}[Joint ergodicity for generalized Hardy sequences and totally ergodic systems {\cite[Conjecture 2]{Ts22}}]
        Let $0<b_1 < \cdots < b_\ell$. Are the sequences $\sfloor{n^{b_1}}^2, \cdots \sfloor{n^{b_\ell}}^2$ jointly ergodic for any totally ergodic system $(X, \CX, \mu, T)$?
    \end{problem}
    \begin{problem}[Joint ergodicity for generalized Hardy sequences and ergodic systems]
        Let $0<b_1 < \cdots < b_\ell$ be not in $\frac{1}{2}\Z$. Are the sequences $\sfloor{n^{b_1}}^2, \cdots \sfloor{n^{b_\ell}}^2$ jointly ergodic for any ergodic system $(X, \CX, \mu, T)$?
    \end{problem}
    \begin{problem}[Multiple recurrence for generalized Hardy sequences {\cite[Conjecture 3]{Ts22}}]
        Let $0<b_1 < \cdots < b_\ell$, $(X, \CX, \mu, T)$ be a system, and $E\in\CX$ with $\mu(E)>0$. Does there exist $n\in\N$ for which 
            \begin{align*}
               \mu(E\cap T^{-\sfloor{n^{b_1}}^2}E\cap\cdots \cap T^{-\sfloor{n^{b_\ell}}^2}E)>0?
            \end{align*}
    \end{problem}
    % \begin{problem}
    %     Let $0<b_1 < \cdots < b_\ell$.
    %     \begin{enumerate}
    %         \item (Norm convergence) Does the average
    %         \begin{align*}
    %             \E_{n\in[N]}T^{\sfloor{n^{b_1}}^2}f_1\cdots T^{\sfloor{n^{b_\ell}}^2}f_\ell 
    %         \end{align*}
    %         converge in $L^2(\mu)$ for any system $(X, \CX, \mu,T)$ and any $f_1, \ldots, f_\ell\in L^\infty(\mu)$?
    %         \item (Seminorm control) Are the sequences $\sfloor{n^{b_1}}^2, \cdots \sfloor{n^{b_\ell}}^2$ good for seminorm control?
    %         \item (Joint ergodicity for totally ergodic systems) Are the sequences $\sfloor{n^{b_1}}^2, \cdots \sfloor{n^{b_\ell}}^2$ jointly ergodic for any totally ergodic system $(X, \CX, \mu, T)$?
    %         \item (Joint ergodicity for ergodic systems) If $2b_j\notin \Z$, are the sequences $\sfloor{n^{b_1}}^2, \cdots \sfloor{n^{b_\ell}}^2$ jointly ergodic for any ergodic system $(X, \CX, \mu, T)$?
    %         \item (Multiple recurrence) If $(X, \CX, \mu, T)$ is a system and $E\in\CX$ has $\mu(E)>0$, does there exist $n\in\N$ for which 
    %         \begin{align*}
    %            \mu(E\cap T^{-\sfloor{n^{b_1}}^2}E\cap\cdots \cap T^{-\sfloor{n^{b_\ell}}^2}E)>0?
    %         \end{align*}
    %     \end{enumerate}
    % \end{problem}

    \subsection{Tempered functions}
    Another class of functions of polynomial growth studied in ergodic theory is the family of tempered functions.
\begin{definition}[Tempered functions]
    Let $d\in\N_0$. A function $a:[t_0,\infty)\to\R$ is called \emph{tempered} of \emph{degree $d$} if it satisfies the following conditions:
    \begin{enumerate}
        \item $a$ is $(d+1)$-times continuously differentiable;
        \item $a^{(d+1)}(t)$ tends monotonically to 0 as $t\to\infty$;
        \item $\lim\limits_{t\to\infty} t |a^{(d+1)}(t)|=\infty$.
        %\item $\lim\limits_{t\to\infty}\frac{|a(t)|}{t^d} = \infty$;
        %\item $\lim\limits_{t\to\infty}\frac{|a(t)|}{t^{d+1}} = 0$;
    \end{enumerate}
    A tempered function of degree $0$ is called \emph{Fej\'er}.
\end{definition}
Just like with Hardy functions, we identify two tempered functions if they eventually agree; hence when we speak of tempered functions, we really mean \emph{germs} of functions. 

Despite superficial similarity and nontrivial overlap, tempered functions differ from Hardy functions in several significant ways.
%Note that in contrast to Hardy functions, 
Tempered functions need not be smooth, nor do their ratios converge in general. No polynomial is tempered. Tempered functions are not closed under arithmetic operations or differentiation (for instance, $t^{3/2}+t,\; t^{3/2}$ are both tempered, but their difference is not).  
To avoid various problematic cases, we typically restrict to special classes of tempered functions, such as the class $\CT$ defined below:
\begin{align*}
    \CR &:= \rem{a\in C^\infty\colon \lim_{t\to\infty}\frac{ta^{(i+1)}(t)}{a^{(i)}(t)}\in\R\; \textrm{for all}\; i\in\N_0};\\
    %\CT_i &:= \rem{a\in \CR\colon \lim_{t\to\infty}\frac{ta'(t)}{a(t)}=\alpha\; \textrm{for\; some}\; i< \alpha\leq i+1\; \textrm{and}\; \lim_{t\to\infty}a^{(i+1)}(t) = 0};\\
    \CT_i &:= \rem{a\in \CR\colon \lim_{t\to\infty}\frac{ta'(t)}{a(t)}\in (i,i+1]\; \textrm{and}\; \lim_{t\to\infty}a^{(i+1)}(t) = 0};\\
    \CT &:= \bigcup_{i=0}^\infty \CT_i.
\end{align*}
By \cite{BeHK09}, a function $a\in\CT_i$ has degree $i$ and satisfies the growth condition $t^i \log t \prec a(t) \prec t^{i+1}$; hence elements of $\CT$ are nonpolynomial in a quantitatively precise sense. %\footnote{Real polynomials of degree $d$ are also tempered of degree $d$; they are just not in $\CT$.}

Examples of functions in $\CT$ include $t^b (\log t)^c$ for $b\in\R_+\backslash\N$, %or for $b\in\N, c>1$, 
as well as functions with sufficiently well-controlled oscillation, as in the example below.
\begin{example}
Let $g(t) = t^b(2 + \cos\sqrt{\log t})$ for $b\in\R_+\backslash\N$. Then $$g^{(i)}(t) =(1+o(1))c_{b,i} t^{b-i}(2 + \cos\sqrt{\log t})$$
for $c_{b,i}:=b(b-1)\cdots(b - (i-1))$, which for every $i\in\N_0$ is nonzero by the assumption $b\notin\N$.
Since the oscillatory factor is bounded by 1, we deduce that this function is in $\CT_i$ for $i = \floor{b}= \ceil{b}-1$.
\end{example}

% \BK{Probably remove the statement below. Check Nikos' joint ergodicity before}
% If we want to work inside a class of tempered functions that include (real) polynomials, we can consider
% \begin{align*}
%     \CP_i &:= \rem{a\in C^\infty\colon \lim_{t\to\infty}a^{(i+1)}(t)\in\R\backslash\{0\}\; \textrm{and}\; \lim_{t\to\infty}t^j a^{(i+j+1)}(t) = 0\; \textrm{for all}\; j\in\N};\\
%     \CP &:= \bigcup_{i=0}^\infty \CP_i.
% \end{align*}
% Note that any degree-$i$ real polynomial is in $\CP_i$.

Joint ergodicity for multiple ergodic averages has been studied by Bergelson and H{\aa}land-Knutson \cite{BeHK09}, and later by Koutsogiannis \cite{Kouts21}. They proved a number of joint ergodicity results by imposing various conditions on the tempered functions, some under weakly mixing and some under ergodic assumptions. We will list here samples of their results, directing the reader to both papers for more general statements.

The first of these results can be seen as an extension of Furstenberg's joint ergodicity result for weakly mixing systems (Theorem \ref{T: Furstenberg averages weakly mixing}).
\begin{theorem}[{\cite[Corollary 5.14(i)]{BeHK09}}]
    Let $a$ be a tempered function and $0<c_1<\cdots <c_\ell$ be real numbers. Then the sequences $c_1 a(n), \cdots, c_\ell a(n)$ are jointly ergodic for any weakly mixing system $(X, \CX, \mu, T)$.
\end{theorem}

The second one concerns tempered functions of distinct growth. %The stronger the assumptions, we get different statements.
\begin{theorem}[{\cite[Proposition 2.5, Corollary 2.6, Corollary 2.7]{Kouts21}}]\label{T: Ko tempered}
    Let $a_1, \ldots, a_\ell\in \CT$ be tempered functions with the growth conditions $a_1 \prec \cdots \prec a_\ell$.
    \begin{enumerate}
        \item Then $a_1, \ldots, a_\ell$ admit Host-Kra seminorm control. % for any system $(X, \CX, \mu,$\! $T_1, \ldots, T_\ell)$. 
        In particular, they are jointly ergodic for any system $(X, \CX, \mu,$\! $T_1, \ldots, T_\ell)$ with $T_1, \ldots, T_\ell$ weakly mixing. 
        \item\label{i: Ko tempered jointly ergodic} Suppose that any nontrivial linear combination of $a_1, \ldots, a_\ell$ is in $\CT$. Then $a_1, \ldots, a_\ell$ are controlled by the invariant factor. %for any system $(X, \CX, \mu,$\! $T_1, \ldots, T_\ell)$. 
        In particular, they are jointly ergodic for any system $(X, \CX, \mu,$\! $T_1, \ldots, T_\ell)$ with $T_1, \ldots, T_\ell$ ergodic. 
    \end{enumerate}
\end{theorem}
The extra assumption in the second case is needed to invoke equidistribution results for tempered functions. As the example below shows, the result may fail without it, and hence the fact that $\CT$ is not closed under taking linear combinations (as opposed to $\CH$) has real consequences.  
%Note that whereas the Hardy class $\CH$ is closed under taking linear combinations, the same is not true for $\CT$.
\begin{example}
    Let
    \begin{align*}
     a_1(t) = t^{3/2}(2 + \cos\sqrt{\log t})+2t\quad \textrm{and}\quad a_2(t) = t^{3/2}(2 + \cos\sqrt{\log t}). 
    \end{align*}
    %$a_1(t) = t^{3/2}(2 + \cos\sqrt{\log t})+2t$ and $a_2(t) = t^{3/2}(2 + \cos\sqrt{\log t})$. 
    While they are both tempered of degree 1, their linear combination $a_1(t)-a_2(t) = 2t$ is not tempered.
    Hence they do not satisfy the assumptions of Theorem \ref{T: Ko tempered}\eqref{i: Ko tempered jointly ergodic}, nor, as we see shortly, its conclusion. Indeed, on taking $T_1 = T_2$ to be the map $x\mapsto x +1$ on $\Z/2\Z$ and letting $f_1(x)=\overline{f_2}(x) = (-1)^x$, we see that
    \begin{align*}
        %\lim_{N\to\infty}\E_{n\in[N]}
        f_1(T_1^{\sfloor{a_1(n)}}x)\cdot f_2(T_2^{\sfloor{a_2(n)}}x) = 1
    \end{align*}
    holds for all $x\in \Z/2\Z$, $n\in\N$, and hence the pointwise limit of the Ces\`aro average is 1.
    By contrast, $\int f_1\; d\mu  = \int f_2\; d\mu = 0$; hence $a_1, a_2$ are not jointly ergodic for this system. %This shows that some additional assumption in the second part of Theorem \ref{T: Ko tempered} compared to the first part is needed.
\end{example}

For a single transformation, and whenever $a_1, \ldots, a_\ell$ have distinct positive degrees, the first part of Theorem \ref{T: Ko tempered} was covered by \cite[Corollary 5.14(iii)]{BeHK09}

\begin{theorem}[{\cite[Theorem 2.4]{Kouts21}, \cite[Theorem 7.3]{BeHK09}}]
    Let $a_1, \ldots, a_\ell\in \CT$ be Fej\'er functions with the growth conditions $a_1 \prec \cdots \prec a_\ell$. Suppose moreover that the ratios $a'_j(t)/a_{j+1}'(t)$ are eventually monotone for every $1\leq j < \ell$. Then $a_1, \ldots, a_\ell$ are controlled by the invariant factor. %for any system $(X, \CX, \mu,$\! $T_1, \ldots, T_\ell)$. 
    In particular, they are jointly ergodic for any system $(X, \CX, \mu,$\! $T_1, \ldots, T_\ell)$ with $T_1, \ldots, T_\ell$ ergodic. 
\end{theorem}
The conclusion for $T_1, \ldots, T_\ell$ ergodic was first obtained by Bergelson and H{\aa}land-Knutson; Koutsogiannis then removed the ergodicity assumption.

In the light of recent progress on polynomials and Hardy sequences outline in previous sections, it is natural to inquire whether the assumptions of different growth in the results above can be weakened to some variants of independence and pairwise independence. 
\begin{problem}
    Find a reasonable notion of ``pairwise independence'' on $\CT$ so that all pairwise independent functions $a_1, \ldots, a_\ell\in\CT$ admit Host-Kra seminorm control (and hence are jointly ergodic for every system $(X, \CX, \mu,$\! $T_1, \ldots, T_\ell)$ with $T_1, \ldots, T_\ell$ weakly mixing).
\end{problem}
\begin{problem}
    Find a reasonable notion of ``independence'' on $\CT$ so that all independent functions $a_1, \ldots, a_\ell\in\CT$ are controlled by the invariant factor (and hence are jointly ergodic for every system $(X, \CX, \mu,$\! $T_1, \ldots, T_\ell)$ with $T_1, \ldots, T_\ell$ ergodic).
\end{problem}
%\BK{Open problems: replace different growth by pairwise independence and independence}
% BHK:
% \begin{enumerate}
%     \item tempered along APs and WM (Corollary 5.14(i));
%     \item tempered of different degree and WM (Corollary 5.14(iii));
%     \item Fejer of different rations and ergodic (Theorem 7.3).
% \end{enumerate}

% Kouts:
% \begin{enumerate}
%     \item Fejer of different rations and ergodic, commuting (Theorem 2.4).
%     \item tempered of different growth, WM/ergodic, commuting (Corollary 2.6, 2.8);
% \end{enumerate}

%which gives a very concrete description to the nonpolynomiality of elements of $\CT$. 

% \BK{Examples}

% DKS: JOINT ERGODICITY FOR FUNCTIONS OF POLYNOMIAL GROWTH

% Bergelson-Harald-Knutson

% Koutsogiannis: MULTIPLE ERGODIC AVERAGES FOR
% TEMPERED FUNCTIONS

    \subsection{Sequences of superpolynomial growth and with oscillations}\label{SSS: superpolynomial}

    The scope of methods employed to deliver Host-Kra seminorm estimates is constrained by the growth of the sequence: existing techniques cannot deal with 
    %too slow sequences like $\log$ (and indeed, such sequences are bad), nor can they deal 
    sequences growing faster than polynomials. Differentiating a polynomial finitely many times brings it to a linear one. Much for the same reasons, finitely many steps of the PET induction scheme (see Section \ref{SSS: PET}) applied to an average along ``nice'' sequences of polynomial growth (polynomials, Hardy sequences, tempered functions, etc.) bring it down to an average with linear iterates that can be controlled by Host-Kra or box seminorms as in Examples \ref{Ex: HK example} and \ref{Ex: Host example}. While lacunary sequences tend to be bad for ergodic-theoretic purposes (for instance, the sequence $(2^n\alpha)_n$ is not dense for a set of $\alpha\in\T$ of full Hausdorff dimension, as shown by Erd\"os and Taylor \cite{ET57}\footnote{For example, take any $\alpha$ with no consecutive 1's in its binary expansion. It is easy to see that $2^n\alpha\in[0,3/4)$, as the binary expansion of any $x\in[3/4,1)$ starts with two 1's.}), there is a grey zone of sequences that grow superpolynomially but subexponentially, like $n^{(\log n)^b}$ (for small $b>0$) or $n^{\log\log n}$. In contrast to lacunary sequences, some of these intermediate ones are known to be equidistributed \cite{Ka71}.

    A similar issue arises in the study of sequences with oscillations, like $n \sin n$. While known to be equidistributed \cite{BBK02, BBK95, BK90}, they are not amenable to current techniques for seminorm estimates because their successive derivatives do not exhibit any decay. 

    Combining known equidistribution results for such sequences with Corollary \ref{C: joint ergodicity criteria for nilsystems}, Frantzikinakis %used the single-transformation case of Theorem \ref{T: joint ergodicity criteria} to 
    deduced the following results on ergodic nilsystems (he states them in the single-transformation case, but the exact same reasoning works in the commuting case).
    \begin{theorem}[{\cite[Theorem 1.10]{Fr21}}]\label{T: Fr joint ergodicity for weird sequences}
        %Let $(Y, \CY, \nu, S)$ be an ergodic nilsystem. 
        The following families of sequences are jointly ergodic for any nilsystem $(X, \CX, \mu,$\! $T_1, \ldots, T_\ell)$ with $T_1, \ldots, T_\ell$ ergodic: 
        \begin{enumerate}
            \item $n^{(\log n)^{b_1}}, \ldots, n^{(\log n)^{b_\ell}}$ for any $0<b_1<\cdots < b_\ell < 1/2$;
            \item $a_1(n)\sin n, \ldots, a_\ell(n) \sin n$ for affinely independent $a_1, \ldots, a_\ell\in\Z[t]$;
            \item $n^k\sin(n), n^k\sin(2n) \ldots, n^k \sin(\ell n)$ for any $k,\ell\in\N$.
        \end{enumerate}
    \end{theorem}
    %The key point about nilsystems in \eqref{T: Fr joint ergodicity for weird sequences} is that an $s$-step nilsystem is a $Z_s$-system, meaning that the Host-Kra seminorm of degree $s+1$ is a norm; hence any sequence admits Host-Kra seminorm contol on a nilsystem. It there

    The next problem inquires whether the restriction to nilsystems in Theorem \ref{T: Fr joint ergodicity for weird sequences} can be dropped; in other words, whether the abovementioned families of sequences are amenable to the Host-Kra machinery.
    \begin{problem}[{\cite[Problem 2]{Fr21}}]
        Do the families in Theorem \ref{T: Fr joint ergodicity for weird sequences} admit Host-Kra seminorm control (even in the single-transformation case)?
    \end{problem}
    
\section{Joint ergodicity classification problem}\label{S: classification problem}
    The previous section listed rather general classes of sequences and systems for which we have joint ergodicity; hence it dealt with \emph{sufficient} conditions for this phenomenon. This section concentrates on equivalent characterization of joint ergodicity, i.e. sufficient \emph{and} necessary conditions. 
    
    One such set of necessary and sufficient conditions is provided by Theorem \ref{T: joint ergodicity criteria}. Vaguely speaking, one can view the Host-Kra seminorm control and equidistribution conditions as dealing with the weakly mixing and spectral behaviors respectively. This section focuses on two different conditions, which relate joint ergodicity to ergodicity of simpler actions. Before we state them, we record one definition used profusely throughout this section.
    \begin{definition}[Strongly and weakly ergodic actions]
        A sequence of measure preserving-transformations $(T_n)_n$ on $(X, \CX, \mu)$ is \emph{strongly/weakly ergodic} if
    \begin{align*}
        \lim_{N\to\infty}\E_{n\in[N]}T_n f = \int f\; d\mu
    \end{align*}
    for any $f\in L^\infty(\mu)$ in the strong/weak sense. If we do not specify the mode of convergence, then ``ergodicity'' stands for ``strong ergodicity''.
    \end{definition}
    Here are the conditions that we shall work with. 
    \begin{definition}[Ergodicity conditions]\label{D: ergodicity conditions}
        Let $a_1,\ldots,a_\ell\colon \N\to \Z$ be sequences and $(X, \CX, \mu,$\! $T_1, \ldots, T_\ell)$ be a system. For this system, the sequences satisfy:
    \begin{enumerate}
    \item the \emph{product ergodicity condition} if  $(T_1^{{a_1(n)}}\times \cdots \times T_\ell^{{a_\ell(n)}})_n$ is ergodic for the product system $(X^\ell, \CX^{\otimes \ell}, \mu^\ell)$;
    \item the \emph{difference ergodicity condition} if $(T_i^{{a_i(n)}}T_j^{-{a_j(n)}})_n$ is ergodic for $(X, \CX, \mu)$ for all distinct $1\leq i, j\leq \ell$.
    \end{enumerate} 
    \end{definition}
    Many results defined in this section concern the case when all the iterates are the same. While discussing such results    
        %One case of Problem \ref{Pr: joint ergodicity problem} that received particular interest is when all the sequences are the same. While discussing this case,
    we apply the following convention.
    \begin{convention}
        We say that a property defined for sequences $a_1, \ldots, a_\ell:\N\to\Z$ holds for a sequence $a:\N\to\Z$ if it holds with $a = a_1 = \cdots = a_\ell$.
    \end{convention}
    % In Definition \ref{D: ergodicity conditions} and elsewhere, we say that a sequence of measure preserving-transformations $(T_n)_n$ on $(X, \CX, \mu)$ is \emph{strongly/weakly ergodic} if
    % \begin{align*}
    %     \lim_{N\to\infty}\E_{n\in[N]}T_n f = \int f\; d\mu
    % \end{align*}
    % for any $f\in L^\infty(\mu)$ in the strong/weak sense. If we do not specify the mode of convergence, then ``ergodicity'' stands for ``strong ergodicity''.
    For instance, we often say 
    %If $a:=a_1 = \cdots = a_\ell$, we simply say 
    that $a$ satisfies the product/difference ergodicity condition or is jointly ergodic for a system if the relevant property is satisfied for $a=a_1 = \cdots = a_\ell$. If our sequences are real-valued, then we say that they satisfy the product/difference ergodicity condition if their integer parts do, in accordance with Convention \ref{Conv: real sequences}.

    Theorem \ref{T: BB characterization}, the foundational joint ergodicity result of Berend and Bergelson, can then be restated as follows.
    \begin{theorem}\label{T: BB characterization 2} The sequence $a(n) = n$ is jointly ergodic for a system $(X, \CX, \mu,$\! $T_1, \ldots, T_\ell)$ if and only if it satisfies the product and difference ergodicity conditions for this system. 
    \end{theorem}    
    It is instructive to see how the product and difference ergodicity conditions interact with the conditions present in Theorem \ref{T: joint ergodicity criteria}. To this end, we sketch a modern-day proof of Theorem \ref{T: BB characterization 2} that uses Theorem \ref{T: joint ergodicity criteria} as a black box.
    \begin{proof}%[Sketch of proof]
        We start by relating the product ergodicity condition to being good for equidistribution. It is classical (see e.g. \cite[Chapter 4, Theorem 16]{HK18}) that the space of $T_1\times \cdots \times T_\ell$-invariant functions 
        %$I(T_1\times\cdots\times T_\ell)$ 
        is spanned by functions 
        \begin{align*}
            g_1 \otimes \cdots \otimes g_\ell,
        \end{align*}
        where each $g_j$ is an eigenfunction of $T_j$ with eigenvalue $\beta_j$ subject to the constraint $\beta_1 + \cdots + \beta_\ell = 0$. Now, the function $g_1 \otimes \cdots \otimes g_\ell$ is constant if and only if $\beta_1 = \cdots = \beta_\ell = 0$, from which it follows that $T_1 \times \cdots \times T_\ell$ is ergodic if and only if $a(n) = n$ is good for equidistribution for the system $(X, \CX, \mu,$\! $T_1, \ldots, T_\ell)$. 
        
        We similarly relate Host-Kra seminorm control to the difference ergodicity condition.
%     We will similarly claim that under the assumption that $T_1, \ldots, T_\ell$ are ergodic, Host-Kra seminorm control is equivalent to the transformations $T_iT_j\inv$ being ergodic for distinct $i\neq j$. 
        Recall the seminorm estimate \eqref{E: box seminorm control}, which we restate here for convenience:
        \begin{align*}%\label{E: box seminorm control}
            \limsup_{N\to\infty}\norm{\E_{n\in[N]} T_1^n f_1 \cdots T_\ell^n f_\ell}_{L^2(\mu)} \leq \nnorm{f_\ell}_{T_\ell, T_\ell T_1\inv, \ldots, T_\ell T_{\ell-1}\inv}.
        \end{align*}
        On one hand, if $T_iT_j\inv$ are ergodic, then we can bound the right-hand side by $\nnorm{f_\ell}_{\ell, T_\ell}$, yielding Host-Kra seminorm control (note that the choice of the index $\ell$ here is arbitrary). Hence the product and difference ergodicity conditions together imply joint ergodicity by Theorem \ref{T: joint ergodicity criteria}. On the other hand, if $T_iT_j\inv$ is not ergodic for some $i\neq j$, take a nonconstant $T_iT_j\inv$-invariant $f\in L^\infty(\mu)$. 
        If we set $f_i = f$, $f_j = \overline{f}$, then $T_i^n f_i \cdot T_j^n f_j = T_i^n|f|^2$. Hence
        \begin{align*}
            \lim_{N\to\infty} \abs{\E_{n\in[N]}\int T_i^n f \cdot T_j^n f_j\; d\mu} %= \lim_{N\to\infty} \abs{\E_{n\in[N]}\int T_i^n |f^2| \; d\mu} 
            = \int |f|^2\; d\mu.
        \end{align*}
        %by the mean ergodic theorem and the ergodicity of $T_i$. On the other hand, the er..
        If however $a(n) = n$ is jointly ergodic, then 
                \begin{align*}
            \lim_{N\to\infty} \abs{\E_{n\in[N]}\int T_i^n f \cdot T_j^n f_j\; d\mu} 
            = \abs{\int f\; d\mu}^2.
        \end{align*}
        Since $f$ is assumed to be nonconstant, this is strictly smaller than $\int |f|^2\; d\mu$, yielding a contradiction to joint ergodicity.
        % Since 
        % \begin{align*}
        %     \lim_{N\to\infty}\E_{n\in[N]}(T_1 \times \cdots \times T_\ell)^n(g_1 \otimes \cdots \otimes g_\ell) = (g_1 \otimes \cdots \otimes g_\ell)\cdot \lim_{N\to\infty}\E_{n\in[N]}e((t_1 + \cdots + t_\ell)n),
        % \end{align*}
        % it follows from Weyl's equidistribution theorem that $T_1 \times \cdots \times T_\ell$ is ergodic if and only if $n$ is good for equidistribution for the system $(X, \CX, \mu,$\! $T_1, \ldots, T_\ell)$.
        % Suppose first that $n$ is jointly ergodic for the system $(X, \CX, \mu,$\! $T_1, \ldots, T_\ell)$.
    \end{proof}
    The insights presented in the proof above hold more generally: the product ergodicity condition is related to being good for equidistribution while the difference ergodicity condition can be used to upgrade box seminorm control to Host-Kra seminorm control; the latter is vastly preferable because we generally have no good structure theorems for box seminorms. For more discussion on the relationships between these various notions, consult \cite[Section 12]{DKKST24}.

    A problem that has driven a lot of joint ergodicity research is to understand the extent to which the product and difference ergodicity conditions characterize joint ergodicity in general.
        \begin{problem}[Joint ergodicity classification problem, {\cite[Problem 1]{DKS22}}]\label{Pr: joint ergodicity problem}
    Classify the sequences $a_1, \ldots, a_\ell:\N\to\Z$ for which the following equivalence holds for any system: $a_1, \ldots, a_\ell$ are jointly ergodic for a system if and only if they satisfy the product and difference ergodicity conditions for this system.
    %Classify the sequences $a_1, \ldots, a_\ell:\N\to\Z$, with the following property: $a_1, \ldots, a_\ell$ are jointly ergodic for an (arbitrary) system  $(X, \CX, \mu,$\! $T_1, \ldots, T_\ell)$ if and only if they satisfy the product and difference ergodicity conditions for this system.
    \end{problem}

    %while the equidistribution condition handles the spectral behavior. 

    Most of the work in establishing Problem \ref{Pr: joint ergodicity problem} for some class of iterates goes into showing that product and difference ergodicity conditions are sufficient for joint ergodicity; hence it makes sense to think of this direction as ``difficult''.

    The remainder of this section will present known positive (and negative) results on this problem, as well as further open questions.

\subsection{Integer polynomials}
    For integer polynomials, Problem \ref{Pr: joint ergodicity problem} has been resolved by Frantzikinakis and the author. %following a number of partial results.
    \begin{theorem}[{\cite{FrKu22b}}]\label{T: FrKu classification problem}
        Problem \ref{Pr: joint ergodicity problem} has an affirmative answer for integer polynomials. That is, integer polynomials are jointly ergodic for a system if and only if they satisfy the product and difference ergodicity conditions for this system.
    \end{theorem}
    Theorem \ref{T: FrKu classification problem} was proved after a number of partial results:
    \begin{enumerate}
        \item Donoso, Koutsogiannis, and Sun resolved Problem \ref{Pr: joint ergodicity problem} for a single integer polynomial, i.e. when $a_1 = \cdots = a_\ell\in\Z[t]$ \cite{DKS22};
        %\item Donoso, Koutsogiannis, Ferr\'e-Moragues, and Sun covered more cases of integer polynomials \cite{DFKS22};
        \item Frantzikinakis and the author \cite{FrKu22a} answered Problem \ref{Pr: joint ergodicity problem} in the complementary case of pairwise independent integer polynomials.
    \end{enumerate}
These earlier works
%, particularly \cite{DFKS22} and \cite{FrKu22a} 
were important because they introduced new techniques for deriving seminorm estimates presented in Section \ref{S: seminorm smoothing}. 

For polynomials in $\Z^\ell[t]$, the following more general problem remains open. For a system $(X, \CX, \mu, T_1, \ldots,T_\ell)$ and $b\in\Z^\ell$, we denote $T^\b = T_1^{b_1}\cdots T_\ell^{b_\ell}$.
\begin{problem}[Joint ergodicity classification problem for polynomials over $\Z^\ell$ {\cite[Conjecture 1.5]{DKS22}, \cite[Conjecture 1.3]{DFKS22}}]\label{Pr: general classification problem for polys}
    Let $\ba_1, \ldots, \ba_\ell\in\Z^\ell[t]$ be polynomials and $(X, \CX, \mu$, $T_1, \ldots, T_\ell)$ be a system. Prove or disprove: $(T^{\ba_1(n)}, \ldots, T_\ell^{\ba_\ell(n)})_n$ is jointly ergodic on $(X, \CX, \mu)$ if and only if the following hold:
    \begin{enumerate}
    \item $(T^{{\ba_1(n)}}\times \cdots \times T^{{\ba_\ell(n)}})_n$ is ergodic for the product system $(X^\ell, \CX^{\otimes \ell}, \mu^\ell)$;
    \item $(T^{{\ba_i(n)}}T_j^{-\ba_j(n)})_n$ is ergodic for $(X, \CX, \mu)$ for all distinct $1\leq i, j\leq \ell$.
    \end{enumerate}
\end{problem}
Theorem \ref{T: FrKu classification problem} corresponds to the case when $\ba_j = \be_j a_j$ for some $a_j\in\Z[t]$. Several cases towards Problem \ref{Pr: general classification problem for polys} not covered by Theorem \ref{T: FrKu classification problem} have been addressed by Donoso, Koutsogiannis, Ferr\'e-Moragues, and Sun \cite{DFKS22}.

\subsection{Hardy sequences}

For Hardy sequences, Problem \ref{Pr: joint ergodicity problem} admits a positive answer in the following two special cases, resolved respectively by Donoso, Koutsogiannis, and Sun as well as Donoso, Koutsogiannis, Sun, Tsinas, and the author.
%, which correspond to two cases for integer polynomials listed below Theorem \ref{T: FrKu classification problem} that have been resolved before the full proof of Theorem
\begin{theorem}[{\cite[Theorem 1.1]{DKS23}}]\label{T: DKS Hardy}
    Let $a\in\CH$, and suppose that $a$ \emph{stays logarithmically away from real polynomials} in that $a-p\succ \log$ for any $p\in\R[t]$. Then Problem \ref{Pr: joint ergodicity problem} has an affirmative answer for $a$.
\end{theorem}
Theorem \ref{T: DKS Hardy} thus covers $a(t)$ being any of $t^{3/2}, t\log t, t + (\log t)^2$, but neither of $\sqrt{2}t^2, t^2 + \log t$. In \cite[Theorem 6.1]{DKS23}, the authors also extend it to $a\in\CT$.
\begin{theorem}[{\cite[Theorem 1.9]{DKKST24}}]
    Let $a_1, \ldots, a_\ell\in\CH$ be pairwise independent. Then Problem \ref{Pr: joint ergodicity problem} has an affirmative answer for $a$.
\end{theorem}
For general Hardy iterates of polynomial growth, Problem \ref{Pr: joint ergodicity problem} turns out to be more intricate than anticipated. In \cite{DKKST25}, Donoso, Koutsogiannis, Sun, Tsinas, and the author resolve affirmatively the ``difficult'' direction of Problem \ref{Pr: joint ergodicity problem}, establishing the following. 
    \begin{theorem}[{\cite[Theorem 1.5]{DKKST25}}]\label{T: DKKST sufficient condition}
    Suppose that $a_{1},\dots,a_{\ell}\in\mathcal{H}$ satisfy the difference and product ergodicity conditions for a system $(X, \CX, \mu,$\! $T_1, \ldots, T_\ell)$.
Then they are jointly ergodic for the system.
     \end{theorem}

The same work establishes the following partial converse. Recall Definition \ref{D: almost rational poly} for the definition of almost rational polynomials.
%We call $a\in\mathcal{H}$ an \emph{almost rational polynomial} if $a-p$ converges for some $p$ in $\Q[t]$.

%We also establish the following partial converse.

% \begin{definition}
%     For  %(resp. $\R[x]$). 
%     We say that $a$ \emph{stays logarithmically away from $\Q[x]$} (\emph{resp. $\R[x]$}) if $a-p\succ \log$ for all $p$ in $\Q[x]$ (resp. $\R[x]$).
% \end{definition}

\begin{theorem}[{\cite[Theorem 1.7]{DKKST25}}]\label{T: DKKST partial converse}
        Suppose that $a_{1},\dots,a_{\ell}\in\mathcal{H}$ are jointly ergodic for a system $(X, \CX, \mu, T_1, $ $ \ldots, T_\ell)$. Then the following holds:
        \begin{enumerate}
            \item The sequences satisfy the product ergodicity condition on the system.
            \item Suppose additionally that for every $i\neq j$, $c_i, c_j\in\R$, the function $c_i a_i - c_j a_j$ is either an almost rational polynomial  or stays logarithmically away from $\Q[x]$. Then the sequences satisfy the difference ergodicity condition on the system.
        \end{enumerate}
\end{theorem}

Combining Theorem \ref{T: DKKST sufficient condition} and \ref{T: DKKST partial converse} gives an affirmative answer to Problem \ref{Pr: joint ergodicity problem} for a large class of Hardy sequences, which includes all fractional polynomials (with any nonnegative exponents, including positive integer ones).
\begin{corollary}[{\cite[Corollary 1.8]{DKKST25}}]\label{C: DKKST classification}
    Let $a_1, \ldots, a_\ell\in\CH$ be Hardy sequences with the property that for every $i\neq j$ and $c_i, c_j\in\R$ the function $c_i a_i - c_j a_j$ is either an almost rational polynomial or it stays logarithmically away from rational polynomials.    
    Then the sequences are jointly ergodic for a system if and only  if they satisfy the product and difference ergodicity conditions for the system. 
\end{corollary}

Unexpectedly, Problem \ref{Pr: joint ergodicity problem} has a negative answer for Hardy sequences not covered by Corollary \ref{C: DKKST classification}. 
\begin{theorem}[{\cite[Theorem 1.10]{DKKST25}}]\label{T: DKKST counterexample to classification problem}
    Let $(X, \CX, \mu, T)$ be a system. Then $(T^n, T^{n+\floor{\log_2 n}})_n$ is jointly ergodic if and only if $T$ is mixing. By contrast, $(T^{\floor{\log_2 n}})_n$ is ergodic if and only the system is conjugate to a one-point system. Consequently, if $T$ is mixing but not conjugate to a one-point system, then $(T^n, T^{n+\floor{\log_2 n}})_n$ is jointly ergodic even though $(T^{\floor{\log_2 n}})_n$ is not ergodic.
\end{theorem}
%What is particularly interesting about Theorem \ref{T: DKKST counterexample to classification problem} is that - to the author's knowledge - it is a unique result in the field of multiple ergodic averages where imposing a stronger assumption than weak mixing changes the rules of the game. In all the other results, there is no advantage in considering mixing, Bernoulli, or one-point systems compared to weak mixing systems.

The issue that Theorem \ref{T: DKKST counterexample to classification problem} exploits is that joint ergodicity for a given system need not imply the difference ergodicity condition for this system because the ergodic average along $(T_i^{a_i(n)}T_j^{-a_j(n)})_n$ may not converge in norm. However, it is always the case that the joint ergodicity of $(T_1^{a_1(n)}, \ldots, T_\ell^{a_\ell(n)})_n$ implies that 
\begin{align}
    \E_{n\in[N]}T_i^{a_i(n)}T_j^{-a_j(n)} f
\end{align} 
converges weakly to $\int f\; d\mu$ for any $f\in L^\infty(\mu)$ since
\begin{align*}
    \lim_{N\to\infty}\E_{n\in[N]}\int g\cdot T_i^{a_i(n)}T_j^{-a_j(n)} f\; d\mu &= \lim_{N\to\infty}\E_{n\in[N]}\int T_j^{a_j(n)}g\cdot T_i^{a_i(n)} f\; d\mu\\
    &= \int g\; d\mu \cdot \int f\; d\mu
\end{align*}
by joint ergodicity. In the light of that, \cite{DKKST25} proposes two variants of Problem \ref{Pr: joint ergodicity problem} that rely on the weak convergence version of the notions used in Problem \ref{Pr: joint ergodicity problem}. 
\begin{problem}[{\cite[Question 8.3]{DKKST25}}]\label{Pr: mending joint ergodicity problem 1}
    Does Problem \ref{Pr: joint ergodicity problem} admit affirmative answer for all $a_1, \ldots, a_\ell\in\CH$ if we define the difference ergodicity condition in terms of weak rather than strong ergodicity?
\end{problem}

\begin{problem}[{\cite[Question 8.4]{DKKST25}}]\label{Pr: mending joint ergodicity problem 2}
    Does Problem \ref{Pr: joint ergodicity problem} admit affirmative answer for all $a_1, \ldots, a_\ell\in\CH$ if joint ergodicity as well as product and difference ergodicity conditions are all defined in terms of weak rather than strong limits?
\end{problem}

%We remark that question \ref{c2} can be very difficult, as it implies that 2-mixing implies 3-mixing (see the following proposition).
Problem \ref{Pr: mending joint ergodicity problem 1} can be very difficult; a positive answer would give that 2-mixing implies a property resembling 3-mixing: see \cite[Proposition 8.5]{DKKST25}.

Another possible way to fix Problem \ref{Pr: joint ergodicity problem} for Hardy sequences would be to consider weighted averaging schemes as in Section \ref{SSS: weighted averages}. 
\begin{problem}[{\cite[Conjecture 8.8]{DKKST25}}]
    Fix $a_1, \ldots, a_\ell\in\CH$. Can we always find $1 \prec W(t) \ll t$ such that Problem \ref{Pr: joint ergodicity problem} admits affirmative answer for $a_1, \ldots, a_\ell$ if we define joint ergodicity as well as product and difference ergodicity conditions using $W$-averaging schemes?
    %Does Problem \ref{Pr: joint ergodicity problem} admit affirmative answer for all $a_1, \ldots, a_\ell\in\CH$ if joint ergodicity as well as product and difference ergodicity conditions are all defined for $W$-averaging schemes for some $1 \prec W(t) \ll t$?
\end{problem}
We invite the reader interested in possible fixes to Problem \ref{Pr: joint ergodicity problem} to consult \cite[Section 8]{DKKST25}, which provides an ample discussion of three questions listed above. 

\subsection{Generalized polynomials}
Problem \ref{Pr: joint ergodicity problem} can be posed for other classes of sequences. In this survey, we restrict our attention to one more such class: generalized polynomials, defined in Section \ref{SS: generalized polynomials}
\begin{problem}\label{Pr: joint ergodicity problem for generalized polys}
    Resolve Problem \ref{Pr: joint ergodicity problem} for generalized polynomials. 
\end{problem}

The reason we discuss this particular class is that the following special case of Problem \ref{Pr: joint ergodicity problem for generalized polys} has been resolved by Bergelson, Leibman, and Son. %\cite{BLS16}.
\begin{theorem}[{\cite[Theorem 0.4]{BLS16}}]\label{T: BLS generalized linear}
    Problem \ref{Pr: joint ergodicity problem} admits the affirmative answer for generalized linear functions. 
\end{theorem}
\emph{Generalized linear functions} are precisely what one thinks: they are functions constructed from applying finitely many times the operations of addition, scalar multiplication, and taking integer parts to linear functions $\alpha t + \beta$. Examples include
\begin{align*}
    \alpha_1\floor{\alpha_2 t + \alpha_3} + \alpha_4 t + \alpha_5,\; \rem{\alpha_1 t} + \floor{\alpha_2 t}, \; \alpha_1\floor{\alpha_2\floor{\alpha_3 t + \alpha_4}+\alpha_5}+\alpha_6
\end{align*}
for $\alpha_j\in\R$. We do emphasize, however, that generalized linear functions are a very particular subclass of generalized polynomials, and it is not clear how much of the intuition behind Theorem \ref{T: BLS generalized linear} extends to arbitrary generalized polynomials.

\section{Technical breakthroughs}\label{S: technical breakthroughs}
This section aims to highlight the main ideas behind two major technical breakthroughs that lie at the foundation of recent joint ergodicity advances:
\begin{enumerate}
    \item the \emph{degree lowering} argument used to prove the difficult direction in Theorem \ref{T: joint ergodicity criteria};
    \item the new techniques to prove {Host-Kra seminorm estimates} for multiple ergodic averages of commuting transformations, which involve a refined version of the \emph{PET induction scheme}, (qualitative and quantitative) \emph{concatenation} arguments for box seminorms and factors, and the \emph{seminorm smoothing} technique that upgrades box seminorm control of an average to one by Host-Kra seminorms.
\end{enumerate}

\subsection{Degree lowering}\label{S: degree lowering proof}

In this subsection, we will outline the proof of the difficult direction of Theorem \ref{T: invariant control criteria} via degree lowering. %This is one of two existing expositions of degree lowering in ergodic theory. 
Given its complexity, it is impossible to outline the argument by giving equal importance to its every aspect. In this exposition, we focus on the complex induction scheme that underpins the argument and the idea of ``transferring seminorm control'' lying at its base. At the same time, we try to sweep under the carpet various (non-trivial and highly ingenious) technical maneuvers such as dual-difference interchange and removing low-complexity functions. 
%an exposition has to choose to which aspects to assign importance and which to play down. 
For an alternative take on degree lowering, which gives comparatively more space to technical tricks and less to induction in the more concrete case of the Furstenberg-Weiss average
\begin{align*}
    \lim_{N\to\infty}\E_{n\in[N]}T^n f_1 \cdot T^{n^2}f_2,
\end{align*}
we recommend the exposition of Frantzikinakis \cite[Section 2]{Fr21}.

Theorem \ref{T: invariant control criteria} is a local result, in the sense that it holds for fixed integer sequences and system. Therefore, for the entirety of this section, we fix sequences $a_1,\ldots,a_\ell\colon \N\to \Z$ and a system $(X, \CX, \mu,$\! $T_1, \ldots, T_\ell)$. We thus restate the statement whose proof we are about to sketch.
\begin{theorem}[The difficult direction of Theorem \ref{T: invariant control criteria}]\label{T: degree lowering}
        Suppose that the sequences $a_1,\ldots,a_\ell$ admit Host-Kra seminorm control and are good for equidistribution for the system $(X, \CX, \mu,$\! $T_1, \ldots, T_\ell)$. Then they are controlled by the invariant factor for the system, i.e.     \begin{align}\label{E: weak joint ergodicity restated}
        \lim_{N\to\infty}\norm{\E_{n\in[N]}\prod_{j=1}^\ell T_j^{{a_{j}(n)}}f_j - \prod_{j=1}^\ell \E(f_j|\CI(T_j))}_{L^2(\mu)} = 0
    \end{align}
    holds for all $f_1, \ldots, f_\ell\in L^\infty(\mu)$.
\end{theorem}

\subsubsection{Induction scheme}
The proof of Theorem \ref{T: invariant control criteria} follows a double induction argument. In the outer layer of induction, we induct on the parameter $0\leq m \leq \ell$: it is the complexity of the average, measuring the number of functions in the average that are not of a ``simple'' form. The inner induction scheme, relegated to the next subsection, deals with successively lowering the degree of the Host-Kra seminorm that controls our average.

The inductive property that we derive in the outer induction scheme is as follows. 
\begin{definition}\label{D: P_m}
    Let $0\leq m\leq \ell$. We say that Property $(P_m)$ holds if the identity \eqref{E: weak joint ergodicity restated} holds whenever $f_{m+1}, \ldots, f_\ell$ are nonergodic eigenfunctions of $T_{m+1}, \ldots, T_\ell$ respectively.
    %$f_{m+1}\in\CE(T_{m+1}), \ldots, f_\ell\in \CE(T_\ell)$, 
\end{definition}
By definition, Property $(P_0)$ corresponds to the case when all functions are nonergodic eigenfunctions. It is tantamount to the property of being good for equidistribution, and so it holds by assumption. This fixes the base case. Theorem \ref{T: degree lowering} is equivalent to property $(P_\ell)$. 

In the outer inductive step, we thus need to consider averages in which some of the functions are nonergodic eigenfunctions, and hence they are much simpler than arbitrary functions in $L^\infty(\mu)$. Therefore, the inductive step will use the following Host-Kra seminorm control assumption. It is strictly weaker than the assumption stated in Theorem \ref{T: invariant control criteria} and \ref{T: degree lowering}; but it is all that we need (see the discussion below Definition \ref{D: seminorm control}).
\begin{definition}[Weak Host-Kra seminorm control]\label{D: weak control}
    We say that $a_1, \ldots, a_\ell$ \emph{admit weak (Host-Kra) seminorm control} for the system if there exists $s\in\N$ such that for all $1\leq m\leq \ell$ and $f_1, \ldots, f_\ell\in L^\infty(\mu)$ with $f_{m+1}, \ldots, f_\ell$ being nonergodic eigenfunctions of $T_{m+1}, \ldots, T_\ell$ respectively, we have
    \begin{align*}
        \lim_{N\to\infty}\norm{\E_{n\in[N]}\prod_{j=1}^\ell T_j^{{a_{j}(n)}}f_j}_{L^2(\mu)} = 0
    \end{align*}
    whenever $\nnorm{f_m}_{s, T_m} = 0$. 
\end{definition}
Note that the order of sequences matters in Definition \ref{D: weak control}. In applications, we order them according to the increasing ``degree'' (whatever that means in a particular context), so that $a_j$ has no greater ``degree'' than $a_{j+1}$.
\begin{example}
Weak seminorm control is easier to verify than full Host-Kra seminorm control because nonergodic eigenfunctions quickly disappear in the analytic arguments used to establish seminorm control. To illustrate this point, we 
    %To see why weak seminorm control is easier to establish than full Host-Kra seminorm control, 
    consider the average
    \begin{align*}
        \E_{n\in[N]}T_1^n f_1\cdot T_2^{n^2}f_2.
    \end{align*}
    If $f_2$ is a nonergodic eigenfunction of $T_2$, then two applications of the van der Corput inequality (Lemma \ref{L: vdC}) followed by the Cauchy-Schwarz inequality eliminate $T_2^{n^2}f_2$ and allow us to control its $L^2(\mu)$ limit by $\nnorm{f_1}_{3, T_1}$. By contrast, without any assumption of $f_2$, we need to first use the PET induction scheme to control the average by $\nnorm{f_2}_{3, T_2}$ (see Example \ref{Ex: PET n, n^2} for the details), then use the nonergodic form of the Host-Kra structure theorem to replace $f_2$ by a function $f_2'$ such that $(f_2'(T_2^n x))_n$ is a 2-step nilsequence pointwise a.e., and then use either multiple applications of the van der Corput lemma or equidistribution results on nilsystems to deduce $\nnorm{f_1}_{s, T_1}$ control for some $s\in\N$. 
\end{example}

\subsubsection{Inductive step: reduction to the degree lowering property}\label{SS: reduction to degree lowering}
We now present the outer inductive step. For the simplicity of exposition, we assume $\ell=2$, and we shall derive Property $(P_2)$ from Property $(P_1)$. This case captures all the difficulties of the general case while allowing for a much cleaner presentation.
%We also restrict this exposition to the case when $T_1, T_2$ are ergodic; the nonergodic setting requires additional nontrivial maneuvers that obfuscate the main idea. 
Then Property $(P_1)$, which we shall invoke inductively, reduces to the pleasant statement that 
\begin{align*}
%\lim_{N\to\infty} \norm{\E_{n\in[N]} T_1^{a_1(n)}f_1 \cdot T_2^{a_1(n)}\chi_2 - \int f_1\; d\mu\cdot \int \chi_2\; d\mu}_{L^2(\mu)} = 0
    \lim_{N\to\infty} \norm{\E_{n\in[N]} T_1^{a_1(n)}f_1 \cdot T_2^{a_1(n)}\chi_2 - \E(f_1|\CI(T_1))\cdot \E(\chi_2|\CI(T_2))}_{L^2(\mu)} = 0
\end{align*}
%whenever $\chi_2$ is an eigenfunction of $T_2$.
whenever $\chi_2$ is a nonergodic eigenfunction of $T_2$.

In the outer inductive step, the goal is to show invariant factor control for averages
\begin{align}\label{E: degree lowering average}
    \limsup_{N\to\infty}\norm{\E_{n\in[N]}T_1^{a_1(n)}f_1 \cdot T_2^{a_1(n)}f_2}_{L^2(\mu)}.
\end{align}
%in which $f_{m+1}, \ldots, f_\ell$ are nonergodic eigenfunctions of $T_{m+1}, \ldots, T_\ell$ respectively. 
The proof consists of the following substeps, each of which is about obtaining sharper seminorm control of the average \eqref{E: degree lowering average}:
%More precisely, the outer inductive step will consist of the following substeps:
\begin{enumerate}
    \item (Weak seminorm control) control the average \eqref{E: degree lowering average} by $\nnorm{f_2}_{s, T_2}$ for some $s\in\N$;
    \item (Transferring control to weakly structured function) control the average \eqref{E: degree lowering average} by $\nnorm{F_2}_{s, T_{2}}$ for a \emph{weakly structured} function $F_2$ defined below;
    \item (Degree lowering) control the average \eqref{E: degree lowering average} by $\nnorm{F_2}_{1, T_{2}}$;
    \item (Transferring optimal control\footnote{In deriving Property $(P_{m+1})$ from $(P_m)$ for general $m,\ell$, at this step we would transfer control to $f_j$ for any $j\neq m+1$.} to $f_1$) control the average \eqref{E: degree lowering average} by $\nnorm{f_1}_{1, T_1}$;
    \item (Extending optimal control to $f_2$) control the average \eqref{E: degree lowering average} by $\nnorm{f_2}_{1, T_{2}}$.
\end{enumerate}
As this roadmap suggests, the derivation of $(P_{2})$ from $(P_1)$ will involve a lot of back-and-forth between the seminorms of different functions. This is an essential feature of both the degree lowering and seminorm smoothing arguments. 

%How do we deduce property $(P_{2})$ from $(P_1)$? As the name \emph{degree lowering} suggests, the main idea is to lower the degree of the seminorm controlling our average. This is indeed what we will do, except that we will replace $f_2$ by a more structured function $F_2$ (note: $2$ is the largest index so that the function appearing in the proof of property $(P_{2})$ can be arbitrary). 

Before we define the weakly structured function $F_2$, we make one standing assumption for the argument in this section:

\emph{\ We assume that for all $f_1, f_2\in L^\infty(\mu)$, the average 
\begin{align*}
    \E_{n\in[N]}T_1^{a_1(n)}f_1 \cdot T_2^{a_1(n)}f_2
\end{align*}
converges in $L^2(\mu)$.
}

This assumption is not needed; the degree lowering arguments from \cite{Fr21, FrKu22a} circumvent it by using limsups in inequalities and subsequential weak limit in the definition of $F_2$ below. However, this assumption allows us to simplify the formula for the weakly structured function $F_2$, thus facilitating the exposition. 

Weak seminorm control, the first step of the roadmap above, is part of the assumptions. We therefore move on to the second step, i.e. to transferring the control from $f_2$ to $F_2$. The starting point for this step is the lemma below, which also gives the definition of $F_2$.
%The replacement of $f_2$ by $F_2$ is performed in the following lemma. In the additive combinatorics literature, this maneuver is called \emph{stashing}, a name due to Manners \cite{Man21}.
\begin{lemma}[Stashing: replacing $f_2$ by a weakly structured function]\label{L: stashing in degree lowering}
    Let $f_1, f_2\in L^\infty(\mu)$ be 1-bounded. 
    Then
    \begin{align*}
        \lim_{N\to\infty}\norm{\E_{n\in[N]}T_1^{a_1(n)}f_1 \cdot T_2^{a_1(n)}f_2}_{L^2(\mu)}^4\leq \lim_{N\to\infty}\norm{\E_{n\in[N]}T_1^{a_1(n)}f_1 \cdot T_2^{a_1(n)}F_2}_{L^2(\mu)},
    \end{align*}
    % If 
    % \begin{align}\label{E: replacing f_2 0}
    %     \lim_{N\to\infty}\norm{\E_{n\in[N]}T_1^{a_1(n)}f_1 \cdot T_2^{a_1(n)}f_2}_{L^2(\mu)}>0,
    % \end{align}
    % then 
    % \begin{align*}%\label{E: replacing f_2 1}
    %             \lim_{N\to\infty}\norm{\E_{n\in[N]}T_1^{a_1(n)}f_1 \cdot T_2^{a_1(n)}F_2}_{L^2(\mu)}>0,
    %     %\lim_{N\to\infty}\norm{\E_{n\in[N]}T_{2}^{a_{2}(n)}F_2\cdot \prod_{\substack{1\leq j\leq \ell\colon\\ j\neq 2}} T_j^{{a_{j}(n)}}f_j}_{L^2(\mu)}>0,
    % \end{align*}
    where\footnote{Without assuming $L^2(\mu)$ convergence, we define $F_2$ as follows. We first take a sequence $(N_k)$ along which the averages $g_k:=\E_{n\in[N_k]}T_1^{a_1(n)}f_1 \cdot T_2^{a_1(n)}f_2$ converge weakly. Then we define $F_2$ as \emph{weak} limit
    \begin{align*}
                   F_2 = \lim_{k\to\infty}\E_{n\in[N_k]}T_2^{-a_2(n)}\overline{g_k}\cdot T_1^{a_1(n)}T_2^{-a_2(n)}\overline{f_1}. 
    \end{align*}
    This turns out to be sufficient for the applications.
    }
    \begin{align}\label{E: structured}
           F_2 := \lim_{N\to\infty}\E_{n\in[N]}T_2^{-a_2(n)}\overline{f_0}\cdot T_1^{a_1(n)}T_2^{-a_2(n)}\overline{f_1} 
       % F_2:=\lim_{N\to\infty} \E_{n\in[N]} \prod_{\substack{0\leq j\leq \ell\colon\\ j\neq 2}} T_j^{{a_{j}(n)}}T_{2}^{-a_{2}(n)}\overline f_j
    \end{align}
    for %some $f_0\in L^\infty(\mu)$.
    \begin{align*}
        f_0:= \lim_{N\to\infty}\E_{n\in[N]}T_1^{a_1(n)}\overline{f_1} \cdot T_2^{a_1(n)}\overline{f_2}.
    \end{align*}
\end{lemma}
%As complicated as it may look, the average appearing in \eqref{E: replacing f_2 1} is the same as the one appearing in \eqref{E: replacing f_2 0} except that $f_2$ is replaced by $F_2$.
\begin{proof}
    % Set
    % \begin{align*}
    %     f_0:= \lim_{N\to\infty}\E_{n\in[N]}T_1^{a_1(n)}f_1 \cdot T_2^{a_1(n)}f_2.
    % \end{align*}
    Expanding the square of the $L^2(\mu)$ norm and composing the resulting integral with $T_{2}^{-a_{2}(n)}$ gives
    \begin{align*}
        0< \lim_{N\to\infty}\norm{\E_{n\in[N]}T_1^{a_1(n)}f_1 \cdot T_2^{a_1(n)}f_2}_{L^2(\mu)}^2 &= \lim_{N\to\infty} \E_{n\in[N]} \int f_0\cdot T_1^{a_1(n)}f_1 \cdot T_2^{a_1(n)}f_2\; d\mu\\
        &=\lim_{N\to\infty} \E_{n\in[N]} \int T_2^{-a_2(n)}{f_0}\cdot T_1^{a_1(n)}T_2^{-a_2(n)}{f_1}\cdot f_2\; d\mu\\ 
        %\prod_{\substack{0\leq j\leq \ell\colon\\ j\neq 2}} T_j^{{a_{j}(n)}}T_{2}^{-a_{2}(n)}f_j\; d\mu\\
        &= \int \overline{F_2}\cdot f_2\; d\mu.
    \end{align*}
%    on expanding the square and composing the average with $T_{2}^{-a_{2}(n)}$ (in the expression above, $T_0^{a_0(n)} := \Id_X$).
    %Recall that this expression assumption \eqref{E: replacing f_2 1} and 
    Applying the Cauchy-Schwarz inequality, expanding the inner product, and composing back with $T_{2}^{a_{2}(n)}$, we get
    \begin{align*}
        0< \norm{F_2}_{L^2(\mu)}^2 &= \int  \overline{F_2}\cdot F_2 \; d\mu\\
        &=\lim_{N\to\infty} \E_{n\in[N]} \int T_2^{-a_2(n)}{f_0}\cdot T_1^{a_1(n)}T_2^{-a_2(n)}{f_1}\cdot F_2\; d\mu\\
        &= \lim_{N\to\infty} \E_{n\in[N]} \int f_0\cdot T_1^{a_1(n)}{f_1}\cdot T_{2}^{a_{2}(n)}F_2\; d\mu.
    \end{align*}
    The claim follows from one more application of the Cauchy-Schwarz inequality.
\end{proof}
The Cauchy-Schwarz gymnastics performed in the proof of Lemma \ref{L: stashing in degree lowering} is known in additive combinatorics as \emph{stashing}, a name due to Manners \cite{Man21}.

As a consequence of Proposition \ref{L: stashing in degree lowering} and the weak seminorm control assumption, we get the following auxiliary seminorm control, which completes the second step of the roadmap. 
\begin{corollary}[Seminorm control in terms of weakly structured function]\label{C: seminorm control by structured}
    There exists $s\in\N$ such that the following holds. Let $f_1, f_2\in L^\infty(\mu)$. If 
    \begin{align}\label{E: replacing f_2 0}
        \lim_{N\to\infty}\norm{\E_{n\in[N]}T_1^{a_1(n)}f_1 \cdot T_2^{a_1(n)}f_2}_{L^2(\mu)}>0,
    \end{align}
    %\eqref{E: replacing f_2 0} holds,
    % \begin{align}\label{E: replacing f_2 0}
    %     \lim_{N\to\infty}\norm{\E_{n\in[N]}\prod_{j=1}^\ell T_j^{{a_{j}(n)}}f_j}_{L^2(\mu)}>0,
    % \end{align}
     then $\nnorm{F_2}_{s, T_{2}}>0$.
\end{corollary}
Note that Corollary \ref{C: seminorm control by structured} is equivalent to saying that $\nnorm{F_2}_{s, T_{2}}$ controls \eqref{E: degree lowering average}.
%under the additional assumptions on $f_{m+2}, \ldots, f_\ell$. 
This contrapositive reformulation turns out to be more handy for degree lowering.

Everything so far has been basic preparatory maneuvers; the essence of the argument lies in the result below.
\begin{proposition}[Degree lowering]\label{P: degree lowering}
    Assume that Property $(P_1)$ holds. If $\nnorm{F_2}_{s, T_{2}}>0$ for some $s\geq 2$, then also $\nnorm{F_2}_{s-1, T_{2}}>0$. 
\end{proposition}
The proof of Proposition \ref{P: degree lowering} will be sketched in the following subsection. Before, let us see how it can be used to derive property $(P_{2})$ from $(P_1)$, and hence complete the proof of Theorem \ref{T: degree lowering}.

Iterating Proposition \ref{P: degree lowering} $s-1$ times forms the inner induction cycle in the double induction scheme that underpins the proof of Theorem \ref{T: degree lowering}. It results in the following statement. 
\begin{corollary}[Iterated degree lowering]\label{C: iterated degree lowering}
    Assume that Property $(P_1)$ holds. If $\nnorm{F_2}_{s, T_{2}}>0$ for some $s\geq 2$, then also $\nnorm{F_2}_{1, T_{2}}>0$. 
\end{corollary}
Combining Corollary \ref{C: iterated degree lowering} with Corollary \ref{C: seminorm control by structured}, we get the following corollary, which completes the third step in the roadmap.
\begin{corollary}[Optimal seminorm control in terms of weakly structured function]\label{C: optimal control by dual}
Assume that Property $(P_1)$ holds. Let $f_1, f_2\in L^\infty(\mu)$. % with $f_{m+2}, \ldots, f_\ell$ being nonergodic eigenfunctions of $T_{m+2}, \ldots, T_\ell$ respectively. 
If \eqref{E: replacing f_2 0} holds, then $\nnorm{F_2}_{1, T_{2}}>0$.
\end{corollary}

We thus control the original average \eqref{E: degree lowering average} by the smallest possible Host-Kra seminorm of $F_2$. However, $F_2$ does not appear in the original average, and so this seminorm control is only useful as an intermediate step in controlling \eqref{E: degree lowering average} by degree-1 seminorms of $f_1, f_2$. The remaining two steps of the roadmap are precisely about transferring the optimal seminorm control from $F_2$ to  $f_1, f_2$. First, we do that for $f_1$, which is the fourth step of the roadmap.
\begin{proposition}\label{P: transferring control to f_j except 2}
    Assume that Property $(P_1)$ holds. Let $f_1, f_2\in L^\infty(\mu)$.
    %with $f_{m+2}, \ldots, f_\ell$ being nonergodic eigenfunctions of $T_{m+2}, \ldots, T_\ell$ respectively. 
    If \eqref{E: replacing f_2 0} holds, then $\nnorm{f_1}_{1,T_1}>0$.
    %then $\nnorm{f_j}_{1, T_j}>0$ for all $j\neq 2$.
\end{proposition}
\begin{proof}%[Proof of Property $(P_{2})$ assuming Property $(P_1)$ and everything so far]
    Suppose that \eqref{E: replacing f_2 0} holds. Then $\nnorm{F_2}_{1, T_1}>0$ by Corollary \ref{C: optimal control by dual}. Using the formula \eqref{E: dual identity} for the degree-1 seminorm as well as the definition \eqref{E: structured} of $F_2$, we deduce that 
    \begin{multline*}
        \nnorm{F_2}_{1, T_1}^2 = \int \overline{F_2}\cdot \E(F_2|\CI(T_{2}))\; d\mu\\
        = \lim_{N\to\infty} \E_{n\in[N]} \int T_2^{-a_2(n)}{f_0}\cdot T_1^{a_1(n)}T_2^{-a_2(n)}{f_1}\cdot \E(F_2|\CI(T_{2}))\; d\mu >0.
    \end{multline*}
    %for $f_2 := \E(F_2|\CI(T_{2}))$.
    Composing the integral with $T_{2}^{a_{2}(n)}$, we obtain
    \begin{align*}
        \lim_{N\to\infty} \E_{n\in[N]} \int f_0\cdot T_1^{a_1(n)}{f_1}\cdot T_{2}^{a_{2}(n)}\E(F_2|\CI(T_{2}))\; d\mu > 0.
        %\cdot \prod_{\substack{0\leq j\leq \ell\colon\\ j\neq 2}} T_j^{{a_{j}(n)}}f_j\; d\mu >0.
    \end{align*}
    Crucially, the function $\E(F_2|\CI(T_{2}))$ is $T_{2}$-invariant, and so 
    \begin{align*}
        \lim_{N\to\infty} \E_{n\in[N]} \int (f_0\cdot \E(F_2|\CI(T_{2})))\cdot T_1^{a_1(n)}{f_1}\; d\mu > 0.
        %\cdot \prod_{\substack{0\leq j\leq \ell\colon\\ j\neq 2}} T_j^{{a_{j}(n)}}f_j\; d\mu >0.
    \end{align*}
    %    $T_{2}$ disappears from the expression above. 
    From the Cauchy-Schwarz inequality we infer that
    \begin{align*}
        \lim_{N\to\infty} \norm{\E_{n\in[N]} T_1^{a_1(n)}{f_1}}_{L^2(\mu)} > 0.
    \end{align*}
    %(recall that $T_0^{a_0(n)}f_0 = f_0$, hence $f_0$ can also be removed while applying the Cauchy-Schwarz inequality).
    This new average is now amenable to Property $(P_1)$ (with 1, clearly a nonergodic eigenfunction of $T_{2}$, being the function at the index $2$). By inductively invoking $(P_1)$, we deduce that
    \begin{align*}
       \nnorm{f_1}_{1, T_1} = \norm{\E(f_1|\CI(T_1))}_{L^2(\mu)} > 0.
    \end{align*}
    %from which it follows that $\nnorm{f_j}_{1, T_j}>0$ for every $j\neq 2$.
\end{proof}
Lastly, we establish optimal seminorm control of \eqref{E: degree lowering average} in terms of $f_2$. This is the last step in the roadmap, and it completes the proof of Theorem \ref{T: degree lowering} modulo the proof of Proposition \ref{P: degree lowering} in the next section.
\begin{proposition}\label{P: transferring control to f_2}
    Assume that Property $(P_1)$ holds. Let $f_1, f_2\in L^\infty(\mu)$. %with $f_{m+2}, \ldots, f_\ell$ being nonergodic eigenfunctions of $T_{m+2}, \ldots, T_\ell$ respectively. 
    If \eqref{E: replacing f_2 0} holds, then $\nnorm{f_2}_{1, T_{2}}>0$.
\end{proposition}
\begin{proof}
    Suppose that \eqref{E: replacing f_2 0} holds. We split $f_1 = \E(f_1|\CI(T_1)) + g_1$.
    %For $j\neq 2$, we split $f_j = \E(f_j|\CI(T_j)) + g_j$. 
    The error term $g_1$ has vanishing seminorm $\nnorm{\cdot}_{1, T_1}$, contributing 0 to \eqref{E: replacing f_2 0} by the contrapositive of Proposition \ref{P: transferring control to f_j except 2}. From this and \eqref{E: replacing f_2 0} we deduce that
    \begin{align*}
                \lim_{N\to\infty} \norm{\E_{n\in[N]}T_1^{{a_{1}(n)}}\E(f_1|\CI(T_1))\cdot T_{2}^{a_{2}(n)}f_2}_{L^2(\mu)} > 0,
    \end{align*}
    and hence 
    \begin{align*}
                \lim_{N\to\infty} \norm{\E_{n\in[N]}T_{2}^{a_{2}(n)}f_2}_{L^2(\mu)} > 0
    \end{align*}
    by the $T_1$-invariance of $\E(f_1|\CI(T_1))$ and the H\"older inequality.  % $T_j^{{a_{j}(n)}}\E(f_j|\CI(T_j)) = \E(f_j|\CI(T_j))$, we deduce from the Cauchy-Schwarz inequality that 
    % \begin{align*}
    %     \nnorm{f_2}_{1, T_{2}} = \lim_{N\to\infty} \norm{\E_{n\in[N]}T_{2}^{a_{2}(n)}f_2}_{L^2(\mu)} > 0.
    % \end{align*}
    % %Hence $\nnorm{f_2}_{1, T_{2}} > 0$ by the mean ergodic theorem and definition of degree-1 Host-Kra seminorms. 
    %we have $\nnorm{f_j}_{1, T_j}>0$ for all $j\neq 2$.
    The conclusion $\nnorm{f_2}_{1, T_{2}}>0$ then follows by a mini-rendition of the degree lowering argument that can be found in the proof of \cite[Theorem 1.1]{Fr21}.
\end{proof}

\subsubsection{Outline of the proof of Proposition \ref{P: degree lowering}}
We conclude the proof of Theorem \ref{T: degree lowering} with the proof of Proposition \ref{P: degree lowering}, the deepest and most interesting step in the entire argument. The proofs presented in Section \ref{SS: reduction to degree lowering} used rather basic maneuvers: simple applications of the Cauchy-Schwarz inequality, measure invariance, and the induction hypothesis; if there were any complications, they came from setting up the induction. By contrast, the proof of Proposition \ref{P: degree lowering} involves several formidably technical steps.

% We therefore restrict to the case $\ell = 2$, which captures all the complexity of the general case yet allows for a significant simplification of the notation. Then $F_2$ takes the much simpler form 
% \begin{align*}
%     F_2 = \lim_{N\to\infty}\E_{n\in[N]}T_2^{-a_2(n)}\overline{f_0}\cdot T_1^{a_1(n)}T_2^{-a_2(n)}\overline{f_1}. 
% \end{align*}
% We also restrict this exposition to the case when $T_1, T_2$ are ergodic; the nonergodic setting introduces nontrivial technicalities that obfuscate the main idea. Then Property $(P_1)$, which we shall invoke inductively, reduces to the fairly pleasant statement that 
% \begin{align*}
% \lim_{N\to\infty} \norm{\E_{n\in[N]} T_1^{a_1(n)}f_1 \cdot T_2^{a_1(n)}f_2 - \int f_1\; d\mu\cdot \int f_2\; d\mu}_{L^2(\mu)} = 0
%     %\lim_{N\to\infty} \norm{\E_{n\in[N]} T_1^{a_1(n)}f_1 \cdot T_2^{a_1(n)}f_2 - \E(f_1|\CI(T_1))\cdot \E(f_2|\CI(T_2))}_{L^2(\mu)} = 0
% \end{align*}
% whenever $f_2$ is an eigenfunction of $T_2$.
% %is controlled by $\nnorm{f_1}_{s, T_1}$

We will prove two instances of Proposition \ref{P: degree lowering}, corresponding to $s=2,3$. The first of these cases avoids much of the technicalities present in the general case, yet it shows why the degree reduction happens and where Property $(P_1)$ is invoked. The second of these cases will be considerably more technical, but the notation will not be as prohibitive as in the general case. We list these two cases as separate statements and prove them one by one.

The proof that we present below differs from the argument of Frantzikinakis and the author from \cite{FrKu22a}; rather, it is based on a later argument of the author in the discrete setting \cite{Kuc23}. The proof in \cite{FrKu22a} covers Theorem \ref{T: Krat control criteria}, which is more tricky than Theorem \ref{T: invariant control criteria} and therefore requires a more subtle treatment, one that the relatively simpler argument below does not provide.

\begin{proposition}[Degree lowering, baby case I]\label{P: degree lowering s=2}
    Assume %that $T_1, T_2$ are ergodic, and 
    that Property $(P_1)$ holds. If $\nnorm{F_2}_{2, T_2}>0$, then also $\nnorm{F_2}_{1, T_2}>0$. 
\end{proposition}
For the reader who has not seen the degree lowering argument before, we advise to assume on the first read that $T_2$ is ergodic. Under this assumption, we can simply take $U$ to be all of $X$ and $\chi_2$ to be an eigenfunction of $T_2$ in the proof that follows. Furthermore, \eqref{E: chi_2 2} and Property $(P_1)$ will imply that $\chi_2$ is constant.
\begin{proof}
    By the inverse theorem for the degree-2 Host-Kra seminorm, there exists a nonergodic eigenfunction $\chi_2$ of $T_2$ such that 
    \begin{align}\label{E: positivity s=2}
        \int F_2 \cdot \overline{\chi_2}\; d\mu > 0\quad \textrm{and} \quad \E(F_2\cdot \overline{\chi_2}|\CI(T_2))\geq 0
    \end{align}
    (if $T_2$ is ergodic, these two properties are the same).
    The essence of the proof lies in showing that this can only happen if $\chi_2$ has special structure, and specifically it is $T_2$-invariant on a positive-measure set (if $T_2$ is ergodic, then $\chi_2$ has to be constant). Once we establish this property, it will be straightforward to show that $F_2$ indeed has positive degree-1 Host-Kra seminorm.
    
    The inequalities above imply that there exists $\delta > 0$ and a $T_2$-invariant set $U\subseteq X$ such that
    \begin{align}\label{E: lower bound on U}
        \E(F_2\cdot \overline{\chi_2}|\CI(T_2))(x) \geq \delta \quad \textrm{for all}\quad x\in U;
    \end{align}
    in other words, $F_2\cdot \overline{\chi_2}$ is bounded away from 0 on a positive proportion of the ergodic components of $\mu$ with respect to $T_2$ (if $T_2$ is ergodic, then $U=X$, and the conclusion of being bounded away from 0 is immediate from \eqref{E: positivity s=2}).
    Hence
    \begin{align*}
        \int F_2 \cdot \overline{\chi_2}\cdot 1_U\; d\mu = \int \E(F_2\cdot \overline{\chi_2}|\CI(T_2))\cdot 1_U\; d\mu > 0.
    \end{align*}
    On expanding the definition of $F_2$, we have
    \begin{align}\label{E: expanding dual}
        \lim_{N\to\infty}\E_{n\in[N]}\int T_2^{-a_2(n)}\overline{f_0}\cdot T_1^{a_1(n)}T_2^{-a_2(n)}\overline{f_1}\cdot \overline{\chi_2}\cdot 1_U\; d\mu > 0.
    \end{align}
    Composing with $T_2^{a_2(n)}$ and taking the complex conjugate, we deduce that
    \begin{align}\label{E: chi_2}
        \lim_{N\to\infty}\E_{n\in[N]}\int {f_0}\cdot T_1^{a_1(n)}f_1\cdot T_2^{a_2(n)}(\chi_2\cdot 1_U)\; d\mu > 0.
    \end{align}
    The Cauchy-Schwarz inequality then gives
    \begin{align}\label{E: chi_2 2}
        \lim_{N\to\infty}\norm{\E_{n\in[N]} T_1^{a_1(n)}f_1\cdot T_2^{a_2(n)}(\chi_2\cdot 1_U)}_{L^2(\mu)}> 0.
    \end{align}
     
     Since $\chi_2\cdot 1_U$ is a nonergodic eigenfunction, this average is amenable to the application of Property $(P_1)$. This is the key step where we reduce the complexity of the original average in a manner that allows us to invoke the induction hypothesis.
     Indeed, Property $(P_1)$ combined with \eqref{E: chi_2 2} implies that 
     \begin{align*}%\label{E: structure on nonergodic eigenfunction}
         \norm{\E(\chi_2\cdot 1_U|\CI(T_2))}_{L^2(\mu)}>0,
     \end{align*} 
     and so we can find a $T_2$-invariant subset $U'\subseteq U$ with $\mu(U')>0$ such that $$\abs{\E(\chi_2|\CI(T_2))(x)}>0$$ for all $x\in U'$. Now, if $\lambda_2$ is the nonergodic eigenvalue of $\chi_2$, then from the pointwise ergodic theorem and the fact that $|\lambda_2(x)|\in\{0,1\}$ we deduce that
     \begin{equation}\label{E: projection of eigenfunction}
         \begin{split}
         \E(\chi_2|\CI(T_2))(x) &= \chi_2(x)\cdot \lim_{N\to\infty}\E_{n\in[N]}\lambda(x)^n\\ &= \chi_2(x)\cdot 1_{\lambda(x)=1} = \chi_2(x)\cdot 1_{\chi_2(x) = \chi_2(T_2 x)}    
         \end{split}
     \end{equation}
     % \begin{align}\label{E: projection of eigenfunction}
     %     \E(\chi_2|\CI(T_2))(x) = \chi_2(x)\cdot \lim_{N\to\infty}\E_{n\in[N]}\lambda(x)^n = \chi_2(x)\cdot 1_{\lambda(x)=1} = \chi_2(x)\cdot 1_{\chi_2(x) = \chi_2(T_2 x)}
     % \end{align}
     for $\mu$-a.e. $x\in U'$. Removing the exceptional null set from $U'$, we conclude that $\chi_2$ is $T_2$-invariant on $U'$. 
     %(If $T_2$ is ergodic, then this entire reduction simplifies, and it follows rather immediately from the invocation of Property $(P_1)$ that $\chi_2$ is constant.)
     % Now, let $\mu = \int \mu_x\; d\mu(x)$ be the ergodic decomposition of $\mu$ with respect to $T_2$. From \eqref{E: structure on nonergodic eigenfunction} we deduce that
     % \begin{align*}
     %    \abs{\E(\chi_2|\CI(T_2))(x)} = \abs{\int \chi_2 \; d\mu_x} >0
     % \end{align*}
     % for some $T_2$-invariant subset $U'\subseteq U$ with $\mu(U')>0$. Since $\chi_2$ is an eigenfunction of $T_2$ in $L^2(\mu_x)$ for $\mu$-a.e. $x\in X$, we deduce that $\chi_2$ must be constant $\mu_x$-a.e. for all $x\in U'$. 
     
     We have thus established that on a positive-measure set, $\chi_2$ is $T_2$-invariant. %(if $T_2$ were ergodic, we would reach the simpler conclusion that $\chi_2$ is constant).
     %positive proportion of ergodic components, $\chi_2$ becomes constant. 
     Hence $\chi_2$ possesses an extra structure that it a priori would not need to have. This conclusion plays a key part in degree lowering, and we shall see now how exactly it is applied.
     Combining the lower bound \eqref{E: lower bound on U}, the inclusion $U'\subseteq U$, and the fact that $U'$ is $T$-invariant and has positive measure, we deduce that
     \begin{align*}
        \int F_2 \cdot \overline{\chi_2}\cdot 1_{U'}  \; d\mu= \int \E(F_2 \cdot \overline{\chi_2}|\CI(T_2))\cdot 1_{U'}\; d\mu > 0.
     \end{align*}
     The function $\overline{\chi_2}\cdot 1_{U'}$ is $T_2$-invariant, which allows us to replace $F_2$ by $\E(F_2|\CI(T_2))$ in the first expression above. The Cauchy-Schwarz inequality then gives
     \begin{align*}
         \nnorm{F_2}_{1, T_2} = \norm{\E(F_2|\CI(T_2))}_{L^2(\mu)} > 0,
     \end{align*}
     as claimed

     Note that if $T_2$ is ergodic, then the argument following \eqref{E: chi_2 2} radically simplifies. It follows rather immediately from \eqref{E: chi_2 2} and the invocation of Property $(P_1)$ that $\chi_2$ is constant. On plugging this back to \eqref{E: positivity s=2}, we immediately reach the conclusion that $\nnorm{F_2}_{1,T_2} = \abs{\int F_2\; d\mu} >0$.
\end{proof}
%Without the ergodicity assumption, we would only be able to conclude that $\chi_2$ is $T_2$-invariant on a positive-measure set. This conclusion might suffice to reach the desired conclusion in the proof of Proposition \ref{P: degree lowering s=2}, but it would not work for $s>2$. In the nonergodic case, we instead need to work on ergodic components with respect to $T_2$ and use the fact that any nonergodic eigenfunction $\chi_2$ of $T_2$ is a (classical) eigenfunction on almost every ergodic component. The fact that ergodic components with respect to $T_2$ are not $T_1$-invariant gives rise to a number of spine-chilling technical issues that require great case. We direct the reader interested in these nightmarish subtleties to \cite[Section 7]{FrKu22a}.

Here is another perspective on the maneuvers carried out in the proof of Proposition \ref{P: degree lowering s=2}. We start with the average \eqref{E: degree lowering average} with arbitrary $f_2$. After the first stashing, we replace it by the auxiliary function $F_2$. This function is ``weakly structured'' in the sense that it encapsulates the behavior of the original average. In the proof of Proposition \ref{P: degree lowering s=2}, we reduce to the average \eqref{E: degree lowering average} in which the place of $f_2$ is taken by the nonergodic eigenfunction $\chi_2\cdot 1_U$ (see \eqref{E: chi_2}), which is ``properly structured''. In other words, we progressively replace $f_2$ by a more and more structured function. To this average we then apply Property $(P_1)$, the induction hypothesis.

We now state the second baby case of degree lowering, which introduces further technicalities compared to Proposition \ref{P: degree lowering s=3}.
\begin{proposition}[Degree lowering, baby case II]\label{P: degree lowering s=3}
    Assume %that $T_1, T_2$ are ergodic, and 
    that Property $(P_1)$ holds. If $\nnorm{F_2}_{3, T_2}>0$, then also $\nnorm{F_2}_{2, T_2}>0$. 
\end{proposition}
%For the reader who has not seen the degree lowering argument before, we advise to assume on the first read that $T_2$ is ergodic. Then simply $U=X$ and $\chi_2$ is an eigenfunction of $T_2$, hence \eqref{E: chi_2 2} and Property $(P_1)$ will imply that $\chi_2$ is constant.
\begin{proof}
    First, we observe that the inverse theorem for degree-3 Host-Kra seminorm involves much more complicated objects than nonergodic eigenfunctions. To land in the situation where Property $(P_1)$ is applicable, we apply the inductive formula for Host-Kra seminorms. It gives
    %concerns eigenfunctions, we need to bring our assumption to a situation in which 
    \begin{align*}
        \lim_{H\to\infty}\E_{h\in[H]}\nnorm{\Delta_{h\be_2} F_2}_{2, T_2}^4 >0,
    \end{align*}
    where $\Delta_{h\be_2}f:=f\cdot T_2^h\overline{f}$.
    %  By the inverse theorem for the degree-2 Host-Kra seminorm, for every $h\in\N$ there exists an eigenfunction $\chi_{2,h}$ of $T_2$ such that 
    % \begin{align*}
    %     \lim_{H\to\infty}\E_{h\in[H]}\int \Delta_{h\be_2} F_2 \cdot \overline{\chi_{2,h}}\; d\mu >0,
    % \end{align*}     
    % and furthermore
    % \begin{align}\label{E: condition on positivity}       
    % \end{align}
    In particular, there exist $\delta>0$ and  a set $\Lambda\subseteq \N$ of positive lower density such that 
    \begin{align}\label{E: lower bound on Lambda}
        \nnorm{\Delta_{h\be_2} F_2}_{2, T_2}^4 \geq \delta \quad \textrm{for all}\quad h\in \Lambda.
    \end{align}
    By the inverse theorem for the degree-2 Host-Kra seminorm, for every $h\in\Lambda$ there exists an eigenfunction $\chi_{2,h}$ of $T_2$ such that 
    \begin{align*}%\label{E: positivity s=2}
        \int \Delta_{h\be_2} F_2 \cdot {\chi_{2,h}}\; d\mu \geq \delta\quad \textrm{and} \quad \E(\Delta_{h\be_2} F_2\cdot {\chi_{2,h}}|\CI(T_2))\geq 0.
    \end{align*}
    We will show once again that the nonergodic eigenfunctions $\chi_{2,h}$ admit some special structure, and specifically that on positive-measure sets, they take the form $\chi\cdot \lambda_h$ for some nonergodic eigenfunctions $\chi$ of $T_2$ and $T_2$-invariant $\lambda_h$'s. (If $T_2$ is ergodic, then simply $\chi_{2,h} = \chi$ for $h\in\Lambda$.)
    %(for $h\in\Lambda$)
    
    By the popularity principle, for every $h\in\Lambda$ we can find a $T_2$-invariant set $U_h$ with $\mu(U_h)\geq \delta/2$ such that
    %Consequently, we can find $\delta > 0$ and a $T$-invariant set $U\subseteq X$ so that
    \begin{align}\label{E: lower bound on U s=3}
        \E(\Delta_{h\be_2} F_2\cdot {\chi_{2,h}}|\CI(T_2))(x) \geq \delta/2 \quad \textrm{for all}\quad x\in U_h.
    \end{align}
    Hence
    \begin{align}\label{E: original positivity s=3}
        \liminf_{H\to\infty}\E_{h\in[H]} 1_\Lambda(h)\int \Delta_{h\be_2} F_2 \cdot {\chi_{2,h}}\cdot 1_{U_h}\; d\mu > 0.
    \end{align}

    The key complication in this case compared to the one covered by Proposition \ref{P: degree lowering s=2} is the presence of the multiplicative derivative $\Delta_{h\be_2} F_2$ in place of $F_2$. We want to move the differentiation ``inside'' the definition of $F_2$. This step, not present in Proposition \ref{P: degree lowering s=2}, is called \emph{dual-difference interchange}, and we carry it now. 
    
    The rule of thumb in ergodic theory is that if we do not know what to do, it does not hurt to try applying the Cauchy-Schwarz inequality. This is indeed what we will do. First, however, we expand the definition of $T_2^h F_2$ (but not of the second $F_2$), obtaining
    \begin{align*}
        \liminf_{H\to\infty}\E_{h\in[H]}1_\Lambda(h) \lim_{N\to\infty}\E_{n\in[N]}\int F_2\cdot T_2^h\brac{T_2^{-a_2(n)}{f_0}\cdot T_1^{a_1(n)}T_2^{-a_2(n)}{f_1}}\cdot {\chi_{2,h}}\cdot 1_{U_h}\; d\mu > 0.
    \end{align*}
    %Removing the complex conjugates and 
    Moving the (finite) average over $h$ inside the integral and applying the Cauchy-Schwarz inequality, we get
        \begin{align*}
        \liminf_{H\to\infty} \lim_{N\to\infty}\E_{n\in[N]}\norm{\E_{h\in[H]}1_\Lambda(h) \cdot T_2^h\brac{T_2^{-a_2(n)}{f_0}\cdot T_1^{a_1(n)}T_2^{-a_2(n)}{f_1}}\cdot {\chi_{2,h}}\cdot 1_{U_h}}_{L^2(\mu)}^2> 0.
    \end{align*}
    We then expand the square to get the nasty inequality
    \begin{multline*}
         \liminf_{H\to\infty} \limsup_{N\to\infty}\E_{n\in[N]}\E_{h,h'\in[H]}1_{\Lambda^2}(h,h')\\
         \int T_2^h\brac{T_2^{-a_2(n)}{f_0}\cdot T_1^{a_1(n)}T_2^{-a_2(n)}{f_1}}\cdot T_2^{h'}\brac{T_2^{-a_2(n)}\overline{f_0}\cdot T_1^{a_1(n)}T_2^{-a_2(n)}\overline{f_1}}\\ \chi_{2,h}\cdot\overline{\chi_{2,h'}}\cdot 1_{U_h}\cdot 1_{U_{h'}}\; d\mu > 0.
    \end{multline*}

    Because we started with a degree-3 Host-Kra seminorm, the dual-difference interchange required only one application of the Cauchy-Schwarz inequality to perform dual-difference interchange. If we started with the degree-$s$ seminorm, $s-2$ applications would be needed. Given how much notational mess one Cauchy-Schwarz has introduced, we leave it to the reader's vivid imagination to envisage the nightmare that the general case of dual-difference interchange leads to.
    
    Fortunately, we are able to make the notation more palatable by a few manipulations. Composing the integral with $T_2^{-h'}$, defining
    \begin{align*}
        f_{j,h,h'} := \Delta_{(h-h')\be_2}\overline{f_j},\quad \chi_{2,h,h'} := T_2^{-h'}(\chi_{2,h}\cdot\overline{\chi_{2,h'}}),\quad U_{h,h'} := U_h\cap U_{h'},
    \end{align*}
    and swapping the limsup over $N$ with the average over $h,h'$, we get
    \begin{multline*}%\label{E: after dual-difference}
         \liminf_{H\to\infty} \E_{h,h'\in[H]}1_{\Lambda^2}(h,h')\\
         \limsup_{N\to\infty}\abs{\E_{n\in[N]}\int T_2^{-a_2(n)}f_{0,h,h'}\cdot T_1^{a_1(n)}T_2^{-a_2(n)}{f_{1,h,h'}}\cdot \chi_{2,h,h'}\cdot 1_{U_{h,h'}}\; d\mu} > 0.
    \end{multline*}
    
    This expression looks quite like \eqref{E: expanding dual}, and we proceed analogously to the proof of Proposition \ref{P: degree lowering s=2}. Composing the integral with $T_2^{a_2(n)}$ and applying the Cauchy-Schwarz inequality gives us
    \begin{multline*}
        \liminf_{H\to\infty} \E_{h,h'\in[H]}1_{\Lambda^2}(h,h')
         \limsup_{N\to\infty}\norm{\E_{n\in[N]}T_1^{a_1(n)}{f_{1,h,h'}}\cdot T_2^{a_2(n)}(\chi_{2,h,h'}\cdot 1_{U_{h,h'}})}_{L^2(\mu)} > 0.
    \end{multline*}

    The functions $\chi_{2,h,h'}\cdot 1_{U_{h,h'}}$ are nonergodic eigenfunctions of $T_2$, and so invoking Property $(P_1)$ much the same way as in the proof of Proposition \ref{P: degree lowering s=2}, we deduce that 
    \begin{align*}
                \liminf_{H\to\infty} \E_{h,h'\in[H]}1_{\Lambda^2}(h,h')
         \norm{\E(\chi_{2,h,h'}\cdot 1_{U_{h,h'}}|\CI(T_2))}_{L^2(\mu)} > 0.
    \end{align*}
    By the pigeonhole principle, we can find a single\footnote{Technically, this $h'$ may vary with $H$, but we omit this technicality in order not to clutter the notation.} $h'\in\N$ for which
    \begin{align*}
                \liminf_{H\to\infty} \E_{h\in[H]}1_{\Lambda}(h)
         \norm{\E(\chi_{2,h,h'}\cdot 1_{U_{h,h'}}|\CI(T_2))}_{L^2(\mu)} > 0.
    \end{align*}
    Just like in Proposition \ref{P: degree lowering s=2}, we can find $T_2$-invariant sets $U'_h\subseteq U_{h,h'}$ for which
    \begin{align}\label{E: structure on s=3}
        \liminf_{H\to\infty} \E_{h\in[H]}1_{\Lambda}(h)
         \mu(U'_h) > 0,
    \end{align}    
    and such that $\chi_{2,h,h'}$ is $T_2$-invariant on $U'_h$. Given the formula for $\chi_{2,h,h'}$, this means that $\chi_{2,h} = \chi\cdot\lambda_h$ (where we set $\chi:=\chi_{2,h'}$) on $U'_h$ for some $T_2$-invariant $\lambda_h$. Just like in Proposition \ref{P: degree lowering s=2}, we have uncovered some structure on the eigenfunctions $\chi_{2,h}$ although this time the structure is in a sense ``shared'' across different $h$. 

We now combine the lower bound \eqref{E: lower bound on U s=3} and \eqref{E: structure on s=3} with the fact that $U'_h\subseteq U_h$ is $T_2$-invariant to infer that
    \begin{align*}
         \liminf_{H\to\infty} \E_{h\in[H]}1_{\Lambda}(h)\int \E(\Delta_{h\be_2} F_2\cdot {\chi_{2,h}}|\CI(T_2))\cdot 1_{U'_h}\; d\mu > 0.
    \end{align*}
    Using the special structure of  $\chi_{2,h}$ on $U'_h$ and the $T_2$-invariance of $U'_h$, we deduce that
    \begin{align}\label{E: plugging special structure}
                 \liminf_{H\to\infty} \E_{h\in[H]}1_{\Lambda}(h)\int \Delta_{h\be_2} F_2\cdot \chi\cdot \lambda_h\cdot 1_{U'_h}\; d\mu > 0.
    \end{align}
    % By setting $\lambda_h=0$ for $h\notin\Lambda$ (this does not affect any of the discussion above), we get
    % \begin{align*}
    %              \liminf_{H\to\infty} \E_{h\in[H]}\int \Delta_{h\be_2} F_2\cdot \chi\cdot \lambda_h\cdot 1_{U'_h}\; d\mu > 0.
    % \end{align*}
    We invite the reader to compare this expression with \eqref{E: original positivity s=3}, the starting point, to see how arbitrary $T_2$-eigenfunctions $\chi_{2,h}$ have been replaced by the more structured ones. 

    We are almost at the end of the argument; all that remains is to remove the ``low complexity'' terms $1_{\Lambda}(h)\cdot \chi\cdot \lambda_h\cdot 1_{U'_h}$ from the integral. By applying the Cauchy-Schwarz inequality in $h$, we can remove $\chi$, getting
    \begin{align*}
        \liminf_{H\to\infty} \norm{\E_{h\in[H]}T_2^h \overline{F_2}\cdot 1_{\Lambda}(h)\cdot \lambda_h\cdot 1_{U'_h}}_{L^2(\mu)}^2 > 0.
    \end{align*}
    Expanding the square and using the $T_2$-invariance of $1_{\Lambda}(h)\cdot\lambda_h\cdot 1_{U'_h}$, we further infer from the Cauchy-Schwarz inequality that
    \begin{align*}
        \liminf_{H\to\infty} \E_{h,h'\in[H]}\norm{\E(T_2^h F_2\cdot T_2^{h'}\overline{F_2}|\CI(T_2))}_{L^2(\mu)}>0.
    \end{align*}
    A simple exercise shows that this implies $\nnorm{F_2}_{2, T_2}>0,$ as claimed.
    
    Note that if $T_2$ is ergodic, then from \eqref{E: structure on s=3}, we simply get that $\chi_{2,h} = \chi$ for the positive-lower-density set $h\in\Lambda$. Then the conclusion of \eqref{E: plugging special structure} simplifies to
    \begin{align*}
        \liminf_{H\to\infty} \E_{h\in[H]}1_\Lambda(h)\int \Delta_{h\be_2} F_2\cdot \chi\; d\mu %> 0.
        = \liminf_{H\to\infty} \int F_2\cdot \chi \cdot \E_{h\in[H]} T_2^h \overline{F_2}\cdot 1_\Lambda(h)\; d\mu > 0.
    \end{align*}
    The Cauchy-Schwarz inequality gives 
    \begin{align*}
        \lim_{H\to\infty}\norm{\E_{h\in[H]}T_2^h \overline{F_2}\cdot 1_\Lambda(h)}_{L^2(\mu)}>0,
    \end{align*}
    and by expanding the square and pulling the indicators out of the integral, it is easy to conclude that $\nnorm{F_2}_{2, T_2}>0$.
    % The maneuver of defining $\lambda_h = 0$ for $h\notin\Lambda$ has the same effect as taking absolute values outside the integral and using nonnegativity to get
    %     \begin{align*}
    %     \liminf_{H\to\infty} \E_{h\in[H]}\abs{\int \Delta_{h\be_2} F_2\cdot \chi\; d\mu} > 0.
    %     %= \liminf_{H\to\infty} \int F_2\cdot \chi \cdot \E_{h\in[H]} T_2^h \overline{F_2}\; d\mu > 0,
    % \end{align*}
    % and the application of the Cauchy-Schwarz inequality gives
    % \begin{align*}
    %     \lim_{H\to\infty}
    % \end{align*}
\end{proof}

\subsubsection{Other degree lowering arguments}
The argument presented above is one of four existing degree lowering arguments to prove Theorem \ref{T: invariant control criteria} and its variants (Theorem \ref{T: joint ergodicity criteria} and \ref{T: Krat control criteria}), in addition to those from \cite{Fr21, FrKu22c, FrKu22a} (see also \cite{BFM22} for a variant of the degree lowering from \cite{Fr21} carried out for more general ring actions). While they all differ slightly at the level of details, the overarching strategy in all of them is more or less the same. Specifically, all of them use the inverse theorem for the degree-2 Host-Kra seminorm to reduce the seminorm degree and aim at degree-1 control or its variant (rational Kronecker factor control). The degree lowering philosophy is, however, so robust that it can be used to reach other ``optimal seminorm control'' statements. Theorems \ref{T: FrKu degree lowering on APs}, \ref{T: FrKu sparse corners}, \ref{T: FrKu poly corners}, and \ref{T: degree reduction} are also proved via degree lowering arguments that use other inverse theorems as their starting points and therefore reach different conclusions.

\subsection{New techniques for seminorm control}\label{S: seminorm smoothing}

The second big breakthrough that we now outline are new ways to prove Host-Kra seminorm control. %This section requires the reader to be familiar with the notions of and notations for box seminorms introduced in Appendix \ref{A: Host-Kra theory}.
Previously available technology relied on variants of the PET induction scheme of Bergelson \cite{Ber87} (see Example \ref{Ex: PET n, n^2}), and it was only capable of delivering seminorm control in two cases of interest:
\begin{enumerate}
    \item in the single-transformation case, for ``essentially distinct'' sequences;
    \item in the commuting case, for sequences of ``distinct growth''.
\end{enumerate}
Current machinery blends the PET induction scheme with concatenation and seminorm smoothing techniques, yielding three types of results:
\begin{enumerate}
    \item Host-Kra seminorm estimates for ``pairwise independent'' sequences in the commuting case;
    \item Host-Kra seminorm estimates for general (i.e. potentially ``pairwise dependent'') sequences in the commuting case under additional ergodicity assumptions on the system;
    \item box seminorm estimates in the commuting case.
\end{enumerate}
In this survey, we will only present results of the first (Proposition \ref{P: smoothing} and Theorem \ref{T: HK estimates for pairwise independent polys}) and third type (Theorem \ref{T: DKKST box seminorm control}). For a result of the second type, we direct the reader to \cite[Theorem 1.1]{FrKu22b} or \cite[Theorem 4.10]{DKKST25}. %A result of the third type is covered in \cite[Theorem 2.5]{DFKS22} (see also \cite[Proposition 8.1]{FrKu22a} for an alternative presentation) and \cite[Theorem 10.1]{DKKST25}.

\subsubsection{PET induction scheme: strengths and limitations}\label{SSS: PET}
The one and only way of proving Host-Kra seminorm control before the advent of new techniques was the PET induction scheme. It relies on methodically applying the van der Corput and Cauchy-Schwarz inequalities to an average until we reach one with linear iterates that can be controlled via Host-Kra seminorms as in Example \ref{Ex: HK example}.
% \begin{proposition}
% Let $\ell\in \N$ and $(X, \CX, \mu, T)$ be a system. Then there exists a constant $C_\ell>0$ and a seminorm $\nnorm{f}_{\ell, T}$ on $L^\infty(\mu)$ such that for any 1-bounded $f_1, \ldots, f_\ell\in L^\infty(\mu)$, we have
% \begin{align*}
% \limsup_{N\to\infty}\norm{\E_{n\in[N]} T^n f_1 \cdots T^{\ell n} f_\ell}_{L^2(\mu)} \leq C_\ell \nnorm{f_\ell}_{\ell, T}.
% \end{align*}
% \end{proposition}
The following example illustrates this procedure.
\begin{example}[PET for $(T_1^n, T_2^{n^2})_n$]\label{Ex: PET n, n^2}
    Let $(X, \CX, \mu,$\! $T_1, T_2)$ be a system, and let $f_1, f_2\in L^\infty(\mu)$. Suppose that 
    \begin{align*}
        \lim_{N\to\infty}\norm{\E_{n\in[N]}T_1^n f_1\cdot T_2^{n^2}f_2}_{L^2(\mu)}>0.
    \end{align*}
    By the van der Corput inequality (Lemma \ref{L: vdC}), 
    \begin{align*}
        \limsup_{H_1\to\infty}\E_{h_1\in[H_1]}\lim_{N\to\infty}\abs{\E_{n\in[N]}\int T_1^n f_1\cdot T_1^{n+h_1} \overline f_1\cdot T_2^{n^2}f_2\cdot T_2^{(n+h_1)^2}f_2\; d\mu}>0.
    \end{align*}
    Composing the average with $T_1^{-n}$, we get
    \begin{align*}
        \limsup_{H_1\to\infty}\E_{h_1\in[H_1]}\lim_{N\to\infty}\abs{\E_{n\in[N]}\int f_1\cdot T_1^{h_1} \overline f_1\cdot T_2^{n^2}T_1^{-n}f_2\cdot T_2^{(n+h_1)^2}T_1^{-n}\overline f_2\; d\mu}>0.
    \end{align*}
    We then apply the Cauchy-Schwarz inequality to eliminate the terms independent of $n$, obtaining 
    \begin{align*}
        \limsup_{H_1\to\infty}\E_{h_1\in[H_1]}\lim_{N\to\infty}\norm{\E_{n\in[N]} T_2^{n^2}T_1^{-n}f_2\cdot T_2^{(n+h_1)^2}T_1^{-n}\overline f_2}_{L^2(\mu)}>0.
    \end{align*}

    This new average may look more complicated than the one we started with, not least because it involves two quadratic polynomials rather than one. However, a key observation is that both quadratic polynomials have the same leading coefficients. This relative simplicity of the new average compared to the original one is a baby example of the \emph{complexity reduction} that underpins the PET induction scheme, and it will come in handy shortly. 
    
    Applying the van der Corput lemma once more, we deduce that
        \begin{multline*}
        \limsup_{H_2\to\infty}\E_{h_2\in[H_2]}\limsup_{H_1\to\infty}\E_{h_1\in[H_1]}\\
        \lim_{N\to\infty}
        \left|\E_{n\in[N]}\int T_2^{n^2}T_1^{-n}f_2\cdot T_2^{(n+h_2)^2}T_1^{-(n+h_2)}\overline f_2\right.\\ 
        \left. \vphantom{\E_{n\in[N]}\int} T_2^{(n+h_1)^2}T_1^{-n}\overline f_2\cdot T_2^{(n+h_1+h_2)^2}T_1^{-(n+h_2)}f_2\; d\mu\right|>0.
        %\limsup_{H_2\to\infty}\E_{h_2\in[H_2]}\limsup_{H_1\to\infty}\E_{h_1\in[H_1]}\lim_{N\to\infty}\\
        %\abs{\E_{n\in[N]}\int T^{n^2-n}f_2\cdot T^{(n+h_2)^2-(n+h_2)}\overline f_2\cdot T^{(n+h_1)^2-n}\overline f_2\cdot T^{(n+h_1+h_2)^2-(n+h_2)}f_2\; d\mu}>0.
    \end{multline*}
    Because all four terms have the same leading coefficient in $n$, composing the integral with $T_1^nT_2^{-n^2}$ allows us to reduce the average to one linear in $n$, giving
    \begin{multline*}
        \limsup_{H_2\to\infty}\E_{h_2\in[H_2]}\limsup_{H_1\to\infty}\E_{h_1\in[H_1]}\lim_{N\to\infty}\\
        \abs{\E_{n\in[N]}\int f_2\cdot T_2^{2h_2 n + h_2^2}T_1^{-h_2}\overline f_2\cdot T_2^{2h_1n + h_1^2}\overline f_2\cdot T_2^{2(h_1+h_2)n+(h_1+h_2)^2}T_1^{-h_2}f_2\; d\mu}>0.
    \end{multline*}
    Crucially, for all $h_1, h_2\in\N$ except the zero-density set $\{(h_1,h_2)\in\N\colon h_1 = h_2\}$, the leading coefficients $2h_1, 2h_2, 2(h_1+h_2)$ are all different. Therefore an application of the bound \eqref{E: box seminorm control} jointly with the scaling property \eqref{E: scaling 1} implies that $\nnorm{f_2}_{3, T_2} >0$. The claimed seminorm control follows by contrapositive.

    There are many variants of the argument above. For instance, by using a quantitative version of the van der Corput lemma, we can ensure that $H_1, H_2$ are the same; we can also let $H$ grow slowly with $N$. By taking $H_1, H_2$ large but finite, we can also obtain a quantitative statement of the form
     \begin{align*}
        \lim_{N\to\infty}\norm{\E_{n\in[N]}T_1^n f_1\cdot T_2^{n^2}f_2}_{L^2(\mu)}^{O(1)}\ll \nnorm{f_2}_{3, T_2}
    \end{align*}
    under the assumptions of 1-boundedness on $f_1, f_2$.
    %Such modifications have proved necessary in the study of multiple ergodic averages along Hardy sequences.
\end{example}

Unfortunately, the argument above fails to give Host-Kra seminorm control even for $(T_1^{n^2}, T_2^{n^2+n})_n$. In this case, analogous manipulations only allow us to control the corresponding multiple ergodic average by an average of box seminorms.
\begin{example}[PET for $(T_1^{n^2}, T_2^{n^2+n})_n$]\label{Ex: PET n^2, n^2+n}
    Let $(X, \CX, \mu,$\! $T_1, T_2)$ be a system, and let $f_1, f_2\in L^\infty(\mu)$ be 1-bounded. Then 
    \begin{align}\label{E: box seminorm control n^2, n^2+n}
        \lim_{N\to\infty}\norm{\E_{n\in[N]}T_1^{n^2}f_1\cdot T_2^{n^2+n}}_{L^2(\mu)}^{O(1)}\ll \limsup_{H\to\infty}\E_{h_1,h_2,h_3\in[H]}\nnorm{f_2}_{\bc_1(h), \ldots, \bc_7(h)},
    \end{align}
    where the vectors $\bc_1(h), \ldots, \bc_7(h)$ are given by
    \begin{gather*}
        2\be_2 h_1, 2(\be_2-\be_1)h_2, 2(\be_2-\be_1)h_3,\\
        2\be_2h_1 + 2(\be_2-\be_1)h_2, 2 \be_2 h_1 + 2(\be_2-\be_1)h_3, 2(\be_2-\be_1)(h_2+h_3),\\
        2\be_2h_1 + (\be_2-\be_1)(h_2+h_3).
    \end{gather*}
    A priori, this is much worse than the conclusion of Example \ref{Ex: PET n, n^2} since there is ostensibly no good reason why the beastly expression on the right-hand side of \eqref{E: box seminorm control n^2, n^2+n} would be of any use. Yet the explicit description of the vectors $\bc_1(h), \ldots, \bc_7(h)$ offers certain hope.
    Two crucial points that we shall exploit later is that these vectors are linear polynomials in $h$ (for polynomials of higher degree than $n^2, n^2 + n$, they will be multilinear), and their coefficients are (multiples of) the leading coefficients of the polynomials $\be_2 (n^2+n)$ and $\be_2 (n^2 + n) - \be_1 n^2$.
\end{example}

The computation performed in Example \ref{Ex: PET n^2, n^2+n} can be found in \cite[Example 7]{DFKS22} (for an alternative presentation in the discrete setting, see also \cite[Example 2]{Kuc23} or \cite[Section 4.1]{KKL24a}). In \cite[Section 4]{DFKS22}, Donoso, Ferr\'e-Moragues, Koutsogiannis, and Sun developed a general \emph{coefficient tracking scheme} (later refined by Kravitz, Leng and the author \cite[Section 4]{KKL24a}; see also \cite[Section 7.1]{DKKST24})) that allows us to upper bound
\begin{align}\label{E: general average with products}
    \lim_{N\to\infty}\norm{\E_{n\in[N]}T^{\ba_1(n)}f_1 \cdots T^{\ba_\ell(n)}f_\ell}_{L^2(\mu)}
\end{align}
for some $\ell$-tuples $\ba_1, \ldots, \ba_\ell:\N\to\Z^\ell$ of polynomials, Hardy sequences, etc.
%, the PET induction scheme allows us to upper bound the average with polynomial bounds 
by an \emph{average of box seminorms} of the form
\begin{align*}
    \lim_{H\to\infty}\E_{h\in [H]^s}\nnorm{f_\ell}_{\bc_1(h), \ldots, \bc_s(h)},
\end{align*}
where $\bc_1, \ldots,\bc_s$ are multilinear polynomials in $h$ whose coefficients depend only on the ``leading coefficients''\footnote{This notion means precisely what one thinks it does for polynomials, but it is trickier to define for Hardy sequences. See \cite[Theorem 10.1]{DKKST24} and the discussion underneath for illustration.} of $\ba_\ell, \ba_\ell-\ba_1, \ldots, \ba_\ell-\ba_{\ell-1}$. The statements of such results, each wrapped in a formidable amount of formalism and terminology, can be found e.g. in \cite[Proposition 5.3]{DFKS22}, \cite[Proposition 4.9]{KKL24a}, and \cite[Proposition 7.7]{DKKST24}.

\subsubsection{Concatenation: control by a single box seminorm}\label{SSS: concatenation}
The estimate obtained in Example \ref{Ex: PET n^2, n^2+n} looks rather daunting. It turns out, however, that we can control the troublesome average of box seminorms on the right-hand side of \eqref{E: box seminorm control n^2, n^2+n} by a single box seminorm using a technique known as \emph{concatenation}. Its end result is presented below.
\begin{example}[Concatenation for $(T_1^{n^2}, T_2^{n^2+n})_n$]\label{Ex: concatenation n^2, n^2+n}
    Let $\bc_1, \ldots, \bc_7$ be the polynomials showing up in \eqref{E: box seminorm control n^2, n^2+n}. Then there exists $s\in\N$ such that
    \begin{align}\label{E: concatenation n^2, n^2+n}
        \brac{\limsup_{H\to\infty}\E_{h_1,h_2,h_3\in[H]}\nnorm{f}_{\bc_1(h), \ldots, \bc_7(h)}}^{O(1)}\ll \nnorm{f}_{\be_2^{\times s}, (\be_2-\be_1)^{\times s}}
    \end{align}
    (recall that $\be_2^{\times s}$ denotes $s$ copies of $\be_2$). 
\end{example}

The idea behind the estimate \eqref{E: concatenation n^2, n^2+n} is to \emph{merge} (``concatenate'') different directions $\bc_1(h), \ldots, \bc_7(h)$ with each other in a way that allows us to recover the \emph{principal directions} $\be_2, \be_2-\be_1$. The modern way to do this (performed e.g. in \cite{DKKST24, KKL24a, Kuc23}) is to apply the Cauchy-Schwarz inequality many times to the left-hand side of \eqref{E: concatenation n^2, n^2+n} in a targeted way. A priori, this upper bounds the left-hand side of \eqref{E: concatenation n^2, n^2+n} by a much larger and complicated-looking average of box seminorms. If done correctly, however, the directions of the new seminorms are ``equidistributed'' as the parameters $h$ vary. Using elementary counting arguments, we can then replace these equidistributed directions by $\be_2, \be_2-\be_1$, or the group $\langle \be_2, \be_2-\be_1\rangle$. We present a very simple case of this argument below, warning the reader that the general case involves a considerably more complicated induction.
\begin{example}[Simple example of concatenation]
    Let $(X, \CX, \mu,$\! $T_1, T_2)$ be a system, and let $f\in L^\infty(\mu)$ be 1-bounded.
    We show that 
    \begin{align*}
        \brac{\limsup_{H\to\infty}\E_{h_1,h_2\in[H]}\nnorm{f}_{\be_1 h_1 + \be_2 h_2}^2}^2\ll \nnorm{f}_{\langle\be_1,\be_2\rangle,\langle\be_1,\be_2\rangle},
    \end{align*}
    and hence the right-hand side can be bounded by $\min\limits_{j=1,2}\nnorm{f}_{\be_j, \be_j}$ using \eqref{E: scaling 1}.
    Let
    \begin{align*}
        \delta &:= \limsup_{H\to\infty}\E_{h_1,h_2\in[H]}\nnorm{f}_{\be_1 h_1 + \be_2 h_2}^2\\
        &= \limsup_{H\to\infty}\E_{h_1,h_2\in[H]}\lim_{M\to\infty}\E_{m\in[M]}\int f\cdot T_1^{h_1 m}T_2^{h_2 m}\overline{f}\; d\mu\\
        &=\limsup_{H\to\infty}\lim_{M\to\infty}\int f\cdot \E_{h_1,h_2\in[H]}\E_{m\in[M]} T_1^{h_1 m}T_2^{h_2 m}\overline{f}\; d\mu.
    \end{align*}
    By applying the Cauchy-Schwarz inequality, expanding the square, and changing variables, we get
    \begin{align*}
        \delta^2 &\leq \limsup_{H\to\infty}\limsup_{M\to\infty}\int \E_{\substack{h_{11},h_{12},\\ h_{21},h_{22}\in[H]}}\E_{m_1,m_2\in[M]} T_1^{h_{11} m_1}T_2^{h_{21} m_1}{f}\cdot T_1^{h_{12} m_2}T_2^{h_{22} m_2}\overline{f}\; d\mu\\
        &\leq \limsup_{H\to\infty}\limsup_{M\to\infty}\int \E_{\substack{h_{11},h_{12},\\ h_{21},h_{22}\in[H]}}\E_{m_1,m_2\in[M]} {f}\cdot T_1^{h_{12} m_2 -h_{11} m_1}T_2^{h_{22} m_2-h_{21} m_1}\overline{f}\; d\mu.
    \end{align*}
    We repeat this procedure once more to get
    \begin{multline*}
        \delta^4\leq \limsup_{H\to\infty}\limsup_{M\to\infty}\\
        \int \E_{\substack{h_{1j},h_{2j}\in[H]:\\ 1\leq j\leq 4}}\E_{\substack{m_j\in[M]:\\ 1\leq j\leq 4}} {f}\cdot T_1^{h_{14} m_4 -h_{13} m_3- h_{12} m_2 +h_{11} m_1}T_2^{h_{24} m_4 -h_{23} m_3- h_{22} m_2 +h_{21} m_1}\overline{f}\; d\mu.
    \end{multline*}
    The key point is that as $h_{1j}, h_{2j}$'s vary over $[H]$ and $m_j$'s vary over $[M]$, then the polynomial 
    \begin{align}\label{E: multilinear poly}
        (h_{14} m_4 -h_{13} m_3- h_{12} m_2 +h_{11} m_1, h_{24} m_4 -h_{23} m_3- h_{22} m_2 +h_{21} m_1)
    \end{align}
    is ``equidistributed'' over its range $[\pm 2HM]^2$ in the sense that each element $(n_1,n_2)\in[\pm 2HM]^2$ admits at most $O(1)$ times the expected number of representations of the form \eqref{E: multilinear poly} (for large $H,M$). This bound, which we do not prove here, follows from a much more general estimate \cite[Proposition 7.5]{KKL24a}.
    What happens is that two applications of the Cauchy-Schwarz inequality give us more variables to play with at the cost of losing a power of $\delta$, and the polynomials in more variables (like \eqref{E: multilinear poly}) tend to be more uniformly distributed over their range than those in fewer variables (like the original polynomial $(h_1 m, h_2 m)$).
    
    Plugging this equidistribution property into the last inequality, we get
    \begin{align*}
        \delta^4 \ll \limsup_{H\to\infty}\limsup_{M\to\infty}\E_{n_1,n_2\in[\pm 2 HM]}\abs{\int f\cdot T_1^{n_1}T_2^{n_2}\overline{f}\; d\mu}.
    \end{align*}
    One more application of the Cauchy-Schwarz inequality gives
    \begin{align*}
        \delta^8 \ll  \limsup_{M\to\infty}\E_{n_1,n_2\in[\pm M]}\abs{\int f\cdot T_1^{n_1}T_2^{n_2}\overline{f}\; d\mu}^2 = \nnorm{f\otimes \overline f}_{\langle \be_1,\be_2\rangle}^2 \leq\nnorm{f}_{\langle \be_1,\be_2\rangle, \langle \be_1,\be_2\rangle}^4,
    \end{align*}
    where we leave the last inequality as an exercise. 
    %upper bounds the right-hand side by $\nnorm{f}_{\langle\be_1,\be_2\rangle,\langle\be_1,\be_2\rangle},$ as claimed.    
\end{example}

The original concatenation argument was developed by Tao and Ziegler \cite{TZ16} to study polynomial progressions in primes. It differed from the argument above in that it was about concatenating \emph{factors}, not seminorms, and hence it was very qualitative in nature. This argument was later applied in the works of Donoso, Koutsogiannis, and Sun \cite{DKS22, DKS23} as well as Donoso, Ferr\'e-Moragues, Koutsogiannis, and Sun \cite{DFKS22} to give \emph{qualitative} comparison control averages of box seminorms 
by a single box seminorm. In particular, in \cite[Example 8]{DFKS22}, they show the following qualitative version of \eqref{E: box seminorm control n^2, n^2+n}:
\begin{align}\label{E: box seminorm control n^2, n^2+n infinitary}
    \limsup_{H\to\infty}\E_{h_1,h_2,h_3\in[H]}\nnorm{f}_{\bc_1(h), \ldots, \bc_7(h)} = 0\quad \textrm{whenever}\quad \nnorm{f}_{\be_2^{\times s}, (\be_2-\be_1)^{\times s}} = 0.
\end{align}
%(compare it with \eqref{E: box seminorm control n^2, n^2+n}).

In a different direction, Peluse and Prendiville \cite{PP19} and Peluse \cite{Pel20} developed the first \emph{quantitative concatenation} argument in their study of single-dimensional polynomial progressions. A more flexible version of this argument was later developed by the author \cite{Kuc23} in the context of multidimensional polynomial patterns over the finite fields, and later extended to the integer setting jointly with Kravitz and Leng \cite{KKL24a}. The argument was subsequently translated to ergodic theory by Donoso, Koutsogiannis, Sun, Tsinas, and the author \cite{DKKST24}. The exposition provided above is based on the argument stemming from \cite{DKKST24, KKL24a, Kuc23}.

By combining the PET induction scheme from Section \ref{SSS: PET} with the concatenation arguments outlined above, we obtain the following rather general statement.
\begin{theorem}[Box seminorm control for general polynomial averages]\label{T: DKKST box seminorm control}
    Let $\ba_1, \ldots, \ba_\ell\in\Z^\ell[t]$ be essentially distinct. Then there exist $C>0$, $s\in\N$, and vectors $\bv_1, \ldots, \bv_s\in\Z^\ell$ belonging to the set of leading coefficients of $\be_\ell, \be_\ell-\be_1, \ldots, \be_\ell-\be_1$, such that for all systems $(X, \CX, \mu, T_1, \ldots, T_\ell)$ and all 1-bounded functions $f_1, \ldots, f_\ell\in L^\infty(\mu)$, we have
    \begin{align*}
        \lim_{N\to\infty}\norm{\E_{n\in[N]}T^{\ba_1(n)}f_1\cdots T^{\ba_\ell(n)}f_\ell}_{L^2(\mu)}^C \leq C \nnorm{f_\ell}_{\bv_1, \ldots, \bv_s}.
    \end{align*}
\end{theorem}
Theorem \ref{T: DKKST box seminorm control} is a special case of \cite[Proposition 8.6]{DKKST24}. Its discrete version can be found in \cite[Theorem 1.1]{KKL24a}, and its qualitative version was previously established in \cite[Theorem 2.5]{DFKS22}. Theorem \ref{T: DKKST box seminorm control} extends to (tuples of) Hardy sequences (see \cite[Theorem 10.1]{DKKST24}). However, this extension is much more challenging to state because the notion of leading coefficients for Hardy sequences is more tricky, and because we need a more robust notion of box seminorms (defined in \cite[Section 3]{DKKST24}) for tuples like $(T_1^{\sfloor{\sqrt{2}n^2}}, T_2^{\sfloor{\sqrt{3}n^2 + n}})_n$ in which the leading coefficient $\sqrt{3}\be_2-\sqrt{2}\be_1$ of the difference $\sfloor{\sqrt{3}n^2 + n}\be_2-\sfloor{\sqrt{2}n^2}\be_1$ has rationally independent entries. 

\subsubsection{Seminorm smoothing}
By combining the arguments from Examples \ref{Ex: PET n^2, n^2+n} and \ref{Ex: concatenation n^2, n^2+n}, we obtain the following inequality:
\begin{align}\label{E: single box seminorm control n^2, n^2+n}
    \lim_{N\to\infty}\norm{\E_{n\in[N]}T_1^{n^2}f_1\cdot T_2^{n^2+n}}_{L^2(\mu)}^{O(1)}\ll \nnorm{f}_{\be_2^{\times s}, (\be_2-\be_1)^{\times s}},
\end{align}
which looks much nicer than \eqref{E: box seminorm control n^2, n^2+n}. Unfortunately, it is not very useful in itself, as we know no useful inverse theorems for the box seminorm $\nnorm{f}_{\be_2^{\times s}, (\be_2-\be_1)^{\times s}}$ for $s>1$. The last step, therefore, is to upgrade \eqref{E: single box seminorm control n^2, n^2+n} to a proper Host-Kra seminorm estimate. This is done via a \emph{seminorm smoothing} technique: it consists in ``smoothing'' the seminorm $\nnorm{f}_{\be_2^{\times s}, (\be_2-\be_1)^{\times s}}$ by replacing each vector $\be_2-\be_1$ one-by-one by (possibly many) copies of the vector $\be_2$. In this particular example, this result takes the following form.
\begin{proposition}\label{P: smoothing}
    Let $(X, \CX, \mu,$\! $T_1, T_2)$ be a system and $f_1, f_2\in L^\infty(\mu)$ be 1-bounded. Let $s_1, s_2\in\N$. There exists $s_1'\in\N$ so that if 
    \begin{align*}
        \brac{\lim_{N\to\infty}\norm{\E_{n\in[N]}T_1^{n^2}f_1\cdot T_2^{n^2+n}}_{L^2(\mu)}}^{O(1)}\ll \nnorm{f}_{\be_2^{\times s_1}, (\be_2-\be_1)^{\times s_2}},
    \end{align*}
    then also
    \begin{align*}
        \brac{\lim_{N\to\infty}\norm{\E_{n\in[N]}T_1^{n^2}f_1\cdot T_2^{n^2+n}f_2}_{L^2(\mu)}}^{O(1)}\ll \nnorm{f}_{\be_2^{\times s'_1}, (\be_2-\be_1)^{\times (s_2-1)}},
    \end{align*}
\end{proposition}
Iterating Proposition \ref{P: smoothing}, we can get rid of all copies of $\be_2-\be_1$. 

The proof of Proposition \ref{P: smoothing} uses a lot of tricks similar to degree lowering, but the underlying philosophy is different: we do not lower the overall degree of the seminorm, only the number of the problematic vectors $\be_2-\be_1$. We present the proof of Proposition \ref{P: smoothing} in the relatively simple case $(s_1, s_2) = (0,1)$, noting that the general case requires nasty maneuvers like dual-difference interchange present in the proof of Proposition \ref{P: degree lowering s=3}. For more exposition of seminorm smoothing, consult \cite[Section 8]{FrKu22a}, \cite[Section 4]{FrKu22b}, or \cite[Section 4.2]{DKKST25}.
\begin{proof}[Proof for $(s_1, s_2) = (0,1)$]
    Let 
    \begin{align}\label{E: smoothing beginning}
        \delta := \lim_{N\to\infty}\norm{\E_{n\in[N]}T_1^{n^2}f_1\cdot T_2^{n^2+n}f_2}_{L^2(\mu)}.
    \end{align}
    By Lemma \ref{L: stashing in degree lowering}, we have
    \begin{align*}
        \norm{\E_{n\in[N]}T_1^{n^2}f_1\cdot T_2^{n^2+n}F_2}_{L^2(\mu)}\geq \delta^4,
    \end{align*}
    where 
    \begin{align*}
        F_2 :=\lim_{N\to\infty}\E_{n\in[N]}T_2^{-n^2-n}\overline{f_0}\cdot T_1^{n^2}T_2^{-n^2-n}\overline{f_1}
    \end{align*}
    for some 1-bounded $f_0\in L^\infty(\mu)$.
    Our assumption gives $\nnorm{F_2}_{\be_2-\be_1}\gg \delta^{O(1)}.$ Using the inverse theorem for the degree-1 box seminorm \eqref{E: deg 1 inverse theorem}, we have %we can find $u_2\in I(T_2T_1\inv)$ such that %$\int F_2\cdot \overline{u_2}\; d\mu \gg \delta^{O(1)}$. 
    \begin{align*}
        \int F_2\cdot \overline{u_2}\; d\mu \gg \delta^{O(1)}
    \end{align*}
    for $u_2 =\E(\overline{F_2}|\CI(T_2T_1\inv))$.
    Expanding the definition of $F_2$, taking complex conjugates, and composing the resulting integral with $T_2^{n^2 + n}$, we get 
    \begin{align*}
        \lim_{N\to\infty}\E_{n\in[N]}\int f_0 \cdot T_1^{n^2}f_1\cdot T_2^{n^2+n}u_2\; d\mu \gg \delta^{O(1)}.
    \end{align*}
    Much like in degree lowering, we have replaced the arbitrary function $f_2$ by a special function $u_2$. The $T_2T_1\inv$-invariance of $u_2$ means that $T_2^{n^2+n}u_2 = T_1^{n^2+n}u_2$, and hence
    \begin{align*}
        \lim_{N\to\infty}\E_{n\in[N]}\int f_0 \cdot T_1^{n^2}f_1\cdot T_1^{n^2+n}u_2\; d\mu \gg \delta^{O(1)}.
    \end{align*}
    This is the essence of seminorm smoothing: to reduce the original average of commuting transformations to one involving a single transformation. Using known Host-Kra seminorm estimates for the averages in the single-transformation case, we deduce that $\nnorm{f_1}_{s, T_1} \gg \delta^{O(1)}$ for some $s\in\N$. In other words, we have shown that $\nnorm{f_1}_{s, T_1}$ controls the original average with polynomial bounds.

    So far, we have passed from box seminorm control for $f_2$ to Host-Kra seminorm control for $f_1$. In the next step, we will transfer seminorm control back to $f_1$. Because of this going back-and-forth between seminorms of $f_1$ and $f_2$, we have called these two steps \emph{ping} and \emph{pong} in \cite{FrKu22a}.
    %In \cite{FrKu22a}, we have called this first step \emph{ping}, as it transferred seminorm control from $f_2$ to $f_1$. In the next step

    We proceed to the \emph{pong} step. Defining
        \begin{align*}
        F_1 :=\lim_{N\to\infty}\E_{n\in[N]}T_1^{-n^2}\overline{f_0}\cdot T_2^{n^2+n}T_1^{-n^2}\overline{f_2},
    \end{align*}
    we deduce once again from Lemma \ref{L: stashing in degree lowering} applied to the original average that 
       \begin{align*}
        \lim_{N\to\infty}\norm{\E_{n\in[N]}T_1^{n^2}F_1\cdot T_1^{n^2+n}f_2}_{L^2(\mu)} \geq\delta^4.
    \end{align*} 
    Applying the seminorm estimate from the \emph{ping} step with $F_1$ in place of $f_1$, we deduce that $\nnorm{F_1}_{s,T_1}\gg\delta^{O(1)}$. By the weak inverse theorem for the Host-Kra seminorms \eqref{E: dual identity}, we have
    \begin{align*}
        \int F_1 \cdot \CD\; d\mu \gg \delta^{O(1)},
    \end{align*}
    where $\CD = \CD_{s,T_1}(F_1)$ is the degree-$s$ dual function of $F_1$. Expanding the definition of $F_1$ and composing the resulting integral with $T_1^{n^2 + n}$, we get
    \begin{align*}
                \lim_{N\to\infty}\E_{n\in[N]}\int f_0 \cdot T_1^{n^2}\overline{\CD}\cdot T_2^{n^2+n}f_2\; d\mu \gg \delta^{O(1)}.
    \end{align*}
    The point now is that the sequence $(T_1^n \overline{\CD})_n$ is a so-called \emph{dual sequence}, and it behaves in many ways like a nilsequence. In particular, it was shown in \cite[Proposition 6.1]{Fr12} that finitely many applications of the Cauchy-Schwarz and van der Corput inequalities completely eliminate this term from the average at the cost of duplicating $T_2^{n^2+n}f_2$ each time the inequalities are applied. As an end result of this procedure, we get an expression like
    \begin{align*}
        \lim_{H\to\infty}\E_{h\in[H]^s}\lim_{N\to\infty}\norm{\E_{n\in[N]}\prod_{\eps\in\{0,1\}^s}T_2^{(n+\eps\cdot h)^2+(n+\eps\cdot h)}f_2}_{L^2(\mu)}\gg \delta^{O(1)}.
    \end{align*}
    For a full-density set of $h$'s, this is an average of essentially distinct polynomials of a single transformation, and so seminorm estimates in the single-transformation case combined with a quantitative concatenation argument like the one outlined in Section \ref{SSS: concatenation} give the claimed conclusion $\nnorm{f_2}_{s'_1,T_2}\gg\delta^{O(1)}$ for some $s'_1\in\N$.
\end{proof}
The seminorm smoothing argument was developed by Frantzikinakis and the author to prove Theorem \ref{T: FrKu pairwise affine independent}\eqref{i: poly HK seminorm control}, seminorm estimates for pairwise independent integer polynomials. At that time, the quantitative concatenation arguments from \cite{DKKST24, KKL24a, Kuc23} were not yet known, and so the proof concluded with the qualitative conclusion from Theorem \ref{T: FrKu pairwise affine independent}\eqref{i: poly HK seminorm control}. The fully quantitative Host-Kra seminorm estimates were provided by Donoso, Koutsogiannis, Sun, Tsinas, and the author \cite[Theorem 1.16]{DKKST24} for pairwise independent Hardy sequences. This result used essentially the same seminorm smoothing argument, but with a much more robust PET argument and quantitative rather than qualitative concatenation.
%We state below a special case of \cite[Theorem 1.16]{DKKST24} for pairwise independent fractional polynomials with arbitrary positive powers (hence it also covers real polynomials).
\begin{theorem}[Quantitative Host-Kra seminorm estimates for pairwise independent Hardy sequences]\label{T: HK estimates for pairwise independent polys}
For any pairwise independent %$0<b_1 < \cdots < b_d$, and for $1\leq j\leq \ell$, define $a_j(t) = \sum_{i=1}^d \alpha_{ji}t^{b_i}$ for some $\alpha_{ji}\in\R$. If the functions 
$a_1, \ldots, a_\ell\in\CH$, we can find $C>0$ and $s\in\N$ so that for any system $(X, \CX, \mu, T_1, \ldots,T_\ell)$ and any 1-bounded functions $f_1, \ldots, f_\ell\in L^\infty(\mu)$, we have
\begin{align*}
    \brac{\limsup_{N\to\infty}\norm{\E_{n\in[N]}T_1^{\sfloor{a_1(n)}}f_1\cdots T_\ell^{\sfloor{a_\ell(n)}}f_\ell}_{L^2(\mu)}}^C \leq C \min_{1\leq j\leq \ell}\nnorm{f_j}_{s,T_j}.
\end{align*}
% Let $0<b_1 < \cdots < b_d$, and for $1\leq j\leq \ell$, define $a_j(t) = \sum_{i=1}^d \alpha_{ji}t^{b_i}$ for some $\alpha_{ji}\in\R$. If the functions $a_1, \ldots, a_\ell$ are pairwise independent, then we can find $C>0$ and $s\in\N$ so that for any system $(X, \CX, \mu, T_1, \ldots,T_\ell)$ and any 1-bounded functions, we have
% \begin{align*}
%     \brac{\limsup_{N\to\infty}\norm{\E_{n\in[N]}T_1^{\sfloor{a_1(n)}}f_1\cdots T_\ell^{\sfloor{a_\ell(n)}}f_\ell}_{L^2(\mu)}}^C \leq C \min_{1\leq j\leq \ell}\nnorm{f_j}_{s,T_j}.
% \end{align*}
\end{theorem}

\subsubsection{Pairwise dependent case and relative concatenation}\label{SSS: pairwise dependence}
For ergodic averages along pairwise independent sequences, %whatever that means for a particular class of sequences, 
one can typically hope for Host-Kra seminorm control in accordance with Heuristic \ref{H: pairwise independent seminorm}. For pairwise dependent sequences, not only is this normally not possible, but in fact there seem to be no ``optimal'' seminorm estimates. For instance, the works of Donoso, Ferr\'e-Moragues, Koutsogiannis, and Sun \cite{DFKS22} as well as Donoso, Koutsogiannis, Sun, Tsinas and the author \cite{DKKST25} establish two different box seminorm estimates for averages along pairwise dependent polynomials.
\begin{example}[Different seminorm estimates for $(T_1^{n^2}, T_2^{n^2}, T_3^{n^2+n})_n$]\label{Ex: n^2, n^2, n^2+n}
    As a special case of Theorem \ref{T: DKKST box seminorm control}, it follows that
    %\cite[Example 8]{DFKS22}, a special case of \cite[Theorem 2.5]{DFKS22},
    \begin{align}\label{E: n^2, n^2, n^2+n}
        \lim_{N\to\infty}\norm{\E_{n\in[N]}T_1^{n^2}f_1\cdot T_2^{n^2}f_2\cdot T_3^{n^2+n}f_3}_{L^2(\mu)} = 0
    \end{align}
    whenever $\nnorm{f_3}_{\be_3^{\times s}, (\be_3-\be_2)^{\times s}, (\be_3-\be_1)^{\times s}} = 0$ for some $s\in\N$. By contrast, \cite[Theorem 4.5]{DKKST25} gives \eqref{E: n^2, n^2, n^2+n} whenever $\nnorm{f_3}_{\be_3^{\times s'}, (\be_2-\be_1)^{\times s'}} = 0$ for some $s'\in\N$. 
    The two estimates are not comparable as each box seminorm involves some vectors absent from the second one. 
    %They can also be used for different purposes. In \cite{FrKu22b}, Frantzikinakis and the author used the first estimate to show that 
    
    Any of these two seminorm estimates (and analogous estimates for $f_1,f_2$) can be used to show that if $\CI(T_1\inv T_2) = \CI(T_1)\cap \CI(T_2)$, i.e. the only sets invariant under $T_1\inv T_2$ are those simultaneously invariant under $T_1$ and $T_2$, then the average in  \eqref{E: n^2, n^2, n^2+n} admits Host-Kra seminorm control. Different proofs of this assertion can be found in \cite{FrKu22b} and \cite{DKKST25}. This condition is essentially sharp, in the sense that if there exists a $T_1\inv T_2$-invariant $f$ that is not $T_1,T_2$-invariant, then on taking $f_1 = f$, $f_2 = \overline{f_2}$, $f_3 =1$, we get 
    \begin{align*}
        \lim_{N\to\infty}\norm{\E_{n\in[N]}T_1^{n^2}f_1\cdot T_2^{n^2}f_2\cdot T_3^{n^2+n}f_3}_{L^2(\mu)} = \lim_{N\to\infty}\norm{\E_{n\in[N]}T_1^{n^2}|f_1|^2}_{L^2(\mu)},
    \end{align*}
    which is controlled by $\nnorm{|f_1|^2}_{2,T_1}$, but in general not by $\nnorm{f_1}_{s,T_1}$ for any $s\in\N$.
    %holds whenever $\nnorm{f_3}_{s'',T_3} = 0$ for some $s''\in\N$.
\end{example}

In \cite{FrKu22b}, Frantzikinakis and the author developed a more sophisticated variant of the seminorm smoothing argument with the goal of obtaining Host-Kra seminorm control for averages involving pairwise dependent polynomials under some necessary invariance assumptions. The Host-Kra seminorm control result mentioned in Example \ref{Ex: n^2, n^2, n^2+n} is a special case. Unfortunately, this argument relies on a delicate induction that breaks for sequences that are pairwise dependent with irrational coefficients, like $(T_1^{\sfloor{\sqrt{2}n^2}}, T_2^{\sfloor{\sqrt{3}n^2}}, T_3^{n^2+n})_n$. In \cite[Section 4]{DKKST25}, Donoso, Koutsogiannis, Sun, Tsinas, and the author found a more robust way of running the seminorm smoothing argument in the pairwise dependent context, one that extends to Hardy sequences. As already mentioned above, this method gives different box seminorm estimates for averages along Hardy sequences than those provided by Theorem \ref{T: DKKST box seminorm control}. Those new estimates played a key role in the proof of Theorem \ref{T: DKKST sufficient condition}, the difficult direction of the joint ergodicity classification problem for Hardy sequences.

One important technique that the smoothing argument from \cite{DKKST25} added to the toolbox is \textit{relative concatenation}. In their foundational work on concatenation, Tao and Ziegler \cite[Theorem 1.16]{TZ16} showed that for any subgroups $G_1, \ldots, G_s, G_1',\ldots, G_s'\subseteq \Z^\ell$,
\begin{align}\label{E: TZ concatenation}
    \CZ(G_1, \ldots, G_s) \cap \CZ(G_1', \ldots, G_{s'}')\subseteq \CZ(G_i + G'_j\colon\; 1\leq i\leq s,\; 1\leq j\leq s'),
\end{align}
i.e. an intersection of two box factors lies in a box factor of higher degree but larger groups. For instance, if $s=s'=2$, then
\begin{align*}
    \CZ(G_1, G_2)\cap\CZ(G_1', G_2')\subseteq \CZ(G_1 + G_1', G_1 + G_2', G_2+G_1', G_2+G_2').
\end{align*}
Such a result is particularly useful if the groups $G_i + G_j'$ are much larger than $G_i,G_j'$ individually, e.g. if $G_i + G_j' = \Z^\ell$ for all $i,j'$. For instance, if $\ell = 2$, $G_i = \langle \be_1\rangle$ and $G_j' = \langle \be_2\rangle$ for all $i,j'$, then \eqref{E: TZ concatenation} gives that
\begin{align*}
    \CZ_{s-1}(\be_1) \cap \CZ_{s'-1}(\be_2) \subseteq \CZ_{s+s'-1}(\Z^2).
\end{align*}

An important corollary of \eqref{E: TZ concatenation} is \cite[Corollary 1.21]{TZ16}, which asserts that any $f\in L^\infty(\mu)$ can be decomposed as $f_1 +f_2+f_3$, where 
\begin{gather*}
    \E(f_1|\CZ(G_1, \ldots, G_s)) = 0,\quad 
    %\textrm{and}\quad 
    \E(f_2|\CZ(G_1', \ldots, G_{s'}'))=0\\
    \E(f_3|\CZ(G_i + G_j'\colon\; 1\leq i\leq s,\; 1\leq j\leq s')) = 0.
\end{gather*}

In \cite[Theorem 3.1]{DKKST25}, Donoso, Koutsogiannis, Sun, Tsinas, and the author obtain a relative version of this result: for any subgroups $H_1, \ldots, H_r\subseteq \Z^\ell$, we get the following generalization of \eqref{E: TZ concatenation}:
\begin{multline}\label{E: DKKST concatenation}
    \CZ(H_1, \ldots, H_r, G_1, \ldots, G_s) \cap \CZ(H_1, \ldots, H_r, G_1', \ldots, G_{s'}')\\
    \subseteq \CZ(H_1, \ldots, H_r, G_i + G_j'\colon\; 1\leq i\leq s,\; 1\leq j\leq s').
\end{multline}
For instance,
\begin{align*}
        \CZ(H, G_1, G_2)\cap\CZ(H, G_1', G_2')\subseteq \CZ(H, G_1 + G_1', G_1 + G_2', G_2+G_1', G_2+G_2').
\end{align*}

While \eqref{E: DKKST concatenation} may seem like a small improvement over \eqref{E: TZ concatenation}, the corresponding decomposition result for functions found a vital application in the seminorm smoothing argument for pairwise dependence sequences from \cite{DKKST25}. We demonstrate its utility with the following example; for a more robust example, we encourage the reader to consult \cite[Section 4.2]{DKKST25}.
\begin{example}
    We want to show how the second estimate from Example \ref{Ex: n^2, n^2, n^2+n} can be deduced from the first. Suppose that we know that for some values $s_1, s_2, s_3\in\N$, we have \eqref{E: n^2, n^2, n^2+n} whenever 
    \begin{align}\label{E: relative concatenation 1}
        \nnorm{f_3}_{\be_3^{\times s_1}, (\be_2-\be_1)^{\times s_2}, (\be_3-\be_2)^{\times s_3}} = 0.
    \end{align}
    (Such an estimate forms an intermediate step of passing between two estimates from Example \ref{Ex: n^2, n^2, n^2+n}. We arrive at it after removing all the vectors $\be_3-\be_1$ at the cost of getting some vectors $\be_2-\be_1$ and increasing the number of vectors $\be_3$.)
    To obtain the desired estimate, we want to get rid of the vectors $\be_3-\be_2$ altogether. In the process of running the \textit{ping} and \textit{pong} step of the seminorm smoothing argument, we establish that \eqref{E: n^2, n^2, n^2+n} also holds whenever
    \begin{align}\label{E: relative concatenation 2}
        \nnorm{f_3}_{\be_3^{\times s_1'}, (\be_2-\be_1)^{\times s_2'}, (\be_3-\be_2)^{\times (s_3-1)}, \be_2^{\times s_4'}} = 0
    \end{align}
    for some $s_1',s_2',s_3'\in\N$. Without loss of generality, we can assume that $s_1 = s_1'$ and $s_2=s_2'$; if not, we simply replace them with $\max(s_1,s_1'), \max(s_2,s_2')$ respectively and use the monotonicity property of box seminorms. 

    We now want to use relative concatenation to get that \eqref{E: n^2, n^2, n^2+n} holds whenever
    \begin{align*}%\label{E: relative concatenation 3}
        \nnorm{f_3}_{\be_3^{\times (s_1+s_4')}, (\be_2-\be_1)^{\times s_2}, (\be_3-\be_2)^{\times (s_3-1)}} = 0.
    \end{align*}
    Since $\be_3 \in \langle \be_3-\be_2, \be_2\rangle$, the claim will follow from \eqref{E: scaling 1} if we can show that \eqref{E: n^2, n^2, n^2+n} holds whenever 
    \begin{align}\label{E: relative concatenation 3}
        \nnorm{f_3}_{\be_3^{\times s_1}, (\be_2-\be_1)^{\times s_2}, (\be_3-\be_2)^{\times (s_3-1)}, \langle \be_3 -\be_2, \be_2 \rangle^{\times s_4'}} = 0.
    \end{align}
    Indeed, if \eqref{E: relative concatenation 3} holds, then using 
    \begin{multline*}
        \CZ(\be_3^{\times s_1}, (\be_2-\be_1)^{\times s_2}, (\be_3-\be_2)^{\times s_3})\cap\CZ(\be_3^{\times s_1}, (\be_2-\be_1)^{\times s_2}, (\be_3-\be_2)^{\times (s_3-1)}, \be_2^{\times s_4'})\\
        \subseteq \CZ(\be_3^{\times s_1}, (\be_2-\be_1)^{\times s_2}, (\be_3-\be_2)^{\times (s_3-1)}, \langle \be_3 -\be_2, \be_2 \rangle^{\times s_4'}),
    \end{multline*}
    a special form of \eqref{E: DKKST concatenation}, we can split $f_3 = f_3'+f_3''$, where $f_3'$ satisfies \eqref{E: relative concatenation 1} while $f_3''$ satisfies \eqref{E: relative concatenation 2}. Using our previous two estimates, we thus deduce that both $f_3'$ and $f_3''$ contribute 0 to the average, and hence does $f_3$. 
\end{example}

Despite sharing many features, the proofs and statements of concatenation of factors differ markedly from those for quantitative concatenation. It would be interesting to find a uniform perspective on both types of results, leading to the rather open question below.
\begin{problem}[Uniform perspective on concatenation]
    Can we find a single concatenation result that implies both the results on concatenation of factors, proved in \cite{DKKST25, TZ16} and presented in Section \ref{SSS: pairwise dependence}, and the results on quantitative concatenation, proved in \cite{DKKST24, KKL24a, Kuc23,  Pel20, PP19} and presented in Section \ref{SSS: concatenation}?
\end{problem}

\section{Beyond joint ergodicity}\label{S: beyond joint ergodicity}
Joint ergodicity is the dream scenario, one where the limit is as simple and explicit as it could be. But what if it is not there? This section aims to briefly summarize the main results and open problems on the limits of multiple ergodic averages outside the realm of joint ergodicity.

% \subsection{Identities for the same iterates}
% We start with the case that is as far from joint ergodicity as possible: when all the sequences are the same, or at least multiples of each other. This section thus deals with the following rather general problem.
% \begin{problem}\label{Pr: identity along corners}
%     For which sequences $a:\N\to\Z$ and systems $(X, \CX, \mu,$\! $T_1, \ldots, T_\ell)$ do we have the identity
%     \begin{align}\label{E: identity along corners}
%         \lim_{N\to\infty}\norm{\E_{n\in[N]}\prod_{j=1}^\ell T_j^{a(n)}f_j - \E_{n\in[N]}\prod_{j=1}^\ell T_j^nf_j}_{L^2(\mu)} = 0
%     \end{align}
%     for all $f_1, \ldots, f_\ell\in L^\infty(\mu)$?
% \end{problem}
% For this problem, the single-transformation case corresponds to taking $T_j = T^j$ for some system $(X, \CX, \mu, T)$, in which case we aim at the identity
% \begin{align}\label{E: identity along APs}
%         \lim_{N\to\infty}\norm{\E_{n\in[N]}\prod_{j=1}^\ell T^{j a(n)}f_j - \E_{n\in[N]}\prod_{j=1}^\ell T^{jn}f_j}_{L^2(\mu)} = 0.
%     \end{align}

% The single-transformation case of Problem \ref{Pr: identity along corners} was studied by Frantzikinakis, who proved the following two results.
% \begin{theorem}
%     Let $a\in\Z[t]$ be nonconstant and $(X, \CX, \mu, T)$ be totally ergodic. Then \eqref{E: }
% \end{theorem}

\subsection{Identities for arithmetic progressions and corners}
We start with a case that is as far from joint ergodicity as possible: when all the sequences are the same, or at least multiples of each other. In this case, we can hope to describe a limit neatly by comparison as explained in Section \ref{SSS: limit by comparison}. The properties that we shall investigate in this subsection are summarized in the following definition.
\begin{definition}[Good behavior along APs and corners]
    %Let $\ell\in\N$. 
    We say that $a:\N\to\Z$ is:
    \begin{enumerate}
        \item \emph{well-behaved along $\ell$-APs\footnote{Of course, ``$\ell$-APs'' stand for ``$\ell$-term arithmetic progressions''.} for the system $(X, \CX, \mu, T)$} if 
        \begin{align}\label{E: identity along APs}
        \lim_{N\to\infty}\norm{\E_{n\in[N]}\prod_{j=1}^\ell T^{j a(n)}f_j - \E_{n\in[N]}\prod_{j=1}^\ell T^{jn}f_j}_{L^2(\mu)} = 0
    \end{align}
    holds for all $f_1, \ldots, f_\ell\in L^\infty(\mu)$;
    \item \emph{well-behaved along APs for the system $(X, \CX, \mu, T)$} if it is well-behaved along $\ell$-APs for the system for any $\ell\in\N$;
    \item \emph{well-behaved along $\ell$-corners for the system $(X, \CX, \mu,$\! $T_1, \ldots, T_\ell)$} if 
        \begin{align}\label{E: identity along corners}
        \lim_{N\to\infty}\norm{\E_{n\in[N]}\prod_{j=1}^\ell T_j^{a(n)}f_j - \E_{n\in[N]}\prod_{j=1}^\ell T_j^nf_j}_{L^2(\mu)} = 0
    \end{align}
    for all $f_1, \ldots, f_\ell\in L^\infty(\mu)$;
    \item \emph{well-behaved along corners for the system $(X, \CX, \mu,$\! $T_1, \ldots, T_\ell)$} if it is well-behaved along $\ell$-corners for the system for any $\ell\in\N$;
    \item \emph{well-behaved along ($\ell$-)APs/corners} if it is well-behaved along ($\ell$-)APs/corners for any system.
    %\item \emph{well-behaved along APs}
    \end{enumerate}
\end{definition}
In \cite{FrKu22c}, a property very similar to good behavior along $\ell$-APs was called ``being good for mean convergence along $\ell$-APs''.  

This section  deals with the following rather general problems.
\begin{problem}\label{Pr: identity along APs}
    Let $\ell\in\N$. When is a sequence $a:\N\to\Z$ well-behaved along $\ell$-APs for a system $(X,\CX,\mu,T)$?
\end{problem}
\begin{problem}\label{Pr: identity along corners}
Let $\ell\in\N$. When is a sequence $a:\N\to\Z$ well-behaved along $\ell$-corners for a system $(X,\CX,\mu,T_1, \ldots, T_\ell)$?
    % For which sequences $a:\N\to\Z$ and systems $(X, \CX, \mu,$\! $T_1, \ldots, T_\ell)$ do we have the identity
    % \begin{align}\label{E: identity along corners}
    %     \lim_{N\to\infty}\norm{\E_{n\in[N]}\prod_{j=1}^\ell T_j^{a(n)}f_j - \E_{n\in[N]}\prod_{j=1}^\ell T_j^nf_j}_{L^2(\mu)} = 0
    % \end{align}
    % for all $f_1, \ldots, f_\ell\in L^\infty(\mu)$?
\end{problem}
We can consider many variants of Problems \ref{Pr: identity along APs} and \ref{Pr: identity along corners}: we can fix $\ell$ and the system, or we may want the good behavior property to hold for all $\ell$, all systems, or both.
%for fixed $\ell$, system, or both, or we can

In general, good behavior along $\ell$-APs does not imply good behavior along $(\ell+1)$-APs: to see this, take $a(n)$ to be the indicator function of $\{n\in\N\colon 1/4<\norm{\sqrt{2}n^{\ell+1}}<1/2\}$ (see \cite[Theorem A]{FLW06} or the remark below \cite[Theorem 3.4]{FrKu22c} for explanation).

\subsubsection{Good behavior along APs for specific sequences}
%In early work on Problem \ref{Pr: identity along APs}, Frantzikinakis proved the following two results.
For integer polynomials and primes, i.e. sequences not uniformly distributed in residue classes, Frantzikinakis, Frantzikinakis-Host-Kra, as well as Moreira-Richter proved the following results.
\begin{theorem}[{\cite[Theorem A]{Fr08}}]\label{T: Fr polys along APs}
    Any nonconstant $a\in\Z[t]$ is well-behaved along APs for any totally ergodic system.
    %$(X, \CX, \mu, T)$ be totally ergodic
\end{theorem}
For the necessity of total ergodicity, recall Example \ref{Ex: n^2}.
\begin{theorem}\label{T: primes along APs}
    The sequence of primes is well-behaved along APs for any totally ergodic system.
    %$(X, \CX, \mu, T)$ be totally ergodic
\end{theorem}
For 2-APs, Theorem \ref{T: primes along APs} was proved by Frantzikinakis, Host, and Kra \cite[Theorem 5]{FHK07}, who also outlined the proof of the general case conditional on the Gowers uniformity of the modified von Mangoldt function. A slightly different proof of the general case was provided by Moreira and Richter \cite[Theorem 2.9]{MR19}.

For Hardy sequences that are good for equidistribution, Frantzikinakis established the following.
\begin{theorem}[{\cite[Theorem 2.2]{Fr10}}]\label{T: Fr Hardy along APs}
    Any $a\in\CH$ satisfying
    \begin{align}\label{E: staying logarithmically away from cQ[t]}
        \lim_{t\to\infty}\abs{\frac{a(t) - c p(t)}{\log t}} = \infty\quad \textrm{for\; all}\quad c\in\R,\; p\in\Q[t]
    \end{align}
    is well-behaved along APs.
    %$(X, \CX, \mu, T)$ be totally ergodic
\end{theorem}
Koutsogiannis and Tsinas combined Theorem \ref{T: Fr Hardy along APs} with Theorem \ref{T: KoTs transference} to obtain an analogous result when the Hardy sequence is sampled along primes.
\begin{theorem}[{\cite[Theorem 1.3]{KoTs23}}]\label{T: KoTs Hardy along APs}
    Let $a\in\CH$ satisfy \eqref{E: staying logarithmically away from cQ[t]}. Then $a(p_n)$ is well-behaved along APs, where we recall that $\P = \{p_1<p_2<\cdots\}$.
\end{theorem}
Both Theorems \ref{T: Fr Hardy along APs} and \ref{T: KoTs Hardy along APs} also hold when $a\in\CH$ is replaced by $a\in\CT$; see \cite[Section 2.1.5]{Fr10} and \cite[Theorem 7.4]{KoTs23}.

Since Theorems \ref{T: Fr Hardy along APs} and \ref{T: KoTs Hardy along APs} make no assumptions on the system, they yield new extensions of the Szemer\'edi theorem (a similar comment holds for Theorem \ref{T: FrKu sparse corners} below).

Note that the assumption \eqref{E: staying logarithmically away from cQ[t]} is stronger than \eqref{E: staying away}, as in \eqref{E: staying logarithmically away from cQ[t]} we require $a$ to stay logarithmically away from all \emph{real} multiples of rational polynomials. This assumption is there to prevent \emph{irrational obstructions} to equidistribution, and its necessity is evidenced by the following example.
\begin{example}\label{Ex: sqrt 2 n}
    Take $a(n) = \sqrt{2}n$, and consider $Tx = x+\sqrt{2}$ on $X=\T$. Setting $f(x) = e(x)$, we have
    \begin{align*}
        \lim_{N\to\infty}\E_{n\in[N]}f(T^{\sfloor{\sqrt{2}n}}x) &= f(x)\cdot\lim_{N\to\infty}\E_{n\in[N]}e(\sqrt{2}\sfloor{\sqrt{2}n})\\
        &= f(x)\cdot\lim_{N\to\infty}\E_{n\in[N]}e(-\sqrt{2}\srem{\sqrt{2}n})\\
        &= f(x)\cdot\int_0^1 e(-\sqrt{2}x)\;dx \neq 0
    \end{align*}
    by Weyl's equidistribution theorem whereas the pointwise ergodic theorem gives
    \begin{align*}
        \lim_{N\to\infty}\E_{n\in[N]}f(T^n x) = 0.
    \end{align*}
\end{example}

As shown by Moreira and Richter, the sequence $\sqrt{2}n$ is well-behaved along APs as long as particular irrational obstructions do not arise in the system. 
%The systematic study of Problem \ref{Pr: identity along APs} was carried out by Moreira and Richter \cite{MR19}, who derived a number of further results. The first of these concerns sequences like the one in Example \ref{Ex: sqrt 2 n}.
\begin{theorem}[{\cite[Theorem 2.5]{MR19}}]\label{T: MR generalized linear}
    Let $a(n) = c n + d$ for $c\in\R_+,d\in\R$. Then the sequence $a$ is well-behaved along APs for any system $(X,\CX,\mu,T)$ satisfying the spectral condition $\spec(T)\cap\frac{1}{c}\Z = \{0\}$.
\end{theorem}
 It would be interesting to know whether Theorem \ref{T: MR generalized linear} extends to other real polynomials not covered by Theorem \ref{T: Fr Hardy along APs}, or to multiples of primes.
 \begin{problem}
     Let $a(n) = c p(n) + d$ for some $c\in\R_+, d\in\R$, and nonconstant $p\in\Q[t]$. Is the sequence $a$ well-behaved along APs for any system $(X,\CX,\mu,T)$ satisfying the spectral condition $\spec(T)\cap\frac{1}{c}\Q = \{0\}$?
 \end{problem}
  \begin{problem}
     Let $a(n) = c p_n + d$ for some $c\in\R_+, d\in\R$, where $(p_n)_n$ is the sequence of primes. Is the sequence $a$ well-behaved along APs for any system $(X,\CX,\mu,T)$ satisfying the spectral condition $\spec(T)\cap\frac{1}{c}\Q = \{0\}$?
 \end{problem}

Moreira-Richter's method proceeds by analyzing the spectrum of the multicorrelation sequences arising on a given system. Specifically, their main technical tool is the following spectral refinement of Bergelson-Host-Kra's decomposition \cite[Theorem 1.9]{BHK05}. We refer the reader to Appendix \ref{A: nilsystems} for definitions.
\begin{theorem}[Spectral refinement of BHK decomposition {\cite[Theorem 2.1]{MR19}}]\label{T: MR spectral refinement}
    Let $(X, \CX, \mu, T)$ be a system and $f_0, \ldots, f_\ell\in L^\infty(\mu)$. For every $\veps>0$, there exists a decomposition
    \begin{align*}
        \int f_0 \cdot T^nf_1\cdots T^{\ell n}f_\ell\; d\mu = \phi(n) + \omega(n) + \gamma(n),
    \end{align*}
    where $\omega$ is a null-sequence, $\norm{\gamma}_\infty\leq \veps$, and $\phi(n)=F(S^ny)$ for some nilmanifold $(Y, \CY, \nu, S)$ with $\spec(S)\subseteq \spec(T)$, $y\in Y$, and $F\in C(Y)$.
\end{theorem}
It is unclear to what extent Theorem \ref{T: MR spectral refinement} extends to sparse sequences.
The following question was hinted at in \cite{MR19}, at least for integer polynomials.
\begin{problem}
    Let $(X, \CX, \mu, T)$ be a system and $f_0, \ldots, f_\ell\in L^\infty(\mu)$. Let $a:\N\to\Z$ be any of the following: (i) the sequence of primes; (ii) a nonconstant integer polynomial; (iii) (the integer part) of a Hardy sequence. Can we find for every $\veps>0$ a decomposition
    \begin{align*}
        \int f_0 \cdot T^{a(n)}f_1\cdots T^{\ell a(n)}f_\ell\; d\mu = \phi(n) + \omega(n) + \gamma(n),
    \end{align*}
    where $\omega$ is a null-sequence, $\norm{\gamma}_\infty\leq \veps$, and $\phi(n)=F(S^ny)$ for some nilmanifold $(Y, \CY, \nu, S)$ with $\spec(S)\subseteq \spec(T)$, $y\in Y$, and $F\in C(Y)$?
\end{problem}

\subsubsection{General criteria for good behavior along APs}
The results mentioned so far address Problem \ref{Pr: identity along APs} for rather concrete examples of sequences. In \cite{FrKu22c}, Frantzikinakis and the author provide an equivalent characterization for good behavior along APs for general sequences. 
\begin{theorem}[Equivalent characterization of good behavior along APs {\cite[Theorem 3.11]{FrKu22c}}]\label{T: FrKu degree lowering on APs}
    Let $\ell\in\N$. A sequence $a:\N\to\Z$ is well-behaved along $\ell$-APs for the system $(X, \CX, \mu, T)$ if and only if the following two conditions hold:
    \begin{enumerate}
        \item the sequences $a, 2a, \ldots, \ell a$ admit Host-Kra seminorm control for $(X, \CX, \mu, T)$;
        \item $a$ is well-behaved along $\ell$-APs for $(X, \CZ_\ell(T), \mu, T)$.
    \end{enumerate}
\end{theorem}
Theorem \ref{T: FrKu degree lowering on APs} amounts to saying the following: if 
\begin{align*}%\label{E: average along APs}
        \lim_{N\to\infty}\norm{\E_{n\in[N]}T^{a(n)}f_1\cdots T^{\ell a(n)}f_\ell}_{L^2(\mu)}
        %\lim_{N\to\infty}\norm{\E_{n\in[N]}\prod_{j=1}^\ell T^{j a(n)}f_j}_{L^2(\mu)}
\end{align*}
%\eqref{E: average along APs} 
admits Host-Kra seminorm control, then the identity \eqref{E: identity along APs} holds for all functions if and only if it holds for $\CZ_\ell(T)$-measurable functions. While the assumption of Host-Kra seminorm control reduces the task of verifying \eqref{E: identity along APs} to asserting it for all $\CZ_s(T)$-measurable functions for some possibly large $s\in\N$, Theorem  \ref{T: FrKu degree lowering on APs} says that it suffices to consider $\CZ_\ell(T)$-measurable functions.
%Thus, it reduces the problem of verifying \eqref{E: identity along APs} to verifying it over $\CZ_\ell(T)$, whereas the assumption of Host-Kra seminorm control. 
%Note that if $a$ is well-behaved along $\ell$-APs for the system $(X, \CX, \mu, T)$, then the factor $\CZ_{\ell-1}(T)$ is characteristic for \eqref{E: average along APs}. 
%     \begin{align}\label{E: average along APs}
%         \lim_{N\to\infty}\norm{\E_{n\in[N]}\prod_{j=1}^\ell T^{j a(n)}f_j}_{L^2(\mu)} = 0
%     \end{align}
% for this system. So, 

Theorem \ref{T: FrKu degree lowering on APs} gives two global corollaries that reduce the good behavior along APs to equidistribution on nilsystems. 
\begin{theorem}[Criteria for convergence along APs - general systems {\cite[Theorem 3.1]{FrKu22c}}]\label{T:APsConvergence}
	Let $\ell\in\N$ and $a\colon \N\to \Z$ be a sequence. Assume that $a, 2a, \ldots, \ell a$ admit Host-Kra seminorm control for all systems $(X, \CX, \mu, T)$. Then the  following   properties are equivalent:
	\begin{enumerate}
		\item $a$ is well-behaved along $\ell$-APs  for all  systems.
		
		\item $a$ is well-behaved along $\ell$-APs
		for all  ergodic $\ell$-step nilsystems.
		
		\item  $a$ equidistributes on all $\ell$-step nilsystems $(Y, \CY, \nu, S)$.
	\end{enumerate}
\end{theorem}

\begin{theorem}[Criteria for  convergence along $\ell$-term APs - totally ergodic systems {\cite[Theorem 3.2]{FrKu22c}}]\label{T:APsConvergenceTE}
	Let $\ell\in\N$ and $a\colon \N\to \Z$ be a sequence. Assume that $a, 2a, \ldots, \ell a$ admit Host-Kra seminorm control for all systems $(X, \CX, \mu, T)$. Then the following   properties are equivalent:
	\begin{enumerate}
		\item $a$ is well-behaved along $\ell$-APs for all  totally ergodic systems.
		
		\item  $a$ is well-behaved along $\ell$-APs for all  totally ergodic $\ell$-step nilsystems.
		\suspend{enumerate}
		Furthermore, both properties are implied by the following one:
		\resume{enumerate}
		\item   $a$  equidistributes on all totally ergodic $\ell$-step nilsystems $(Y, \CY, \nu, S)$.
	\end{enumerate}
\end{theorem}

%It is unclear if Theorem \ref{T: FrKu degree lowering on APs} extends to the commuting case. There, the role of Host-Kra seminorms would be played by box seminorms. Hence, before we formulate the surmised commuting version of Theorem \ref{T: FrKu degree lowering on APs}, we introduce one more definition.

% The proposed commuting case of Theorem \ref{T: FrKu degree lowering on APs} then takes the following form.
% \begin{problem}
%     Let $a:\N\to\Z$ be a sequence, and $(X, \CX, \mu,$\! $T_1, \ldots, T_\ell)$. Is is true that $a$ is well-behaved along $\ell$-corners for the system if and only if it admits box seminorm control for the system, and the identity \eqref{E: identity along corners} holds for all
%     \begin{align}\label{E: complicated condition}
%         f_j\in \bigcup_{\substack{0\leq k \leq \ell\colon\\ k\neq j}}Z_{(T_jT_k\inv)^{\times 2}, \{T_j T_i\inv\colon\; 0\leq i\leq \ell,\; i\neq j,k\}}
%         % T_jT_1\inv, \ldots, T_j T_{j-1}\inv, T_j T_{j+1}\inv, \ldots, T_j T_\ell\inv}
%     \end{align}
%     and all $1\leq j\leq \ell$.
% \end{problem}
% The condition \eqref{E: complicated condition} is rather complicated to understand in general; for $\ell=2$, it simply means that we want to check \eqref{E: identity along corners} for
% \begin{align*}
%     f_1 \in Z_{T_1, T_1, T_1T_2\in}\cup Z_{T_1, T_1T_2\inv, T_1T_2\inv}\quad \textrm{and}\quad f_2 \in Z_{T_2, T_2, T_2T_1\in}\cup Z_{T_2, T_2T_1\inv, T_2T_1\inv}.
% \end{align*}

\subsubsection{Good behavior along corners}
While there are many results towards Problem \ref{Pr: identity along APs}, very little is known for Problem \ref{Pr: identity along corners}. The main reason is that it is notoriously hard to find useful inverse theorems for the box seminorms that arise in the study of the averages
\begin{align*}%\label{E: average along corners}
        \limsup_{N\to\infty}\norm{\E_{n\in[N]}T_1^{a(n)}f_1\cdots T_\ell^{a(n)}f_\ell}_{L^2(\mu)}.
        %\lim_{N\to\infty}\norm{\E_{n\in[N]}\prod_{j=1}^\ell T^{j a(n)}f_j}_{L^2(\mu)}
\end{align*}
In particular, the following problem stands unresolved.
\begin{problem}\label{Pr: Hardy well-behaved along corners}
    Show that any Hardy sequence $a\in\CH$ satisfying \eqref{E: staying logarithmically away from cQ[t]} is well-behaved along $\ell$-corners for any $\ell\in\N.$
\end{problem}
It is not even known if 
\begin{align*}
    \lim_{N\to\infty}\norm{\E_{n\in[N]}T_1^{\sfloor{n^{3/2}}}f_1\cdot T_2^{\sfloor{n^{3/2}}}f_2-\E_{n\in[N]}T_1^n f_1\cdot T_2^n f_2}_{L^2(\mu)} = 0;
\end{align*}
nor even if the average on the left converges in norm.
%present in \eqref{E: identity along corners}.

Recently, Daskalakis made the following progress towards Problem \ref{Pr: Hardy well-behaved along corners}.
\begin{theorem}[{\cite[Theorem 1.1]{Das25a}}]\label{T: Daskalakis}
    For any $b\in(1,23/22)$, the sequence $n^b$ is well-behaved along 2-corners. 
\end{theorem}
In fact, the result holds if $n^b$ is replaced by any $b$-regularly varying function with $b$ in this range (see \cite[Definition 1.3]{Das25a}). Daskalakis' method is very finitary; it proceeds by expressing the original average as a weighted ergodic average
\begin{align*}
    \E_{n\in[N]}T_1^{\sfloor{n^b}}f_1\cdot T_2^{\sfloor{n^b}}f_2 = \E_{n\in[N^c]}w_{N,n} \cdot T_1^n f_1\cdot T_2^n f_2,
\end{align*}
and then establishing the Gowers uniformity of the difference $w_{N,n}-1$. The fact that the weight $w_{N,n}$ is far from being 1-bounded is one reason why the author can only consider $b$ just slightly better than 1 to make sure that all the quantitative estimates yield cancellations.

Soon after, Frantzikinakis and the author removed the restriction on $b$ from Theorem \ref{T: Daskalakis} at the cost of taking a weaker mode of convergence.
\begin{theorem}[Limiting formula for Hardy corners {\cite[Theorem 1.1]{FrKu25}}]\label{T: FrKu sparse corners}
	   For every Hardy sequence $a\in\CH$ satisfying \eqref{E: staying logarithmically away from cQ[t]}, every system $(X,\CX,\mu,T_1, T_2)$, and all functions $f_0,f_1,f_2 \in L^\infty(\mu)$, we have the identity
 	\begin{align}\label{E: sparse corner identity}
 	    \lim_{N\to\infty} \frac{1}{N}\sum_{n=1}^N \, \int f_0\cdot T_1^{\sfloor{a(n)}}f_1\cdot T_2^{\sfloor{a(n)}}f_2\, d\mu =
	\lim_{N\to\infty} \frac{1}{N}\sum_{n=1}^N\,\int f_0\cdot   T_1^n f_1\cdot  T_2^n f_2\, d\mu.
 	\end{align}
\end{theorem}
Interestingly, Theorem \ref{T: FrKu sparse corners} only establishes an identity of \emph{weak} limits, which is enough to yield recurrence along Hardy corners. The method very quickly breaks when applied to strong limits or longer corners (see \cite[Section 3.4]{FrKu25} for discussion). This is one of the rare examples in the study of multiple ergodic averages when a result for weak limits cannot easily be upgraded to strong limits.

The same paper establishes a double limiting identity for polynomials.
\begin{theorem}[Limiting formula for polynomial corners {\cite[Theorem 1.2]{FrKu25}}]\label{T: FrKu poly corners}
	For every nonconstant $a\in \Z[t]$ with zero constant terms, every system $(X,\CX,\mu,T_1,T_2)$, and all functions $f_0,f_1, f_2 \in L^\infty(\mu)$, we have the identity
 	\begin{multline}\label{E: poly corner identity}
 	 \lim_{k\to \infty}\lim_{N\to\infty}\frac{1}{N}\sum_{n=1}^N\,  \int f_0\cdot T_1^{a(k!n)}f_1\cdot T_2^{a(k!n)}f_2\; d\mu\\
     =\lim_{k\to \infty}\lim_{N\to\infty}\frac{1}{N}\sum_{n=1}^N\,  \int f_0\cdot T_1^{k!n}f_1\cdot T_2^{k!n}f_2\, d\mu.
 	\end{multline}
\end{theorem}
A slight modification of the left-hand side of \eqref{E: poly corner identity} yields a double limiting identity for all intersective polynomials; see \cite[Theorem 1.2]{FrKu25} for the details. Combined with a result of Chu \cite[Theorem 1.1]{Chu11}, Theorems \ref{T: FrKu sparse corners} and \ref{T: FrKu poly corners} establish a variant of the popular common difference results for sparse corners \cite[Theorems 1.4 and 1.5]{FrKu25}.
%namely, for every $A\subseteq\Z^2$ and $\veps>0$, the set

The limitations of the method leading up to Theorem \ref{T: FrKu sparse corners} motivate the following two subcases of Problem \ref{Pr: Hardy well-behaved along corners}.
\begin{problem}
    Extend \eqref{E: sparse corner identity} to an identity for $L^2(\mu)$ limits.
\end{problem}
\begin{problem}
    Extend \eqref{E: sparse corner identity} and \eqref{E: poly corner identity} to 3-corners. 
\end{problem}

The proof of Theorem \ref{T: FrKu sparse corners} proceeds by a variant of the degree lowering argument. It gives rather general sufficient conditions on a sequence $a:\N\to\Z$ so that the identity
	\begin{align}\label{E: weak limit identity for corners}
		\lim_{N\to\infty}\E_{n\in[N]}\int f_0\cdot T_1^{a(n)}f_1\cdot T_2^{a(n)}f_2 \, d\mu = \lim_{N\to\infty}\E_{n\in[N]} \int f_0\cdot T_1^n f_1\cdot T_2^n f_2\, d\mu
	\end{align}
    holds for all systems. Before we state this result, we need one more definition.
\begin{definition}[Box seminorm control]\label{D: box seminorm control}
% We say that sequences $a_1,\ldots,a_\ell\colon \N\to \Z$: {\em admit box seminorm control for the system  $(X, \CX, \mu,T_1,\ldots, T_\ell)$} if there exists $s\in\N$ such that for all  $f_1,\ldots, f_\ell\in L^\infty(\mu)$, we have
% \begin{align}\label{E: vanishing}
%     \lim_{N\to\infty}\norm{\E_{n\in[N]} T_1^{a_1(n)}f_1\cdots T_\ell^{a_\ell(n)}f_\ell}_{L^2(\mu)} = 0
% \end{align}
We say that $a\colon \N\to \Z$: {\em admits box seminorm control for the system  $(X, \CX, \mu,T_1,\ldots, T_\ell)$} if there exists $s\in\N$ such that for all  $f_1,\ldots, f_\ell\in L^\infty(\mu)$, we have
\begin{align*}%\label{E: vanishing}
    \lim_{N\to\infty}\norm{\E_{n\in[N]} T_1^{a(n)}f_1\cdots T_\ell^{a(n)}f_\ell}_{L^2(\mu)} = 0
\end{align*}
whenever $\nnorm{f_\ell}_{T_\ell^{\times s}, (T_\ell T_{\ell-1}\inv)^{\times s}, \ldots, (T_\ell T_1\inv)^{\times s}} = 0$.
%, and likewise on permuting the indices $1, \ldots, \ell$.
%replacing $\ell$ with any of $1, \ldots, \ell-1$.
%We say that they admit \emph{box seminorm control} if they admit it for every system. If $a :=a_1 = \cdots = a_\ell$, then we say that $a$ admits box seminorm control (for a system).
\end{definition}
\begin{theorem}[{\cite[Theorem 3.3]{FrKu25}}]\label{T: FrKu weak limit identity for corners}
	Suppose that $a\colon \N\to\Z$ admits box seminorm control on every system $(X, \CX, \mu,$\! $T_1, T_2)$ and equidistributes on nilsystems $(Y, \CY, \nu, S)$. Then the identity \eqref{E: weak limit identity for corners} holds for every system $(X, \CX, \mu,$\! $T_1, T_2)$ and all functions $f_0,f_1, f_2\in L^\infty(\mu)$.
\end{theorem}
A variant of Theorem \ref{T: FrKu weak limit identity for corners} {\cite[Theorem 3.4]{FrKu25}} gives Theorem \ref{T: FrKu poly corners}.

It is unclear whether the two conditions from Theorem \ref{T: FrKu weak limit identity for corners} will suffice for
	\begin{align}\label{E: weak limit identity for long corners}
		\lim_{N\to\infty}\E_{n\in[N]}\int f_0\cdot T_1^{a(n)}f_1\cdots T_\ell^{a(n)}f_\ell \, d\mu = \lim_{N\to\infty}\E_{n\in[N]} \int f_0\cdot T_1^n f_1\cdots T_\ell^n f_\ell\, d\mu
	\end{align}
for $\ell>2$. While the box seminorm control assumption will surely be needed, it is likely that equidistribution on nilsystems may not be enough to guarantee \eqref{E: weak limit identity for long corners}. This is an interesting open problem, as it ties with the important question of understanding inverse theorems for box seminorms.
\begin{problem}
    If $a\colon \N\to\Z$ admits box seminorm control on every system $(X, \CX, \mu$, $T_1, \ldots, T_\ell)$ and equidistributes on nilsystems $(Y, \CY, \nu, S)$, does it necessarily satisfy \eqref{E: weak limit identity for long corners} for all $f_0, \ldots, f_\ell\in L^\infty(\mu)$?
\end{problem}
There are grounds to believe that even for $\ell\geq 3$, we may need to deal with proper ``multidimensional'' obstructions that go beyond nilsystems. See \cite[Theorem 3.5]{FrKu25} for a flavor of those.

\subsection{Optimal characteristic factors in the single-transformation case}
As explained in Section \ref{S: introduction}, in the absence of joint ergodicity and nice identities, the best we can do is to identify a factor or a seminorm that controls an average of interest. As a consequence of the seminorm estimates presented throughout this survey, we know that large classes of multiple ergodic averages are controlled by some Host-Kra factor. Since these factors are nested (see \eqref{E: monotonicity of factors}), and the first two of them admit a much more explicit description than the rest, it is often useful to identify the \emph{optimal} characteristic factor for a given average (this can be used, e.g., to get new results on popular common differences). For polynomial averages of a single transformation, variants of this question have been studied by Bergelson, Leibman, and Lesigne \cite{BLL07}, Leibman \cite{L09}, Frantzikinakis \cite{Fr08}, and the author \cite{Kuc22a}. Before presenting open questions and conjectures, we introduce several definitions to organize the presentation.
\begin{definition}[Host-Kra and Weyl complexities of polynomial families]\label{D: complexity}
    Let $s\in\N_0$. We say that essentially distinct polynomials $a_1, \ldots, a_\ell\in\Z[t]$ have
    \begin{enumerate}
        \item \emph{Host-Kra complexity $s$} if $\CZ_s(T)$ the smallest Host-Kra factor that is characteristic for $a_1, \ldots, a_\ell$ along any totally ergodic system $(X, \CX, \mu, T)$;
%         \footnote{Meaning that for arbitrary $f_1, \ldots, f_\ell\in L^\infty(\mu)$,
%         \begin{align}\label{E: complexity average}
%             \lim_{N\to\infty}\norm{\E_{n\in[N]}T^{a_1(n)}f_1\cdots T^{a_\ell(n)}f_\ell}_{L^2(\mu)} = 0 \quad\textrm{whenever}\quad \E(f_j|\CZ_s(T)) = 0\quad \textrm{for\; some}\quad 1\leq j\leq \ell.
%             %\min_{1\leq j\leq \ell}\nnorm{f_j}_{s, T} = 0
%         \end{align}}
% %        for arbitrary $f_1, \ldots, f_\ell\in L^\infty(\mu)$.}
        \item \emph{Weyl complexity $s$} if $\CZ_s(T)$ the smallest Host-Kra factor that is characteristic for $a_1, \ldots, a_\ell$ along any totally ergodic Weyl system $(X, \CX, \mu, T)$.
        % \item \emph{algebraic complexity $s$} if $s$ is the smallest nonnegative integer such that the progression
        % \begin{align}\label{E: polynomial progression}
        %     m,\; m + a_1(n),\; \ldots,\; m+a_\ell(n)
        % \end{align}
        % is \emph{algebraically independent of degree $s+1$}, in that there are no polynomials $Q_j\in\Z[t]$ with $\max\limits_{0\leq j\leq \ell}\deg Q_j = s+1$ such that\footnote{The work of Bergelson, Leibman, and Lesigne \cite{BLL07} has an alternative presentation of this condition in terms of a matrix involving the polynomials $a_1, \ldots, a_\ell$ rather than the progression \eqref{E: polynomial progression}. We believe, however, that the formulation using algebraic relations \eqref{E: algebraic independence} is conceptually clearer. To motivate how such progressions arise out of the polynomial family, we expand the average as
        % \begin{align*}
        %     \lim_{N\to\infty}\norm{\E_{n\in[N]}T^{a_1(n)}f_1\cdots T^{a_\ell(n)}f_\ell}_{L^2(\mu)}^2 &= \lim_{N\to\infty}\E_{n\in[N]}\int f_0\cdot T^{a_1(n)}f_1\cdots T^{a_\ell(n)}f_\ell\; d\mu\\
        %     &= \lim_{N\to\infty}\E_{m,n\in[N]}\int T^m f_0\cdot T^{m+a_1(n)}f_1\cdots T^{m+a_\ell(n)}f_\ell\; d\mu
        % \end{align*}
        % for some $f_0\in L^\infty(\mu)$ using measure invariance. 
        % }
        % \begin{align}\label{E: algebraic independence}
        %     Q_0(m) + Q_1(m+a_1(n)) + \cdots + Q_\ell(m+a_\ell(n)) = 0.
        % \end{align}
    \end{enumerate}
\end{definition}

The nomenclature employed in Definition \ref{D: complexity} comes from \cite{Kuc22a}; previous works called the aforementioned notions differently, in some cases also shifting the indices (e.g. Weyl complexity 2 according to Definition \ref{D: complexity} corresponds to Weyl complexity 3 in \cite{Fr08}). 
We assume total ergodicity in the definition so that affinely independent polynomials have Weyl complexity 0, which would fail due to local obstructions if we considered only ergodic systems.

A \emph{Weyl system} is an ergodic system $(X, \CX, \mu, T)$, where $X$ is a compact abelian Lie group and $T$ is a \emph{unipotent affine transformation} on $X$, i.e. $Tx = \phi(x) + a$ for $a\in X$ and an automorphism $\phi$ of $X$ satisfying $(\phi - \rm{Id}_X)^s = 0$ for some $s\in\N$. The examples to have in mind are skew products
\begin{align*}
    T(x_1, \ldots, x_d) = (x_1 + \alpha, x_2 + x_1, \ldots, x_d + x_{d-1})
\end{align*}
on $X=\T^d$ for some $\alpha\notin\Q$. The motivation to study complexity over Weyl systems comes from 
%In the light of 
Leibman's equidistribution theorem, which asserts that a polynomial orbit is equidistributed on a connected nilmanifold if and only if it is equidistributed on its maximal-factor torus \cite{L05c, L05b}. Frantzikinakis and Kra then showed that for a connected, ergodic (i.e. totally ergodic) nilsystem, its maximal factor torus is a Weyl system \cite[Proposition 3.1]{FrKr05}. Thereafter, it has been conjectured on many occasions \cite[Section 0.10]{L09}, \cite[Section 0.B]{BLL07}, \cite[Section 5.1.1]{Fr16}, \cite[Conjecture 1.9]{Kuc22a} that these two notions of complexity should be the same.
\begin{problem}\label{Pr: Host and Weyl complexities are the same}
    Show that for any essentially distinct polynomials, Host-Kra and Weyl complexities are the same.
\end{problem}
For 3 polynomials, Problem \ref{Pr: Host and Weyl complexities are the same} has been resolved by Frantzikinakis \cite{Fr08}, but it remains open even for 4 polynomials.

Informed by works  of Gowers and Wolf on true complexity in additive combinatorics \cite{gowers_wolf_2010, gowers_wolf_2011a, gowers_wolf_2011b, gowers_wolf_2011c}, the author introduced the notion of algebraic complexity \cite{Kuc22a}.
\begin{definition}
    Essentially distinct polynomials $a_1, \ldots, a_\ell\in\Z[t]$ have \emph{algebraic complexity $s$} if $s$ is the smallest nonnegative integer such that the progression
        \begin{align}\label{E: polynomial progression}
            m,\; m + a_1(n),\; \ldots,\; m+a_\ell(n)
        \end{align}
        is \emph{algebraically independent of degree $s+1$}, in that there are no polynomials $Q_j\in\Z[t]$ with $\max\limits_{0\leq j\leq \ell}\deg Q_j = s+1$ such that\footnote{The work of Bergelson, Leibman, and Lesigne \cite{BLL07} has an alternative presentation of this condition in terms of a matrix involving the polynomials $a_1, \ldots, a_\ell$ rather than the progression \eqref{E: polynomial progression}. We believe, however, that the formulation using algebraic relations \eqref{E: algebraic independence} is conceptually clearer. To motivate how such progressions arise out of the polynomial family, we expand the average as
        \begin{align*}
            \lim_{N\to\infty}\norm{\E_{n\in[N]}T^{a_1(n)}f_1\cdots T^{a_\ell(n)}f_\ell}_{L^2(\mu)}^2 &= \lim_{N\to\infty}\E_{n\in[N]}\int f_0\cdot T^{a_1(n)}f_1\cdots T^{a_\ell(n)}f_\ell\; d\mu\\
            &= \lim_{N\to\infty}\E_{m,n\in[N]}\int T^m f_0\cdot T^{m+a_1(n)}f_1\cdots T^{m+a_\ell(n)}f_\ell\; d\mu
        \end{align*}
        for some $f_0\in L^\infty(\mu)$ using measure invariance. 
        }
        \begin{align}\label{E: algebraic independence}
            Q_0(m) + Q_1(m+a_1(n)) + \cdots + Q_\ell(m+a_\ell(n)) = 0.
        \end{align}
\end{definition}
\begin{example}
%~\ 
The algebraic complexity equals:
\begin{enumerate}
        \item 0 for affinely independent polynomials like $n,\; n^2,\; \ldots,\; n^\ell$;
        \item $\ell-1$ for $n,\; 2n,\; \ldots, \; \ell n$; for instance, for $\ell=3$, we have the relationship
        \begin{align*}
            m^2 - 3(m+n)^2 + 3(m+2n)^2 - (m+3n)^2 = 0;
        \end{align*}
        \item 1 for $n,\; n^2,\; n+n^2$ because they satisfy the linear relation
        \begin{align*}
            m - (m+n) - (m +n^2) + (m+n+n^2) = 0;
        \end{align*}
        but no quadratic relation.
        \item 2 for the family $n,\; 2n,\; n^2$ because of the quadratic relation
        \begin{align*}
            (m^2 + 2m) - 2(m+n)^2 + (m+2n)^2 - 2(m+n^2) = 0.
        \end{align*}
        This relation is ``inhomogeneous'' in that it is quadratic in the first three terms and linear in the last one. Studying complexity for such families is difficult, as the corresponding polynomial orbits on nilmanifolds cannot be as neatly described as for the previous three examples. 
    \end{enumerate}
    % \begin{enumerate}
    %     \item Affinely independent polynomials like $n,\; n^2,\; \ldots,\; n^\ell$ have algebraic complexity 0.
    %     \item The family $n,\; 2n,\; \ldots, \; \ell n$ has algebraic complexity $\ell-1$; for instance, for $\ell=3$, we have the relationship
    %     \begin{align*}
    %         m^3 - 3(m+n)^2 + 3(m+2n)^2 - (m+3n)^2 = 0.
    %     \end{align*}
    %     \item The family $n,\; n^2,\; n+n^2$ has algebraic complexity 1 because it satisfies
    %     \begin{align*}
    %         m - (m+n) - (m +n^2) + (m+n+n^2) = 0,
    %     \end{align*}
    %     but no quadratic relation.
    %     \item The family $n,\; 2n,\; n^2$ has algebraic complexity 2 because of the quadratic relation
    %     \begin{align*}
    %         (m^2 + 2m) - 2(m+n)^2 + (m+2n)^2 - 2(m+n^2) = 0.
    %     \end{align*}
    %     This relation is however ``inhomogeneous'' in that it is linear in the first three terms and quadratic in the last one. Studying complexity for such families is difficult, as the corresponding polynomial orbits on nilmanifolds cannot be as neatly described as for the previous three examples. 
    % \end{enumerate}
\end{example}
In \cite[Theorem 1.18]{Kuc22a}, the author showed that algebraic and Weyl complexities are the same.
This is important insofar as it 
%shows that Weyl complexity is a purely algebraic construct, and so it 
reduces Problem \ref{Pr: Host and Weyl complexities are the same} to 
%showing that all obstructions to Host-Kra complexity come from
a purely algebraic question of studying algebraic relations of the form \eqref{E: algebraic independence}. The author then resolved Problem \ref{Pr: Host and Weyl complexities are the same} for a broad class of polynomial families called \emph{homogeneous} \cite[Theorem 1.11]{Kuc22a}, which includes among others all progressions of algebraic complexity 0 and 1. However, the general case
%this is only the easier half 
of Problem  \ref{Pr: Host and Weyl complexities are the same} remains open.

A second major problem, mentioned e.g. in \cite[Section 0.9]{L09}, \cite[Section 5]{BLL07}, \cite[Problem 9]{Fr16}, \cite[Conjecture 1.9]{Kuc22a} is to upper bound Weyl/algebraic complexity uniformly in terms of the length of the progression.
\begin{problem}[Upper bound on complexity]\label{Pr: upper bound on complexity}
        Show that for any $\ell$ essentially distinct polynomials, their algebraic complexity is at most $\ell-1$.
\end{problem}
The upper bound is sharp, as evidenced e.g. by $n, 2n, \ldots, \ell n$.
Despite the semblance of a simple algebraic exercise, Problem \ref{Pr: upper bound on complexity} is notoriously difficult. Although it has been resolved for $\ell\leq 4$ in \cite[Proposition 3.5]{Fr08} and \cite[Theorem 1.1]{McClendon}, the proofs consist of tedious case-by-case computations and offer no insight into the general case.
%and has resisted any attempts so far.

There remain many other questions where characteristic factors are not known. For instance, the following problem has been proposed by Moreira (private communication).
\begin{problem}\label{Pr: Moreira}
    Let $\alpha_1, \ldots, \alpha_\ell\in\R$ be distinct and nonzero. What is the optimal Host-Kra factor controlling the average along $\sfloor{\alpha_1 n}, \ldots, \sfloor{\alpha_\ell n}$ for all ergodic systems $(X, \CX, \mu, T)$? In particular, if $\alpha_1, \alpha_2, 1$ are $\Q$-independent, is the average along $\sfloor{\alpha_1 n},$ $\sfloor{\alpha_2 n}$  controlled by the Kronecker factor $\CZ_1(T)$?
\end{problem}
For integer polynomials, Problem \ref{Pr: Host and Weyl complexities are the same} inquires whether all the obstructions to uniformity come from rational algebraic relations \eqref{E: algebraic independence}. Problem \ref{Pr: Moreira} asks whether irrational obstructions also play a role. 

%One interest in Problem \ref{Pr: Moreira} is because it could yield new results on popular common differences; see Problem \ref{?} below.

We conclude this section with a result of Frantzikinakis and the author that gives general conditions for when a particular Host-Kra seminorm/factor control a given average.
\begin{theorem}[General degree reduction property {\cite[Theorem 3.11]{FrKu22c}}]\label{T: degree reduction}
Let $s\in\N$, $a_1,\ldots, a_\ell\colon \N\to\Z$ be sequences, and $(X, \CX, \mu, T)$ be a system. Then the Host-Kra seminorm $\nnorm{\cdot}_{s,T}$ (equivalently, the factor $\CZ_{s-1}(T)$) controls $a_1, \ldots, a_\ell$ on $(X, \CX, \mu, T)$ if and only if the following two conditions hold:
\begin{enumerate}
    \item $a_1, \ldots, a_\ell$ admit Host-Kra seminorm control for the system;
    \item $\nnorm{\cdot}_{s,T}$ controls $a_1, \ldots, a_\ell$ on $(X, \CZ_s(T), \mu, T)$.
\end{enumerate}
\end{theorem}
Thus, Theorem \ref{T: degree reduction} shows that under the assumption of Host-Kra seminorm control, to check for $\CZ_{s-1}(T)$-control, it suffices to verify it ``one level up'', i.e. for $\CZ_{s}(T)$-measurable functions. In particular, when $s=1$, Theorem \ref{T: degree reduction} reduces to the single-transformation case of Theorem \ref{T: joint ergodicity criteria}.

\subsection{Obstructions in the commuting case}
In contrast to the single-transformation case, extremely little is known regarding optimal seminorms and characteristic factors in the commuting case. Even the following simple problem remains open.
% \begin{problem}\label{Pr: optimal factor for double n^2}
%     Let $(X, \CX, \mu,$\! $T_1, T_2)$ be a system. Show that 
%     \begin{align*}
%         \lim_{N\to\infty}\norm{\E_{n\in[N]}T_1^{n^2}f_1\cdot T_2^{n^2}f_2}_{L^2(\mu)} = 0
%     \end{align*}
%     whenever $\nnorm{f_1}_{T_1,T_1T_2\inv} = 0$ or $\nnorm{f_2}_{T_2, T_2T_1\inv} = 0$.
% \end{problem}
% If $n^2$ is replaced by $n$, then this is known due to Host \cite{H09}; see Example \ref{Ex: Host example}. For $n^2$, the best we have is the weaker conclusion that the double average
% \begin{align*}
%         \lim_{k\to\infty}\lim_{N\to\infty}\norm{\E_{n\in[N]}T_1^{(k!n)^2}f_1\cdot T_2^{(k!n)^2}f_2}_{L^2(\mu)}
% \end{align*}
% vanishes under the claimed assumptions, which follows from Theorem \ref{T: FrKu poly corners}.
%
% In Problem \ref{Pr: optimal factor for double n^2}, we can at least make a reasonable conjecture for what the optimal seminorm/factor should be. The same cannot be said for even slightly more complicated averages.
\begin{problem}\label{Pr: optimal factor n, n^2, n^2+n commuting}
    Let $(X, \CX, \mu, T_1, T_2, T_3)$ be a system. What are the optimal seminorms/ factors for the ergodic average along $(T_1^n, T_2^{n^2}, T_3^{n^2+n})_n$?
\end{problem}
The family $n,\; n^2,\; n^2+n$ is the simplest polynomial family whose terms are pairwise independent but not affinely independent. If $T:=T_1 = T_2 = T_3$, we know thanks to Frantzikinakis \cite[Theorem B]{Fr08} that  $\CZ_1(T)$ is characteristic. By the seminorm control from Theorem \ref{T: FrKu pairwise affine independent}, we know that the Host-Kra factors $\CZ_s(T_j)$ will be characteristic for some $s\in\N$; but it is not clear whether $s=1$ would do. The study of Problem \ref{Pr: optimal factor n, n^2, n^2+n commuting} would likely involve equidistribution problems on nonergodic nilsystems, which can be quite challenging (see \cite[Section 7]{CFH11} for a foretaste).

Another important open problem concerns the structure of multicorrelation sequences of commuting transformations.
\begin{problem}[Decomposition of polynomial multicorrelation sequences]\label{Pr: multicorrelation decomposition}
    Let $a_1, \ldots, a_\ell\in\Z[t]$. Show that for any system $(X, \CX, \mu,$\! $T_1, \ldots, T_\ell)$ and any $f_1, \ldots, f_\ell\in L^\infty(\mu)$, we can decompose 
    \begin{align}\label{E: multicorrelation}
        \int f_0\cdot T_1^{a_1(n)}f_1\cdots T_\ell^{a_\ell(n)}f_\ell\; d\mu = \phi(n) + \omega(n),
    \end{align}
    where $\phi$ is a nilsequence and $\omega$ is a null-sequence.
\end{problem}
When $T_1 = \cdots = T_\ell$, such a decomposition is known due to Bergelson-Host-Kra for $a_j(n) = j n$ \cite[Theorem 1.9]{BHK05} and Leibman for general polynomials \cite[Theorem 3.1]{L10}. In the commuting case, Frantzikinakis and the author resolved Problem \ref{Pr: multicorrelation decomposition} for pairwise independent polynomials \cite[Theorem 2.17]{FrKu22a}. For pairwise dependent polynomials, only several cases are known. Recently, Leng established \eqref{E: multicorrelation} for $a_1(n) = \cdots = a_\ell(n) = n$ \cite[Theorem 1.6]{Leng25}, and later Frantzikinakis and the author extended this to $\ell=2$ and arbitrary $a_1 = a_2\in\Z[t]$ \cite[Theorem 1.6]{FrKu25}. 

In a different direction, Frantzikinakis showed that for every $\veps>0$ we can get \eqref{E: multicorrelation} with a weaker condition $\limsup\limits_{N\to\infty}\E\limits_{n\in[N]}|\omega(n)|^2 \leq \veps$ regardless of the choice of polynomials \cite[Theorem 1.1]{Fr15}. But with $\veps = 0$, Problem \ref{Pr: multicorrelation decomposition} remains unresolved for general polynomials, even for tuples $(T_1^{n^2}, T_2^{n^2}, T_3^{n^2})_n$ and $(T_1^{n^2}, T_2^{n^2}, T_3^{n^2+n})_n$.

\subsection{Multiple recurrence}
Given that multiple ergodic averages originate in the study of multiple recurrence, it is only fitting to record here several problems on the latter topic that still remain open. The most important one concerns the extension of the multidimensional polynomial Szemer\'edi theorem to jointly intersective polynomials.
\begin{definition}[Joint intersectivity]
    Polynomials $a_1, \ldots, a_\ell\in\Z[t]$ are called \emph{jointly intersective} if for every $r\in\N$ there exists $n\in\Z$ such that $r$ divides $a_1(n), \ldots, a_\ell(n)$.
\end{definition}
For background regarding intersective and jointly intersective polynomials, we refer the reader to \cite{BLL08}. %The following problem is probably the most important open problem in the study of multiple recurrence.
\begin{problem}[Multidimensional polynomial Szemer\'edi conjecture for jointly intersective polynomials {\cite[Conjecture 6.3]{BLL08}}]\label{Pr: joint intersective}
    Show that for any jointly intersective polynomials $a_1, \ldots, a_\ell\in\Z[t]$, any system $(X, \CX, \mu,$\! $T_1, \ldots, T_\ell)$, and any $E\in\CX$ with $\mu(E)>0$, we have
    \begin{align*}
       \lim_{N-M\to\infty}\E_{n\in [M,N)}\mu(E\cap T_1^{-a_1(n)}E\cap \cdots \cap T_\ell^{-a_\ell(n)}E)>0.
    \end{align*}
    %for some $n\in\Z$.
\end{problem}
In fact, {\cite[Conjecture 6.3]{BLL08}} is stated for more general polynomials over $\Z^k$; we leave the extension to the interested reader.

When $T_1 = \cdots = T_\ell$, the problem has been resolved by Bergelson, Leibman, and Lesigne \cite[Theorem 1.1]{BLL08}. For commuting transformations, the case $\ell=2$ has been addressed by Frantzikinakis and the author \cite[Theorem 1.3]{FrKu25}. The general case remains open. Just like with many other open problems on commuting commuting transformations in this survey, a major stumbling block is the absence of usable structure theory for box seminorms and factors.
%Clearly, all polynomials with zero constant terms are jointly intersective (take any $n\in r\Z$).

The next question inquires about the multiple recurrence counterpart of Problem \ref{Pr: Hardy well-behaved along corners}.
\begin{problem}[Multiple recurrence for Hardy corners]\label{Pr: multiple recurrence for Hardy corners}
    Let $a\in\CH$ be a Hardy sequence satisfying \eqref{E: staying logarithmically away from cQ[t]} (such as $a(n) = n^{3/2}$). Show that for any system $(X, \CX, \mu,$\! $T_1, \ldots, T_\ell)$ and any $E\in\CX$ with $\mu(E)>0$, we have
    \begin{align*}
       \lim_{N\to\infty}\E_{n\in [N]}\mu(E\cap T_1^{-\sfloor{a(n)}}E\cap \cdots \cap T_\ell^{-\sfloor{a(n)}}E)>0.
    \end{align*}
\end{problem}
Obviously, the successful resolution of Problem \ref{Pr: Hardy well-behaved along corners} would automatically address Problem \ref{Pr: multiple recurrence for Hardy corners}. Alternatively, one could aim for a proof along the lines of Frantzikinakis and Wierdl \cite{FrW09} that does not require the full identity \eqref{E: identity along corners}.
The fact that Problem \ref{Pr: multiple recurrence for Hardy corners}  remains unresolved demonstrates how far we are from having a satisfactory picture of multiple recurrence for real polynomials, fractional polynomials and more general Hardy sequences.

%real polys

%fractional polys 

Lastly, we record one problem regarding popular common differences. In what follows, we use the following definition.
\begin{definition}[Popular common differences]
    We say that $a_1, \ldots, a_\ell\in\Z[t]$ \emph{admit popular common differences} if for every ergodic system $(X, \CX, \mu, T)$, set $E\in\CX$, and $\veps>0$, the set
    \begin{align*}
        \rem{n\in\Z\colon \mu(E\cap T^{-a_1(n)}E\cap \cdots \cap T^{-a_\ell(n)}E) > \mu(E)^{\ell+1}-\veps}
    \end{align*}
    is syndetic.
\end{definition}
Whereas multiple recurrence is a very general phenomenon, popular common differences constitute a peculiarity that arises only for special patterns. For instance, for nonzero $c_1<c_2<c_3$, the pattern $c_1n,\; c_2n,\; c_3n$ admits popular differences if and only if $c_3=c_1+c_2$. Furthermore, linear families of length 4 do not admit them at all. The next problem aims to fully classify those families that admit this behavior.
\begin{problem}[Polynomial families admitting popular common differences]\label{Pr: popular common differences}
%    Let $a_1, \ldots, a_\ell\in\Z[t]$ be essentially distinct polynomials with zero constant terms. Show that they do not admit popular common differences unless they fall into one of the following classes:
Prove or disprove: essentially distinct polynomials $a_1, \ldots, a_\ell\in\Z[t]$ with zero constant terms do not admit popular common differences unless they fall into one of the following classes:
    \begin{enumerate}
        \item they have algebraic complexity at most 1;
        \item they are of the form $c_1 a(n),\; c_2 a(n), c_3 a(n)$ for some nonconstant $a\in\Z[t]$ and nonzero $c_1<c_2<c_3$ satisfying $c_3=c_1+c_2$.
    \end{enumerate}
\end{problem}
The exceptional cases listed above are known to contain popular common differences thanks to Bergelson-Host-Kra \cite[Theorem 1.2]{BHK05}, Frantzikinakis-Kra \cite[Theorem 1.3]{FrKr06}, Frantzikinakis \cite[Theorem C]{Fr08}, and the author \cite[Theorem 1.16]{Kuc22a}. To the reader interested in this topic, we also highly recommend the additive combinatorial work of Sah, Sawhney, and Zhao \cite{SSZ21} that classifies all linear patterns admitting popular common differences and has extensive references to other combinatorial works on the subject.
%For the state of the art on the additive combinatorics counterpart of Problem \ref{Pr: popular common differences}, we direct the reader to \cite{SSZ21}.
%Popular common differences is a much
%Popular common differences for general polynomial progressions. 
%\subsection{Norm convergence}
%Like multiple recurrence, norm convergence is such a fundamental question in the study of multiple ergodic averages that we list several problems still awaiting solution. 

%generalized polys

%fractional polys

\subsection{Asymptotic discorrelation}
We conclude with one generalization of joint ergodicity that we find worth exploring. We start with a rather immediate restatement of joint ergodicity.
\begin{lemma}
    Sequences $a_1, \ldots, a_\ell:\N\to\Z$ are jointly ergodic for the system $(X, \CX, \mu$, $T_1, \ldots, T_\ell)$ if and only if the following two conditions are satisfied:
    \begin{enumerate}
        \item the actions $(T_j^{a_j(n)})_n$ are ergodic on $(X, \CX, \mu)$ for all $1\leq j\leq \ell$;
        \item the actions $(T_j^{a_j(n)})_n$ are \emph{asymptotically discorrelated} in the sense that
        \begin{align}\label{E: asymptotic discorrelation}
            \lim_{N\to\infty}\norm{\E_{n\in[N]}\prod_{1\leq j\leq \ell}T_j^{a_j(n)}f_j - \prod_{1\leq j\leq \ell}\E_{n\in[N]}T_j^{a_j(n)}f_j}_{L^2(\mu)}=0
        \end{align}
        for all $f_1, \ldots, f_\ell\in L^\infty(\mu)$. 
    \end{enumerate}
\end{lemma}
Thus, joint ergodicity = ergodicity + asymptotic discorrelation. Given that we by now understand well joint ergodicity, it makes sense to reach a similar level of understanding for asymptotic discorrelation. We propose one possible characterization of this phenomenon in the spirit of Theorem \ref{T: joint ergodicity criteria}.
\begin{problem}[Characterization of asymptotic discorrelation]\label{Pr: discorrelation}
    Suppose that the sequences $a_1, \ldots, a_\ell:\N\to\Z$ admit Host-Kra seminorm control for the system $(X, \CX, \mu,$\! $T_1, \ldots, T_\ell)$. Prove or disprove the following: they are asymptotically discorrelated for that system if and only if \eqref{E: asymptotic discorrelation} holds whenever $f_j$ is a nonergodic eigenfunction of $T_j$ for every $1\leq j\leq \ell$.
\end{problem}
Since nonergodic eigenfunctions of $T_j$ generate the entire $\CZ_1(T_j)$ factor, Problem \ref{Pr: discorrelation} inquires whether asymptotic discorrelation for all functions can be inferred from \eqref{E: asymptotic discorrelation} holding just for $\CZ_1(T_j)$-measurable functions. It may well be that the characterization proposed by Problem \ref{Pr: discorrelation} is too good to hope for, and to get asymptotic discorrelation, we need it to hold at least on $\CZ_2(T_j)$. However, we do not have any explicit counterexamples showing the need for a larger factor than $\CZ_1(T_j)$.

\section{Beyond commutativity}\label{S: beyond commutativity}
The entire discussion up to now focused on systems in which the transformations commute with each other. In this section, we dispose of this assumption, 
%entering the \emph{terra incognita} of noncommutative actions. This section 
presenting the most important results and open questions for \emph{nilpotent} systems and beyond. 
\begin{definition}[Nilpotent systems]
    Let $k\in\N$. We call $(X, \CX, \mu,$\! $T_1, \ldots, T_\ell)$ a \emph{$k$-step nilpotent system} if $T_1, \ldots, T_\ell$ generate a $k$-step nilpotent group of measure-preserving transformations on a standard probability space $(X, \CX, \mu)$. We call a system \emph{nilpotent} if it is $k$-step nilpotent for some $k\in\N$.
\end{definition}

\subsection{Nilpotent heuristic}

For nilpotent systems, we have two foundational results of Leibman and Walsh.
\begin{theorem}[Leibman's nilpotent multiple recurrence theorem {\cite[Theorem 0.4]{L98}}]\label{T: Leibman nilpotent}
    Let $a_1, \ldots, a_\ell\in\Z[t]$ have zero constant terms. For every nilpotent system $(X, \CX, \mu,$\! $T_1, \ldots, T_\ell)$ and every $E\in\CX$ with $\mu(E)>0$, there exists $n\in\N$ such that
    \begin{align*}
        %\liminf_{N\to\infty}\E_{n\in[N]}
        \mu(E\cap T_1^{-a_1(n)}E\cap \cdots \cap T_\ell^{-a_\ell(n)}E)>0.
    \end{align*}
\end{theorem}
\begin{theorem}[Walsh's norm convergence theorem {\cite[Theorem 1.1]{W12}}]\label{T: Walsh}
    Let $a_1, \ldots, a_\ell\in\Z[t]$. For every nilpotent system $(X, \CX, \mu,$\! $T_1, \ldots, T_\ell)$ and all $f_1, \ldots, f_\ell\in L^\infty(\mu)$, the average
    \begin{align*}
        \E_{n\in[N]}T_1^{a_1(n)}f_1\cdots T_\ell^{a_\ell(n)}f_\ell
    \end{align*}
    converges in $L^2(\mu)$.
\end{theorem}
Both Theorems \ref{T: Leibman nilpotent} and \ref{T: Walsh} hold for more general expressions, but we have opted to just state the simpler case for the clarity of presentation. Before Theorem \ref{T: Walsh}, the only truly nilpotent (i.e. involving non-commutative actions) average for which norm convergence was known was the double ergodic average
\begin{align}\label{E: BL nilpotent average}
    \E_{n\in[N]}T_1^n f_1\cdot T_2^n f_2,
\end{align}
%along $(T_1^n, T_2^n)_n$, 
a result due to Bergelson and Leibman \cite{BL02}. %Afterwards, Theorem \ref{T: Walsh} has been strengthened by Zorin-Kranich \cite[Theorem 1.1]{ZK16} to cover e.g. uniform averages.

Walsh's method for proving norm convergence requires rather minimal understanding of the structure of multiple ergodic averages under consideration compared to other results in this survey. Despite its incredible scope, its reach outside polynomials remains rather unexplored - hence the following problem. 
%The first problem in this section inquires whether it extends past polynomials. 
%In the spirit of Heuristic \ref{H: nilpotent heuristic}, the first problem in this section considers the extension of Theorem \ref{T: Walsh} to fractional polynomials.
\begin{problem}[Norm convergence along fractional polynomials]\label{Pr: norm convergence fractional polynomials}
        Let $a_1, \ldots, a_\ell$ be fractional polynomials. Does the average
    \begin{align*}
        \E_{n\in[N]}T_1^{a_1(n)}f_1\cdots T_\ell^{a_\ell(n)}f_\ell
    \end{align*}
    converge in $L^2(\mu)$ for every nilpotent system $(X, \CX, \mu,$\! $T_1, \ldots, T_\ell)$ and all $f_1, \ldots, f_\ell\in L^\infty(\mu)$?
\end{problem}
Problem \ref{Pr: norm convergence fractional polynomials} remains open even in the commuting case.

We conclude this section's introduction with a heuristic that guides our thinking on nilpotent actions.
\begin{heuristic}[Nilpotent heuristic]\label{H: nilpotent heuristic}
    All good properties (multiple recurrence, norm and pointwise convergence, structural properties) that hold for multiple ergodic averages for commuting systems should extend to nilpotent systems.
\end{heuristic}

\subsection{Groups of superpolynomial growth}
While nilpotent actions are good for the purpose of multiple recurrence and norm convergence, the same cannot be said for solvable actions of exponential growth. In \cite[Section 4]{BL02}, Bergelson and Leibman constructed examples of such actions which are neither good for multiple recurrence nor for norm convergence.
%Bergelson and Leibman \cite[Section 4]{BL02} constructed systems $(X, \CX,$ $\mu, T_1, T_2)$, with $T_1, T_2$ generating solvable groups of exponential growth, for which the multiple recurrence property fails and the average \eqref{E: BL nilpotent average} fails to converge. 
On the other hand, it is unclear what happens for groups of intermediate (i.e. superpolynomial but subexponential) growth.
\begin{problem}[Multiple recurrence in groups of intermediate growth]
    Suppose that $T_1, T_2$ generate a group of intermediate growth of measure-preserving transformations on a standard probability space $(X, \CX, \mu)$. Does this system admit the \emph{multiple recurrence property}, i.e. given $E\in\CX$ with $\mu(E)>0$, does there exist $n\in\N$ for which $\mu(E\cap T_1^{-n}E\cap T_2^{-n}E)>0$?
    %$(X, \CX, \mu,$\! $T_1, T_2)$ be a system with $T_1, T_2$ gener
\end{problem}
\begin{problem}[Norm convergence in groups of intermediate growth]
    Suppose that $T_1, T_2$ generate a group of intermediate growth of measure-preserving transformations on a standard probability space $(X, \CX, \mu)$. Does this system admit the \emph{norm convergence property}, i.e. the average \eqref{E: BL nilpotent average}
    % \begin{align*}
    %     \E_{n\in[N]}T_1^n f_1\cdot T_2^n f_2
    % \end{align*}
    converges in $L^2(\mu)$ for all $f_1, f_2\in L^\infty(\mu)$?
    %$(X, \CX, \mu,$\! $T_1, T_2)$ be a system with $T_1, T_2$ gener
\end{problem}
An example of an amenable group of intermediate growth is the Grigorchuk group. The study of multiple ergodic averages for groups of intermediate growth encounters obstructions somewhat similar to those faced by sequences of superpolynomial growth (see Section \ref{SSS: superpolynomial}): finitely many applications of the van der Corput inequality are incapable of bringing the average to one sufficiently simple to be amenable to known techniques.

For more discussion and conjectures on multiple ergodic averages in solvable groups, we direct the reader to \cite[Section 5]{BL02}.

\subsection{Independent and distinct-growth sequences}
As is the case with commuting systems, the ergodic averages are easier to understand when the sequences are pairwise independent, and even better - if they have distinct growth. Perhaps the strongest validation of Heuristic \ref{H: nilpotent heuristic} would come from resolving affirmatively the following two problems, which extend Heuristics \ref{H: independent} and \ref{H: pairwise independent seminorm} to nilpotent systems.
\begin{problem}\label{Pr: affinely independent nilpotent}
    Show that any affinely independent polynomials $a_1, \ldots, a_\ell\in\Z[t]$ admit rational Kronecker factor control for any nilpotent system $(X, \CX, \mu,$\! $T_1, \ldots, T_\ell)$.
\end{problem}
\begin{problem}\label{Pr: pairwise independent nilpotent}
    Show that any pairwise independent polynomials $a_1, \ldots, a_\ell\in\Z[t]$ admit Host-Kra seminorm control for any nilpotent system $(X, \CX, \mu,$\! $T_1, \ldots, T_\ell)$.
\end{problem}

Problems \ref{Pr: affinely independent nilpotent} and \ref{Pr: pairwise independent nilpotent} seem out of reach of current methods given that both the degree lowering and seminorm smoothing methods used to establish their commutative counterparts rely on commutativity (specifically, the dual-difference interchange step uses commutativity in an essential way). The experience from the commutative case suggests that sequences of distinct growth should do much better. For that reason, we record separately the following two problems, the first of which is a special case of Problems \ref{Pr: affinely independent nilpotent} and \ref{Pr: pairwise independent nilpotent}.
\begin{problem}\label{Pr: distinct degree nilpotent}
    Show that any nonconstant polynomials $a_1, \ldots, a_\ell\in\Z[t]$ of distinct degrees admit Host-Kra seminorm control/rational Kronecker factor control for any nilpotent system. %$(X, \CX, \mu,$\! $T_1, \ldots, T_\ell)$.
\end{problem}
While rational Kronecker factor control is a priori harder to establish than Host-Kra seminorm control, obtaining the latter for distinct-degree polynomials and nilpotent systems should give the former when combined with the equidistribution results of Chu, Frantzikinakis, and Host \cite[Section 7.4]{CFH11} - at least for monomials $n, n^2, \ldots, n^\ell$.
\begin{problem}\label{Pr: distinct powers nilpotent}
    Let $0<b_1< \cdots < b_\ell$. Show that the sequences $n^{b_1}, \ldots, n^{b_\ell}$ admit Host-Kra seminorm control/invariant factor control for any nilpotent system. %$(X, \CX, \mu,$\! $T_1, \ldots, T_\ell)$.
\end{problem}

Given that virtually nothing is known on Problems \ref{Pr: distinct degree nilpotent} and \ref{Pr: distinct powers nilpotent}, we explicitly present two simplest open cases.
\begin{problem}
    Show that $n, n^2$ are jointly ergodic for any 2-step nilpotent system $(X, \CX, \mu,$\! $T_1, T_2)$ with $T_1, T_2$ totally ergodic/weakly mixing.
\end{problem}
\begin{problem}
    Show that $n, n^{3/2}$ are jointly ergodic for any 2-step nilpotent system $(X, \CX, \mu,$\! $T_1, T_2)$ with $T_1, T_2$ ergodic/weakly mixing.
\end{problem}
\subsection{Joint ergodicity on nilpotent actions}
We conclude the discussion of nilpotent actions with a joint ergodicity result of Bergelson and Leibman and its potential generalization.
\begin{theorem}[Characterization of joint ergodicity on nilpotent systems {\cite[Theorem B]{BL02}}]\label{T: BL nilpotent joint ergodicity}
    Let $(X, \CX, \mu,$\! $T_1, T_2)$ be a $k$-nilpotent system. Then $(T_1^n, T_2^n)_n$ is jointly ergodic for $(X,\CX,\mu)$ if and only if the following two conditions hold:
    \begin{enumerate}
        \item $T_1 \times T_2$ is ergodic on $(X\times X, \CX\otimes \CX, \mu\times\mu)$;
        \item the group $\langle T_1^nT_2^{-n}\colon 1\leq n\leq k\rangle$ acts ergodically on $(X, \CX, \mu)$.
    \end{enumerate}
\end{theorem}
A glimpse at the proof of Theorem \ref{T: BL nilpotent joint ergodicity} shows that it exploits heavily the fact that the authors deal with two transformations. Therefore, a more sophisticated method would probably be required to extend it to an arbitrary number of transformations.
\begin{problem}[{\cite[Conjecture 5.5]{BL02}}]
    Prove or disprove the following: for any nilpotent system $(X, \CX, \mu,$\! $T_1, \ldots, T_\ell)$, the action $(T_1^n, \ldots, T_\ell^n)_n$ is jointly ergodic for $(X,\CX,\mu)$ if and only if the following two conditions hold:
    \begin{enumerate}
        \item $T_1 \times \cdots \times T_\ell$ is ergodic on $(X^\ell, \CX^{\otimes \ell}, \mu^\ell)$;
        \item for any distinct $1\leq i, j\leq \ell$, the group $\langle T_i^nT_j^{-n}\colon n\in\Z\rangle$ acts ergodically on $(X, \CX, \mu)$.
    \end{enumerate}
\end{problem}
%\subsection{Entropy assumptions}
%/$\Z^\ell$ actions}

\section{Other topics}\label{S: other topics}
Despite the size of this survey, there are way too many topics related to joint ergodicity that we have not managed to cover. This section aims to list such topics with relevant references but without going into too detailed discussion.
\begin{enumerate}
    \item \emph{Interactions of additive and multiplicative systems.} In \cite{BR22}, Bergelson and Richter initiated an interesting study of the mutual interactions between additive and multiplicative systems. Charamaras \cite{Ch25} and Xiao \cite{Xiao25} then proved a number of joint ergodicity and asymptotic discorrelation results for averages involving both additive and multiplicative actions, an example of which is the average
    \begin{align*}
        \E_{n\in[N]}T_1^n f_1 \cdot T_2^{\Omega(n)}f_2
    \end{align*}
    for (not necessarily commuting) measure-preserving transformations $T_1, T_2$.
    \item \emph{Total joint ergodicity.} Koutsogiannis and Sun \cite{KoSu23} introduced the notion of \emph{total joint ergodicity}, i.e. joint ergodicity along all arithmetic progressions. They proved a number of results and posed several questions for the total joint ergodicity of real polynomials and Hardy sequences. 
    \item \emph{Multiple ergodic averages with entropy assumptions.} Frantzikinakis \cite{Fr22b} and Frantzikinakis-Host \cite{FH23} proved norm convergence and structural results on averages
    \begin{align*}
        \E_{n\in[N]}T_1^n f_1 \cdot T_2^{a(n)}f_2,
    \end{align*}
    where $a(n)$ is a nonlinear polynomial or a fractional power, $T_1, T_2$ do not commute, and $T_1$ has zero entropy. In contrast to earlier works on multiple ergodic averages, they proceeded by investigating Furstenberg systems of sequences $(f(T^{a(n)}x))_n$ and applying disjointness arguments.
    This line of research was carried on by Huang-Shao-Ye \cite{HSY24a, HSY24b, HSY25, HSY26}, Austin \cite{Au25}, Ryzhikov \cite{Ry24}, and Kosloff-Sanadhya \cite{KoSa25, Kosa24}. Most of these works construct counterexamples to the phenomena observed by Frantzikinakis and Host when either $n$ is replaced by a nonlinear polynomial or $a(n) = n$. Their upshot is that the behavior distilled by Frantzikinakis and Host is very rare and breaks for even slight modifications of the averages. 
    \item \emph{Topological analogs of joint ergodicity.} Many questions on multiple ergodic averages have natural topological analogs; joint ergodicity is one of them. In the topological setting, the point is to show that for fixed sequences $a_1, \ldots, a_\ell:\N\to\Z$ and under suitable assumptions on $T_1, \ldots, T_\ell$, homeomorphisms of a compact metric space $X$, the orbit $$(T_1^{a_1(n)}x, \ldots, T_\ell^{a_\ell(n)}x)_n$$ is dense in $X^\ell$ for some $x\in X$. Starting with Glasner's topological analog of Theorem \ref{T: Furstenberg averages weakly mixing} \cite{Gl94}, this question has been extensively studied e.g. by Moothathu \cite{Mo10}, Kwietniak-Oprocha \cite{KwOp12}, Huang-Shao-Ye \cite{HSY19}, Zhang-Zhao \cite{ZhZh21}, Qiu \cite{Qiu23}, and Donoso-Koutsogiannis-Sun \cite{DKS24}. 
    \item \emph{Joint ergodicity of piecewise monotone maps.} Bergelson and Son \cite{BS23, BS22} proved interesting (norm and pointwise) joint ergodicity results for various piecewise monotone maps on $\T$, including beta-transformations, Gauss maps, and interval exchange transformations. These results have applications to the problems of joint normality of numbers and independence of their digit and continuous fraction expansions. 
    \item \emph{Pointwise convergence of multiple ergodic averages.} Starting with the Krause-Mirek-Tao \cite{KMT20} breakthrough, recent years have seen an outbreak of activity on pointwise convergence of multiple ergodic averages due to Daskalakis \cite{Das25b}, Kosz-Mirek-Peluse-Wan-Wright \cite{KMPWW24}, Krause-Mirek-Tao-Ter\"av\"ainen \cite{KMTT24}, Ter\"av\"ainen \cite{Ter24}, and Wan \cite{Wan25}. While not directly related to joint ergodicity, they nonetheless imply pointwise versions of many results described in this survey. Earlier results on the topic include the classical work of Bourgain \cite{Bo90} on double recurrence, as well as the works of Huang-Shao-Ye \cite{HSY19b} and Donoso-Sun \cite{DS21} that cover pointwise convergence on distal systems.   
    \item \emph{Asymptotic total ergodicity.} Recent joint ergodicity results appeared in parallel to analogous quantitative discorrelation estimates in the discrete setting, starting with the result of Bourgain and Chang \cite{BC17}
    \begin{align*}
        \abs{\E_{m,n\in\F_p}f_0(m)f_1(m+n)f_2(m+n^2) - \E_{m\in\F_p}f_0(m)\cdot \E_{m\in\F_p}f_1(m)\cdot \E_{m\in\F_p}f_2(m)}\ll p^{-1/10}
    \end{align*}
    for 1-bounded functions $f_0,f_1,f_2:\F_p\to\C$. More results on this form were proved by Peluse \cite{Pel18, Pel19}, Dong-Li-Sawin \cite{DLS17}, Han-Lacey-Yang \cite{HLY21}, Bergelson-Best \cite{BB23a}, and the author \cite{Kuc23, Kuc22b}. Due to parallels with joint ergodicity, Bergelson and Best refer to this phenomenon as \emph{asymptotic joint ergodicity} \cite{BB23b}.
    \item \emph{Joint ergodicity for general group actions.} Bergelson and Rosenblatt \cite{BR99} as well as Son \cite{Son20} investigated variants of Theorem \ref{T: BB characterization} for general group actions. 
    In a different direction, Ackelsberg and Bergelson \cite{AB24} as well as Best and Ferr\'e-Moragues \cite{BFM22} extended Theorem \ref{T: FrKr affine independence totally ergodic} to actions by more general number fields and rings respectively. The Ackelsberg-Bergelson work followed the Old Joint Ergodicity Strategy, and specifically derived joint ergodicity from an equidistribution result on Weyl systems. By contrast, Best and Ferr\'e-Moragues extended Theorem \ref{T: joint ergodicity criteria} to second-countable locally compact abelian groups and then used it to deduce their result for polynomials.
    % In a different direction, Best and Ferr\'e-Moragues \cite{BFM22} extended the single-transformation case of Theorem \ref{T: joint ergodicity criteria} to second-countable locally compact abelian groups. Then they applied this result to extend Theorem \ref{T: FrKr affine independence totally ergodic} to more general polynomial ring actions.
    \item \emph{Joint ergodicity for flows.} In largely overlooked \cite[Chapter 7]{Fr21}, Frantzikinakis proved pointwise joint ergodicity results for not necessarily commuting flows $T_1,\ldots, T_\ell$. Specifically, he showed that if $a_1, \ldots, a_\ell\in\CH$ are Hardy sequences of ``distinct but comparable growth'', then 
    \begin{align*}
        \lim_{N\to\infty}\int_0^N f_1(T_1^{a_1(t)}x)\cdots f_\ell(T_\ell^{a_\ell(t)}x)\; dt = \E(f_1|\CI(T_1))(x)\cdots \E(f_\ell|\CI(T_\ell))(x)
    \end{align*}
    almost everywhere.
    The method fails for sequences whose growth is either ``too similar'' (like $t^2,\; t^2+t$ and $t,\; t\log t$) or ``too different'' (like $t,\; 2^t$). The case of distinct noninteger fractional powers was previously established by Bergelson, Boshernitzan, and Bourgain \cite{BBB94}.
\end{enumerate}
%\section{Interactions of additive and multiplicative systems}

\appendix
\section{Notation, conventions, and preliminaries}
%The following basic notation is used throughout the paper.

\subsection{Notation}
The symbols $\C, \R, \R_+, \Q, \Z, \N_0, \N, \P$ respectively stand for the sets of complex numbers, reals, positive reals, rationals, integers, nonnegative integers, positive integers, and primes. With $\T$, we denote the one dimensional torus, and we often identify it with $\R/\Z$ or  with $[0,1)$. Likewise, $\S^1$ stands for the unit circle inside $\C$. For $p$ prime, $\F_p$ stands for the finite field with $p$ elements. For $N\in\N$, we denote $[N]:=\{1, \ldots, N\}$ and $[\pm N]:=\{-N, \ldots, N\}$. %as is customary in the field. 
%For a ring $R$, we denote the collection of polynomials over $R$ in variables $x_1, \ldots, x_k$ via $R[x_1, \ldots, x_k]$. 
If $E\subseteq G$ is a subset of a group $G$, then $\langle E \rangle$ is the subgroup generated by $E$.

%We frequently use the interval notation to mean either a real interval or its restriction to $\Z$; the meaning would typically be clear from the context. If we restrict to $\Z$, then we write $[M,N] = \{M, \ldots, N\}$ and  $[N]=\{1, \ldots, |N|\}$ for integers $M\leq N$.

For a finite set $E,$ we define the average of a sequence $a:E\to \C$ as 
\[\E_{n\in E}a(n) :=\frac{1}{|E|}\sum_{n\in E}a(n).\] 
%We will mainly use finite averages along subsets of integers of the form $[N]^s$ or $[\pm N]^s$, also known as \emph{Ces\`aro averages}.

We let $\CC z := \overline{z}$ be the complex conjugate of $z\in \C$.

For a set $E$, we define its indicator function by $1_E$. %Whenever convenient, we also use the notation $1(E)$ for an event $E$.

We denote $\ell$-tuples with the bold notation $\bh = (h_1, \ldots, h_\ell)$. Given a system $(X, \CX, \mu,$  $T_1, \ldots, T_\ell)$ and $\bh\in\Z^\ell$, we let $T^\bh:=T_1^{h_1}\cdots T_\ell^{h_\ell}$.

%Typically, we use underlined symbols $\uh$ to denote elements of $\R^s$ (which we think as tuples of parameters), and we employ the bold notation $\bx$ to denote elements of $\R^\ell$ (which we think as points in a space); sometimes we also use bold notation to denote $s$-tuples of points in $\R^\ell$, e.g., $\bm=(\bm_1, \ldots, \bm_s)\in(\R^s)^\ell$. We denote the integer and fractional parts of $\bx$ via
% \begin{align*}
%     \floor{\bx} = (\floor{x_1}, \ldots, \floor{x_\ell})\quad \textrm{and}\quad \rem{\bx} =(\rem{x_1}, \ldots, \rem{x_\ell}).
% \end{align*}

% We often write $\ueps\in\{0,1\}^s$ for a vector of 0s and 1s of length $s$. For $\ueps\in\{0,1\}^s$ and $\uh, \uh'\in\R^s$, we set $\ueps\cdot \uh:=\eps_1 h_1+\cdots+ \eps_s h_s$, $\abs{\uh} := |h_1|+\cdots+|h_s|$ and
% \begin{gather}\label{E: h^eps}
%     h_j^{\eps_j} =\begin{cases}
%         h_j,\; &\eps_j =0\\
%         h_j',\; &\eps_j = 1
%     \end{cases}
%     \quad \textrm{for}\quad j\in[s].
% \end{gather}

% Given a system $(X, \CX, \mu,$\! $T_1, \ldots, T_\ell)$ and $\bm=(m_{1},\dots,m_{\ell})\in\R^\ell$, we define $$T^{\floor{\bm}}:=T_1^{\floor{m_1}}\cdots T_\ell^{\floor{m_\ell}}.$$
% For $j\in[\ell]$, we set $\be_j$ to be the unit vector in $\R^\ell$ in the $j$-th direction, and we let $\be_0 = \mathbf{0}$, so that $T^{\be_j} = T_j$ for $j\in[\ell]$ and $T^{\be_0}$ is the identity transformation. 

%For $M>0$, 
Let $(X, \CX, \mu)$ be a probability space. We call a function $f\in L^\infty(\mu)$ \emph{$1$-bounded} if $\norm{f}_{L^\infty(\mu)}\leq 1$.

We denote the space of complex-valued continuous functions a topological space $Y$ by $C(Y).$

For a $\sigma$-algebra $\CA\subseteq\CX$ and $f\in L^1(\mu)$, we let $\E(f|\CA)$ be the conditional expectation of $f$ with respect to $\CA$. %By abuse of notation, we can identify the $\sigma$-algebra $\CA$ with the algebra of functions $L^2(\CA)$. 
%One $\sigma$-algebra that we often work with is the \emph{invariant factor} $\CI(T) = \{A\in \CX:\; T\inv A = A\}$ of a transformation $T$.

Throughout, all the equalities between functions and sets are understood to hold up to sets of measure 0.

Given two functions $a,b:(t_0,\infty)\to\mathbb{R}$ for some $t_0\in\R$, we write
\begin{enumerate}
    \item $a\ll b$ or $a = O(b)$ if there exists $C>0$ %and $x_1\geq x_0$ 
    such that $|a(x)| \leq C |b(x)|$ for all $t\geq t_0$; 
    \item $a \asymp b$ if $a\ll b$ and $b\ll a$;
    \item $a \prec b$ or $a = o(b)$  if $\lim\limits_{x\to\infty}\frac{a(x)}{b(x)} = 0$.
    %\item $a \lll b$ if there exists $\delta > 0$ such that $|a(x)|x^\delta \prec |b(x)|$.
\end{enumerate}
%If the absolute constant in (i) depends on some parameters $x_1, \ldots, x_n$, then we write $a\ll_{x_1, \ldots, x_n} b$, $a = O_{x_1, \ldots, x_n}(b)$, etc. If in (ii) we want to emphasize what parameter we are taking to infinity, we can also write $a = o_{x\to\infty; x_1, \ldots, x_n} (b)$ to emphasize that $x_1, \ldots, x_n$ are kept fixed while $x$ goes to infinity.

\subsection{Van der Corput inequality}

The van der Corput inequality is a ubiquitous tool in ergodic theory and additive combinatorics. Here is one version.
\begin{lemma}[Van der Corput inequality]\label{L: vdC} 
    Let $(v_n)_n$ be a 1-bounded sequence in a Hilbert space. Then
    \begin{align*}
        \limsup_{N\to\infty}\norm{\E_{n\in[N]}v_n}^2 \leq \limsup_{H\to\infty}\E_{h\in[H]}\limsup_{N\to\infty}\abs{\E_{n\in[N]}\langle v_n, v_{n+h}\rangle}. 
    \end{align*}
\end{lemma}
For an extensive discussion on the role of Lemma \ref{L: vdC} in ergodic theory, consult the survey of Bergelson and Moreira \cite{BM16}.

\section{Host-Kra theory}\label{A: Host-Kra theory}
This section is dedicated to the basic properties of box seminorms, defined by Host \cite{H09}, and their special but important subclass: Host-Kra seminorms originating in \cite{HK05a}. We refer the reader e.g. to \cite{DKS22, FrKu22b, HK05a, HK18, TZ16} for the proofs of their various properties listed in this section.

\subsection{Box seminorms: definition and basic properties}\label{AA: box seminorms}
Let $(X, \CX, \mu,$\! $T_1, \ldots, T_\ell)$ be a system and $f\in L^{\infty}(\mu)$. For $\bh = (h_1, \ldots, h_\ell)\in\Z^\ell$, we set $T^\bh:=T_1^{h_1}\cdots T_\ell^{h_\ell}$.
%$R\in\angle{T_1, \ldots, T_\ell}$, 
We define the \emph{multiplicative derivative} of $f$ along $\bh$ via
 $$
 \Delta_{\bh} f :=  f \cdot T^\bh \overline{f} 
 $$
   and for $\bh_1, \ldots, \bh_s\in\Z^\ell$, we denote the \emph{iterated multiplicative derivative} via
$$
\Delta_{\bh_1, \ldots, \bh_s} f  := \Delta_{\bh_1}\cdots\Delta_{\bh_s} f = \prod_{\ueps\in\{0,1\}^s} \CC^{|\ueps|} T^{\eps_1 \bh_1 + \cdots \eps_s\bh_s }f.
$$

% If each transformation $R_j$ appears $k_j$ times, we also denote
% \begin{align*}
%     \Delta_{R_1^{\times k_1}, \dots, R_s^{\times k_s}}f := \Delta_{\underbrace{R_1, \ldots, R_1}_{k_1},\; \ldots,\; \underbrace{R_s, \ldots, R_s}_{k_s}}f.
% \end{align*}
% Lastly, we employ the conventions
% \begin{align*}
%     \Delta_{R_1, \ldots, R_s; h}f &:=\Delta_{R_1^{h_1},\ldots, R_s^{h_s}}f\quad &&\textrm{for}\quad h\in\Z^s\\
%     \textrm{and}\quad \Delta_{R_1^{\times k_1}, \ldots, R_s^{\times k_s}; h}f &:=\Delta_{R_1^{h_{11}},\ldots, R_1^{h_{1k_1}}}\cdots\Delta_{R_s^{h_{s1}},\ldots,  R_s^{h_{sk_s}}}f\quad &&\textrm{for}\quad h\in\Z^{k_1+\cdots+k_s}
% \end{align*}
% whenever convenient.

For subgroups $G_1, \ldots, G_s\subseteq\Z^\ell$, we define the \emph{box seminorm of $f$ along $G_1, \ldots, G_s$} by 
\begin{align}\label{E: box seminorm}
	\nnorm{f}_{G_1, \ldots, G_s}^{2^{s}}:=  \E_{\bh_s\in G_s}\cdots \E_{\bh_1\in G_1}\int \Delta_{\bh_1, \ldots, \bh_s}f\; d\mu.
    %\lim_{H\to\infty}\E_{h\in[H]}\nnorm{\Delta_{R_s^h}f}_{R_1, \ldots, R_{s-1}}^{2^{s-1}}
\end{align}
Above, we set
\begin{align*}
    \E_{\bh\in G} := \lim_{H\to\infty}\E_{\bh\in G\cap[\pm H]^\ell},
\end{align*}
noting that the limit is well-defined, $(G\cap [\pm H]^\ell)_H$ can be replaced by any F{\o}lner sequence on $G$, and we can combine several limits into one (e.g. we can replace $\E_{\bh_s\in G_s}\cdots \E_{\bh_1\in G_1}$ by $\E_{(\bh_1, \ldots, \bh_s)\in G_1\times \cdots \times G_s}$). Whenever convenient, we use the following shorthand notations:
\begin{enumerate}
    \item if each of $G_1, \ldots, G_s$ appears $k_1, \ldots, k_s$ times respectively, we write the seminorm as $\nnorm{f}_{G_1^{\times k_1}, \ldots, G_s^{\times k_s}}$;
    \item if $G_j = \langle \bv_j\rangle $ for $1\leq j\leq s$, we also write $\nnorm{f}_{G_1, \ldots, G_s} = \nnorm{f}_{\bv_1, \ldots, \bv_s}$;
    \item if additionally $R_j = T^{\bv_j}$ for $1\leq j\leq s$, we denote $\nnorm{f}_{G_1, \ldots, G_s} = \nnorm{f}_{R_1, \ldots, R_s}$;
    \item if additionally $R:=R_1 = \cdots R_s$, we write $\nnorm{f}_{G_1, \ldots, G_s} = \nnorm{f}_{s, R}$.
\end{enumerate}
In the latter case, we call $\nnorm{f}_{s, R}$ the \emph{degree-$s$ Host-Kra seminorm of $f$ along $R$}.

% ide 
% for $s\in\N$ and $R_1, \ldots, R_s\in\angle{T_1, \ldots, T_\ell}$. Expanding the definition, we thus have
% \begin{align*}
%     \nnorm{f}_{R_1, \ldots, {R_s}}^{2^{s}} = \lim_{H_s\to\infty}\E_{h_s\in[H_s]}\cdots \lim_{H_1\to\infty}\E_{h_1\in[H_1]}\int \Delta_{R_1, \ldots, R_s; h}f\; d\nu.
% \end{align*}
% As explained in \cite[Lemma 1]{H09}, we can replace $[H_i]$ in the average by any F{\o}lner sequence in $\Z$. Combined with \cite[Lemma 1.1]{BL15}, this allows us to replace any $s'$ iterated limits by the single limit $\lim\limits_{H\to\infty}\E_{h\in[H]^{s'}}$.

Box seminorms satisfy several standard properties freely used throughout the paper: %Their proofs can be found e.g. in \cite{DKKST24, FrKu22b, H09}. In what follows, we take $R_1, \ldots, R_{s}\in\angle{T_1, \ldots, T_\ell}$ and $f\in L^\infty(\mu)$.
\begin{enumerate}
    % \item (Inductive formula) we have
    % \begin{align*}
    %     \nnorm{f}_{G_1, \ldots, G_s}^{2^s}=\E_{\bh_s\in G_s}\nnorm{\Delta_{\bh_s}f}_{G_1, \ldots, G_{s-1}}^{2^{s-1}};
    % \end{align*}
    \item (Permutation invariance) for
    any permutation $\sigma:[s]\to[s]$, 
    \begin{align*}
        \nnorm{f}_{G_1, \ldots, G_s} = \nnorm{f}_{G_{\sigma(1)},\ldots, G_{\sigma(s)}};
    \end{align*}
    \item (Monotonicity) we have
    \begin{align*}
    \nnorm{f}_{G_1}\leq \nnorm{f}_{G_1, G_2}\leq \cdots\leq    \nnorm{f}_{G_1, \ldots, G_s};
    \end{align*}
    \item (Inductive formula) for any $1\leq s'\leq s$, 
    \begin{align*}
        \nnorm{f}_{G_1, \ldots, G_s}^{2^s} = \E_{\bh_s\in G_s}\cdots \E_{\bh_{s'+1}\in G_{s'+1}}\nnorm{\Delta_{\bh_{s'+1}, \ldots, \bh_s}f}_{G_1, \ldots, G_{s'}}^{2^{s'}};
        %\lim_{H_s\to\infty}\E_{h_s\in[H_s]}\cdots \lim_{H_{s'+1}\to\infty}\E_{h_{s'+1}\in[H_{s'+1}]}\nnorm{\Delta_{G_{s'+1}, \ldots, G_s; h'}f}_{G_1, \ldots, G_{s'}}^{2^{s'}},
    \end{align*}
    %where $h' = (h_{s'+1}, \ldots, h_s)$; %Moreover, the iterated limit can be replaced by the single limit $\lim\limits_{H\to\infty}\E_{h'\in [H]^{s-s'}}$.
    % \item (Gowers-Cauchy-Schwarz inequality) for any $(f_\ueps)_{\ueps\in\{0,1\}^s}\subseteq L^\infty(\mu)$, we have 
    % \begin{align*}
    %     \Big|\lim_{H_s\to\infty}\E_{h_s\in[H_s]}\cdots \lim_{H_1\to\infty}\E_{h_1\in[H_1]}\int \prod_{\ueps\in\{0,1\}^s}\CC^{|\ueps|} G_1^{\eps_1}\cdots G_s^{\eps_s}f_\ueps\; d\mu\Big|\leq \prod_{\ueps\in\{0,1\}^s}\nnorm{f_\ueps}_{G_1, \ldots, G_s}.
    % \end{align*}
    % Once again, the iterated limit on the left-hand side can be replaced by the single limit $\lim\limits_{H\to\infty}\E_{h\in [H]^{s}}$.
    \item \label{I:scaling} (Scaling) for any subgroups $G'_1\subseteq G_1, \ldots, G'_s\subseteq G_s$,
    \begin{align}\label{E: scaling 1}
        \nnorm{f}_{G_1, \ldots, G_s}\leq \nnorm{f}_{G_1', \ldots, G_s'};
    \end{align}
    %(We caution that $G_i^{r_i}$ denotes the $r_i$-th power of $G_i$, not $r_i$ copies of $G_i$, denoted by $G_i^{\times r_i}$.) 
    conversely, if $s\geq 2$ and $|G_i/G_i'| = r_i$ for all $1\leq i\leq s$, then
    \begin{align}\label{E: scaling 2}
        \nnorm{f}_{G_1', \ldots, G_s'}\leq |r_1\cdots r_s|^{1/2^s}\nnorm{f}_{G_1, \ldots, G_s}.
    \end{align}
\end{enumerate}

\subsection{Box factors and dual functions}\label{AA: box factors}
Once again, let $(X, \CX, \mu,$\! $T_1, \ldots, T_\ell)$ be a system and $G_1, \ldots, G_s\subseteq\Z^\ell$ be subgroups. Set $\{0,1\}^s_* := \{0,1\}^s\setminus\{{0}\}$. For $(f_\ueps)_{\ueps\in \{0,1\}^s_*}\subseteq L^\infty(\mu)$, we define the \emph{dual function of  $(f_\ueps)_{\ueps\in \{0,1\}^s_*}$ along $G_1, \ldots, G_s$} via
\begin{align*}
    \CD_{G_1, \ldots, G_s}((f_\ueps)_\ueps) := \E_{\bh_s\in G_s}\cdots \E_{\bh_1\in G_1} \prod_{\ueps\in\{0,1\}^s_*}\CC^{|\ueps|} T_1^{\eps_1 \bh_1 +\cdots +\eps_s\bh_s}f_\ueps.
    %\CD_{R_1, \ldots, R_s}((f_\ueps)_\ueps) := \lim_{H_s\to\infty}\E_{h_s\in[H_s]}\cdots \lim_{H_1\to\infty}\E_{h_1\in[H_1]} \prod_{\ueps\in\{0,1\}^s_*}\CC^{|\ueps|} R_1^{\eps_1}\cdots R_s^{\eps_s}f_\ueps.
\end{align*}
%The limit exists in $L^2(\mu)$ (see e.g. \cite[Proposition 2.2]{TZ16}) and can be replaced by the single limit $\lim\limits_{H\to\infty}\E_{h\in [H]^{s}}$ thanks to \cite[Lemma 1.1]{BL15}. 
We call $s$ the \emph{degree} of the dual function. If $f_\ueps = f$ for all $\ueps\in\{0,1\}^s_*$, we simply denote
\begin{align*}
    \CD_{G_1, \ldots, G_s}(f) := \CD_{G_1, \ldots, G_s}((f_\ueps)_\ueps),
\end{align*}
and when $G:=G_1 = \cdots = G_s$, we let
\begin{align*}
    \CD_{G_1, \ldots, G_s}((f_\ueps)_\ueps) = \CD_{s, G}((f_\ueps)_\ueps).
\end{align*}

Dual functions arise in the identity
\begin{align}\label{E: dual identity}
        \nnorm{f}_{R_1, \ldots, R_s}^{2^s} = \int f \cdot \CD_{R_1, \ldots, R_s}(f)\, d\mu,
\end{align}
which should be seen as a \emph{weak inverse theorem} for box seminorms: if $\nnorm{f}_{R_1, \ldots, R_s}$ is large, then $f$ correlates with a dual function along $R_1, \ldots, R_s$.
While too weak for applications regarding norm convergence, this inverse theorem is nevertheless sufficient for the analytic manipulations performed in the seminorm smoothing argument in order to show Host-Kra seminorm control.

While far from obvious, the $L^2(\mu)$-closed linear span of dual functions $\CD_{G_1, \ldots, G_s}((f_\ueps)_\ueps)$ forms an algebra $Z(G_1, \ldots, G_\ell)$ invariant under $T_1, \ldots, T_\ell$. Using the usual identification between closed invariant subalgebras of $L^2(\mu)$ and invariant sub-$\sigma$-algebras of $\CX$, there exists a factor $\CZ({G_1, \ldots G_s})\subseteq\CX$, called the \emph{box factor} along $G_1, \ldots, G_s$, such that $Z(G_1, \ldots, G_\ell) = L^2(\CZ({G_1, \ldots G_s}))$.
%We also define the \emph{box factor} $\CZ({G_1, \ldots G_s})$ to be the smallest factor with respect to which all dual functions $\CD_{G_1, \ldots, G_s}((f_\ueps)_\ueps)$ are measurable (equivalently, $L^2(\CZ({G_1, \ldots G_s}))$ is the closed linear span of such dual functions). 

The factor $\CZ({G_1, \ldots, G_s})$ satisfies the well-known property
\begin{align}\label{E: factor property}
    \nnorm{f}_{G_1, \ldots, G_s} = 0 \quad \Longleftrightarrow\quad \E(f|\CZ(G_1, \ldots, G_s)) = 0;
\end{align}
one direction follows from \eqref{E: dual identity} while the other one is a consequence of the Gowers-Cauchy-Schwarz inequality. In addition, box factors satisfying the following properties that can be directly deduced from the similar properties for seminorms:
\begin{enumerate}
    % \item (Inductive formula) we have
    % \begin{align*}
    %     \nnorm{f}_{G_1, \ldots, G_s}^{2^s}=\E_{\bh_s\in G_s}\nnorm{\Delta_{\bh_s}f}_{G_1, \ldots, G_{s-1}}^{2^{s-1}};
    % \end{align*}
    \item (Permutation invariance) $\CZ(G_1, \ldots, G_s) = \CZ(G_{\sigma(1)},\ldots, G_{\sigma(s)})$ for
    any permutation $\sigma:[s]\to[s]$; 
    \item (Monotonicity) we have
    \begin{align}\label{E: monotonicity of factors}
    \CZ(G_1)\subseteq \CZ({G_1, G_2})\subseteq \cdots\subseteq    \CZ({G_1, \ldots, G_s});
    \end{align}
    \item  (Scaling) for any subgroups $G'_1\subseteq G_1, \ldots, G'_s\subseteq G_s$, we have $$\CZ(G_1, \ldots, G_s)\subseteq \CZ(G_1', \ldots, G_s'),$$ with equality if $s\geq 2$ and every $G_i/G_i'$ is finite.
\end{enumerate}

When $G:=G_1 = \cdots = G_s$, we denote $\CZ_{s-1}(G) := \CZ(G_1, \ldots, G_s)$ and call it \emph{Host-Kra factor} along $G$ of degree $s-1$. The rather unfortunate choice of indices is justified by the \emph{Host-Kra structure theorem} (Theorem \ref{T: HK structure theorem}), which uncovers a deep connection between the (ergodic) Host-Kra factor of a fixed degree to nilsystems of the same degree. The Host-Kra structure theorem is the main reason why Host-Kra seminorms/factors are vastly preferable to box seminorms/factors. Except in a few special cases, the structure theory for the latter is completely unknown. The simplest possible open case is provided in the following problem.
\begin{problem}[Structure theory for box factors]\label{Pr: structure theorem for box}
    Let $(X, \CX, \mu,$\! $T_1, T_2)$ be a system. Find a useful structure theorem for the factor $\CZ(T_1, T_2, T_1T_2)$ (possibly on a suitable extension).
\end{problem}
The additive combinatorial version of Problem \ref{Pr: structure theorem for box} is currently the major stumbling block for proving reasonable bounds in the multidimensional Szemer\'edi theorem - see \cite[Section 4.2]{Pel24} for the discussion.

\subsection{Low-degree Host-Kra factors}\label{A: low-degree HK factors}
We now specialize to the case of a single transformation. Let $(X, \CX, \mu, T)$ be a system; we summarize the properties of the factors $\CZ_0(T), \CZ_1(T)$.
For $\CZ_0(T)$, we observe that $\CD_{1,T}(f) = \E(\overline f|\CI(T))$; hence
\begin{align}\label{E: deg 1 inverse theorem}
    \nnorm{f}_{1,T}^2 = \int f \cdot \E(\overline f|\CI(T))\; d\mu,
\end{align}
and $\CZ_0(T) = \CI(T)$
% \begin{align*}
%   \CZ_0(T) = \CI(T) = \{E\in \CX\colon\;  T^{-\bh} E = E\; \textrm{for\; all}\; \bh\in G\}
%   %\CZ(G) = \CI(G) = \{E\in \CX\colon\;  T^{-\bh} E = E\; \textrm{for\; all}\; \bh\in G\}
% \end{align*}
is the $\sigma$-algebra of $T$-invariant functions/sets. %Here and elsewhere, all the equalities are understood to hold up to sets of measure 0.
%\section{Nonergodic eigenfunctions}\label{A: nonergodic eigenfunctions}

To describe the $\CZ_1(T)$ factor, we need the notion of nonergodic eigenfunctions. %A \emph{eigenfunction} of $T$ is a function $f$ satisfying $Tf = \lambda f$ for $\lambda\in\C$. A  

\begin{definition}[Nonergodic eigenfunctions]
A function $\chi\in L^\infty(\mu)$ is called a \emph{nonergodic eigenfunction} of $T$ if it satisfies the following properties:
    \begin{enumerate}
        \item $T\chi = \lambda \chi$ for some $T$-invariant $\lambda\in L^\infty(\mu)$ (which we call the \emph{nonergodic eigenvalue} of $f$);
        \item $|\chi(x)|\in\{0,1\}$ for $\mu$-a.e. $x\in X$, and $\lambda(x)=0$ whenever $\chi(x)=0$.
\end{enumerate}
%We denote the set of nonergodic eigenfunctions of $T$ by $\CE(T)$.
\end{definition}

		 For ergodic systems, a nonergodic eigenfunction is either the zero function or a classical unit-modulus eigenfunction. For general systems, each nonergodic eigenfunction $\chi$ satisfies
		 \begin{equation*}%\label{E:nonergodiceigen}
		 	\chi(Tx)=1_E(x) \cdot e(\phi(x))\cdot \chi(x)
		 \end{equation*}
		  for some $T$-invariant set $E\in \CX$ and measurable $T$-invariant function $\phi \colon X\to \T$. 
          \begin{example}
              Consider the nonergodic system $X = \T^2$, $T(x,y) = (x, y+x)$. All its classical eigenfunctions are constant; by contrast, functions $\chi(x,y) = e(kx+ly)$ are nonergodic eigenfunctions for any $k,l\in\Z$ since $\chi(T(x,y)) = e(lx)\chi(x,y)$.
          \end{example}
          
          By \cite[Theorem 5.2]{FH18}, the factor $\CZ_1(T)$ is generated by nonergodic eigenfunctions.

Nonergodic eigenfunctions can be used to give a \emph{quantitative} inverse theorem for the degree-2 Host-Kra seminorm. The version below is used in the proof of Proposition \ref{P: degree lowering}; since it is stronger than those stated in the literature, we provide a full proof. 
\begin{proposition}
    Let $(X, \CX, \mu, T)$ be a system. For every 1-bounded $f\in L^\infty(\mu)$ and every $\veps>0$ we can find a nonergodic eigenfunction $\chi$ such that 
    \begin{align*}
        \nnorm{f}_{2,T}^4 \leq \int f\cdot \chi\; d\mu + \veps,
    \end{align*}
    and moreover $\E(f\cdot \chi|\CI(T))\geq 0$.
\end{proposition}
\begin{proof}
    By \cite[Proposition 4.3]{FrKu22a}, we can find a nonergodic eigenfunction $\chi_0$ such that
    \begin{align*}
        \nnorm{f}_{2,T}^4 \leq \Re\brac{\int f\cdot \chi_0\; d\mu} + \veps.
    \end{align*}
    %Modifying $\chi_0$ by a phase, we can remove the real part. It remains to show that we can modify $\chi_0$ into $\chi$ satisfying $\E(f\cdot \chi|\CI(T))\geq 0$. 
    Observe that
    \begin{align*}
        \Re\brac{\int f\cdot \chi_0\; d\mu} = \Re\brac{\int \E(f\cdot \chi_0|\CI(T))\; d\mu}\leq \int \abs{\E(f\cdot \chi_0|\CI(T))}\; d\mu.
    \end{align*}
    Then there exists a function $\lambda\in L^\infty(\mu)$ for which 
    \begin{align*}
        \abs{\E(f\cdot \chi_0|\CI(T))} = \lambda\cdot \E(f\cdot \chi_0|\CI(T)),
    \end{align*}
    and which takes value on the unit circle whenever $\E(f\cdot \chi_0|\CI(T)) \neq 0$ and equals 0 otherwise. Moreover, the definition clearly shows that $\lambda$ is $T$-invariant. Hence
    \begin{align*}
        \int \abs{\E(f\cdot \chi_0|\CI(T))}\; d\mu = \int \lambda\cdot\E(f\cdot \chi_0|\CI(T))\; d\mu = \int f\cdot (\chi_0\cdot \lambda)\; d\mu,
    \end{align*}
    and so the result follows with $\chi := \chi_0\cdot \lambda$, which can be easily seen to be be a nonergodic eigenfunction.
\end{proof}

When $T$ is ergodic, $\CZ_1(T)$ equals the \emph{Kronecker factor}, the largest factor of $\CX$ measure-isomorphic to a group rotation, generated by \emph{eigenfunctions} of $T$, i.e. functions $f\in L^\infty(\mu)$ satisfying $Tf = e(\beta) f$ for some $\beta\in\T$ (called the \emph{eigenvalue} of $f$). 
Many properties of $\CZ_1(T)$ can then be read off from its \emph{spectrum}, i.e. the subgroup of $\T$ consisting of all the eigenvalues of $T$.
% \begin{align*}
%     \Spec(T) :=\rem{\beta\in\T\colon\; Tf = e(\beta) f\; \textrm{for some}\; f\in L^\infty(\mu)}.
% \end{align*}

\subsection{Nilsystems}\label{A: nilsystems}
While the point of modern joint ergodicity tools is to avoid proving complicated equidistribution results on nilsystems, the latter nevertheless arise in many places throughout the survey. Therefore, we state the most important definitions, directing the reader to \cite[Chapters 10 \& 11]{HK18} for definitions and major properties. 

An {\em $s$-step  nilmanifold} is a compact homogeneous space $Y=G/\Gamma$ where $G$ is an $s$-step nilpotent Lie group and $\Gamma$ is a discrete cocompact subgroup.  %We denote by  $e_X$ the image  of the identity element  of $G$ in $X$. 
For $b\in G$, the transformation $S\colon Y\to Y$ defined by $Sy=by$ is called an {\em $s$-step  nilrotation} on $S$ and the system $(Y,\CY, m,S)$, where $\CY$ is the Borel $\sigma$-algebra and $m$ is the Haar measure on $Y$,  is called an {\em $s$-step nilsystem.}
%If $X=G/\Gamma$ is a nilmanifold and $b\in G$, $x\in X$, then the closure $Y$ of the sequence $(b^n\cdot x)_n$ admits the structure of a nilmanifold, and the action  of $b$ on $Y$ is uniquely ergodic  (\cite[Chapter~11, Theorem~17]{HK18} or \cite[Section~2]{Lei05a}).  
\begin{example}[Examples of nilsystems]\label{Ex: nilsystems}
    Three classical examples of nilsystems are:
    \begin{enumerate}
        \item Rotations on $\T$;
        \item Skew products $T(x_1, \ldots, x_d) = (x_1 + \alpha, x_2 + x_1, \ldots, x_d + x_{d-1})$ on $\T^d$;
        \item Left-multiplication on the Heisenberg nilmanifold $\begin{pmatrix} 1 & \R & \R \\ 0 & 1 & \R \\ 0 & 0 & \R \end{pmatrix}/\begin{pmatrix} 1 & \Z & \Z \\ 0 & 1 & \Z \\ 0 & 0 & \Z \end{pmatrix}$.
    \end{enumerate}
\end{example}

Following \cite{BHK05}, we say that $\psi:\Z\to\C$ is a {\em nilsequence} if it is a uniform limit of sequences of the form $n\mapsto F(b^ny)$ for some  $b\in G$, $y\in Y$, and $F\in C(Y)$ (we note a notational disparity between the ergodic and combinatorial works on the subject: additive combinatorists call such functions \emph{generalized nilsequences}, referring only to functions like $n\mapsto F(b^ny)$ as nilsequences). We call $\omega:\Z\to\C$ a \emph{null-sequence} if $\lim\limits_{N\to\infty}\E\limits_{n\in[N]}|\omega(n)|^2 =0$. 
\begin{example}[Examples of nilsequences]\label{Ex: nilsequences}
    The nilsystems in Example \ref{Ex: nilsystems} give rise to the following three classes of nilsequences:
    \begin{enumerate}
        \item (Linear phases) $\psi(n) = e(\alpha n)$ for any $\alpha\in\T$;
        \item (Polynomial phases) $\psi(n) = e(p(n))$ for any $p\in\R[t]$;
        \item (Bracket/generalized polynomials) $\psi(n) = e(\alpha n\sfloor{\beta n})$.
    \end{enumerate}
\end{example}

\section{Hardy fields}\label{A: Hardy}
Let $B$ be the collection of equivalence classes (``germs'') of real valued functions defined on some halfline $(t_0,\infty)$ for $t_0\geq 0$, where we identify two functions that eventually agree.  A \emph{Hardy field} is a subfield of the ring $(B, +, \cdot)$ that is closed under differentiation~\cite{Hardy12}. We call $a\colon (t_0,\infty)\to\R$ a \emph{Hardy-field function} if it belongs to some Hardy field, in which case we call $(a(n))_n$ a \emph{Hardy sequence}.
 By abuse of notation, we speak of ``functions'' rather than ``germs'', understanding that all the operations defined and statements made for elements of $\CH$ are considered only for sufficiently large values of $t\in \mathbb{R}$. 
  We say that $a\in\CH$ has \emph{polynomial growth} if there exists $d>0$ such that $a(t)\ll t^d$.   
 Basic properties of Hardy-field functions can be found in \cite{Boshernitzan-equidistribution, Fr10, Kho}.

A typical example of a Hardy field is the Hardy field $\mathcal{LE}$ of \emph{logarithmico-exponential functions}, i.e., functions constructed using a finite combination of symbols $+, -, \times, \div$, $\exp,$ $\log$ acting on the real variable $t$ and on real constants. These include, for example,  the functions $t^b (\log{t})^c$, where $b,c\in \R$. For many results stated in this survey, we need the Hardy field to be closed under composition and compositional inverses (and occasionally also under variable shifts $t\mapsto t + c$). Since there exists a Hardy field satisfying these closure properties and containing $\CL\CE$, we fix such a field throughout the survey, and we let $\CH$ be the collection of Hardy functions of polynomial growth lying inside this field.
%Here, we just recall that by $\CH$ we denote the collection of all Hardy functions of polynomial growth lying in some fixed Hardy field that contains $\CL\CE$, all logarithmico-exponential functions, and enjoys a bunch of other nice properties (it is closed under compositions, compositional inverses, and variable shifts).
%{\em In this article, we focus on Hardy-field functions in $\mathcal{LE}$.} 
    % Our results also apply beyond
    % this setting, to more general Hardy-fields, provided they satisfy the mild assumptions in \cite[Section 2.2]{DKKST24}.
    % {In our statements, we also assume that the Hardy-field functions are defined on all of $\R_+$, but all arguments remain valid without modification if they are defined only on a halfline.}
    
  \bibliography{library}
\bibliographystyle{plain}

\end{document}